\documentclass[10pt]{article}
\usepackage[a4paper, total={6.5in, 10in}]{geometry}

\usepackage[normalem]{ulem} 

\usepackage{mathtools}
\mathtoolsset{showonlyrefs=true}

\usepackage[numbers]{natbib}
\usepackage{amsthm}
\usepackage{mathrsfs}
\usepackage[breaklinks,bookmarks=false]{hyperref}

\hypersetup{colorlinks, linkcolor=blue, citecolor=blue,
urlcolor=blue, plainpages=false, pdfwindowui=false,
pdfstartview={FitH}}

\usepackage{amsmath,amsfonts,cleveref,amssymb,color,enumitem,tensor,mathtools,graphicx,epstopdf,pgfplotstable,tikz,subcaption,xcolor,comment}
\mathtoolsset{showonlyrefs}
\usetikzlibrary{positioning, arrows.meta,calc}

\usepackage{algorithm}
\usepackage{algpseudocode}
\usepackage[only,llbracket,rrbracket,llparenthesis,rrparenthesis]{stmaryrd} 
\usepackage[textsize=tiny,color=blue!10!white]{todonotes}
\usepackage{bm}
\usepackage{adjustbox}
\usepackage{bm}

\usepackage[medium,compact]{titlesec}
\titleformat*{\section}{\large\bfseries}
\titleformat*{\subsection}{\bfseries}

\numberwithin{equation}{section}

\newtheorem{theorem}{Theorem}[section]{\bfseries}{\it}
\newtheorem{proposition}[theorem]{Proposition}{\bfseries}{\it}
\newtheorem{lemma}[theorem]{Lemma}{\bfseries}{\it}
\newtheorem{corollary}[theorem]{Corollary}{\bfseries}{\it}
{\bfseries}{\it}
\newtheorem{eg}{Example}[section]{\bfseries}{\rmfamily}
{\bfseries}{\it}
\theoremstyle{definition}
\newtheorem{remark}{Remark}[section]{\bfseries}{\rmfamily}
\newcommand{\proofbox}{\qed}

\newcommand{\R}{\mathbb{R}}

\newcommand{\jump}[1]{\llbracket #1 \rrbracket}

\newcommand{\p}{\partial}
\newcommand{\abs}[1]{\lvert#1\rvert} 
\newcommand{\tends}{\rightarrow}
\newcommand{\norm}[1]{\lVert#1\rVert}
\newcommand{\expectation}[1]{\mathbb{E}\left[#1\right]}
\newcommand{\Ip }{\mathcal{I}_{+}^k}
\newcommand{\In}{\mathcal{I}_{-}^k}
\newcommand{\qPikn}{\Pi_{-}^{k,N}}
\newcommand{\qPikp}{\Pi_{+}^{k,N}}

\newcommand{\Ipm}{\mathcal{I}_{\pm}}

\newcommand{\rewrite}[1]{}

\newcommand{\Mk}{M_k}

\newcommand{\ds}{{\mathrm{d}s}}
\newcommand{\LLspace}{{{L}^2(\Omega;\mathbb{R}^d)}}

\newcommand{\Lip}{{\rm Lip}}

\newcommand{\logLogSlopeTriangle}[5]
{
	
	\pgfplotsextra
	{
		\pgfkeysgetvalue{/pgfplots/xmin}{\xmin}
		\pgfkeysgetvalue{/pgfplots/xmax}{\xmax}
		\pgfkeysgetvalue{/pgfplots/ymin}{\ymin}
		\pgfkeysgetvalue{/pgfplots/ymax}{\ymax}
		
		\pgfmathsetmacro{\xArel}{#1}
		\pgfmathsetmacro{\yArel}{#3}
		\pgfmathsetmacro{\xBrel}{#1-#2}
		\pgfmathsetmacro{\yBrel}{\yArel}
		\pgfmathsetmacro{\xCrel}{\xArel}
		
		\pgfmathsetmacro{\lnxB}{\xmin*(1-(#1-#2))+\xmax*(#1-#2)} 
		\pgfmathsetmacro{\lnxA}{\xmin*(1-#1)+\xmax*#1} 
		\pgfmathsetmacro{\lnyA}{\ymin*(1-#3)+\ymax*#3} 
		\pgfmathsetmacro{\lnyC}{\lnyA+#4*(\lnxA-\lnxB)}
		\pgfmathsetmacro{\yCrel}{\lnyC-\ymin)/(\ymax-\ymin)} 
		
		\coordinate (A) at (rel axis cs:\xArel,\yArel);
		\coordinate (B) at (rel axis cs:\xBrel,\yBrel);
		\coordinate (C) at (rel axis cs:\xCrel,\yCrel);
		
		\draw[#5]   (A)-- node[pos=0.5,anchor=north] {1}
		(B)-- 
		(C)-- node[pos=0.5,anchor=west] {$1$}
		cycle;
	}
}

\newcommand{\logLogSlopeTrianglei}[5]
{
	
	\pgfplotsextra
	{
		\pgfkeysgetvalue{/pgfplots/xmin}{\xmin}
		\pgfkeysgetvalue{/pgfplots/xmax}{\xmax}
		\pgfkeysgetvalue{/pgfplots/ymin}{\ymin}
		\pgfkeysgetvalue{/pgfplots/ymax}{\ymax}
		
		\pgfmathsetmacro{\xArel}{#1}
		\pgfmathsetmacro{\yArel}{#3}
		\pgfmathsetmacro{\xBrel}{#1-#2}
		\pgfmathsetmacro{\yBrel}{\yArel}
		\pgfmathsetmacro{\xCrel}{\xArel}
		
		\pgfmathsetmacro{\lnxB}{\xmin*(1-(#1-#2))+\xmax*(#1-#2)} 
		\pgfmathsetmacro{\lnxA}{\xmin*(1-#1)+\xmax*#1} 
		\pgfmathsetmacro{\lnyA}{\ymin*(1-#3)+\ymax*#3} 
		\pgfmathsetmacro{\lnyC}{\lnyA+#4*(\lnxA-\lnxB)}
		\pgfmathsetmacro{\yCrel}{\lnyC-\ymin)/(\ymax-\ymin)} 
		
		\coordinate (A) at (rel axis cs:\xArel,\yArel);
		\coordinate (B) at (rel axis cs:\xBrel,\yBrel);
		\coordinate (C) at (rel axis cs:\xCrel,\yCrel);
		
		\draw[#5]   (A)-- node[pos=0.5,anchor=north] {1}
		(B)-- 
		(C)-- node[pos=0.5,anchor=west] {$\frac{1}{2}$}
		cycle;
	}
}

\newcommand{\logLogSlopeTriangleii}[5]
{
	
	\pgfplotsextra
	{
		\pgfkeysgetvalue{/pgfplots/xmin}{\xmin}
		\pgfkeysgetvalue{/pgfplots/xmax}{\xmax}
		\pgfkeysgetvalue{/pgfplots/ymin}{\ymin}
		\pgfkeysgetvalue{/pgfplots/ymax}{\ymax}
		
		\pgfmathsetmacro{\xArel}{#1}
		\pgfmathsetmacro{\yArel}{#3}
		\pgfmathsetmacro{\xBrel}{#1-#2}
		\pgfmathsetmacro{\yBrel}{\yArel}
		\pgfmathsetmacro{\xCrel}{\xArel}
		
		\pgfmathsetmacro{\lnxB}{\xmin*(1-(#1-#2))+\xmax*(#1-#2)} 
		\pgfmathsetmacro{\lnxA}{\xmin*(1-#1)+\xmax*#1} 
		\pgfmathsetmacro{\lnyA}{\ymin*(1-#3)+\ymax*#3} 
		\pgfmathsetmacro{\lnyC}{\lnyA+#4*(\lnxA-\lnxB)}
		\pgfmathsetmacro{\yCrel}{\lnyC-\ymin)/(\ymax-\ymin)} 
		
		\coordinate (A) at (rel axis cs:\xArel,\yArel);
		\coordinate (B) at (rel axis cs:\xBrel,\yBrel);
		\coordinate (C) at (rel axis cs:\xCrel,\yCrel);
		
		\draw[#5]   (A)-- node[pos=0.5,anchor=north] {\tiny 1}
		(B)-- 
		(C)-- node[pos=0.5,anchor=west] {\tiny $1/3$}
		cycle;
	}
}

\newcommand{\logLogSlopeTriangleiii}[5]
{
	
	\pgfplotsextra
	{
		\pgfkeysgetvalue{/pgfplots/xmin}{\xmin}
		\pgfkeysgetvalue{/pgfplots/xmax}{\xmax}
		\pgfkeysgetvalue{/pgfplots/ymin}{\ymin}
		\pgfkeysgetvalue{/pgfplots/ymax}{\ymax}
		
		\pgfmathsetmacro{\xArel}{#1}
		\pgfmathsetmacro{\yArel}{#3}
		\pgfmathsetmacro{\xBrel}{#1-#2}
		\pgfmathsetmacro{\yBrel}{\yArel}
		\pgfmathsetmacro{\xCrel}{\xArel}
		
		\pgfmathsetmacro{\lnxB}{\xmin*(1-(#1-#2))+\xmax*(#1-#2)} 
		\pgfmathsetmacro{\lnxA}{\xmin*(1-#1)+\xmax*#1} 
		\pgfmathsetmacro{\lnyA}{\ymin*(1-#3)+\ymax*#3} 
		\pgfmathsetmacro{\lnyC}{\lnyA+#4*(\lnxA-\lnxB)}
		\pgfmathsetmacro{\yCrel}{\lnyC-\ymin)/(\ymax-\ymin)} 
		
		\coordinate (A) at (rel axis cs:{1-\xArel},{1-\yArel});
\coordinate (B) at (rel axis cs:{1-\xBrel},{1-\yBrel});
\coordinate (C) at (rel axis cs:{1-\xCrel},{1-\yCrel});
		
		\draw[#5]   (A)-- node[pos=0.5,anchor=south] {\tiny 1}
		(B)-- 
		(C)-- node[pos=0.5,anchor=east] {\tiny $1/4$}
		cycle;
	}
}

\newcommand{\logLogSlopeTriangleiv}[5]
{
	
	\pgfplotsextra
	{
		\pgfkeysgetvalue{/pgfplots/xmin}{\xmin}
		\pgfkeysgetvalue{/pgfplots/xmax}{\xmax}
		\pgfkeysgetvalue{/pgfplots/ymin}{\ymin}
		\pgfkeysgetvalue{/pgfplots/ymax}{\ymax}
		
		\pgfmathsetmacro{\xArel}{#1}
		\pgfmathsetmacro{\yArel}{#3}
		\pgfmathsetmacro{\xBrel}{#1-#2}
		\pgfmathsetmacro{\yBrel}{\yArel}
		\pgfmathsetmacro{\xCrel}{\xArel}
		
		\pgfmathsetmacro{\lnxB}{\xmin*(1-(#1-#2))+\xmax*(#1-#2)} 
		\pgfmathsetmacro{\lnxA}{\xmin*(1-#1)+\xmax*#1} 
		\pgfmathsetmacro{\lnyA}{\ymin*(1-#3)+\ymax*#3} 
		\pgfmathsetmacro{\lnyC}{\lnyA+#4*(\lnxA-\lnxB)}
		\pgfmathsetmacro{\yCrel}{\lnyC-\ymin)/(\ymax-\ymin)} 
		
		\coordinate (A) at (rel axis cs:{1-\xArel},{1-\yArel});
\coordinate (B) at (rel axis cs:{1-\xBrel},{1-\yBrel});
\coordinate (C) at (rel axis cs:{1-\xCrel},{1-\yCrel});
		
		\draw[#5]   (A)-- node[pos=0.5,anchor=south] {\tiny 1}
		(B)-- 
		(C)-- node[pos=0.5,anchor=east] {\tiny $1/5$}
		cycle;
	}
}

\newcommand{\logLogSlopeTrianglev}[5]
{
	
	\pgfplotsextra
	{
		\pgfkeysgetvalue{/pgfplots/xmin}{\xmin}
		\pgfkeysgetvalue{/pgfplots/xmax}{\xmax}
		\pgfkeysgetvalue{/pgfplots/ymin}{\ymin}
		\pgfkeysgetvalue{/pgfplots/ymax}{\ymax}
		
		\pgfmathsetmacro{\xArel}{#1}
		\pgfmathsetmacro{\yArel}{#3}
		\pgfmathsetmacro{\xBrel}{#1-#2}
		\pgfmathsetmacro{\yBrel}{\yArel}
		\pgfmathsetmacro{\xCrel}{\xArel}
		
		\pgfmathsetmacro{\lnxB}{\xmin*(1-(#1-#2))+\xmax*(#1-#2)} 
		\pgfmathsetmacro{\lnxA}{\xmin*(1-#1)+\xmax*#1} 
		\pgfmathsetmacro{\lnyA}{\ymin*(1-#3)+\ymax*#3} 
		\pgfmathsetmacro{\lnyC}{\lnyA+#4*(\lnxA-\lnxB)}
		\pgfmathsetmacro{\yCrel}{\lnyC-\ymin)/(\ymax-\ymin)} 
		
		\coordinate (A) at (rel axis cs:{1-\xArel},{1-\yArel});
\coordinate (B) at (rel axis cs:{1-\xBrel},{1-\yBrel});
\coordinate (C) at (rel axis cs:{1-\xCrel},{1-\yCrel});
		
		\draw[#5]   (A)-- node[pos=0.5,anchor=south] {\tiny 1}
		(B)-- 
		(C)-- node[pos=0.5,anchor=east] {\tiny $1/6$}
		cycle;
	}
}

\newcommand{\drawcube}[6]{%
\begin{scope}[shift={(#1,0)}]
  \coordinate (FBL) at (0,0);    \coordinate (FBR) at (\s,0);
  \coordinate (FTR) at (\s,\s);   \coordinate (FTL) at (0,\s);
  \coordinate (BBR) at (\s+\d,\d); \coordinate (BTR) at (\s+\d,\s+\d);
  \coordinate (BTL) at (\d,\s+\d);

  \fill[gray!12] (FBL)--(FBR)--(FTR)--(FTL)--cycle;
  \fill[gray!22] (FTL)--(FTR)--(BTR)--(BTL)--cycle;
  \fill[gray!17] (FBR)--(FTR)--(BTR)--(BBR)--cycle;
  \draw (FBL)--(FBR)--(FTR)--(FTL)--cycle;
  \draw (FTL)--(BTL)--(BTR)--(FTR);
  \draw (FBR)--(BBR)--(BTR);

  \coordinate (topC)   at (\cc,\s+\d/2);   
  \coordinate (rightC) at (\s+\d/2,\cc);   
  \coordinate (leftC)  at (0,\cc);         
  \coordinate (botC)   at (\cc,0);         

  \draw[vin]   ($(topC)+(0,\al)$)   -- (topC)
        node[pos=0, above, text=black] {#2};
  \draw[horiz] ($(rightC)+(\al,0)$) -- (rightC)
        node[pos=0, above, text=black] {#3};
  \draw[horiz] (leftC) -- ($(leftC)+(-\al,0)$)
        node[pos=1, above,  text=black] {#4};
  \draw[vout]  (botC)  -- ($(botC)+(0,-\al)$)
        node[pos=1, right, text=black] {#5};

  \node[align=center, font=\small] at (\s/2,\s/2) {#6};
\end{scope}%
}

 \title{Numerical analysis of first-order mean field games under displacement monotonicity}
  \author{Alp\'ar R.\ M\'esz\'aros\footnotemark[2]~ and Yohance A.\ P.\ Osborne\footnotemark[2]}

\begin{document}

\maketitle

 \renewcommand{\thefootnote}{\fnsymbol{footnote}}

\footnotetext[2]{Department of Mathematical Sciences, Durham University, Stockton Road, DH1 3LE Durham, United Kingdom (\texttt{alpar.r.meszaros@durham.ac.uk}, \texttt{yohance.a.osborne@durham.ac.uk}).}

\begin{abstract}
We introduce a particle method for the numerical approximation of time-dependent first-order Mean Field Games (MFGs) systems with non-separable, displacement monotone Hamiltonians and terminal costs, for arbitrary time-horizons and (possibly) singular initial player distributions in $\mathcal{P}_2(\mathbb{R}^d)$. The numerical scheme is based on an implicit Euler discretization in time and sampling in space of the characteristic Hamiltonian/Pontryagin system associated with the continuous MFGs system. We prove convergence of the approximations of the player distribution in the $L^{\infty}(\mathcal{W}_2)$-metric and the approximations for the gradient of the value function along optimal trajectories in the $L^{\infty}{(L^2)}$-norm as the number of spatial samples tends to infinity jointly with the temporal time-step vanishing. The error bound that we establish for this convergence further implies rates of convergence of the scheme for a range of spatial sampling techniques. Provided that the Lagrangian and terminal costs are additionally locally Lipschitz continuous, we also establish an asymptotic error bound in the $L^{\infty}(L^1)$-norm for the approximations of the value function along optimal trajectories. This is the first work in the literature on rigorous numerical approximation and analysis of first-order MFG systems {that} handles non-separable Hamiltonians and potentially singular initial agent distributions for arbitrary long time horizons. We illustrate the performance of the scheme in numerical experiments for a range of {initial agent distributions,} time horizons and space dimensions. 
\end{abstract}

\setcounter{tocdepth}{1}
\tableofcontents
\vspace{0.5cm}

\section{Introduction}\label{sec-introduction}
Mean Field Games (MFGs) model Nash equilibria in (stochastic) differential games that involve a continuum of players. MFGs were initiated independently by Lasry--Lions in \cite{lasry2006jeuxI,lasry2006jeuxII,lasry2007mean} in the mathematics literature and by Huang--Caines--Malham\'e in \cite{huang2006large} in the engineering community. MFGs models are often described by a nonlinear system of PDEs in which a Hamilton--Jacobi--Bellman (HJB) equation --- for the value function of a representative agent --- is coupled with a Kolmogorov--Fokker--Planck (KFP) equation that describes the evolution of the players' distribution. 
For an introduction to the theory, numerical analysis, and applications of MFGs, we refer the reader to \cite{achdou2020mean,GomesSaude2014}.

A prototypical first-order MFGs PDE system takes the form
\begin{equation}\label{mfg-pde-sys}
\left\{
	\begin{array}{ll}
		- \partial_tu +H(x,-D_xu,\rho)=0, &\text{ in }(0,T)\times\mathbb{R}^d,
		\\[3pt]
		\partial_t\rho  + D_x\cdot\left(\rho{D_p H}(x,-D_xu,\rho)\right)=0, &\text{ in }(0,T)\times\mathbb{R}^d,
		\\[3pt]
		u(T,x)=g(x,\rho(T))\quad\text{and}\quad \rho(0,\cdot)=\rho_0,   &\text{ in }\mathbb{R}^d,
	\end{array}
\right.
\end{equation}
where $T>0$ is a fixed time-horizon, the first equation being the HJB equation for the value function $u$ of a typical agent and the second equation is the continuity equation (the `first order version' of the KFP equation) for the unknown $(\rho_t)_{t\in[0,T]}$, a flow of probability measures describing the time evolution of  the player population distribution. This system is equipped with an initial player distribution $\rho_0$. The forward-backward PDE system is nonlinearly coupled through both the Hamiltonian $H:\R^d\times\R^d\times\mathcal{P}_2(\R^d)\to\R$ and the terminal cost $g:\R^d\times\mathcal{P}_2(\R^d)\to\mathbb{R}$, associated with the underlying optimal control problem faced by the generic player. 

We would like to emphasise already at this point that Hamiltonian $H$ has in general a so-called \emph{non-separable} structure: the generalised momentum variable $p$ is not supposed to be additively separated from the the distribution $\rho$. More precisely, we do not suppose the existence of maps $H_0:\mathbb{R}^d\times\mathbb{R}^d\to\mathbb{R}$ and $F:\R^d\times\mathcal{P}_2(\mathbb{R}^d)\to \mathbb{R}$ such that the additive decomposition
\begin{equation}\label{sep-H-structure}
    H(x,p,\mu) = H_0(x,p) - F(x,\mu),\quad\forall (x,p,\mu)\in\mathbb{R}^d\times\mathbb{R}^d\times\mathcal{P}_2(\mathbb{R}^d)
\end{equation}
holds. Here we let $\mathcal{P}_2(\mathbb{R}^d)$ denote the metric space of Borel probability measures supported on $\mathbb{R}^d$, with finite second moments, equipped with the $\mathcal{W}_2$-Wasserstein distance (see Section \ref{sec-basic-notation-cont-pb} below). Non-separable MFGs systems appear naturally in applications, such as in macroeconomics \cite{AchdouEtAl2014} or pedestrian dynamics \cite{gomes2015short} for instance. 

\medskip

{\bf Brief literature review relevant to our work.}

\medskip

\emph{Analysis of MFGs systems.} The existence and uniqueness of solutions to MFGs systems attracted increased attention in the past two decades. When establishing existence of solutions, different techniques have been applied, depending on the presence of regularising effect coming from either non-degenerate parabolicity or regularisation effects of the coupling functions $H$ and $g$ with respect to the measure component. When non-degenerate noise is present there are by now well-paved PDE tools establishing existence of smooth or weak solutions both in the local or nonlocal frameworks. We refer to \cite{achdou2020mean} (and more particularly to \cite{CarPor}) for a detailed summary (as of around 2019 or so) of the various developments on these approaches, and to \cite{GomPimVos} for a regularity theory of classical solutions. Existence for first oder models, or for the ones which are degenerate parabolic, is much more intricate to study. Even for regularising couplings, in general one needs to assume further conditions, as regularity/summability/compact support of the initial measure and separability of the Hamiltonian in the sense of \eqref{sep-H-structure} (see for instance \cite[Section 1.3.4]{CarPor}); in the case of local couplings weak solutions can be produced using variational techniques (cf. \cite{Car,CarGra,CarMesSan,San, GriMes}). Remarkably, in the case of strong monotonicity properties such weak solutions can be proved to be more regular (cf. \cite{GraMes, GriMes}) or even smooth (using a joint `elliptic structure' in time-space, see \cite{Mun}).

When it comes to the uniqueness and stability of solutions, usually more structural assumptions need to be imposed on the data $H$ and $g$. These typically boil down to suitable {\it monotonicity conditions}. The most well-known regimes in this context are the so-called {\it Lasry--Lions (LL) monotonicity} and the {\it displacement (D) monotonicity} (for other possible monotonicity conditions and for their comparisons we refer to \cite{GraMes:23}). While for regularising couplings LL-monotonicity is known to be developed only in the case of separable Hamiltonians \eqref{sep-H-structure}, in the case of local couplings this has a suitable analog (see for instance \cite{AchPor:IHP}). D-monotonicity is known to be developed only in the case of non-local Hamiltonians, and in general by default this does not need to assume any separability on $H$ (cf. \cite{meszaros2024mean, GanMesMouZha}), and gives not only uniqueness, but also existence and stability in a unified manner for first and second order models with general initial measures (cf. \cite{GanMes, BanMes:FMS}). All this previous discussion is on results which are {\it global in time}.

Without monotonicity conditions well-posedness of MFGs systems involving general non-separable Hamiltonians from the point of view of classical solutions can be established at the price of suitable {\it smallness conditions} either on the time horizon, the size of the Hamiltonian, the initial measure or the combination of these. For a non-exhaustive list of works in these directions we refer to \cite{Amb:18, AmbMes, AmbGriMes}. 

\medskip

\emph{Numerical analysis of MFG systems.} The analysis of numerical methods for MFGs systems with separable Hamiltonians (i.e.\ where the decomposition \eqref{sep-H-structure} holds) is an active area of research. Early works on the convergence analysis of monotone finite difference methods appeared in \cite{achdou2010mean,achdou2013mean,achdou2016convergence}, with more recent contributions concerning the convergence of policy iteration \cite{cacace2020policy,camilli2022rates} and fictitious play schemes \cite{inoue2023fictitious, tang2024learning}, as well as the derivation of rates of convergence \cite{bonnans2022error}. The convergence  of semi-Lagrangian methods in one space dimension was shown in \cite{CarliniSilva14,CarliniSilve15} for the first-order and degenerate second-order systems, respectively. This resulted in further works, such as the extension of the convergence analysis of \cite{CarliniSilva14} to arbitrary space dimension in \cite{carlini2024lagrange} (see also \cite{hadikhanloo2019finite}), the convergence of a semi-Lagrangian scheme for a non-degenerate second-order system in \cite{carlini2018discretization}, and the approximation of monotone operator formulations of first-order systems in \cite{carlini2025semi}.  In addition, a tensor-train-based high-order semi-Lagrangian policy iteration scheme was introduced in \cite{carlini2025high} for separable MFG systems with high-dimensional state space. Other numerical methods have also been studied, such as augmented Lagrangian methods \cite{benamou2015augmented,andreev2017preconditioning,fu2023high}, kernel methods \cite{liu2021computational,nurbekyan2019fourier,agrawal2022random}, and approaches based on neural networks \cite{lin2021alternating}. 

The analysis and numerical approximation of MFGs systems with (possibly) nondifferentiable, but separable, Hamiltonians was recently initiated in \cite{osborne2022analysis,osborne2024erratum,osborne2023finite}, where in particular convergence of a class of monotone stabilized finite element methods was shown. Since then, finite element error bounds for such schemes have appeared, such as \emph{a priori} quasi-optimal error bounds in \cite{osborne2024near} and \emph{a posteriori} error bounds in \cite{osborne2025posteriori,smears2026timeapost}, each addressing the case of systems with $C^{1,1}$ Hamiltonians, as well as rates of convergence for MFGs systems with merely Lipschitz, nondifferentiable, separable Hamiltonians in \cite{osborne2025rates}. An alternative approach to error analysis based on Brezzi--Rappaz--Raviart approximation theory was introduced in \cite{berry2025approximation} for the finite element approximation of solutions to MFGs systems with $C^2$-separable Hamiltonians that admit multiple solutions; see also \cite{berry2025nonsmooth} for an extension of this approach to the case of $C^{1,1}$-separable Hamiltonians based on the notion of metric-regular mappings. 

The analysis of numerical schemes for second-order MFG systems in the case of non-separable Hamiltonians seems to have been considered only in \cite{lauriere2021policy,ASSOULI2023113802,camilli2025quadratic}. In particular, a policy iteration method was analyzed in \cite{lauriere2021policy} for a second-order non-separable system which was then applied to a finite difference discretization of the system to generate simulations. In \cite{ASSOULI2023113802} an algorithm was studied for a second-order non-separable system, wherein the HJB and KFP equations are each approximated by a deep neural network within an iteration procedure for the entire system. The quadratic convergence of a Newton iteration scheme in infinite-dimensional function spaces was shown in \cite{camilli2025quadratic} for a class of second-order non-separable MFG systems. Turning to the case of systems that are both first-order and non-separable, to the best of our knowledge, only numerical algorithms, without their analysis, have been introduced for solving a very particular MFGs system in the work \cite{mazanti2019minimal}. Therein, an iterative algorithm based on approximating the forward-backward Pontryagin system for the player trajectories was proposed, without analysis, to obtain simulations of minimal-time MFG models with initial distributions being Dirac measures supported on a single point. It appears therefore that the analysis of numerical methods for first-order, non-separable MFG systems has been largely unexplored {until our current work.}

\medskip

{\bf Our contributions.} 

\medskip

In this paper, we focus on the numerical approximation of the MFG system \eqref{mfg-pde-sys} that is both first-order and non-separable. As our main contribution, we develop a {\it particle-based} numerical method for the system \eqref{mfg-pde-sys} where the Hamiltonian $H$ and the terminal cost $g$ are in particular displacement monotone, and the initial player distribution $\rho_0$ is allowed to be an arbitrary measure $\mathcal{P}_2(\mathbb{R}^d)$. For this, we fix  a reference probability space $(\Omega,\mathbb{F},\mathbb{P})$ with $\mathbb{P}$ being non-atomic (to ensure than any $\rho\in\mathcal{P}_2(\R^d)$ can be realised as the law of a random variable $X\in L^2(\Omega;\mathbb{R}^d)$). Without loss of generality, $\Omega$ can be taken to be a compact set of a Euclidean space, with unit Lebesgue measure and $\mathbb{P}$ can be chosen to be the restriction of the Lebesgue measure to $\Omega$.

Suppose that $(\rho(s))_{s\in[0,T]}$ is the MFG equilibrium flow of probability measures and let $X_0\in L^2(\Omega;\mathbb{R}^d)$ be a given initial random variable distributed according to $\rho_0$. The starting point for the method is the deterministic forward-backward Hamiltonian type system 
\begin{equation}\label{characteristics-continuous}
	\left\{
	\begin{array}{ll}
		 X(t) &=X_0+\displaystyle \int_0^tD_pH\left(X(s),Y(s),\rho(s)\right)\mathrm{d}s, \quad\text{in }C([0,T];L^2(\Omega;\mathbb{R}^d)), 
		\\
		 Y(t) &\displaystyle= -D_xg\left(X(T),\rho(T)\right)+\int_t^TD_xH\left(X(s),Y(s),\rho(s)\right)\mathrm{d}s,\quad\text{in }C([0,T];L^2(\Omega;\mathbb{R}^d)), 
	\end{array}
	\right.
\end{equation}
together with the fixed point condition that the law of $X(s)$ is precisely $\rho_s$, i.e.
\begin{equation}\label{eq:FP}
X(s)_\#\mathbb{P} = \rho_s,\ \ \forall s\in[0,T]. 
\end{equation}
It was shown in \cite{meszaros2024mean} that under suitable regularity and D-monotonicity assumptions on $H,g$, this Hamiltonian system has a unique global in time solution pair $(X,Y)$ (in particular $X$ satisfes \eqref{eq:FP}) which is in a one-to-one correspondence with the unique solution $(u,\rho)$ to the MFG system \eqref{mfg-pde-sys}. In formal terms, this relationship is as follows: the momentum flow $Y\in C([0,T];L^2(\Omega;\mathbb{R}^d))$ satisfies $Y(t) = -D_xu(t,X(t))$ in $L^2(\Omega;\mathbb{R}^d)$, $t\in [0,T]$, which is the flow of the {negative of the} value function's gradient $-D_xu$ along the players' characteristic path $X\in C([0,T];L^2(\Omega;\mathbb{R}^d))$, that uniquely satisfies the forward equation of \eqref{characteristics-continuous} with $\rho(t) = (X(t))_{\#}\mathbb{P}$ in $\mathcal{P}_2(\mathbb{R}^d)$, $t\in [0,T]$ (see Section \ref{sec-basic-notation-cont-pb} for further details). While in this system there is no randomness coming from idiosyncratic or common noises, the role of the probability space $(\Omega,\mathbb{F},\mathbb{P})$ and the Lebesgue space $L^2(\Omega;\mathbb{R}^d)$ for the corresponding $L^2$-random variables is to model the dependence on the initial measure $\rho_0$ and its approximations in a robust way, and give a functional analytic setting, suitable for our analysis. We will benefit from this setting throughout the analysis later.

\medskip

In view of this connection between $(u,\rho)$ and $(X,Y)$,  we introduce a discretisation of the continuous Hamiltonian system \eqref{characteristics-continuous} as a numerical scheme to approximate both $\rho$ and $D_xu(t,X)$. More precisely, the numerical scheme is a discrete Hamiltonian system whose solutions $\{X^{k,N}\}_{k,N}$, $\{Y^{k,N}\}_{k,N}$ are piecewise constant in time and defined everywhere in $[0,T]$, with constant step-size $\tau_k>0$, $k\in\mathbb{N}$, whose values are simple random variables in $L^2(\Omega;\mathbb{R}^d)$ that assume $N$ distinct values in $\mathbb{R}^d$ for each $N\in\mathbb{N}$. The formulation of the numerical method, which we present in Section \ref{sec-numerical-discretization}, is structurally consistent with the continuous Hamiltonian system \eqref{characteristics-continuous} and employs merely deterministic empirical approximations of $\rho_0$, thereby allowing for (possibly) purely singular initial player distributions for the MFG model problem. {Here $k\in\mathbb{N}$ is the index of the time step $\tau_k$, while $N\in\mathbb{N}$ refers to the discretization level of the initial measure $\rho_0$.}

The main guiding principle of our analysis is the following observation: first, under D-monotonicity the HJB equation in \eqref{mfg-pde-sys} has a unique classical solution such that $D_xu(t,\cdot)\in C^{0,1}(\R^d)$;  second, this combined with the deterministic nature of the problem implies that if $\rho_0$ is a finitely supported measure, {$t\mapsto\rho_t$ will remain a finitely supported measure flow} throughout $t\in[0,T]$, and collisions of the Dirac masses/crossing of characteristics cannot occur. This precisely means that if $X_0$ is taken to be a simple random variable (i.e. with finite range), this will be maintained for the solution $(X_t,Y_t)_{t\in [0,T]}$ of the Hamiltonian system as well, throughout the time evolution. Our numerical approach is hence based on these crucial facts.

\medskip

Our main results concerning well-posedness for the scheme are given in Theorem \ref{theorem-existence-num-scheme} and Theorem \ref{theorem-uniqueness-numerical-sheme} where we prove, respectively, the existence and uniqueness of the approximations $\{(X^{k,N},Y^{k,N})\}_{{k,N}}$ for all $k\in\mathbb{N}$ sufficiently large and for all $N\in\mathbb{N}$. At the heart of the analysis is a continuous dependence bound that we prove in Theorem \ref{theorem-cont-dependence-discrete-Hamiltonian-sys} for discrete Hamiltonian systems under affine perturbations. Therein, we show the key property that the discrete Hamiltonian system remarkably {\it inherits} the D-monotonicity from the D-monotonicity property satisfied by the continuous Hamiltonian system \eqref{characteristics-continuous}, which ultimately implies {\it uniqueness and stability} of our discrete approximations. As such, we have developed a \emph{D-monotone} numerical scheme, in the sense of displacement monotonicity of the nonlinear couplings, that applies to the MFG system \eqref{mfg-pde-sys} even if the initial player distribution is not absolutely continuous with respect to Lebesgue measure or does not have compact support. To the best of our knowledge, this is the first use of  D-monotonicity in developing a well-posed numerical method for MFGs systems. 

\medskip

The next main result concerns the strong convergence of the numerical scheme in the joint limit as $k,N\to\infty$. In fact, we prove a stronger result in Theorem \ref{theorem-convergence-k-N-joint}, where we establish the asymptotic error bound 
\begin{align}\label{main-error-bound-num-scheme-intro}
	\nonumber\sup_{t\in [0,T]}\mathcal{W}_2\left(\rho(t),\rho^{k,N}(t)\right)&+\sup_{t\in [0,T]}\|(-D_xu(\cdot,X)-Y^{k,N})(t)\|_{L^2(\Omega;\mathbb{R}^d)}\\
    &+\sup_{t\in [0,T]}\|(X-X^{k,N})(t)\|_{L^2(\Omega;\mathbb{R}^d)}
			\leq C\left( \|X_0-X_0^N\|_{\LLspace}+ \tau_k\right)
\end{align}
for all $k,N\in\mathbb{N}$ with $k$ sufficiently large, where the constant $C>0$ depends only on parameters associated with the model data of the continuous problem. Here, $\rho^{k,N}(t)\in\mathcal{P}_2(\mathbb{R}^d)$ denotes the law of $X^{k,N}(t)$ for each $t\in [0,T]$ and $k,N\in\mathbb{N}$. The approximations $\{X_0^N\}_{N\in\mathbb{N}}$ for $X_0$ are a  given sequence of simple random variables that are distributed according to a corresponding sequence of deterministic finitely supported measures $\{\rho_0^N\}_{N\in\mathbb{N}}$ for $\rho_0$, where $\rho_0^N$ is the law of $X_0^N$ for each $N\in\mathbb{N}$ and $\{X_0^N\}_{N\in\mathbb{N}}$ converges strongly to $X_0$ in $L^2(\Omega;\mathbb{R}^d)$. Crucially, we show that the above error bound \eqref{main-error-bound-num-scheme-intro} is a consequence of the continuous dependence bound in Theorem \ref{theorem-cont-dependence-discrete-Hamiltonian-sys} for discrete Hamiltonian systems and the $W^{1,\infty}$-in-time regularity of the solution $(X,Y)$ to the continuous Hamiltonian system \eqref{characteristics-continuous}. If both the terminal cost $g$ and the Lagrangian $L(x,v,\mu)\coloneqq \sup_{p\in\mathbb{R}^d}\{p\cdot v-H(x,p,\mu)\}$ are additionally locally Lipschitz continuous, then we further prove in Theorem \ref{theorem-convergence-k-N-joint-u-val} that the discrete value function  
\begin{align}\label{def:u-N}
u^{k,N}(t)\coloneqq \int_t^TL(X^{k,N}(s),{D_pH(X^{k,N}(s), Y^{k,N}(s),\rho^{k,N}(s))},\rho^{k,N}(s))\mathrm{d}s + g(X^{k,N}(T),\rho^{k,N}(T)),\ \forall t\in [0,T],
\end{align}
constructed from the approximations $\{(X^{k,N},Y^{k,N},\rho^{k,N})\}_{k,N}$, satisfies the error bound
$$\sup_{t\in [0,T]}\|(u(\cdot,X)-u^{k,N})(t)\|_{L^1(\Omega;\mathbb{R}^d)}
			\lesssim \|X_0-X_0^N\|_{\LLspace}+ \tau_k$$ for all $k,N\in\mathbb{N}$ both sufficiently large.

\medskip

We then obtain two corollaries on deriving explicit rates of convergence for the scheme according to the choice of the piecewise constant spatial approximations $\{X_0^N\}_{N\in\mathbb{N}}$. For example, if $\{X_0^N\}_{N\in\mathbb{N}}$ is generated through optimal transport maps associated with the semi-discrete {optimal transport problem for} $\mathcal{W}_2(\rho_0,\rho_0^N)$, for a suitable class of empirical measures $\{\rho_0^N\}_{N\in\mathbb{N}}$, then Corollary \ref{cor-1-eg-rate-of-conv} gives the error bound
\begin{multline}
    \label{cor-1-eg-rate-of-conv-intro}
		\sup_{t\in [0,T]}\mathcal{W}_2\left(\rho(t),\rho^{k,N}(t)\right)+\sup_{t\in [0,T]}\|(-D_xu(\cdot,X)-Y^{k,N})(t)\|_{L^2(\Omega;\mathbb{R}^d)}\\+\sup_{t\in [0,T]}\|(X-X^{k,N})(t)\|_{L^2(\Omega;\mathbb{R}^d)}
		\leq C
		\begin{dcases}
			\max\{N^{-1/d},  N^{-1/2 + 1/q}\} +\tau_k, & \text{if } d \ne \frac{2q}{q-2},\\
			N^{-1/d} \log(1+N)^{1/d} + \tau_k, & \text{if } d = \frac{2q}{q-2},
		\end{dcases} 
\end{multline}
for all $N,k\in\mathbb{N}$ with $k$ sufficiently large, provided $\rho_0\in \mathcal{P}_q(\mathbb{R}^d)$, $q>2$, is absolutely continuous with respect to the $d$-dimensional Lebesgue measure. Alternatively, if $\{X_0^N\}_{N\in\mathbb{N}}$ is obtained through simple  vector quantization of $X_0$, Corollary \ref{cor-2-eg-rate-of-conv} gives the error bound
\begin{multline}\label{cor-2-eg-rate-of-conv-intro}
		\sup_{t\in [0,T]}\mathcal{W}_2\left(\rho(t),\rho^{k,N}(t)\right)+\sup_{t\in [0,T]}\|(-D_xu(\cdot,X)-Y^{k,N})(t)\|_{L^2(\Omega;\mathbb{R}^d)}\\+\sup_{t\in [0,T]}\|(X-X^{k,N})(t)\|_{L^2(\Omega;\mathbb{R}^d)}
		\leq C\left( N^{-1/d}+\tau_k\right),
\end{multline}
for all $N,k\in\mathbb{N}$ with $k$ sufficiently large, provided that $\rho_0$ has a non-trivial absolutely continuous part and $X_0\in L^{2+\delta}(\Omega;\mathbb{R}^d)$ for some $\delta>0$, or equivalently $\rho_0$ has finite $(2+\delta)$-moment, for some $\delta>0$.  In the case where $L$ and $g$ are additionally locally Lipschitz continuous, we deduce corresponding error bounds in $L^{\infty}(L^1)$-norm for the discrete value function approximation $u^{k,N}$ that inherit the above asymptotic convergence rates.

We conclude this work with numerical experiments that illustrate the performance of our numerical scheme. In particular, in Section \ref{sec-num-experi-algorithms} we propose two iterative algorithms that approximate the nonlinear discrete forward-backward Hamiltonian system satisfied by $(X^{k,N},Y^{k,N})$ for given $k,N\in\mathbb{N}$. The first method, Algorithm \ref{alg:standard_picard_iteration}, is based on a standard Picard iteration globally in time. Motivated by the fact that the convergence of such Picard iterations requires that the time horizon $T$ is sufficiently small, we propose an alternative approach, namely Algorithm \ref{alg:local_picard_iteration}, to approximate the discrete problem for longer time horizons. Our approach is inspired by the work of Chassagneux--Crisan--Delarue \cite{chassagneux2019numerical} in which they employ local-in-time Picard iterations to obtain pointwise approximations of solutions to a class of master equations posed on Wasserstein space, where the approximations are shown to remain stable independently of the size of the time horizon. In Algorithm \ref{alg:local_picard_iteration}, we generate local-in-time Picard iterations within a global-in-time Picard fixed-point scheme to approximate the nonlinear discrete Hamiltonian system. Here we find yet another application of the general continuous dependence bound of Theorem \ref{theorem-cont-dependence-discrete-Hamiltonian-sys}, which we use to establish an \emph{a posteriori} error bound for the iteration scheme given by Algorithm \ref{alg:local_picard_iteration}, for the simplified scenario where each of the local-in-time Picard iterations return the exact discrete solution within one iteration. 

\medskip

The numerical experiments involve examples of MFG systems where we  use reference solutions or have access to an exact solution, {and where the initial agent distribution} has compact support or not. In particular, the first two experiments showcase that both Algorithm \ref{alg:standard_picard_iteration} and Algorithm \ref{alg:local_picard_iteration} can approximate the system equally well (in the sense of having nearly identical relative errors) when $T$ is not large (such as $T=1$), but that Algorithm \ref{alg:local_picard_iteration} can be used to solve the discrete system when $T$ assumes larger values such as {$T\in\{2,4,8,16,32\}$} while preserving theoretically predicted asymptotic convergence rates independently of $T$. These experiments hence showcase robustness of our numerical method for longer time horizons. For the final numerical experiment we illustrate the space-dimension dependence of the asymptotic rate of convergence for the numerical scheme when applied to an MFG with absolutely continuous initial measure supported {everywhere} on $\mathbb{R}^d$ with $d\in\{1,2,3,4,5,6\}$. By initializing the numerical scheme using vector quantizers for $X_0$, in this test we observe an asymptotic convergence rate of order $1/d$, which is consistent with the conclusion of Corollary \ref{cor-2-eg-rate-of-conv}. {To the best of our knowledge, our method is the first to approximate first-order, non-separable MFG systems with theoretically predicted rates of convergence.}

\medskip
\medskip

{\bf Outline.} 

\medskip

The rest of the manuscript is structured as follows. We set basic notation and model assumptions for the MFG system \eqref{mfg-pde-sys} in Section \ref{sec-basic-notation-cont-pb}. This is followed by the formulation of the continuous Hamiltonian system in Section \ref{sec-basic-notation-cont-pb}, where in particular we discuss its connection with the MFG system \eqref{mfg-pde-sys} in terms of continuous dependence with respect to discrete approximations of the initial data. We then dedicate Section \ref{sec-numerical-discretization} to introducing notation for a suite of vector spaces consisting of Banach space-valued maps that are piecewise constant in time. We employ such vector spaces in Section \ref{sec-numerical-discretization} where we formulate our discrete numerical scheme for the continuous Hamiltonian system \eqref{characteristics-continuous}. In Section \ref{sec-main-results}, we present the main results on well-posedness and convergence, as well as corollaries on rates of convergence, for the numerical scheme. The proof of the key continuous dependence bound for discrete Hamiltonian systems under affine perturbations is the subject of Section \ref{sec-cont-dep-gen-discrete-ham-sys}. We then employ this bound to prove the existence of discrete approximations in Section \ref{sec-existence} and the convergence of such approximations in Section \ref{sec-convergence}. Algorithms for solving the numerical scheme are given in Section \ref{sec-num-experi-algorithms}, which is then followed by the results of three numerical experiments therein. We include proofs of some auxiliary results in Appendix \ref{app-aux-disc-temp-ana}. 

\section{Setting and Continuous Problem}\label{sec-basic-notation-cont-pb}
\subsection*{Basic notation.}
Given $q\geq 1$, let $\mathcal{P}_q(\mathbb{R}^d)$ denote the set of Borel probability measures {supported on Boreal measurable subsets of $\R^d$} with finite $q$-moment, and for $\mu\in \mathcal{P}_q(\mathbb{R}^d)$ we denote its $q$-moment by $M_q(\mu)\coloneqq \left(\int_{\mathbb{R}^d}|x|^q\mathrm{d}\mu(x)\right)^{\frac{1}{q}}$. Let  $(\Omega,\mathbb{F},\mathbb{P})$ denote a {non-atomic} reference probability space. {Without loss of generality, it is convenient for us to choose} $\Omega\coloneqq [0,1]^d$ to be the unit cube, $\mathbb{F}$ the $\sigma$-algebra of Lebesgue measurable subsets of $\Omega$ and $\mathbb{P}$ the Lebesgue measure {restricted to $\Omega$}. For $q\in [1,\infty)$, we denote by $L^q(\Omega;\mathbb{R}^d)$ the space of $\mathbb{R}^d$-valued random variables $X$ on $\Omega$ that are Lebesgue measurable  and $\int_{\Omega}|X(\omega)|^q\mathrm{d}\mathbb{P}(\omega)<\infty$. We equip $L^q(\Omega;\mathbb{R}^d)$ with the standard norm denoted by $\|\cdot\|_{L^q(\Omega;\mathbb{R}^d)}$. For the special case when $q=2$, we equip $\LLspace$ with the standard inner product that is induced by expectation:
$\mathbb{E}[V\cdot W]\coloneqq \int_{\Omega}V\cdot W\mathrm{d}\mathbb{P}$ for $V,W\in \LLspace$, with the corresponding norm denoted by  $\|\cdot\|_{{L}^2(\Omega;\mathbb{R}^d)}$. As such,  $\LLspace$ is a Hilbert space. Given $q\geq 1$, $\mu\in \mathcal{P}_q(\mathbb{R}^d)$ and $\xi\in L^q(\Omega;\mathbb{R}^d)$, we say that $\mu$ is the law of $\xi$ if $\mathcal{L}_{\xi}\coloneqq \xi_{\#}\mathbb{P}=\mu$. Here, the push-forward notation $\xi_{\#}\mathbb{P}=
\mu$ means $\mu(A)=\mathbb{P}( \xi^{-1}(A))$ for any Borel set $A\subseteq\mathbb{R}^d$. For $1\leq q <\infty$, we equip $\mathcal{P}_q(\mathbb{R}^d)$ with the $\mathcal{W}_q$-Wasserstein metric, which {for all $\mu, \nu\in \mathcal{P}_q(\mathbb{R}^d)$} is defined as
\begin{equation}\label{W2-distance}
\mathcal{W}_q(\mu,\nu)\coloneqq \inf\left\{\left(\mathbb{E}[|\xi-\eta|^q]\right)^{\frac{1}{q}}:\ \ \xi,\eta\in L^q(\Omega;\R^d),\ {\rm with}\  \mu = \mathcal{L}_{\xi},\nu = \mathcal{L}_{\eta}\right\}.
\end{equation}

\subsection*{Vector spaces generated by simple random variables.} 

An $\mathbb{R}^d$-valued simple random variable $\xi$ is a function in $L^2(\Omega;\mathbb{R}^d)$ such that, for some finite $M\in\mathbb{N}$, there exist sets of (nontrivial) weights $\mathcal{A}\coloneq\{\alpha_{j}\}_{j=1}^{M}\subset (0,1]$, points $\{x_{j}\}_{j=1}^{M}\subset\mathbb{R}^d$ and subsets $\mathcal{O}\coloneq \{\Omega_{j}\}_{j=1}^{M}\subset \mathbb{F}$ such that $\xi\coloneqq \sum_{j=1}^{M}x_{j}\chi_{\Omega_{j}}$ where $\sum_{j=1}^{M}\alpha_{j}=1$, $\cup_{j=1}^{M}\overline{\Omega_{j}}=\Omega$, $\mathbb{P}(\Omega_{j})=\alpha_{j}$ for $j\in \{1,\cdots, {M}\}$, and $\Omega_{j}\cap\Omega_{i}=\emptyset$ if $i\neq j$ in $\{1,\cdots,{M}\}$. {In this case, we have that $\mathcal{L}_\xi = \sum_{j=1}^{M}\alpha_{j}\delta_{x_{j}}$.} Furthermore, it is clear that $
\xi$ lives in the finite-dimensional vector space of all $\mathbb{R}^d$-valued simple random variables on $(\Omega,\mathbb{F},\mathbb{P})$ of the form $W = \sum_{j=1}^{M}w_j\chi_{\Omega_{j}}$ such that $\{w_j\}_{j=1}^{M}\subset\mathbb{R}^d$. We let $\mathcal{R}_{d}(\mathcal{A},\mathcal{O})$ denote this ${M}$-dimensional vector space generated by $\xi$ which consists of all $\mathbb{R}^d$-valued simple random variables supported on the subsets $\mathcal{O}$ with corresponding weights $\mathcal{A}$. Notice that for each $V,W\in \mathcal{R}_{d}(\mathcal{A},\mathcal{O})$ the scalar product $V\cdot W$ yields an $\mathbb{R}$-valued simple random variable on  $(\Omega,\mathbb{F},\mathbb{P})$ where $V\cdot W =\sum_{i=1}^{M}v_i\cdot w_i\chi_{\Omega_{i}}$
and taking expectation with respect to $\mathbb{P}$ gives $\mathbb{E}\left[V\cdot W\right]=\sum_{i=1}^{M}\alpha_{i}v_i\cdot w_i$. As such, the space $\mathcal{R}_{d}(\mathcal{A},\mathcal{O})$ generated by $\xi$, when equipped with the norm $\|\cdot\|_{L^2(\Omega;\mathbb{R}^d)}$, is a finite dimensional subspace of $  L^2(\Omega;\mathbb{R}^d)$.

\subsection*{{Standing} assumptions.}
We let $H: \mathbb{R}^d\times\mathbb{R}^d\times\mathcal{P}_2(\mathbb{R}^d)\to\mathbb{R}$ denote the Hamiltonian of the MFG system \eqref{mfg-pde-sys}. We assume that $H$ is continuous, {where the continuity is with respect to $\mathcal{W}_1$ in the measure variable, and which is locally uniform with respect to the $(x,p)$ variable. Furthermore, $H$}
satisfies 
\begin{equation}\label{ass-H:1}
 \bullet\ \  H(\cdot,\cdot,\mu)\in C^2(\mathbb{R}^d\times\mathbb{R}^d),\ \forall \mu\in\mathcal{P}_2(\mathbb{R}^d),
\end{equation}
\begin{equation}\label{ass-H:2}
	\hspace{15pt} \bullet\ \ D_pH,D_xH\in C^{0,1}( \mathbb{R}^d\times\mathbb{R}^d\times\mathcal{P}_2(\mathbb{R}^d);\mathbb{R}^d),
\end{equation}
where the Lipschitz continuity is understood with respect to $\mathcal{W}_1$ in the measure variable. Note that this Lipschitz continuity then also holds with respect to the weaker $\mathcal{W}_q$ metric  in the measure variable for any $q\in [1,\infty)$.

It then follows that there exist constants $C_{D_pH}> 0$ and $ C_{D_xH}> 0$ such that
\begin{subequations}\label{linear-growth-D_pH-D_xH}
	\begin{align}
		&|D_pH(x,p,\mu)|\leq C_{D_pH}\left(|x|+|p|+{\mathcal{W}_q}(\mu,\delta_0)+1\right)\label{linear-growth-D_pH}
		\\
		&|D_xH(x,p,\mu)|\leq C_{D_xH}\left(|x|+|p|+{\mathcal{W}_q}(\mu,\delta_0)+1\right)\label{linear-growth-D_xH}
	\end{align}
\end{subequations}
for all $(x,p,\mu)\in \mathbb{R}^d\times\mathbb{R}^d\times\mathcal{P}_2(\mathbb{R}^d)$ and all $q\in [1,\infty)$. We also assume that
\begin{equation}\label{ass-H:3}
	\bullet\ \ \R^d\ni p\mapsto H(x,p,\mu) \text{ is strongly convex\ } \text{uniformly in }(x,\mu)\in \mathbb{R}^d\times\mathcal{P}_2(\mathbb{R}^d), 
\end{equation} i.e. there exists $c_0>0$ such that  
\begin{equation}\label{ass-H:4}
	D_{pp}^2H\geq c_0I_d\quad \text{on}\quad\mathbb{R}^d\times\mathbb{R}^d\times\mathcal{P}_2(\mathbb{R}^d).
\end{equation} We assume that 
\begin{equation}\label{ass-H:5}
    |D_pH(0,p,\mu)|\leq C(p)\quad \forall (p,\mu)\in \mathbb{R}^d\times \mathcal{P}_2(\mathbb{R}^d),    
\end{equation}
for some positive function $\mathbb{R}^d\ni p\mapsto C(p)>0$ that is independent of $\mu$. 

We further assume that the Hamiltonian $H$ is \emph{displacement monotone} in the sense that
\begin{equation}\label{ass-H:disp-mono-cond-H}
	\begin{split}
		&\bullet\ \ \mathbb{E}\Bigl[\left(D_pH(X_1,Y_1,\mathcal{L}_{X_1}) - D_pH(X_2,Y_2,\mathcal{L}_{X_2})\right)\cdot (Y_1-Y_2)
		\\
		&\qquad\qquad-\left(D_xH(X_1,Y_1,\mathcal{L}_{X_1}) - D_xH(X_2,Y_2,\mathcal{L}_{X_2})\right)\cdot (X_1-X_2)\Bigr]\geq 0
	\end{split}
\end{equation}
for any $X_i,Y_i\in \LLspace$, $i\in\{1,2\}$.

The terminal cost $g: \mathbb{R}^d\times\mathcal{P}_2(\mathbb{R}^d)\to\mathbb{R}$ of the MFG system \eqref{mfg-pde-sys} is assumed to be continuous, {where the continuity is with respect to $\mathcal{W}_1$ in the measure variable. Furthermore, $g$ is supposed to satisfy}
\begin{equation}\label{ass-g:1}
	\hspace{5pt}\bullet\ \  g(\cdot,\mu)\in C^2(\mathbb{R}^d),\ \forall \mu\in\mathcal{P}_2(\mathbb{R}^d),
\end{equation}
\begin{equation}\label{ass-g:2}
	\bullet\ \ D_xg\in C^{0,1}(\mathbb{R}^d\times\mathcal{P}_2(\mathbb{R}^d);\mathbb{R}^d),
\end{equation} 
where the Lipschitz continuity is with respect to $\mathcal{W}_1$ in the measure variable. We note then that this Lipschitz continuity also holds with respect to the weaker $\mathcal{W}_q$ metric  in the measure variable for any $q\in [1,\infty)$. As a consequence, 
there exists a constant $C_{D_xg}>0$ such that
\begin{equation}\label{linear-growth-D_xg}
|D_xg(x,\mu)|\leq C_{D_xg}\left(|x|+\mathcal{W}_q(\mu,\delta_0)+1\right),\quad\forall (x,\mu)\in \mathbb{R}^d\times\mathcal{P}_2(\mathbb{R}^d),
\end{equation}
for all $q\in[1,\infty)$. We also assume that $g$ satisfies the {\it displacement monotonicity} condition, in the sense that
\begin{equation}\label{ass-g:disp-mono-cond-g}
	\bullet\ \ \mathbb{E}\left[\left(D_xg({X}_1,\mathcal{L}_{X_1})-D_xg({X}_2,\mathcal{L}_{X_2})\right)\cdot ({X}_1-{X}_2)\right]\geq 0
\end{equation}
for any ${X}_1,X_2\in \LLspace$. 
It follows that $\R^d\ni x\mapsto g(x,\mu)$ is convex, uniformly in $\mu\in\mathcal{P}_2(\mathbb{R}^d)$ (cf. \cite{meszaros2024mean}). 

Finally, let $L:\mathbb{R}^d\times\mathbb{R}^d\times\mathcal{P}_2(\mathbb{R}^d)\to \mathbb{R}$ be the Lagrangian function associated to $H$, i.e. 
$$L(x,v,\mu)\coloneqq \sup_{p\in\mathbb{R}^d}\left\{p\cdot v - H(x,p,\mu)\right\},$$ for
all $x,v\in\mathbb{R}^d$ and $\mu\in\mathcal{P}_2(\mathbb{R}^d)$. It follows from standard convex analysis that the regularity \eqref{ass-H:1} and uniform Lipschitz continuity of $D_pH$ and $D_xH$ via \eqref{ass-H:2}, together with strong convexity of $H$ with respect to \ $p$ via \eqref{ass-H:3} and \eqref{ass-H:4}, altogether imply  
\begin{equation}\label{property-L:1}
 \bullet\ \  L(\cdot,\cdot,\mu)\in C^2(\mathbb{R}^d\times\mathbb{R}^d),\ \forall \mu\in\mathcal{P}_2(\mathbb{R}^d),
\end{equation}
\begin{equation}\label{property-L:2}
	\hspace{15pt} \bullet\ \ D_vL,D_xL\in C^{0,1}( \mathbb{R}^d\times\mathbb{R}^d\times\mathcal{P}_2(\mathbb{R}^d);\mathbb{R}^d),
\end{equation}where the Lipschitz continuity is understood with respect to $\mathcal{W}_1$ in the measure variable, and which hence also holds with respect $\mathcal{W}_q$ for any $q\in [1,\infty)$. We note that the uniform Lipschitz continuity of $D_xL$ and $D_vL$ imply that there exist constants $C_{D_vL}> 0$ and $ C_{D_xL}> 0$ such that
\begin{subequations}\label{linear-growth-D_qL-D_xL}
	\begin{align}
		&|D_vL(x,v,\mu)|\leq C_{D_vL}\left(|x|+|v|+{\mathcal{W}_q}(\mu,\delta_0)+1\right)\label{linear-growth-D_qL}
		\\
		&|D_xL(x,v,\mu)|\leq C_{D_xL}\left(|x|+|v|+{\mathcal{W}_q}(\mu,\delta_0)+1\right)\label{linear-growth-D_xL}
	\end{align}
\end{subequations}
for all $(x,v,\mu)\in \mathbb{R}^d\times\mathbb{R}^d\times\mathcal{P}_2(\mathbb{R}^d)$ and all $q\in [1,\infty)$.
Throughout, we assume that 
\begin{equation}\label{ass-L-g}
    L(x,v,\mu)\geq \theta_1(|v|) - \theta_2(\mu)(|x|+1),\quad g(x,\mu)\geq -\theta_2(\mu),\quad \forall (x,q,\mu)\in\mathbb{R}^d\times\mathbb{R}^d\times\mathcal{P}_2(\mathbb{R}^d),
\end{equation} for some superlinear function $\theta_1: [0, \infty) \to [0, \infty)$ and some function $\theta_2 : \mathcal{P}_2(\mathbb{R}^d) \to [0, \infty)$ which is bounded in $\{\mu \in \mathcal{P}_2(\mathbb{R}^d) : M_2(\mu) \leq R\}$, for any $R > 0$.

\medskip

{Throughout this paper, the initial agent distribution is taken to be a general measure $\rho_0\in\mathcal{P}_2(\R^d)$, without any further assumptions, unless specified in particular examples.}

\medskip
\medskip

	For notational purposes, for real numbers $a$ and $b$, we write $a\lesssim b$ if there exists a constant $C$ such that $a\leq Cb$ where $C>0$ {(naturally, independent of $a$ and $b$)}  depends only on the problem data $c_0$, $T$, $\Lip(D_pH)$, $\Lip(D_xH)$, $\Lip(D_vL)$, $\Lip(D_xL)$, $\Lip(D_xg)$, but is otherwise independent of all temporal and spatial discretisation parameters. We write $a \eqsim b$ if $a\lesssim b$ and $b\lesssim a$.

\begin{eg}\label{eg-1}
	Let $C_1,C_2>0$ be constants {(whose values will be specified)}. Let $\psi,\nu\in C^{0,1}(\mathbb{R}^d)\cap L^{\infty}(\mathbb{R}^d)$ both be non-negative functions. Let $\phi\in C^2(\mathbb{R}^d\times\mathbb{R}^d)$ and $\upsilon\in C^2(\mathbb{R}^d)$ be such that $p\mapsto\phi(x,p)$ is convex, uniformly in $x\in\mathbb{R}^d$,  $x\mapsto\phi(x,p)$ is concave, uniformly in $p\in\mathbb{R}^d$, $\upsilon$ is convex, and each of the derivatives $D_p\phi$, $D_x\phi$, and $D_x\upsilon$ are uniformly bounded and Lipschitz continuous on $\mathbb{R}^d\times\R^d$ and $\R^d$, respectively. Consider the Hamiltonian  $H: \mathbb{R}^d\times\mathbb{R}^d\times\mathcal{P}_2(\mathbb{R}^d)\to\mathbb{R}$ given by 
    \begin{equation}
        H(x,p,\mu)\coloneqq \frac{C_1}{2}\left(|p|^2 - |x|^2\right)+\phi(x,p)\int_{\mathbb{R}^d}\psi(z)d\mu(z)\quad \forall (x,p,\mu)\in \mathbb{R}^d\times\mathbb{R}^d\times\mathcal{P}_2(\mathbb{R}^d)
    \end{equation}
    and the terminal cost $g: \mathbb{R}^d\times\mathcal{P}_2(\mathbb{R}^d)\to\mathbb{R}$ of the form
    \begin{equation}
        g(x,\mu)\coloneqq \frac{C_2}{2}|x|^2 + \upsilon(x)\int_{\mathbb{R}^d}\nu(z)d\mu(z)\quad\forall (x,\mu)\in \mathbb{R}^d\times\mathbb{R}^d\times\mathcal{P}_2(\mathbb{R}^d).
    \end{equation}
    By a direct calculation, one can show that the above $H$ satisfies the regularity assumptions \eqref{ass-H:1}, \eqref{ass-H:2}, \eqref{ass-H:3}, \eqref{ass-H:4}, \eqref{ass-H:5} and displacement monotonicity condition \eqref{ass-H:disp-mono-cond-H} for $C_1$ satisfying $$ C_1\geq 2\|\psi\|_{L^{\infty}(\mathbb{R}^d)}(\Lip(D_p\phi)+\Lip(D_x\phi))+2\Lip(\psi)(\|D_x\phi\|_{L^{\infty}(\mathbb{R}^d\times\mathbb{R}^d;\mathbb{R}^d)}+\|D_p\phi\|_{L^{\infty}(\mathbb{R}^d\times\mathbb{R}^d;\mathbb{R}^d)}).$$ {Similarly, one can show } that $g$ satisfies the regularity assumptions \eqref{ass-g:1}, \eqref{ass-g:2}, and the displacement monotonicity condition \eqref{ass-g:disp-mono-cond-g} for $C_2\geq \Lip(D_x\upsilon)\|\nu\|_{L^{\infty}(\mathbb{R}^d)}+\Lip(\nu)\|D_x\upsilon\|_{L^{\infty}(\mathbb{R}^d)}$. It can be shown that the constants $C_1,C_2$ can be taken to be sufficiently large so that, in addition, the growth conditions \eqref{ass-L-g} on the Lagrangian $L$ and $g$ hold with $\theta_1(r)\coloneqq \frac{r^2}{2C_1}$, $r\geq 0$, and $\theta_2\coloneqq C_*$ on $\mathcal{P}_2(\mathbb{R}^d)$ with constant $C_*$ sufficiently large that depends only on $\|D_p\phi\|_{L^{\infty}(\mathbb{R}^d\times\mathbb{R}^d;\mathbb{R}^d)}$, $\|D_x\phi\|_{L^{\infty}(\mathbb{R}^d\times\mathbb{R}^d;\mathbb{R}^d)}$, $\|D_x\upsilon\|_{L^{\infty}(\mathbb{R}^d;\mathbb{R}^d)}$, $\|\psi\|_{L^{\infty}(\mathbb{R}^d)}$, and $\|\nu\|_{L^{\infty}(\mathbb{R}^d)}$.
\end{eg}


\subsection*{The Continuous problem: existence and uniqueness of solutions.}
Let $X_0\in L^2(\Omega;\mathbb{R}^d)$ be a random variable such that $\rho_0=\mathcal{L}_{X_0}$. We consider solutions to the Hamiltonian system \eqref{characteristics-continuous} in the following sense: \emph{find $(X,Y,\rho)\in C^1([0,T];L^2(\Omega;\mathbb{R}^d))\times C^1([0,T];L^2(\Omega;\mathbb{R}^d))\times C([0,T];\mathcal{P}_2(\mathbb{R}^d))$ such that}
\begin{equation}\label{definition-soln-of-continuous-pb}
\left\{
	\begin{array}{lll}
        X(t) &\displaystyle =X_0+ \int_0^tD_pH\left(X(s),Y(s),\rho(s)\right)\mathrm{d}s, \qquad&\text{in }[0,T]
		\\[3pt]
		Y(t) &= \displaystyle-D_xg\left(X(T),\rho(T)\right)+\int_t^TD_xH\left(X(s),Y(s),\rho(s)\right)\mathrm{d}s, \qquad&\text{in }[0,T]\\[3pt]
		\rho(t) &= X(t)_{\#}\mathbb{P} \qquad&\text{in }[0,T].
	\end{array}
\right.
\end{equation}

\begin{remark}\label{rem-rv-finite-support-rho}
	For the particular case when $\rho_0$ is a deterministic measure of finite support  of the form $\rho_0=\sum_{i=1}^m\tilde{\alpha_i}\delta_{\tilde{x}_i}$ in $\mathcal{P}_2(\mathbb{R}^d)$, for some finite set of points $\{\tilde{x}_i\}_{i=1}^m\subset\mathbb{R}^d$ and weights {$\{\tilde{\alpha}_i\}_{i=1}^m\subset (0,1]$, $\sum_{j=1}^m\tilde{\alpha}_j = 1,$} $m\in\mathbb{N}$, we assume in the Hamiltonian system \eqref{characteristics-continuous} that $X_0=\sum_{i=1}^mx_{i,m}\chi_{\Omega_{i,m}}$ in $L^2(\Omega;\mathbb{R}^d)$ where $x_{i,m}\coloneqq \tilde{x}_i$ and $\mathbb{P}(\Omega_{i,m})=\alpha_{i,m}$ with $\alpha_{i,m}\coloneqq \tilde{\alpha}_i$ for $i\in\{1,\dots,m\}$.
\end{remark}

We recall that
it is known from \cite[Theorem 4.1 \& Theorem 4.5]{meszaros2024mean} that, for each initial random variable $X_0$ distributed according to $\rho_0$, the Hamiltonian system \eqref{definition-soln-of-continuous-pb} admits a unique solution $(X,Y,
\rho)$. This solution is determined by the unique pair $(u,\rho)$ which satisfies the first-order mean field game PDE system 
\begin{equation}
\left\{
    \begin{array}{ll}
        - \partial_tu+H(x,-D_xu,\rho)=0 &\text{ in }(0,T)\times\mathbb{R}^d,
		\\[3pt]
		\partial_t\rho + D_x\cdot\left(\rho{D_p H}(x,-D_xu,\rho)\right)=0 &\text{ in }(0,T)\times\mathbb{R}^d,
		\\[3pt]
		u(T,x)=g(x,\rho(T))\quad\text{and}\quad \rho(0,x)=\rho_0,  &\text{ in }\mathbb{R}^d,
    \end{array}
    \right.
\end{equation}
in the sense that $u$ is the unique {classical} solution {(of class $C^{1,1}_{\rm loc}$)} of the Hamilton--Jacobi equation 
 and $\rho$ is the unique distributional solution of the {continuity} equation.  %
 The existence and uniqueness of such a pair $(u,\rho)$ is ensured by \cite[Corollary 4.6]{meszaros2024mean} under the regularity and growth assumptions \eqref{ass-H:1}, \eqref{ass-H:2}, \eqref{ass-H:3}, \eqref{ass-H:4}, \eqref{ass-H:5}, \eqref{ass-g:1}, \eqref{ass-g:2}, \eqref{ass-L-g} and the displacement monotonicity conditions \eqref{ass-H:disp-mono-cond-H}, \eqref{ass-g:disp-mono-cond-g} on $H$ and $g$.  
 \cite[Theorem 4.1]{meszaros2024mean} then shows that 
\begin{equation}
    Y(t)\coloneqq -D_xu(t,X(t))\quad\text{in }[0,T],
\end{equation} and that $X$ is the unique solution of the forward equation
\begin{equation}
    X(t) =X_0+ \int_0^{t}D_pH\left(X(s),-D_xu(s,X(s)),\rho(s)\right)\mathrm{d}s, \qquad\text{in }[0,T],
\end{equation}with $\rho(t)= X(t)_{\#}\mathbb{P}$, $t\in[0,T]$. In particular $X,Y\in C^1([0,T];L^2(\Omega;\mathbb{R}^d))$ and $\rho$ is an element in the space $C([0,T];\mathcal{P}_2(\mathbb{R}^d))$. Consequently, the value function evaluated along the optimal trajectory $(X(t),{Y(t)})_{t\in [0,T]}$ has the representation
\begin{equation}\label{u-opt-trajectory-rep}
    u(t,X(t))=\int_t^TL(X(s),{D_pH(X(s),Y(s),\rho(s))},\rho(s))\mathrm{d}s + g(X(T),\rho(T))\quad\forall t\in [0,T].
\end{equation}

\subsection*{Continuous dependence of solutions on discrete initial conditions.}
Given $N\in\mathbb{N}$, let $\rho_0^N=\sum_{j=1}^N\alpha_{j,N}\delta_{x_{j,N}}$ denote a deterministic measure approximation of $\rho_0$ with finite support in $\mathcal{P}_2(\mathbb{R}^d)$ where $\text{card}(\text{spt}(\rho_0^N))=N$ and $\mathcal{A}_N\coloneqq \{\alpha_{j,N}\}_{j=1}^N\subset (0,1]$ with $\sum_{j=1}^{N}\alpha_{j,N}=1$. For each $N\in\mathbb{N}$, let $X_0^N$ denote an $\mathbb{R}^d$-valued simple random variable distributed according to $\rho_0^N$ with respect to the reference probability space $(\Omega,\mathbb{F},\mathbb{P})$ of the form $X_0^N=\sum_{j=1}^Nx_{j,N}\chi_{\Omega_{j,N}}$ for some subsets $\mathcal{O}_N\coloneqq \{\Omega_{j,N}\}_{j=1}^N\subset\mathbb{F}$ such that $\cup_{j=1}^{N}\overline{\Omega_{j,N}}=\Omega$, $\mathbb{P}(\Omega_{j,N})=\alpha_{j,N}$ for $j\in \{1,\cdots, N\}$, and $\Omega_{j,N}\cap\Omega_{i,N}=\emptyset$ if $i\neq j$ in $\{1,\cdots,N\}$. Therefore, for each $N\in \mathbb{N}$, the simple random variable $X_0^N$ generates the subspace $\mathcal{R}_d(\mathcal{A}_N,\mathcal{O}_N)$ of $L^2(\Omega;\mathbb{R}^d)$ with $\mathcal{A}_N=\{\alpha_{j,N}\}_{j=1}^N$ and $\mathcal{O}_{N}=\{\Omega_{j,N}\}_{j=1}^N$. Throughout this work, we assume that the sequences $\{\rho_0^N\}_{N\in\mathbb{N}}$ and $\{X_0^N\}_{N\in\mathbb{N}}$, and hence the sequences of weights $\{\mathcal{A}_N\}_{N\in\mathbb{N}}$ and subsets $\{\mathcal{O}_N\}_{N\in\mathbb{N}}$, are given and fixed, unless otherwise specified. 

Applying once more \cite[Theorem 4.5]{meszaros2024mean}, we have that for each $N\in\mathbb{N}$ there exists a unique solution $(X^N,Y^N,\rho^N)$ to the continuous Hamiltonian system \eqref{definition-soln-of-continuous-pb} with $X_0^N$ as the initial random variable of the system, with the corresponding value function $u^N$ satisfying
\begin{equation}\label{u-opt-trajectory-rep-N}
    u^N(t,X^N(t))=\int_t^TL(X^N(s),{D_pH(X^N(s),Y^N(s),\rho^N(s))},\rho^N(s))\mathrm{d}s + g(X^N(T),\rho^N(T)),\quad\forall t\in [0,T].
\end{equation}
The following result concerns the convergence of the sequence $\{(X^{N},Y^{N},\rho^{N})\}_{N\in\mathbb{N}}$,  whose proof is omitted since it is a direct application of the argument behind the continuous dependence result given in \cite[Theorem 4.5]{meszaros2024mean}.
\begin{lemma}\label{lemma-convegrence-fixed-N-inf}
	Let the Hamiltonian $H$ satisfy the regularity and growth assumptions \eqref{ass-H:1}, \eqref{ass-H:2}, \eqref{ass-H:3}, \eqref{ass-H:4}, \eqref{ass-H:5}, \eqref{ass-L-g} and the displacement monotonicity condition \eqref{ass-H:disp-mono-cond-H}. Let the terminal cost $g$ satisfy the regularity and growth assumptions \eqref{ass-g:1}, \eqref{ass-g:2}, \eqref{ass-L-g} and the displacement monotonicity condition \eqref{ass-g:disp-mono-cond-g}.  Let $X_0\in L^2(\Omega;\mathbb{R}^d)$ be given such that $\rho_0=\mathcal{L}_{X_0}$. For each $N\in\mathbb{N}$, let $\rho_0^N$ be a deterministic measure of finite support with ${\rm card}({\rm spt}(\rho_0^N))=N$ and let $X_0^N$ be a simple random variable distributed according to $\rho_0^N$. Let $(X,Y,\rho)\in C^1([0,T];\LLspace)\times C^1([0,T];\LLspace)\times C([0,T];\mathcal{P}_2(\mathbb{R}^d))$ denote the triple that uniquely satisfies the system \eqref{definition-soln-of-continuous-pb} with initial random variable $X_0$, and let $\{(X^{N},Y^{N},\rho^{N})\}_{N\in\mathbb{N}}\subset  C^1([0,T];\LLspace)\times C^1([0,T];\LLspace)\times C([0,T];\mathcal{P}_2(\mathbb{R}^d))$ denote the sequence of unique solutions to the system \eqref{definition-soln-of-continuous-pb} corresponding to the sequence of initial random variables $\{X_0^N\}_{N\in\mathbb{N}}$. Then,
	\begin{equation}\label{N-convergence-error-bound}
		\begin{split}
			&	\sup_{0\leq t\leq T}\|(-D_xu(\cdot,X) - Y^{N})(t)\|_{\LLspace}+\sup_{0\leq t\leq T}\|(X - X^{N})(t)\|_{\LLspace}  + \sup_{0\leq t\leq T} \mathcal{W}_2(\rho(t),\rho^{N}(t))
			\\
			&\qquad\qquad\qquad\qquad\qquad\qquad\qquad\qquad\qquad\qquad\qquad\qquad\lesssim \|X_0-X_0^N\|_{\LLspace},\quad\forall N\in\mathbb{N}.
		\end{split}
	\end{equation}
	In particular, if $X_0^N\to X_0$ in $\LLspace$ as $N\to\infty$, then 
	\begin{subequations}\label{numerical-scheme-N-inf-only-conv}
		\begin{align}
			&X^{N}\to X \text{ and } Y^{N}\to -D_xu(\cdot,X) {\rm\ in\ }C([0,T];\LLspace) \\
			&\rho^{N} \to \rho=(X(t))_{\#}\mathbb{P} \quad {\rm\ in\ }C([0,T];\mathcal{P}_2(\mathbb{R}^d)),
		\end{align}
	\end{subequations} as $N\to\infty$.
\end{lemma}

We obtain a corollary for continuous dependence of the value function along the optimal trajectory provided that $L$ (equivalently $H$) and $g$ are additionally locally Lipschitz continuous in all variables. 
\begin{corollary}\label{cor-val-func-contin-dep-bound}
{Let  $u$ denote the value function in the continuous MFG, represented along optimal trajectories in \eqref{u-opt-trajectory-rep} and let $u^N$ stand for the discretised value function defined in \eqref{u-opt-trajectory-rep-N}.}
    In addition to the hypotheses of Lemma \ref{lemma-convegrence-fixed-N-inf}, suppose that for each $R>0$ there exists a constant $C_R>0$ such that
    \begin{subequations}
        \begin{align}
            |L(x,v,\mu)-L(y,w,\nu)|&\leq C_R\left(1+|x|+|y|+|v|+|w|+\mathcal{W}_1(\mu,\delta_0)+\mathcal{W}_1(\nu,\delta_0)\right)(|x-y|+|v-w|+\mathcal{W}_1(\mu,\nu)) \label{L-loc-meas-contin}\\
            |g(x,\mu)-g(y,\nu)|&\leq C_R\left(1+|x|+|y|+\mathcal{W}_1(\mu,\delta_0)+\mathcal{W}_1(\nu,\delta_0)\right)(|x-y|+\mathcal{W}_1(\mu,\nu)) \label{g-loc-meas-contin}
        \end{align}
    \end{subequations}
    whenever $(x,v),(y,w)\in\mathbb{R}^d\times\mathbb{R}^d$ and $\mu,\nu\in\mathcal{P}_2(\mathbb{R}^d)$ satisfy $|x-y|+|y-w|+\mathcal{W}_1(\mu,\nu)<R$. If $X_0^N\to X_0$ in $\LLspace$ as $N\to\infty$, then 
    \begin{equation}\label{N-convergence-error-bound-u-value}
		\sup_{0\leq t\leq T}\|(u(\cdot,X) - u^{N}(\cdot,X^N)(t)\|_{L^1(\Omega)}\lesssim \|X_0-X_0^N\|_{\LLspace}
	\end{equation}
    for all $N\in\mathbb{N}$ sufficiently large.
\end{corollary}

This result is a direct consequence of the representation formulae \eqref{u-opt-trajectory-rep} and \eqref{u-opt-trajectory-rep-N}, the standing assumptions on $H$ and $g$, and the additional local Lipschitz hypotheses \eqref{L-loc-meas-contin}, \eqref{g-loc-meas-contin}. For completeness, we present its proof in Section \ref{sec-convergence}.

{Let us emphasise that the additional assumptions on the local Lipschitz properties of $L$ and $g$ (in the measure variable) are not a consequence of the previously imposed regularity assumptions. We note that all the previous regularity assumptions in the measure variable are imposed on $D_x g$, $D_xH, D_pH$ (or equivalently on $D_xL, D_vL$) and in particular no regularity or Lipschitz property was imposed on $L$ and $g$ in the measure variable.}

We conclude this section with discussion of two constructions for the sequence of deterministic measures $\{\rho_0^N\}_{N\in\mathbb{N}}$ of finite support and corresponding random variables $\{X_0^N\}_{N\in\mathbb{N}}$ that converge strongly to $X_0$ in $L^2(\Omega;\mathbb{R}^d)$. Each example is also accompanied by rates of convergence that hold under sufficient conditions.
\begin{eg}[Semi-discrete optimal transport]\label{eg-2}
    Given a measure $\rho_0\in\mathcal{P}_2(\mathbb{R}^d)$, it is well known that $\rho_0$ can be approximated by sequences of empirical measures $\rho_0^N$ in $\mathcal{P}_2(\mathbb{R}^d)$ which are stochastic \cite{boissard2014mean,fournier2015rate,fournier2023convergence} or deterministic \cite{chevallier2018uniform,quattrocchi2024asymptotics,seeger2025error}. In particular, assuming that $\rho_0\in\mathcal{P}_q(\mathbb{R}^d)$ for some $q>2$, \cite[Theorem 1]{seeger2025error} provides an algorithmic construction of a sequence of uniformly weighted deterministic empirical measures $\{\rho_0^N\}_{N\in\mathbb{N}}$ that converge strongly to $\rho_0$ in $\mathcal{P}_2(\mathbb{R}^d)$ as $N\to\infty$, with the error bounds
    \begin{equation}
        \mathcal{W}_2(\rho_0,\rho_0^N)\leq C M_q(\rho_0) \times
		\begin{dcases}
			\max\{N^{-1/d},  N^{-1/2 + 1/q}\}, & {\rm if\ } d \ne \frac{2q}{q-2},\\
			N^{-1/d} \log(1+N)^{1/d}, & {\rm if\ } d = \frac{2q}{q-2},
		\end{dcases}\quad\forall N\in\mathbb{N}
    \end{equation} and $\mathcal{W}_2(\rho_0,\rho_0^N)\leq C N^{-1/2+1/q}$ for $N\in\mathbb{N}$ if $q<\infty$ and $d < \frac{2q}{q-2}$. The constant $C$ in both of these bounds here depends only on $d$ and $q$. Assuming further that $\rho_0$ is absolutely continuous with respect to  $\mathbb{P}$, Brenier's Theorem \cite{brenier1991polar} then ensures that, for each $N\in\mathbb{N}$, there exists an optimal transport map $\mathcal{T}^N:\mathbb{R}^d\to\mathbb{R}^d$, defined everywhere with ${\rm card}({\rm Ran}(T^N))=N$, which satisfies $\mathcal{T}^N(X_0)\in L^2(\Omega;\mathbb{R}^d)$ and $$\|X_0-\mathcal{T}^N(X_0)\|_{L^2(\Omega;\mathbb{R}^d)}=\mathcal{W}_2(\rho_0,\rho_0^N).$$ As such, the sequence $\{X_0^N\}_{N\in\mathbb{N}}$ given by $X_0^N\coloneqq \mathcal{T}^N(X_0)$, $N\in\mathbb{N}$, converges to $X_0$ in $L^2(\Omega;\mathbb{R}^d)$ with rates inherited from those above for the convergence of $\{\rho_0^N\}_{N\in\mathbb{N}}$ to $\rho_0$ in $\mathcal{P}_2(\mathbb{R}^d)$. The sequence of partitions $\{\mathcal{O}_N\}_{N\in\mathbb{N}}$ of $\Omega$ associated with $\{X_0^N\}_{N\in\mathbb{N}}$ are given by the pre-images $\{X_0^{-1}(\mathcal{L}_N)\}_{N\in\mathbb{N}}$ of the Laguerre tessellations $\{\mathcal{L}_N\}_{N\in\mathbb{N}}$ of $\mathbb{R}^d$ of the maps $\{\mathcal{T}^N\}_{N\in\mathbb{N}}$. The practical construction of the optimal transport maps $\{\mathcal{T}^N\}_{N\in\mathbb{N}}$ is an active area of research concerning semi-discrete optimal transport, see \cite{bourne2018semi,tacskesen2023semi,sadhu2024stability} for a range of computational approaches. 
\end{eg}

\begin{eg}[Quantization of random variables]\label{eg-3}
    An alternative construction can be achieved through quantization of the random variable $X_0$, see the seminal papers \cite{max1960quantizing,lloyd1982least} as well as  \cite{graf2000foundations} for an introduction to the subject. To help formulate the main idea of this approach with an error bound, we introduce some notation. Given $N\in\mathbb{N}$, let $\mathcal{F}_N$ denote the set of Borel measurable functions {$f:\R^d\to\R^d$ (defined a.e.)}, where ${\rm card}({\rm Range}($f$))\leq N$ and let the $N$-th quantization error for $\rho_0$ be defined as 
    $$V_{N}(\rho_0)\coloneqq \inf_{f\in\mathcal{F}_N}\|X_0-f(X_0)\|_{L^2(\Omega;\mathbb{R}^d)}^2.$$ 
    The lower quantization dimension for $\rho_0$ is defined as
    $$\mathcal{D}_L(\rho_0)\coloneqq \liminf_{N\to\infty}((-2)\log(N)/\log(V_N(\rho_0)))$$ and the upper quantization dimension for $\rho_0$ is defined as
    $$\mathcal{D}_U(\rho_0)\coloneqq \limsup_{N\to\infty}((-2)\log(N)/\log(V_N(\rho_0))).$$ If $\mathcal{D}_L(\rho_0)=\mathcal{D}_U(\rho_0)$ then the common value $\mathcal{D}(\rho_0)$ is called the quantization dimension of $\rho_0$. Now, suppose that the initial random variable $X_0$ has slightly stronger integrability in $L^{2+\delta}(\Omega;\mathbb{R}^d)$ for some $\delta>0$. Using  \cite[Theorem 4.1, Theorem 4.12 \& Corollary 11.4(c)]{graf2000foundations} and unpacking the definitions of $\mathcal{D}_L(\rho_0),\mathcal{D}_U(\rho_0)$, it follows that, for each $\epsilon\in (0,1)$ there exist $N_{\dagger}=N_{\dagger}(\epsilon,{\mathcal{D}_L(\rho_0)})\in\mathbb{N}$, sequences of weights $\{\mathcal{A}_N\}_{N\geq N_{\dagger}}$, Voronoi tessellations $\{\mathcal{O}_N\}_{N\geq N_{\dagger}}$ of $\Omega$ and maps $\{f_{N}\}_{N\geq N_{\dagger}}$ (which achieve the infimum in $V_N(\rho_0)$ for each $N$)  such that $X_0^N\coloneqq f_N(X_0)\in \mathcal{R}_d(\mathcal{A}_N,\mathcal{O}_N)$ for each $N\geq N_{\dagger}$ and
    \begin{equation}
        N^{-\frac{1}{\mathcal(1-\epsilon){D}_L(\rho_0)}}\leq\|X_0-X_0^N\|_{\LLspace}\leq N^{-\frac{1}{\mathcal{D}_U(\rho_0)+\epsilon\mathcal{D}_L(\rho_0)}}\quad\forall N\geq N_{\dagger}(\epsilon,{\mathcal{D}_L(\rho_0)})
    \end{equation}
    where $0< \mathcal{D}_L(\rho_0)\leq \mathcal{D}_U(\rho_0)\leq d$. The above bound merely provides a non-sharp estimate for the asymptotic rate of convergence for such quantizers $\{X_0^N\}_{N\geq N_{\dagger}}$, where the quantization dimensions need not coincide in general. In the case when ${\rm card}({\rm spt}(\rho_0))$ is finite the error bound is trivial since then $X_0$ is already a simple random variable that does not need approximating. If ${\rm card}({\rm spt}(\rho_0))=\infty$ and $\rho_0$ has non-trivial absolutely continuous part, \cite[Corollary 11.4(c)]{graf2000foundations} and Zador's theorem \cite{zador1964development} imply  $\mathcal{D}_L(\rho_0)=\mathcal{D}_U(\rho_0)$ with $\mathcal{D}(\rho_0)=d$ and one has a more precise bound with a sharp convergence rate: 
\begin{equation}
	\|X_0-X_0^N\|_{\LLspace}\leq CN^{-1/d}\quad\forall N\geq N_{\dagger}
\end{equation}
where the constant $C$ depends only on $\delta$, $X_0$, the absolutely continuous part of $\rho_0$, and $d$. If $\rho_0$ is a distribution supported on a self-similar subset of $\mathbb{R}^d$ then it is possible for $\mathcal{D}_L(\rho_0)=\mathcal{D}_U(\rho_0)$ and $\mathcal{D}(\rho_0)<d$. For example, if $d=1$ and $\rho_0$ is the one-dimensional Cantor distribution, \cite[Ch.\ 14]{graf2000foundations} gives $\mathcal{D}(\rho_0)=\log(2)/\log(3)<1$ where the optimal dyadic $N_m$-quantizer with $N_m=2^m$ satisfies
\begin{equation}
	\|X_0-X_0^{N_m}\|_{\LLspace}\leq \frac{1}{8}N_m^{-\log(3)/\log(2)}\quad\forall m\in\mathbb{N}.
\end{equation} Concerning the practical construction of quantizers $\{X_0^N\}_{N\in\mathbb{N}}$, together with the associated weights $\{\mathcal{A}_N\}_{N\geq N_{\dagger}}$, points $\{\{x_{j,N}\}_{j=1}^{N}\}_{N\geq N_{\dagger}}$ and Voronoi tessellations $\{\mathcal{O}_N\}_{N\geq N_{\dagger}}$ of $\Omega$, one can consider methods based on Lloyd's algorithm \cite{lloyd1982least,kieffer1982exponential}, which include stochastic gradient descent and Newton iterations \cite{du1999centroidal,pages2003optimal}, as well as greedy algorithms \cite{luschgy2015greedy,el2022new}. In this way, for each $N\in\mathbb{N}$, the deterministic {finitely supported} measure $\rho_0^N$ is constructed from the quantizer $X_0^N$ and the weights $\mathcal{A}_N$, where its support is precisely the range of $X_0^N$.
\end{eg}



\section{The Discrete Problem}\label{sec-numerical-discretization}

\subsection{{The setting for} temporal discretization}


We introduce some notation for vector-space valued maps that will be used throughout the manuscript.
Let $(A,\|\cdot\|_A)$ be an arbitrary real, normed vector space. Let $a,b\in\mathbb{R}$ be given such that $0\leq a<b$. Given $k\in\mathbb{N}$, let $\mathcal{J}_k\coloneqq \{I_n\}_{n=1}^{\Mk}$, $M_k\in\mathbb{N}$, denote the partition of the time interval $[a,b]$ where $I_n\coloneqq (t_{n-1},t_n)$ with $t_n\coloneqq a+n\tau_k$ for $n\in\{1,\cdots,\Mk\}$, $t_0\coloneqq a$, and the time-step $\tau_k\coloneqq (b-a)/\Mk$, with  $M_{k+1}\geq M_k$ for $k\in\mathbb{N}$ and $\Mk \tends \infty$ as $k\tends \infty$. 

\subsubsection*{The space $\mathbb{V}_k(a,b;A)$.}

\begin{itemize}
	\item We let $\mathbb{V}_k(a,b;A)$ denote the vector space of $A$-valued functions defined on $(a,b)$ that are piecewise constant-in-time with respect to the temporal partition $\mathcal{J}_k$, i.e.\ $X\in\mathbb{V}_k(a,b;A)$ if and only if $X|_{I_n}\in A$ is constant in time for all $n\in \{1,\cdots,M_k\}$. 
	\item For a function $X\in\mathbb{V}_k(a,b;A)$ and $n\in\{1,\cdots,\Mk-1\}$, we let $X(t_n^{-})$ and $X(t_n^{+})$  denote respectively the left- and right-limits of $X$ in $A$ at time $t_n$. {A function $X\in \mathbb{V}_k(a,b;A) $ also admits a right-limit $X(t_{0}^+)=X|_{I_1}$ in $A$ at $t_0=a$ and a left-limit $X(t_{\Mk}^-)=X|_{I_{\Mk}}$ in $A$ at $t_{\Mk}=b$.} 
\end{itemize}

\begin{remark} 
Since $X\in \mathbb{V}_k(a,b;A)$ is piecewise constant, we clearly have $X(t_n^{-})=X|_{I_n}$ in $A$ for all $n\in\{1,\dots,\Mk\}$, and that $X(t_{n}^+)=X|_{I_{n+1}}$ in $A$ for all $n\in\{0,\dots,\Mk-1\}$. 
\end{remark} 

\subsubsection*{The spaces $\mathbb{V}_k^{\pm}([a,b];A)$.}
In the ensuing analysis it will be useful to consider piecewise constant $A$-valued maps that are defined everywhere in $[a,b]$ through either left-continuity or right-continuity. To be precise, we introduce the following notation and conventions. 
\begin{itemize}
	\item For a function $X\colon ([a,b],|\cdot|) \tends (A,\|\cdot\|_A)$ with $X|_{(a,b)}\in \mathbb{V}_k(a,b;A)$, we say that $X$ is left-continuous if $X(t_n)=X(t_{n}^-)$ in $A$ for all $n=1,\dots,\Mk$, and that $X$ is right-continuous if $X(t_n)=X(t_{n}^+)$ in $A$ for all $n=0,\dots,\Mk-1$. 
	\item We define the vector spaces \begin{subequations}
		\begin{align}
			&\mathbb{V}_k^+([a,b];A)\coloneqq \left\{X\colon[a,b]\tends  A\text{ such that\ } X|_{(a,b)}\in \mathbb{V}_k(a,b;A), \text{ and }X\text{ is {left}-continuous}\right\},\\
			&\mathbb{V}_k^-([a,b];A)\coloneqq \left\{X\colon[a,b]\tends  A \text{ such that\ } X|_{(a,b)}\in \mathbb{V}_k(a,b;A), \text{ and }X\text{ is {right}-continuous}\right\}.
		\end{align}
	\end{subequations}
	\item When $V\in\mathbb{V}_k^{+}([a,b];A)$ and $W\in\mathbb{V}_k^{-}([a,b];A)$, we set $V(t_0^{-})\coloneqq V(a)$ in $A$ and $W(t_{\Mk}^+)\coloneqq W(b)$ in $A$.
	\item {For $V\in \mathbb{V}_k(a,b;A)$, we define the jump $\jump{v}_{n}\in A$ of $v$ at time $t_n\in(a,b)$ by
		\begin{equation}\label{def-time-jump}
			\jump{V}_n\coloneqq V(t_n^{-})-V(t_n^{+})\text{ in }A.
		\end{equation}
		We define also the jump $\jump{V}_{0}=V(a)-V(a^+)$ in $A$ for $V\in\mathbb{V}_k^{+}([a,b];A)$ , and $\jump{W}_{\Mk}=W(b^-)-W(b)$ in $A$ for $W\in \mathbb{V}_k^{-}([a,b];A)$.
	}
\end{itemize}
	Throughout this work, we adopt the convention that functions in $\mathbb{V}_k^{\pm}([a,b];A)$ are defined at all points in $[a,b]$, and that two functions in $\mathbb{V}_k^{\pm}([a,b];A)$ that agree up to a subset of measure zero of $[a,b]$ are \emph{not} identified.

\subsubsection*{Reconstruction of time derivatives.}
Let $a,b\in\mathbb{R}$ be given such that $0\leq a<b$. We let $C([a,b];A)$ be the space of continuous maps from $([a,b],|\cdot|)$ to $(A,\|\cdot\|_A)$. The formulation of the numerical scheme in the next section will be based on the following interpolation-in-time of maps from $\mathbb{V}_k^{\pm}([a,b];A)$:
\begin{itemize}
	\item Let the \emph{forward-in-time} reconstruction operator $\mathcal{I}_+^k:\mathbb{V}_k^+([a,b];A)\to {C([a,b]; A)}$ be defined by 
	\begin{multline}\label{Ip-op}
		(\Ip V)(t)\coloneqq {V(t)+\frac{t_n-t}{\tau_k}\jump{V}_{n-1}}\text{ in }A,\text{ } t\in {(t_{n-1},t_n]},\;n\in\{1,\cdots,\Mk\},  V \in \mathbb{V}_k^+([a,b];A).
	\end{multline}
	\item Define the \emph{backward-in-time} reconstruction operator $\In :\mathbb{V}_k^-([a,b];A)\to C([a,b]; A)$ by 
	\begin{multline}\label{In-op}
		(\In W)(t)\coloneqq {W(t)-\frac{t-t_{n-1}}{\tau_k}\jump{W}_{n}}\text{ in }A,\text{ }{t\in [t_{n-1},t_n)},\; n\in\{1,\cdots,\Mk\},  W\in \mathbb{V}_k^-([a,b];A).
	\end{multline}
\end{itemize}
Throughout, we slightly abuse notation by not indicating the dependence of $\Ipm^k$ on $A$ as there is no risk of confusion. 
\begin{remark}
Since $V\in \mathbb{V}_k^+([a,b];A)$ is by definition left-continuous, it follows that $\Ip V$ is continuous on $[a,b]$ and satisfies $\Ip V(t_{n})=V(t_n)=V(t_{n}^-)$ in $A$ for all $n\in\{0,1,2,\dots,\Mk\}$ and  $\Ip V(a)=V(a)$ in $A$. Likewise, for $W\in \mathbb{V}_k^-([a,b];A)$, the map $\In W$ is continuous on $[a,b]$ and $\In W(t_n)=W(t_n)=W(t_{n}^+)$ in $A$ for all $n\in\{0,\dots,\Mk\}$ and $\In W(b)=W(b)$ in $A$.
\end{remark}

It can be shown that $\Ip V$ and $\In W$ are piecewise continuously differentiable with respect to the time partition $\mathcal{J}_k$ for any $V\in\mathbb{V}_k^+([a,b];A)$ and any $W\in \mathbb{V}_k^-([a,b];A)$, where
\begin{equation}\label{derivatives_Ip}
	\begin{aligned}
		\p_t \Ip V|_{I_n}= -\frac{1}{\tau_k}\jump{V}_{n-1} , && \p_t \In W|_{I_n} = -\frac{1}{\tau_k} \jump{W}_n, \text{ in }A&& \forall n\in\{1,\dots,\Mk\}.
	\end{aligned}
\end{equation}
This shows that $\p_t \Ip$ and $\p_t \In$ are respectively related to the first-order backward and forward difference operators. 

\subsection*{Discrete-time probability measure flows.}
Let $N\in\mathbb{N}$ be given. We let $\mathcal{P}_{d,N}$ denote the set of deterministic, finitely supported probability measures of the form $\sum_{i=1}^N\alpha_{i,N}\delta_{{y}_i}$ where $\{{y}_i\in\mathbb{R}^d:i=1,\dots,N\}$ and  $\{\alpha_{i,N}\}_{i=1}^N\subset (0,1)$ are given {such that $\sum_{i=1}^N\alpha_{i,N}=1$}. We then  introduce  the set of maps 
\begin{equation}
	\begin{split} 
\mathbb{P}_k([0,T];\mathcal{P}_{d,N})
		\coloneqq\Big\{&\upsilon:[0,T]\to (\mathcal{P}_{d,N},\mathcal{W}_2) \text{ s.t.\ } \upsilon|_{I_n}\in\mathcal{P}_2(\mathbb{R}^d) \text{ constant }\forall n\in\{1,\cdots, M_k\} \\ 
&\text{ and }\upsilon\text{ is left-continuous}\Big\},
	\end{split} 
\end{equation}
which thus consist of deterministic, empirical probability measure flows that are piecewise constant in time with respect to the time partition $\mathcal{J}_k$ that are defined everywhere in $[0,T]$. Note that when $\varrho\in \mathbb{P}_k([0,T];\mathcal{P}_{d,N})$ we set $\varrho(t_0^{-})\coloneqq \varrho(0)$ in $\mathcal{P}_2(\mathbb{R}^d)$.

\subsection{Discretization of the Hamiltonian system}
\subsubsection*{{The} numerical scheme.}
We now present a numerical scheme for the approximation of the characteristic system of the first-order MFG system \eqref{mfg-pde-sys}.  For each $N\in\mathbb{N}$, let $\rho_0^N$ denote an arbitrary deterministic {finitely supported} measure approximation of $\rho_0$ of the form $\rho_0^N\coloneqq \sum_{i=1}^{N}\alpha_{i,N}\delta_{x_{i,N}}$ where  $\{x_{i,N}\}_{i=1}^N\subset \mathbb{R}^d$, and let $X_0^N$ denote the corresponding $\rho_0^N$-distributed random variable in $\mathcal{R}_d(\mathcal{A}_N,\mathcal{O}_N)$. 

The numerical method we consider in this work is defined as follows: \emph{Let $N\in\mathbb{N}$ and $k\in\mathbb{N}$ be given. Find $X^{k,N}\in \mathbb{V}_k^+([0,T];\mathcal{R}_{d}(\mathcal{A}_N,\mathcal{O}_N))$, $Y^{k,N}\in \mathbb{V}_k^-([0,T];\mathcal{R}_{d}(\mathcal{A}_N,\mathcal{O}_N))$ and $\rho^{k,N}\in\mathbb{P}_k([0,T];\mathcal{P}_{d,N})$ where 
\begin{subequations}\label{numerical-scheme}
    \begin{align}
		\mathcal{I}_+^kX^{k,N}(t) &=X_0^{N}+ \int_0^tD_pH\left(X^{k,N}(s),Y^{k,N}(s),\rho^{k,N}(s)\right)\mathrm{d}s,\ \forall t\in [0,T], \label{disc-HJ-eqn}
		\\
		\mathcal{I}_-^kY^{k,N}(t) &= -D_xg\left(X^{k,N}(T),\rho^{k,N}(T)\right)+\int_t^TD_xH\left(X^{k,N}(s),Y^{k,N}(s),\rho^{k,N}(s)\right)\mathrm{d}s,\ \forall t\in [0,T], \label{disc-KFP-eqn}\\
		\rho^{k,N}(t) &= (X^{k,N}(t))_{\#}\mathbb{P}, \quad \forall t\in [0,T].
    \end{align}
\end{subequations}
}

\begin{remark}\label{rem-rv-finite-support-rho-N}
	In view of Remark \ref{rem-rv-finite-support-rho}, whenever $\rho_0$ is an measure supported on a finite number of points with corresponding simple random variable $X_0$ we consider the constant sequence $\{X_0^N\}_{N\in\mathbb{N}}$ of initial conditions given by $X_0^N\coloneqq X_0$ in $L^2(\Omega;\mathbb{R}^d)$ for the numerical scheme \eqref{numerical-scheme}.
\end{remark}
\begin{remark} The integral terms in each equation \eqref{disc-HJ-eqn}, \eqref{disc-KFP-eqn} are well-defined. Indeed, write $ X^{k,N}(t) =\sum_{i=1}^{N}X^{k,N}(t)|_{\Omega_{i,N}}\chi_{\Omega_{i,N}}$ and $ Y^{k,N}(t) =\sum_{i=1}^{N}Y^{k,N}(t)|_{\Omega_{i,N}}\chi_{\Omega_{i,N}}$ for each $t\in [0,T]$. Since $X^{k,N}(t)$ is piecewise constant-in-time with values in $\mathcal{R}_{d}(\mathcal{A}_N,\mathcal{O}_N)$ and is left-continuous as maps from $([0,T],|\cdot|)$ to $(\mathcal{R}_{d}(\mathcal{A}_N,\mathcal{O}_N),\|\cdot\|_{L^2(\Omega;\mathbb{R}^d)})$, we have
\begin{equation}\label{nonlinearity-identity-3}
	\begin{split} 
		\rho^{k,N}(t) & =\sum_{i=1}^N\alpha_{i,N}\delta_{X^{k,N}(t)|_{\Omega_{i,N}}}
	\end{split} 
\end{equation}
is a deterministic map from $([0,T],|\cdot|)$ to $(\mathcal{P}_2(\mathbb{R}^d),W_2)$ that is piecewise constant in time. To see that $\rho^{k,N}$ is left-continuous, notice that there holds the bound
\begin{equation}
	\mathcal{W}_2(\rho^{k,N}(s),\rho^{k,N}(r))\leq \|X^{k,N}(s)-X^{k,N}(r)\|_{\LLspace}\quad\forall s,r\in [0,T]
\end{equation}
since $\rho^{k,N}(t)$ is the law of $X^{k,N}(t)$ with respect to the reference probability space $(\Omega,\mathbb{F},\mathbb{P})$ for each $t\in [0,T]$. It then follows that $\rho^{k,N}$ is left-continuous on $[0,T]$ with $\rho^{k,N}(t_0^{-})=\rho_0^N=\rho(0)$ in $\mathcal{P}_2(\mathbb{R}^d)$, so that $\rho^{k,N}$ is in  $\mathbb{P}_k([0,T];\mathcal{P}_{d,N})$.
We then obtain that the map  $[0,T]\ni t\mapsto D_pH\left(X^{k,N}(t),Y^{k,N}(t),\rho^{k,N}(t)\right)$ is piecewise constant in time and defined via
\begin{multline}
	D_pH\left(X^{k,N}(t),Y^{k,N}(t),\rho^{k,N}(t)\right)|_{I_n} 
	=  D_pH\left(\sum_{i=1}^{N}\mathfrak{X}_{\Omega_{i,N}}(t)\chi_{\Omega_{i,N}},\sum_{i=1}^{N}\mathfrak{Y}_{\Omega_{i,N}}(t)\chi_{\Omega_{i,N}},\sum_{i=1}^{N}\alpha_{i,N}\delta_{\mathfrak{X}_{\Omega_{i,N}}(t)}\right),
\end{multline} 
for $t\in I_n$, where $\mathfrak{X}_{\Omega_{i,N}}(t)\coloneqq X^{k,N}(t)|_{\Omega_{i,N}}$ and $\mathfrak{Y}_{\Omega_{i,N}}(t)\coloneqq Y^{k,N}(t)|_{\Omega_{i,N}}$ are constant for $t\in I_n$ since $X^{k,N},Y^{k,N}$ are constant-in-time on $I_n$ for $n\in\{1,\cdots,M_k\}$. It further follows that, for $n\in\{1,\cdots, M_k\}$ and $\omega\in \Omega_{j,N}$ where $i=1,2,\cdots, N$, 
\begin{align*}
		D_pH\left(X^{k,N}(t)(\omega),Y^{k,N}(t)(\omega),\rho^{k,N}(t)\right)|_{I_n} = D_pH\left(\mathfrak{X}_{\Omega_{j,N}}|_{I_n},\mathfrak{Y}_{\Omega_{j,N}}|_{I_n},\sum_{i=1}^{N}\alpha_{i,N}\delta_{\mathfrak{X}_{\Omega_{i,N}}|_{I_n}}\right).
\end{align*} 
Therefore,  the map $I_n\ni s\mapsto D_pH\left(X^{k,N}(t),Y^{k,N}(t),\rho^{k,N}(t)\right)$ has values in the vector space of uniform random variables on  $(\Omega,\mathbb{F},\mathbb{P})$ where
\begin{align}\label{nonlinearity-identity-1}
	D_pH\left(X^{k,N}(t),Y^{k,N}(t),\rho^{k,N}(t)\right)|_{I_n}=\sum_{j=1}^{N} D_pH\left(\mathfrak{X}_{\Omega_{j,N}}|_{I_n},\mathfrak{Y}_{\Omega_{j,N}}|_{I_n},\sum_{i=1}^{N}\alpha_{i,N}\delta_{\mathfrak{X}_{\Omega_{i,N}}|_{I_n}}\right)\chi_{\Omega_{j,N}}
\end{align}
in $\mathcal{R}_{d}(\mathcal{A}_N,\mathcal{O}_N)$ for each $n\in \{1,\cdots, M_k\}$, so $D_pH\left(X^{k,N},Y^{k,N},\rho^{k,N}\right)$ is a map in $\mathbb{V}_k(0,T;\mathcal{R}_{d}(\mathcal{A}_N,\mathcal{O}_N))$ and the integral in \eqref{disc-HJ-eqn} is therefore well-defined as it is the integral of a simple function in time with values in $\mathcal{R}_{d}(\mathcal{A}_N,\mathcal{O}_N)$. Similarly, $D_xH\left(X^{k,N},Y^{k,N},\rho^{k,N}\right)$ is in $\mathbb{V}_k(0,T;\mathcal{R}_{d}(\mathcal{A}_N,\mathcal{O}_N))$ where 
\begin{align}\label{nonlinearity-identity-2}
	D_xH\left(X^{k,N}(t),Y^{k,N}(t),\rho^{k,N}(t)\right)|_{I_n}=\sum_{j=1}^{N} D_xH\left(\mathfrak{X}_{\Omega_{j,N}}|_{I_n},\mathfrak{Y}_{\Omega_{j,N}}|_{I_n},\sum_{i=1}^{N}\alpha_{i,N}\delta_{\mathfrak{X}_{\Omega_{i,N}}|_{I_n}}\right)\chi_{\Omega_{j,N}}
\end{align} 
in $\mathcal{R}_{d}(\mathcal{A}_N,\mathcal{O}_N)$ for all $t\in I_n$, $n\in \{1,\cdots, M_k\}$,  
so the integral in \eqref{disc-KFP-eqn} is well-defined as it is the integral of a simple function in time with values in $\mathcal{R}_{d}(\mathcal{A}_N,\mathcal{O}_N)$.
\end{remark}

\begin{remark}[A numerical approximation for the value function]
    The numerical scheme \eqref{numerical-scheme} also implies an immediate construction of approximations for the value function \eqref{u-opt-trajectory-rep} evaluated along the optimal trajectory. Indeed, given $k,N\in\mathbb{N}$, if $(X^{k,N},Y^{k,N},\rho^{k,N})$ satisfies the discrete system \eqref{numerical-scheme}, define the map $u^{k,N}\in C([0,T];L^1(\Omega))$ via
    \begin{equation}\label{u-opt-trajectory-rep-N-k}
    u^{k,N}(t)=\int_t^TL(X^{k,N}(s),{D_pH(X^{k,N}(s), Y^{k,N}(s),\rho^{k,N}(s))},\rho^{k,N}(s))\mathrm{d}s + g(X^{k,N}(T),\rho^{k,N}(T)),\ \forall t\in [0,T].
\end{equation}
We note that $u^{k,N}$ is indeed $L^1(\Omega)$-valued for each $t\in[0,T]$ due to Fubini's theorem and the fact that $X^{k,N}$,$Y^{k,N}$ are piecewise constant in $(0,T)\times\Omega$ and $\rho^{k,N}$ is piecewise constant in time.
\end{remark}

\subsubsection*{An equivalent {re}formulation.}
The discrete scheme \eqref{numerical-scheme} can be viewed as a probabilistric interpretation of the following problem: \emph{Let $N\in\mathbb{N}$ and $k\in\mathbb{N}$ be given. Find maps $\mathcal{X}_{k,N}\in [\mathbb{V}_k^+([0,T];\mathbb{R}^d)]^{N}$ and $\mathcal{Y}_{k,N}\in [\mathbb{V}_k^-([0,T];\mathbb{R}^d)]^{N}$, and  $\rho^{k,N}\in\mathbb{P}_k([0,T];\mathcal{P}_{d,N})$ of the form $\mathcal{X}_{k,N}(t) = (\mathcal{X}_{k,N}^i(t))_{i=1}^{N}$, $\mathcal{Y}_{k,N}(t) = (\mathcal{Y}_{k,N}^i(t))_{i=1}^{N}$ where $	\rho^{k,N}(t) = \sum_{i=1}^{N}\alpha_{i,N}\delta_{\mathcal{X}_{k,N}^i(t)}$ in $[0,T]$, and for each $i\in\{1,2,\cdots, N\}$
	\begin{subequations}\label{numerical-scheme-equiv}
		\begin{align}
			\partial_t\mathcal{I}_+^k\mathcal{X}_{k,N}^i(t) &= D_pH\left(\mathcal{X}_{k,N}^i(t),\mathcal{Y}_{k,N}^i(t),\rho^{k,N}(t)\right),&&t\in(0,T)\backslash\iota_k, \label{disc-HJ-eqn-equiv}
			\\
			-\partial_t\mathcal{I}_-^k\mathcal{Y}_{k,N}^i(t) &=D_xH\left(\mathcal{X}_{k,N}^i(t),\mathcal{Y}_{k,N}^i(t),\rho^{k,N}(t)\right),&&t\in (0,T)\backslash \iota_k, \label{disc-KFP-eqn-equiv}
			\\
			\mathcal{X}_{k,N}^i(0)&=x_{i,N},\quad \mathcal{Y}_{k,N}^i(T) = - D_xg(\mathcal{X}_{k,N}^i(T),\rho^{k,N}(T)),&&\text{  }\label{disc-equiv-init-term-condition}
		\end{align}
\end{subequations} 
where $\iota_k\coloneqq \{t_{m}:q=1,2,\cdots, M_k-1\}$.}
The following result shows that the  numerical scheme \eqref{numerical-scheme} can be formulated in a mesh-free manner through the above formulation \eqref{numerical-scheme-equiv}, which benefits the practical implementation of \eqref{numerical-scheme} .
\begin{proposition}[Equivalence of numerical schemes]\label{Prop-equiv-scheme}
	The schemes \eqref{numerical-scheme} and \eqref{numerical-scheme-equiv} are equivalent in the following sense :
	\begin{equation}
		X^{k,N}(t) =\sum_{i=1}^{N}\mathcal{X}_{k,N}^i(t)\chi_{\Omega_{i,N}},\quad Y^{k,N}(t) =\sum_{i=1}^{N}\mathcal{Y}_{k,N}^i(t)\chi_{\Omega_{i,N}}\quad\text{ in }\quad\mathcal{R}_{d}(\mathcal{A}_N,\mathcal{O}_N)\quad\forall t\in [0,T].
	\end{equation}
\end{proposition}
For the sake of brevity of the presentation, we include the proof of this proposition in Appendix \ref{app-aux-disc-temp-ana}. 

\section{Main results}\label{sec-main-results}

The first main result concerns the existence of solutions to the numerical scheme \eqref{numerical-scheme}.
\begin{theorem}[Existence of numerical approximations]\label{theorem-existence-num-scheme}
	Let the Hamiltonian $H$ satisfy the regularity assumptions \eqref{ass-H:1}, \eqref{ass-H:2}, \eqref{ass-H:3}, \eqref{ass-H:4}, and the displacement monotonicity condition \eqref{ass-H:disp-mono-cond-H}. Let the terminal cost $g$ satisfy the regularity assumptions \eqref{ass-g:1}, \eqref{ass-g:2} and the displacement monotonicity condition \eqref{ass-g:disp-mono-cond-g}. For each $N\in\mathbb{N}$, let $X_0^N$ denote an initial random variable in $\mathcal{R}_{d}(\mathcal{A}_N,\mathcal{O}_N)$ for some set of weights $\mathcal{A}_N$ and a partition $\mathcal{O}_N$ of $\Omega$. Then, there exists $k_{\dagger}\in\mathbb{N}$ such that, for each $k,N\in\mathbb{N}$ with $k\geq k_{\dagger}$, there exist $X^{k,N}\in \mathbb{V}_k^+([0,T];\mathcal{R}_{d}(\mathcal{A}_N,\mathcal{O}_N))$, $Y^{k,N}\in \mathbb{V}_k^-([0,T];\mathcal{R}_{d}(\mathcal{A}_N,\mathcal{O}_N))$ and $\rho^{k,N}\in\mathbb{P}_k([0,T];\mathcal{P}_{d,N})$ which satisfy the numerical scheme \eqref{numerical-scheme}. The constant $k_{\dagger}$ depends only on the time horizon $T$, the uniform convexity constant $c_0>0$, and the Lipschitz constants associated with $D_pH$, $D_xH$ and $D_xg$.
\end{theorem} Moreover, we have the uniqueness of such solutions, as the following result shows.
\begin{theorem}[Uniqueness of numerical approximation]\label{theorem-uniqueness-numerical-sheme}
	Assume the hypotheses of Theorem \ref{theorem-existence-num-scheme}. Let $k_{\dagger}\in\mathbb{N}$ be given as in Theorem \ref{theorem-existence-num-scheme}. Suppose that for each $k,N\in\mathbb{N}$, with $k\geq k_{\dagger}$, we are given triples $(X_i^{k,N},Y_i^{k,N},\rho_i^{k,N})$, $i=1,2$, that satisfy the scheme \eqref{numerical-scheme} with initial random variables $X_{0,i}^N\in \mathcal{R}_{d}(\mathcal{A}_N,\mathcal{O}_N)$ for $i=1,2$. Then
	\begin{equation}\label{theorem-uniqueness-bound}
		\begin{split}
			&\sup_{t\in [0,T]}\mathcal{W}_2\left(\rho_1^{k,N}(t),\rho_2^{k,N}(t)\right)
            +\sup_{t\in [0,T]}\left(\|(Y_1^{k,N}-Y_2^{k,N})(t)\|_{L^2(\Omega;\mathbb{R}^d)}+\|(X_1^{k,N}-X_2^{k,N})(t)\|_{L^2(\Omega;\mathbb{R}^d)}\right)\\
			&\qquad\qquad\qquad\qquad\qquad\qquad\qquad\qquad\qquad\qquad\qquad\qquad\lesssim \norm{X_{0,1}^N-X_{0,2}^N}_{L^2(\Omega;\mathbb{R}^d)}
		\end{split}
	\end{equation}
	where $k_{\dagger}$ and the hidden constant above depend only on the time horizon $T$, the uniform convexity constant $c_0>0$, and the Lipschitz constants associated with $D_pH$, $D_xH$ and $D_xg$. In particular, the numerical scheme \eqref{numerical-scheme} admits a unique solution for all $k\in\mathbb{N}$ sufficiently large and for all $N\in\mathbb{N}$.
\end{theorem}
The proofs of Theorem \ref{theorem-existence-num-scheme} and Theorem \ref{theorem-uniqueness-numerical-sheme} are given in Section \ref{sec-existence} and Section \ref{sec-cont-dep-gen-discrete-ham-sys}, respectively. In both instances, the analysis we present is based on a continuous dependence bound that we prove in Theorem \ref{theorem-cont-dependence-discrete-Hamiltonian-sys} of Section \ref{sec-cont-dep-gen-discrete-ham-sys} for a class of discrete Hamiltonian systems under displacement monotonicity. 

Furthermore, we prove the convergence of the numerical scheme \eqref{numerical-scheme} together with an asymptotic error bound for the approximations.
\begin{theorem}[Convergence of numerical approximations]\label{theorem-convergence-k-N-joint}
	In addition to the hypotheses of Theorem \ref{theorem-existence-num-scheme}, assume that the Hamiltonian satisfies \eqref{ass-H:5} and that the Lagrangian $L$ and $g$ satisfy the growth assumption \eqref{ass-L-g}. Let $X_0$ be an initial random variable distributed according to $\rho_0$ and assume that $X_0^N\to X_0$ in $\LLspace$ as $N\to\infty$. Furthermore, let $(X,Y,\rho)\in C^1([0,T];L^2(\Omega;\mathbb{R}^d))\times C^1([0,T];L^2(\Omega;\mathbb{R}^d))\times C([0,T];\mathcal{P}_2(\mathbb{R}^d))$ denote the triple which uniquely satisfies the continuous Hamiltonian system \eqref{definition-soln-of-continuous-pb} with initial random variable $X_0$. Then, there exists $k_{\dagger}\in\mathbb{N}$ such that
	\begin{multline}\label{main-error-bound-num-scheme}
				\sup_{t\in [0,T]}\mathcal{W}_2\left(\rho(t),\rho^{k,N}(t)\right)+\sup_{t\in [0,T]}\|(-D_xu(\cdot,X)-Y^{k,N})(t)\|_{L^2(\Omega;\mathbb{R}^d)}+\sup_{t\in [0,T]}\|(X-X^{k,N})(t)\|_{L^2(\Omega;\mathbb{R}^d)}\\
			\lesssim \|X_0-X_0^N\|_{\LLspace}+ \tau_k
	\end{multline}
	for all $N\in\mathbb{N}$ and $k\geq k_{\dagger}$, 
	where $k_{\dagger}$ and the hidden constant above depend only on the time horizon $T$, the uniform convexity constant $c_0>0$, and the Lipschitz constants associated with $D_pH$, $D_xH$ and $D_xg$. In particular, 
	\begin{equation}\label{Y-rho-convergence}
		Y^{k,N}\to Y=-D_xu(\cdot,X)\text{ in }L^{\infty}([0,T];L^2(\Omega;\mathbb{R}^d)),\quad\text{ and }\quad \rho^{k,N}\to \rho\text{ in }L^{\infty}([0,T];\mathcal{P}_2(\mathbb{R}^d))
	\end{equation}
	as $k,N\to\infty$.
\end{theorem}
If in addition $L$ and $g$ are locally Lipschitz continuous in the measure variable in the sense of \eqref{ass-H:5} and \eqref{ass-g:2}, we obtain an error estimate for the value function approximation $u^{k,N}$ given by \eqref{u-opt-trajectory-rep-N-k}.
\begin{theorem}[Approximation of the value function]\label{theorem-convergence-k-N-joint-u-val}
{Let  $u$ denote the value function in the continuous MFG, represented along optimal trajectories in \eqref{u-opt-trajectory-rep} and let $u^{k,N}$ stand for the discretised value function defined in \eqref{u-opt-trajectory-rep-N-k}.}
    In addition to the hypotheses of Theorem \ref{theorem-convergence-k-N-joint}, assume that $L$ and $g$ satisfy  \eqref{L-loc-meas-contin} and \eqref{g-loc-meas-contin}, respectively. Then, there exist $k_{\dagger}^*,N_{\dagger}\in\mathbb{N}$ such that 
    \begin{equation}\label{val-func-error-bound}
        \sup_{t\in [0,T]}\|(u(\cdot,X)-u^{k,N})(t)\|_{L^1(\Omega;\mathbb{R}^d)}
			\lesssim \|X_0-X_0^N\|_{\LLspace}+ \tau_k,
    \end{equation} where $u^{k,N}$ is defined by \eqref{u-opt-trajectory-rep-N-k}, for each $N\geq N_{\dagger}$ and $k\geq k_{\dagger}^*$.
\end{theorem}
The proofs of Theorem \ref{theorem-convergence-k-N-joint} and Theorem \ref{theorem-convergence-k-N-joint-u-val} are given in Section \ref{sec-convergence}.
\begin{remark}
	For the case when the initial measure $\rho_0$ is an measure supported on a finite number of points, the sequence $\{X_0^N\}_{N\in\mathbb{N}}$ can be taken to be constant and equal to $X_0$ (see Remark \ref{rem-rv-finite-support-rho-N}), and so the spatial error term $\|X_0-X_0^N\|_{\LLspace}$ is not present in the error bounds \eqref{main-error-bound-num-scheme},\eqref{val-func-error-bound} for all $k$ sufficiently large. Thus, the asymptotic rate of convergence of the numerical scheme is order $\tau_k$, which is optimal in view of the $H^1$-regularity in time for both the state flow $X$ and momentum flow $Y$ when $\rho_0$ is supported on a finite number of points. Otherwise, the spatial error term in \eqref{main-error-bound-num-scheme} is non-zero in general when $\text{card}(\text{spt}(\rho_0))=\infty$.
\end{remark}

As a first corollary of Theorem \ref{theorem-convergence-k-N-joint}, we deduce rates of convergence for the numerical scheme \eqref{numerical-scheme} when the approximations $\{X_0^N\}_{N\in\mathbb{N}}$ are given by application of suitable discrete optimal transport maps to $X_0$, as illustrated in Example \ref{eg-2}.

\begin{corollary}\label{cor-semi-discrete-opt-transport}
	In addition to the hypothesis of Theorem \ref{theorem-convergence-k-N-joint}, suppose that $\rho_0\in\mathcal{P}_q(\mathbb{R}^d)$, for some $q>2$, is absolutely continuous with respect to $\mathbb{P}$. Furthermore, let $(X,Y,\rho)\in C^1(0,T;L^2(\Omega;\mathbb{R}^d))\times C^1(0,T;L^2(\Omega;\mathbb{R}^d))\times C([0,T];\mathcal{P}_2(\mathbb{R}^d))$ denote the triple which uniquely satisfies the continuous Hamiltonian system \eqref{definition-soln-of-continuous-pb} with initial random variable $X_0$ distributed according to $\rho_0$. Suppose that the sequences $\{\rho_0^N\}_{N\in\mathbb{N}}$ and $\{X_0^N\}_{N\in\mathbb{N}}$ are constructed as in Example \ref{eg-2}.  Then, there exists $k_{\dagger}\in\mathbb{N}$ such that
	\begin{multline}\label{cor-1-eg-rate-of-conv}
		\sup_{t\in [0,T]}\mathcal{W}_2\left(\rho(t),\rho^{k,N}(t)\right)+\sup_{t\in [0,T]}\|(Y-Y^{k,N})(t)\|_{L^2(\Omega;\mathbb{R}^d)}+\sup_{t\in [0,T]}\|(X-X^{k,N})(t)\|_{L^2(\Omega;\mathbb{R}^d)}
		\\ 
		\lesssim 
		\begin{dcases}
			\max\{N^{-1/d},  N^{-1/2 + 1/q}\} +\tau_k, & {\rm if\ } d \ne \frac{2q}{q-2},\\
			N^{-1/d} \log(1+N)^{1/d} + \tau_k,& {\rm if\ } d = \frac{2q}{q-2},
		\end{dcases} 
	\end{multline}
	for all $k\geq k_{\dagger}$ and $N\in\mathbb{N}$. The hidden constant above depends only on $M_q(\rho_0)$, the time horizon $T$, the uniform convexity constant $c_0>0$, and the Lipschitz constants associated with $D_pH$, $D_xH$ and $D_xg$. If in addition $L$ and $g$ satisfy  \eqref{L-loc-meas-contin} and \eqref{g-loc-meas-contin}, respectively, then  
    \begin{equation}\label{cor-1-eg-rate-of-conv-u-val}
		\sup_{t\in [0,T]}\|(u(\cdot,X)-u^{k,N})(t)\|_{L^1(\Omega;\mathbb{R}^d)}
		\lesssim 
		\begin{dcases}
			\max\{N^{-1/d},  N^{-1/2 + 1/q}\} +\tau_k, & {\rm if\ } d \ne \frac{2q}{q-2},\\
			N^{-1/d} \log(1+N)^{1/d} + \tau_k,& {\rm if\ } d = \frac{2q}{q-2},
		\end{dcases} 
	\end{equation} for all $k,N$ sufficiently large.
\end{corollary}
 
In view of Example \ref{eg-3}, we obtain an explicit rate of convergence in the alternative case where the approximations $\{X_0^N\}_{N\in\mathbb{N}}$ are constructed via quantization of $X_0$.
\begin{corollary}\label{cor-quantization-of-rvs}
	In addition to the hypothesis of Theorem \ref{theorem-convergence-k-N-joint}, suppose that $\rho_0$ has a non-trivial absolutely continuous part, and that the initial random variable $X_0$ is in $L^{2+\delta}(\Omega;\mathbb{R}^d)$ for some $\delta>0$. Furthermore, let $(X,Y,\rho)\in C^1(0,T;L^2(\Omega;\mathbb{R}^d))\times C^1(0,T;L^2(\Omega;\mathbb{R}^d))\times C([0,T];\mathcal{P}_2(\mathbb{R}^d))$ denote the triple which uniquely satisfies the continuous Hamiltonian system \eqref{definition-soln-of-continuous-pb} with initial random variable $X_0$.  Suppose that the sequences $\{\rho_0^N\}_{N\in\mathbb{N}}$ and $\{X_0^N\}_{N\in\mathbb{N}}$ are constructed as in Example \ref{eg-3} for some $N_{\dagger}\in\mathbb{N}$ and sequences of weights $\{\mathcal{A}_N\}_{N\geq N_{\dagger}}$, partitions $\{\mathcal{O}_N\}_{N\geq N_{\dagger}}$ of $\Omega$ and simple random variables $\{X_0^N\}_{N\geq N_{\dagger}}$ such that $X_0^N\in \mathcal{R}_d(\mathcal{A}_N,\mathcal{O}_N)$ for each $N\geq N_{\dagger}$. Then, there exists $k_{\dagger}\in\mathbb{N}$ such that
	\begin{multline}\label{cor-2-eg-rate-of-conv}
		\sup_{t\in [0,T]}\mathcal{W}_2\left(\rho(t),\rho^{k,N}(t)\right)+\sup_{t\in [0,T]}\|(Y-Y^{k,N})(t)\|_{L^2(\Omega;\mathbb{R}^d)}+\sup_{t\in [0,T]}\|(X-X^{k,N})(t)\|_{L^2(\Omega;\mathbb{R}^d)}
		\lesssim N^{-1/d}+\tau_k 
	\end{multline}
	for all $k\geq k_{\dagger}$ and $N\geq N_{\dagger}$. The hidden constant above depends only on $\|X_0\|_{L^{2+\delta}(\Omega;\mathbb{R}^d)}$, $k_{\dagger}$, the time horizon $T$, the uniform convexity constant $c_0>0$, and the Lipschitz constants associated with $D_pH$, $D_xH$ and $D_xg$.  If in addition $L$ and $g$ satisfy  \eqref{L-loc-meas-contin} and \eqref{g-loc-meas-contin}, respectively, then  
    \begin{equation}\label{cor-2-eg-rate-of-conv-u-val}
		\sup_{t\in [0,T]}\|(u(\cdot,X)-u^{k,N})(t)\|_{L^1(\Omega;\mathbb{R}^d)}
		\lesssim N^{-1/d}+\tau_k 
	\end{equation} for all $k,N$ sufficiently large.
\end{corollary}

\begin{remark}(On the composite rates of convergence)
    With respect to the temporal component of the error in the upper bound in Theorem \ref{theorem-convergence-k-N-joint}, the $\tau_k$ term is asymptotically optimal for piecewise constant discretizations since $X$ and $Y$ are $W^{1,\infty}$-in-time flows. The asymptotic convergence quality of the approximations therefore, in general, boils down to the freedom in the convergence behaviour of the spatial error term $\|X_0-X_0^N\|_{\LLspace}$ for $N$ large. For instance, the convergence error bound achieved by the  semi-discrete optimal transport approximations for $X_0$ given in Example \ref{eg-2} is optimal among initial measures $\rho_0$ in $\mathcal{P}_q(\mathbb{R}^d)$, $q>2$, except for some critical cases \cite{seeger2025error}. In addition, by Example \ref{eg-3} approximations of $X_0$ that are obtained by piecewise constant optimal vector quantization converge strongly in the $L^2$-norm, with the error asymptotically residing within the closed interval $[N^{-\frac{1}{\mathcal(1-\epsilon){D}_L(\rho_0)}},N^{-\frac{1}{\mathcal{D}_U(\rho_0)+\epsilon\mathcal{D}_L(\rho_0)}}]$, given $0<\epsilon<1$, which depends on the lower and upper quantization dimensions of $\rho_0.$
\end{remark}

\section{Uniqueness for a class of discrete Hamiltonian systems}\label{sec-cont-dep-gen-discrete-ham-sys}
The main result we establish in this section is a continuous dependence bound for a class of generalized discrete Hamiltonian systems, which includes the numerical scheme \eqref{numerical-scheme} as a particular case. This class of systems {is} given as follows. Recall that the sequences $\{\rho_0^N\}_{N\in\mathbb{N}}$ and $\{X_0^N\}_{N\in\mathbb{N}}$, and hence the sequences $\{\mathcal{A}_N\}_{N\in\mathbb{N}}$ and $\{\mathcal{O}_N\}_{N\in\mathbb{N}}$, are assumed given and fixed. Let $k,N\in\mathbb{N}$ and let $\iota_k\coloneqq \{t_1,t_2,\cdots,t_{M_k-1}\}$. For each $\lambda \in [0,1]$, $F,G\in \mathbb{V}_k(0,T;\mathcal{R}_{d}(\mathcal{A}_N,\mathcal{O}_N))$ and $X_0^N\in  \mathcal{R}_{d}(\mathcal{A}_N,\mathcal{O}_N)$, we define the following problem: \emph{find  $\mathscr{X}^{k,N}\in \mathbb{V}_k^+([0,T];\mathcal{R}_{d}(\mathcal{A}_N,\mathcal{O}_N))$, $\mathscr{Y}^{k,N}\in \mathbb{V}_k^-([0,T];\mathcal{R}_{d}(\mathcal{A}_N,\mathcal{O}_N))$ and $\varrho^{k,N}\in\mathbb{P}_k([0,T];\mathcal{P}_{d,N})$ such that}
\begin{equation}\label{numerical-scheme-gen}
	\left\{\begin{aligned}
		\partial_t\mathcal{I}_+^k\mathscr{X}^{k,N}(t) &= \lambda D_pH\left(\mathscr{X}^{k,N}(t),\mathscr{Y}^{k,N}(t),\varrho^{k,N}(t)\right)+F(t),&&t\in (0,T)\backslash \iota_k, 
		\\
		\partial_t\mathcal{I}_-^k\mathscr{Y}^{k,N}(t) &=-\lambda D_xH\left(\mathscr{X}^{k,N}(t),\mathscr{Y}^{k,N}(t),\varrho^{k,N}(t)\right)+G(t),&&t\in (0,T)\backslash \iota_k, 
		\\
		\varrho^{k,N}(t) &= (\mathscr{X}^{k,N}(t))_{\#}\mathbb{P}, && t\in[0,T],
		\\
		\mathscr{X}^{k,N}(0)&=\lambda X_0^N,\quad \mathscr{Y}^{k,N}(T) = - \lambda D_xg(\mathscr{X}^{k,N}(T),\varrho^{k,N}(T))&&\text{ }
	\end{aligned}\right.
\end{equation}\emph{where $\iota_k\coloneqq \{t_{m}:q=1,2,\cdots, M_k-1\}$.} Systems of this form enjoy the following continuous dependence property, which is the main result of this section.

\begin{theorem}[Continuous dependence for generalized discrete Hamiltonian system]\label{theorem-cont-dependence-discrete-Hamiltonian-sys}
	Let the Hamiltonian $H$ satisfy the regularity assumptions \eqref{ass-H:1}, \eqref{ass-H:2}, \eqref{ass-H:3}, \eqref{ass-H:4} and the displacement monotonicity condition \eqref{ass-H:disp-mono-cond-H}. Let the terminal cost $g$ satisfy the regularity assumptions \eqref{ass-g:1}, \eqref{ass-g:2} and the displacement monotonicity condition \eqref{ass-g:disp-mono-cond-g}. Let $N,M,k\in\mathbb{N}$ and 
	\begin{subequations}
		\begin{align}
			&\lambda_1\in (0,1],  \quad F_1,G_1\in \mathbb{V}_k(0,T;\mathcal{R}_{d}(\mathcal{A}_N,\mathcal{O}_N)), \quad X_{0,1}^N\in  \mathcal{R}_{d}(\mathcal{A}_N,\mathcal{O}_N),\label{data-line-1}
			\\
			&\lambda_2\in [0,1], \quad F_2,G_2\in \mathbb{V}_k(0,T;\mathcal{R}_{d}(\mathcal{A}_M,\mathcal{O}_M)),\quad X_{0,2}^M\in  \mathcal{R}_{d}(\mathcal{A}_M,\mathcal{O}_M)\label{data-line-2}
		\end{align}
	\end{subequations}
	be given, and suppose that the triples
	\begin{subequations}
		\begin{align}
			(\mathscr{X}_1^{k,N},\mathscr{Y}_1^{k,N},\varrho_1^{k,N})&\in \mathbb{V}_k^+([0,T];\mathcal{R}_{d}(\mathcal{A}_N,\mathcal{O}_N))\times \mathbb{V}_k^-([0,T];\mathcal{R}_{d}(\mathcal{A}_N,\mathcal{O}_N))\times \mathbb{P}_k([0,T];\mathcal{P}_{d,N})\\
			(\mathscr{X}_2^{k,M},\mathscr{Y}_2^{k,M},\varrho_{2}^{k,M})&\in \mathbb{V}_k^+([0,T];\mathcal{R}_{d}(\mathcal{A}_M,\mathcal{O}_M))\times \mathbb{V}_k^-([0,T];\mathcal{R}_{d}(\mathcal{A}_M,\mathcal{O}_M))\times \mathbb{P}_k([0,T];\mathcal{P}_{d,M})
		\end{align}
	\end{subequations}  each satisfy the discrete Hamiltonian system \eqref{numerical-scheme-gen} with data \eqref{data-line-1} and \eqref{data-line-2}, respectively. There exists $k_{\dagger}\in\mathbb{N}$ such that, for each $k\geq k_{\dagger}$, we have
	\begin{equation}\label{gen-disp-mono-cont-bound}
		\begin{split}
			\sup_{t\in [0,T]}\mathcal{W}_2&\left(\varrho_1^{k,N}(t),\varrho_{2}^{k,M}(t)\right)+\sup_{0\leq t\leq T}\left(\norm{(\mathscr{X}_1^{k,N}-\mathscr{X}_2^{k,M})(t)}_{L^2(\Omega;\mathbb{R}^d)}+\norm{(\mathscr{Y}_1^{k,N}-\mathscr{Y}_2^{k,M})(t)}_{L^2(\Omega;\mathbb{R}^d)}\right)
			\\
			&\lesssim \norm{X_{0,1}^N-\lambda_2\lambda_1^{-1}X_{0,2}^M}_{L^2(\Omega;\mathbb{R}^d)}+\lambda_1^{-1}\left(\|F_1-F_2\|_{L^2(0,T;\LLspace)}+\|G_1-G_2\|_{L^2(0,T;\LLspace)}\right)
			\\
            &+\abs{\lambda_2\lambda_1^{-1}-1}\left(\norm{D_xg(\mathscr{X}_2^{k,M}(T),\varrho_{2}^{k,M}(T))}_{L^{2}(\Omega;\mathbb{R}^d)}+\left\|D_pH(\mathscr{X}_2^{k,M},\mathscr{Y}_2^{k,M},\varrho_{2}^{k,M})\right\|_{L^2(0,T;\LLspace)}\right.
			\\
			&\left.\qquad\qquad\qquad\qquad\qquad\qquad\qquad\qquad+\left\|D_xH(\mathscr{X}_2^{k,M},\mathscr{Y}_2^{k,M},\varrho_{2}^{k,M})\right\|_{L^2(0,T;\LLspace)}\right)
		\end{split}
	\end{equation}
	The number $k_{\dagger}\in\mathbb{N}$ and the hidden constant in \eqref{gen-disp-mono-cont-bound} depend only on the uniform convexity constant $c_0$, $T$, and the Lipschitz constants of $D_pH, D_xH$ and $D_xg$. 
\end{theorem}

We therefore obtain Theorem \ref{theorem-uniqueness-numerical-sheme} as a particular corollary of Theorem  \ref{theorem-cont-dependence-discrete-Hamiltonian-sys}.
\begin{proof}[Proof of Theorem \ref{theorem-uniqueness-numerical-sheme}]
	We deduce the desired bound \eqref{theorem-uniqueness-bound} by setting $N=M\in\mathbb{N}$, $\lambda_1=\lambda_2=1$, $F_1=F_2=G_1=G_2=0$ in $\mathbb{V}_k(0,T;\mathcal{R}_{d}(\mathcal{A}_N,\mathcal{O}_N))$ in \eqref{gen-disp-mono-cont-bound} of Theorem . If in addition we set $X_{0,1}^N\coloneqq X_{0,2}^N$ in $\mathcal{R}_{d}(\mathcal{A}_N,\mathcal{O}_N)$, the uniqueness of solutions to the numerical scheme \eqref{numerical-scheme} is immediate.   
\end{proof}

\begin{remark}[Uniqueness of solutions for generalized discrete Hamiltonian system]\label{remark-uniqueness-gen-ham-sys}
	Theorem  \ref{theorem-cont-dependence-discrete-Hamiltonian-sys} implies that the generalized discrete Hamiltonian system \eqref{numerical-scheme-gen} can admit at most one solution for each fixed $N\in\mathbb{N}$, $\lambda \in (0,1]$, $F,G\in \mathbb{V}_k(0,T;\mathcal{R}_{d}(\mathcal{A}_N,\mathcal{O}_N))$ and $X_0^N\in  \mathcal{R}_{d}(\mathcal{A}_N,\mathcal{O}_N)$, where $k\geq k_{\dagger}$. A direct calculation shows that uniqueness for the system also extends to the case when $\lambda=0$, with the explicit solution being linear in the terms $F$ and $G$. We note, in particular, that in the special case when $F=G=0$ in $\mathbb{V}_k(0,T;\mathcal{R}_{d}(\mathcal{A}_N,\mathcal{O}_N))$, $\lambda=0$ and any $X_0^N$ is fixed, the triple $(\mathscr{X}^{k,N},\mathscr{Y}^{k,N},\varrho^{k,N})=(0,0,\delta_0)$ uniquely satisfies the system \eqref{numerical-scheme-gen}.
\end{remark}

The proof of Theorem  \ref{theorem-cont-dependence-discrete-Hamiltonian-sys} that we present can be considered an extension of \cite[Theorem 4.5]{meszaros2024mean} to discrete Hamiltonian systems of the form \eqref{numerical-scheme-gen}. In particular, the argument we employ shows how displacement monotonicity conditions on $H$ and $g$, together with strong convexity of the Hamiltonian, ensure stability of solutions to discrete Hamiltonian systems under continuous affine perturbations. 

The first step in the proof of Theorem  \ref{theorem-cont-dependence-discrete-Hamiltonian-sys} is the following preliminary stability result. 
\begin{lemma}[Displacement monotonicity stability]\label{lemma-step-1-discrete-uniqueness}
		Let the Hamiltonian $H$ satisfy the regularity assumption \eqref{ass-H:1} and the displacement monotonicity condition \eqref{ass-H:disp-mono-cond-H}. Let the terminal cost $g$ satisfy the regularity assumption \eqref{ass-g:1} and the displacement monotonicity condition \eqref{ass-g:disp-mono-cond-g}. Let $N,M,k\in\mathbb{N}$ and 
	\begin{subequations}
		\begin{align}
			&\lambda_1\in (0,1],  \quad F_1,G_1\in \mathbb{V}_k(0,T;\mathcal{R}_{d}(\mathcal{A}_N,\mathcal{O}_N)), \quad X_{0,1}^N\in  \mathcal{R}_{d}(\mathcal{A}_N,\mathcal{O}_N),\label{data-line-1'}
			\\
			&\lambda_2\in [0,1], \quad F_2,G_2\in \mathbb{V}_k(0,T;\mathcal{R}_{d}(\mathcal{A}_M,\mathcal{O}_M)),\quad X_{0,2}^M\in  \mathcal{R}_{d}(\mathcal{A}_M,\mathcal{O}_M)\label{data-line-2'}
		\end{align}
	\end{subequations}
	be given, and suppose that the triples
	\begin{subequations}
		\begin{align}
			(\mathscr{X}_1^{k,N},\mathscr{Y}_1^{k,N},\varrho_1^{k,N})&\in \mathbb{V}_k^+([0,T];\mathcal{R}_{d}(\mathcal{A}_N,\mathcal{O}_N))\times \mathbb{V}_k^-([0,T];\mathcal{R}_{d}(\mathcal{A}_N,\mathcal{O}_N))\times \mathbb{P}_k([0,T];\mathcal{P}_{d,N}),\label{triple-1}\\
			(\mathscr{X}_2^{k,M},\mathscr{Y}_2^{k,M},\varrho_{2}^{k,M})&\in \mathbb{V}_k^+([0,T];\mathcal{R}_{d}(\mathcal{A}_M,\mathcal{O}_M))\times \mathbb{V}_k^-([0,T];\mathcal{R}_{d}(\mathcal{A}_M,\mathcal{O}_M))\times \mathbb{P}_k([0,T];\mathcal{P}_{d,M})\label{triple-2}
		\end{align}
	\end{subequations}  each satisfy the discrete Hamiltonian system \eqref{numerical-scheme-gen} with data \eqref{data-line-1'} and \eqref{data-line-2'}, respectively.
	For all $n\in\{0,1,2,\cdots,M_k\}$, we have
		\begin{equation}\label{step-one-bound}
		\begin{split}
			&\int_0^{t_n} \lambda_1 \mathbb{E}\left[\left( D_pH(\mathscr{X}_1^{k,N},\mathscr{Y}_1^{k,N},\varrho_1^{k,N}) -  D_pH(\mathscr{X}_2^{k,M},\mathscr{Y}_2^{k,M},\varrho_{2}^{k,M})\right)\cdot (\mathscr{Y}_1^{k,N}-\mathscr{Y}_2^{k,M})\right.
			\\
			&\left.\qquad\qquad-(\mathscr{X}_1^{k,N}-\mathscr{X}_2^{k,M})\cdot\left( D_xH(\mathscr{X}_1^{k,N},\mathscr{Y}_1^{k,N},\varrho_1^{k,N}) -  D_xH(\mathscr{X}_2^{k,M},\mathscr{Y}_2^{k,M},\varrho_{2}^{k,M})\right)\right]\mathrm{d}s
			\\
			&\leq
			(\lambda_2-\lambda_1)\expectation{D_xg(\mathscr{X}_2^{k,M}(T),\varrho_{2}^{k,M}(T))\cdot (\mathscr{X}_1^{k,N}-\mathscr{X}_2^{k,M})(T)} \\&\qquad- \mathbb{E}\left[(\mathscr{X}_1^{k,N}-\mathscr{X}_2^{k,M})(0)\cdot (\mathscr{Y}_1^{k,N}-\mathscr{Y}_2^{k,M})(0)\right]-\mathfrak{I}(0,T),
		\end{split}
	\end{equation} 
	where the map $\mathfrak{I}:[0,T]^2\to\mathbb{R}$ is defined via
	\begin{equation}\label{I-def}
		\begin{split}
			\mathfrak{I}(a,b)\coloneqq &
			\int_{a}^b (\lambda_1-\lambda_2) \mathbb{E}\left[  D_pH(\mathscr{X}_2^{k,M},\mathscr{Y}_2^{k,M},\varrho_{2}^{k,M})\cdot (\mathscr{Y}_1^{k,N}-\mathscr{Y}_2^{k,M})\right.\\
			&\left.\qquad\qquad\qquad\qquad-(\mathscr{X}_1^{k,N}-\mathscr{X}_2^{k,M})\cdot  D_xH(\mathscr{X}_2^{k,M},\mathscr{Y}_2^{k,M},\varrho_{2}^{k,M})\right]\mathrm{d}s
			\\
			&\qquad+\int_{a}^b\mathbb{E}\left[(F_1-F_2)\cdot (\mathscr{Y}_1^{k,N}-\mathscr{Y}_2^{k,M})+(\mathscr{X}_1^{k,N}-\mathscr{X}_2^{k,M})\cdot (G_1-G_2)\right]\mathrm{d}s\\
			&\qquad\qquad\qquad\qquad\qquad\quad \forall a,b\in [0,T].
		\end{split}
	\end{equation}
\end{lemma}
\begin{remark} 
Notice that the upper bound in \eqref{step-one-bound} is uniform in the temporal node index $n\in \{0,1,2,\cdots, M_k\}$. As we now show, this result is a basic stability property of the discrete Hamiltonian system \eqref{numerical-scheme-gen} that is due solely to the displacement monotonicity of $H$ and $g$. 
\end{remark} 

The proofs of the upcoming lemmata will employ the following discrete integration-by-parts formulae that are satisfied by the time reconstruction operators $\Ip$ and $\In$ in \eqref{Ip-op}, \eqref{In-op} in the case when $A=L^2(\Omega;\mathbb{R}^d)$.
\begin{lemma}\label{lemma-ibp-time-formulae} Given $n,m\in \{0,1,\cdots,{M_k}\}$ we have
	\begin{multline}\label{eq:discrete_ibp-expectation_pn}
		\int_{t_m}^{t_n}\mathbb{E}\left[\partial_s\Ip V_1(s)\cdot W_1(s)\right] + \mathbb{E}\left[V_1(s) \cdot\p_s \In W_1(s)\right]\ds
		=\mathbb{E}[V_1(t_n)\cdot W_1(t_n)]-\mathbb{E}[V_1(t_m)\cdot W_1(t_m)], 
	\end{multline}

    \begin{multline}\label{eq:discrete_ibp-expectation_pp}
		\int_{t_m}^{t_n}\mathbb{E}\left[\partial_s\Ip V_1(s)\cdot V_2(s)\right] + \mathbb{E}\left[V_1(s) \cdot\p_s \Ip V_2(s)\right]\ds \\=\mathbb{E}[V_1(t_n)\cdot V_2(t_n)]-\mathbb{E}[V_1(t_m)\cdot V_2(t_m)] + \sum_{j=m+1}^n\mathbb{E}\left[\jump{V_1}_{j-1}\cdot\jump{V_2}_{j-1}\right],
	\end{multline}
	\begin{multline}\label{eq:discrete_ibp-expectation_nn}
		\int_{t_m}^{t_n}\mathbb{E}\left[\partial_s\In W_1(s)\cdot W_2(s)\right] + \mathbb{E}\left[W_1(s) \cdot\p_s \In W_2(s)\right]\ds \\=\mathbb{E}[W_1(t_n)\cdot W_2(t_n)]-\mathbb{E}[W_1(t_m)\cdot W_2(t_m)] - \sum_{j=m+1}^n\mathbb{E}\left[\jump{W_1}_{j}\cdot\jump{W_2}_{j}\right],
	\end{multline} 
	for all $V_1,V_2\in \mathbb{V}_k^+([0,T];L^2(\Omega;\mathbb{R}^d))$ and $W_1,W_2\in  \mathbb{V}_k^-([0,T];L^2(\Omega;\mathbb{R}^d))$ .
\end{lemma}
For completeness, the proof of this lemma is included in Appendix \ref{app-aux-disc-temp-ana}. 

\begin{proof}[Proof of Lemma \ref{lemma-step-1-discrete-uniqueness}]
	Let $k,N,M\in\mathbb{N}$ be given. Suppose that the triples $$(\mathscr{X}_1^{k,N},\mathscr{Y}_1^{k,N},\varrho_1^{k,N})\ {\rm and}\   (\mathscr{X}_2^{k,M},\mathscr{Y}_2^{k,M},\varrho_{2}^{k,M})$$ in \eqref{triple-1}, \eqref{triple-2} each satisfy the discrete Hamiltonian system \eqref{numerical-scheme-gen} with data \eqref{data-line-1} and  \eqref{data-line-2}, respectively. Let $\Delta \mathscr{X}_{N,M}^{k}\coloneqq \mathscr{X}_1^{k,N}-\mathscr{X}_2^{k,M}\in \mathbb{V}_k^+([0,T];L^2(\Omega;\mathbb{R}^d))$ and $\Delta\mathscr{Y}_{N,M}^{k}\coloneqq \mathscr{Y}_1^{k,N}-\mathscr{Y}_2^{k,M}\in \mathbb{V}_k^-([0,T];L^2(\Omega;\mathbb{R}^d))$. For each $n\in \{0,1,2,\cdots,M_k\}$ the systems satisfied by $(\mathscr{X}_1^{k,N},\mathscr{Y}_1^{k,N},\varrho_1^{k,N}),(\mathscr{X}_2^{k,M},\mathscr{Y}_2^{k,M},\varrho_{2}^{k,M})$ respectively imply
	\begin{equation}
		\begin{split}
			&\int_{t_n}^T\mathbb{E}\left[\partial_t\Ip \Delta \mathscr{X}_{N,M}^{k}\cdot \Delta\mathscr{Y}_{N,M}^{k} \right]\mathrm{d}s
			\\
			&\qquad=\int_{t_n}^T \mathbb{E}\left[\left(\lambda_1 D_pH(\mathscr{X}_1^{k,N},\mathscr{Y}_1^{k,N},\varrho_1^{k,N}) - \lambda_2 D_pH(\mathscr{X}_2^{k,M},\mathscr{Y}_2^{k,M},\varrho_{2}^{k,M})\right)\cdot \Delta\mathscr{Y}_{N,M}^{k}\right]\mathrm{d}s
			\\
			&\qquad\qquad+\int_{t_n}^T\expectation{(F_1-F_2)\cdot \Delta\mathscr{Y}_{N,M}^{k}}\ds
		\end{split}
	\end{equation}and 
	\begin{equation}
		\begin{split}
			&\int_{t_n}^T\mathbb{E}\left[ \Delta \mathscr{X}_{N,M}^{k} \cdot\p_t \In \Delta\mathscr{Y}_{N,M}^{k}\right]\mathrm{d}s
			\\
			&=-\int_{t_n}^T \mathbb{E}\left[\left(\lambda_1 D_xH(\mathscr{X}_1^{k,N},\mathscr{Y}_1^{k,N},\varrho_1^{k,N}) - \lambda_2 D_xH(\mathscr{X}_2^{k,M},\mathscr{Y}_2^{k,M},\varrho_{2}^{k,M})\right)\cdot \Delta \mathscr{X}_{N,M}^{k}\right]\mathrm{d}s
			\\
			&\qquad+\int_{t_n}^T\expectation{(G_1-G_2)\cdot \Delta\mathscr{Y}_{N,M}^{k}}\ds.
		\end{split} 
	\end{equation}
	 By summing these two equations we get
	\begin{equation}
		\begin{split}
			&\int_{t_n}^T\mathbb{E}\left[\partial_t\Ip \Delta \mathscr{X}_{N,M}^{k}\cdot \Delta\mathscr{Y}_{N,M}^{k} + \Delta \mathscr{X}_{N,M}^{k} \cdot\p_t \In \Delta\mathscr{Y}_{N,M}^{k}\right]\mathrm{d}s
			\\
			&=\int_{t_n}^T \mathbb{E}\left[\left(\lambda_1 D_pH(\mathscr{X}_1^{k,N},\mathscr{Y}_1^{k,N},\varrho_1^{k,N}) - \lambda_2 D_pH(\mathscr{X}_2^{k,M},\mathscr{Y}_2^{k,M},\varrho_{2}^{k,M})\right)\cdot \Delta\mathscr{Y}_{N,M}^{k}\right.
			\\
			&\left.\qquad\qquad-\Delta \mathscr{X}_{N,M}^{k}\cdot\left(\lambda_1 D_xH(\mathscr{X}_1^{k,N},\mathscr{Y}_1^{k,N},\varrho_1^{k,N}) - \lambda_2 D_xH(\mathscr{X}_2^{k,M},\mathscr{Y}_2^{k,M},\varrho_{2}^{k,M})\right)\right]\mathrm{d}s
			\\
			&\qquad+\int_{t_n}^T\mathbb{E}\left[(F_1-F_2)\cdot \Delta\mathscr{Y}_{N,M}^{k}+\Delta \mathscr{X}_{N,M}^{k}\cdot (G_1-G_2)\right]\mathrm{d}s,
		\end{split}
	\end{equation}
	which is the same as
	\begin{equation}\label{pre-disp-mono-app-one}
		\begin{split}
			&\int_{t_n}^T\mathbb{E}\left[\partial_t\Ip \Delta \mathscr{X}_{N,M}^{k}\cdot \Delta\mathscr{Y}_{N,M}^{k} + \Delta \mathscr{X}_{N,M}^{k} \cdot\p_t \In \Delta\mathscr{Y}_{N,M}^{k}\right]\mathrm{d}s
			\\
			&=\int_{t_n}^T \lambda_1 \mathbb{E}\left[\left( D_pH(\mathscr{X}_1^{k,N},\mathscr{Y}_1^{k,N},\varrho_1^{k,N}) -  D_pH(\mathscr{X}_2^{k,M},\mathscr{Y}_2^{k,M},\varrho_{2}^{k,M})\right)\cdot (\mathscr{Y}_1^{k,N}-\mathscr{Y}_2^{k,M})\right.
			\\
			&\left.\qquad\qquad-(\mathscr{X}_1^{k,N}-\mathscr{X}_2^{k,M})\cdot\left( D_xH(\mathscr{X}_1^{k,N},\mathscr{Y}_1^{k,N},\varrho_1^{k,N}) -  D_xH(\mathscr{X}_2^{k,M},\mathscr{Y}_2^{k,M},\varrho_{2}^{k,M})\right)\right]\mathrm{d}s
			\\
			&\qquad+\mathfrak{I}(t_n,T)
		\end{split}
	\end{equation}
	for all $n\in\{0,1,2,\cdots,M_k\}$ where $\mathfrak{I}$ is defined as in \eqref{I-def}.
	Since $H$ satisfies the displacement monotonicity condition \eqref{ass-H:disp-mono-cond-H} and $\lambda_1>0$ we deduce from \eqref{pre-disp-mono-app-one} that 
	\begin{equation}\label{disp-mono-app-one}
		\begin{split}
			&\int_{t_n}^T\mathbb{E}\left[\partial_t\Ip \Delta \mathscr{X}_{N,M}^{k}\cdot \Delta\mathscr{Y}_{N,M}^{k} + \Delta \mathscr{X}_{N,M}^{k} \cdot\p_t \In \Delta\mathscr{Y}_{N,M}^{k}\right]\mathrm{d}s\geq \mathfrak{I}(t_n,T)
		\end{split}
	\end{equation}
	for all $n\in\{0,1,2,\cdots,M_k\}$.

	Furthermore,  we infer from displacement monotonicity of $g$ via \eqref{ass-g:disp-mono-cond-g} and non-negativity of $\lambda_1\in (0,1]$ that
	\begin{equation}
		\lambda_1 \mathbb{E}\left[(D_xg(\mathscr{X}_1^{k,N}(T),\varrho_1^{k,N}(T)) - D_xg(\mathscr{X}_2^{k,M}(T),\varrho_{2}^{k,M}(T)))\cdot \left(\mathscr{X}_1^{k,N}(T)-\mathscr{X}_2^{k,M}(T)\right)\right]\geq 0
	\end{equation} which is equivalent to
	\begin{multline}
		(\lambda_2-\lambda_1)\expectation{D_xg(\mathscr{X}_2^{k,M}(T),\varrho_{2}^{k,M}(T))\cdot (\mathscr{X}_1^{k,N}(T)-\mathscr{X}_2^{k,M}(T))}
		\\
		\geq \mathbb{E}\left[\left(-\lambda_1 D_xg(\mathscr{X}_1^{k,N}(T),\varrho_1^{k,N}(T))+\lambda_2 D_xg(\mathscr{X}_2^{k,M}(T),\varrho_{2}^{k,M}(T))\right)\cdot\left(\mathscr{X}_1^{k,N}(T)-\mathscr{X}_2^{k,M}(T)\right)\right]\\=\mathbb{E}\left[\Delta\mathscr{Y}_{N,M}^{k}(T)\cdot \Delta \mathscr{X}_{N,M}^{k}(T)\right]
	\end{multline}
	since $\Delta \mathscr{X}_{N,M}^{k}(T)=\mathscr{X}_1^{k,N}(T)-\mathscr{X}_2^{k,M}(T)$ and we have the terminal time expression $$\Delta\mathscr{Y}_{N,M}^{k}(T)=\mathscr{Y}_1^{k,N}(T)-\mathscr{Y}_2^{k,M}(T)=-\lambda_1 D_xg(\mathscr{X}_1^{k,N}(T),\varrho_1^{k,N}(T))+\lambda_2 D_xg(\mathscr{X}_2^{k,M}(T),\varrho_{2}^{k,M}(T)).$$ 
	This bound and the integration-by-parts formula \eqref{eq:discrete_ibp-expectation_pn}  then imply that
	\begin{equation}
		\begin{split}
			&(\lambda_2-\lambda_1)\expectation{D_xg(\mathscr{X}_2^{k,M}(T),\varrho_{2}^{k,M}(T))\cdot (\mathscr{X}_1^{k,N}(T)-\mathscr{X}_2^{k,M}(T))}
			\\
			&\geq \mathbb{E}\left[\Delta\mathscr{Y}_{N,M}^{k}(T)\cdot \Delta \mathscr{X}_{N,M}^{k}(T)\right]
			\\
			&=\mathbb{E}\left[\Delta \mathscr{X}_{N,M}^{k}(t_n)\cdot \Delta\mathscr{Y}_{N,M}^{k}(t_n)\right]+\int_{t_n}^T\mathbb{E}\left[\partial_t\Ip \Delta \mathscr{X}_{N,M}^{k}\cdot \Delta\mathscr{Y}_{N,M}^{k} + \Delta \mathscr{X}_{N,M}^{k} \cdot\p_t \In \Delta\mathscr{Y}_{N,M}^{k}\right]\mathrm{d}s
		\end{split}
	\end{equation} for $n\in\{0,1,2,\cdots,M_k\}$. We then deduce from this and \eqref{disp-mono-app-one} that
	\begin{equation}
		\begin{split}
			&(\lambda_2-\lambda_1)\expectation{D_xg(\mathscr{X}_2^{k,M}(T),\varrho_{2}^{k,M}(T))\cdot \Delta \mathscr{X}_{N,M}^{k}(T)}
			\geq \mathbb{E}\left[\Delta \mathscr{X}_{N,M}^{k}(t_n)\cdot \Delta\mathscr{Y}_{N,M}^{k}(t_n)\right] +\mathfrak{I}(t_n,T)
		\end{split}
	\end{equation}  and thus
	\begin{equation}\label{time-error-joint-mono-bound-gen}
	\begin{split}
		&\mathbb{E}\left[\Delta \mathscr{X}_{N,M}^{k}(t_n)\cdot \Delta\mathscr{Y}_{N,M}^{k}(t_n)\right] 
		\leq (\lambda_2-\lambda_1)\expectation{D_xg(\mathscr{X}_2^{k,M}(T),\varrho_{2}^{k,M}(T))\cdot \Delta \mathscr{X}_{N,M}^{k}(T)} -
	\mathfrak{I}(t_n,T)
	\end{split}
	\end{equation}
	for all $n\in\{0,1,2,\cdots,M_k\}$. 
	
	We also have from the systems satisfied by $(\mathscr{X}_1^{k,N},\mathscr{Y}_1^{k,N},\varrho_1^{k,N}),(\mathscr{X}_2^{k,M},\mathscr{Y}_2^{k,M},\varrho_{2}^{k,M})$ respectively that
	\begin{equation}\label{other-time-interval-formula}
		\begin{split}
			&\int_0^{t_n}\mathbb{E}\left[\partial_t\Ip \Delta \mathscr{X}_{N,M}^{k}\cdot \Delta\mathscr{Y}_{N,M}^{k} + \Delta \mathscr{X}_{N,M}^{k} \cdot\p_t \In \Delta\mathscr{Y}_{N,M}^{k}\right]\mathrm{d}s
			\\
			&=\int_0^{t_n} \lambda_1 \mathbb{E}\left[\left( D_pH(\mathscr{X}_1^{k,N},\mathscr{Y}_1^{k,N},\varrho_1^{k,N}) -  D_pH(\mathscr{X}_2^{k,M},\mathscr{Y}_2^{k,M},\varrho_{2}^{k,M})\right)\cdot \Delta\mathscr{Y}_{N,M}^{k}\right.
			\\
			&\left.\qquad\qquad\qquad-\Delta \mathscr{X}_{N,M}^{k}\cdot\left( D_xH(\mathscr{X}_1^{k,N},\mathscr{Y}_1^{k,N},\varrho_1^{k,N}) -  D_xH(\mathscr{X}_2^{k,M},\mathscr{Y}_2^{k,M},\varrho_{2}^{k,M})\right)\right]\mathrm{d}s
			\\
			&\qquad+\mathfrak{I}(0,t_n)
		\end{split}
	\end{equation} for all $n\in\{0,1,2,\cdots,M_k\}$. Consequently, \eqref{time-error-joint-mono-bound-gen} and the integration-by-parts formula \eqref{eq:discrete_ibp-expectation_pn} give
	\begin{equation}
	\begin{split}
		&(\lambda_2-\lambda_1)\expectation{D_xg(\mathscr{X}_2^{k,M}(T),\varrho_{2}^{k,M}(T))\cdot \Delta \mathscr{X}_{N,M}^{k}(T)} -
		\mathfrak{I}(t_n,T)
		\\
		&\geq \mathbb{E}\left[\Delta \mathscr{X}_{N,M}^{k}(t_n)\cdot \Delta\mathscr{Y}_{N,M}^{k}(t_n)\right]
		\\
		&=\mathbb{E}\left[\Delta \mathscr{X}_{N,M}^{k}(0)\cdot \Delta\mathscr{Y}_{N,M}^{k}(0)\right]+\int_0^{t_n}\mathbb{E}\left[\partial_t\Ip \Delta \mathscr{X}_{N,M}^{k}\cdot \Delta\mathscr{Y}_{N,M}^{k} + \Delta \mathscr{X}_{N,M}^{k} \cdot\p_t \In \Delta\mathscr{Y}_{N,M}^{k}\right]\mathrm{d}s.
	\end{split}
	\end{equation} We then deduce from this and \eqref{other-time-interval-formula}  that
	\begin{equation}
		\begin{split}
			&(\lambda_2-\lambda_1)\expectation{D_xg(\mathscr{X}_2^{k,M}(T),\varrho_{2}^{k,M}(T))\cdot \Delta \mathscr{X}_{N,M}^{k}(T)} -
			\mathfrak{I}(t_n,T)
			\\
			&\geq \mathbb{E}\left[\Delta \mathscr{X}_{N,M}^{k}(0)\cdot \Delta\mathscr{Y}_{N,M}^{k}(0)\right]\\
			&\qquad+\int_0^{t_n} \lambda_1 \mathbb{E}\left[\left( D_pH(\mathscr{X}_1^{k,N},\mathscr{Y}_1^{k,N},\varrho_1^{k,N}) -  D_pH(\mathscr{X}_2^{k,M},\mathscr{Y}_2^{k,M},\varrho_{2}^{k,M})\right)\cdot \Delta\mathscr{Y}_{N,M}^{k}\right.\\
			&\left.\qquad\qquad\qquad-\Delta \mathscr{X}_{N,M}^{k}\cdot\left( D_xH(\mathscr{X}_1^{k,N},\mathscr{Y}_1^{k,N},\varrho_1^{k,N}) -  D_xH(\mathscr{X}_2^{k,M},\mathscr{Y}_2^{k,M},\varrho_{2}^{k,M})\right)\right]\mathrm{d}s
			\\
			&\qquad\qquad+\mathfrak{I}(0,t_n).
		\end{split}
	\end{equation} We then conclude the required bound after rearranging and noting that $\mathfrak{I}(0,t_n)+\mathfrak{I}(t_n,T)=\mathfrak{I}(0,T)$.
\end{proof}
 
 The second step in proving Theorem  \ref{theorem-cont-dependence-discrete-Hamiltonian-sys} showcases a stronger stability property obtained under the additional assumption on strong convexity of $H$ and Lipschitz continuity of both $D_pH$ and $D_xH$.
 \begin{lemma}[Strong convexity stability]\label{lemma-X-bound-Y-discr-bound}
 	Let the Hamiltonian $H$ satisfy the regularity assumptions \eqref{ass-H:1}, \eqref{ass-H:2}, \eqref{ass-H:3}, \eqref{ass-H:4} and the displacement monotonicity condition \eqref{ass-H:disp-mono-cond-H}. Let the terminal cost $g$ satisfy the regularity assumption \eqref{ass-g:1} and the displacement monotonicity condition \eqref{ass-g:disp-mono-cond-g}. Let $N,M,k\in\mathbb{N}$ and 
 	\begin{subequations}
 		\begin{align}
 			&\lambda_1\in (0,1],  \quad F_1,G_1\in \mathbb{V}_k(0,T;\mathcal{R}_{d}(\mathcal{A}_N,\mathcal{O}_N)), \quad X_{0,1}^N\in  \mathcal{R}_{d}(\mathcal{A}_N,\mathcal{O}_N),\label{data-line-1''}
 			\\
 			&\lambda_2\in [0,1], \quad F_2,G_2\in \mathbb{V}_k(0,T;\mathcal{R}_{d}(\mathcal{A}_M,\mathcal{O}_M)),\quad X_{0,2}^M\in  \mathcal{R}_{d}(\mathcal{A}_M,\mathcal{O}_M)\label{data-line-2''}
 		\end{align}
 	\end{subequations}
 	be given, and suppose that the triples
 	\begin{subequations}
 		\begin{align}
 			(\mathscr{X}_1^{k,N},\mathscr{Y}_1^{k,N},\varrho_1^{k,N})&\in \mathbb{V}_k^+([0,T];\mathcal{R}_{d}(\mathcal{A}_N,\mathcal{O}_N))\times \mathbb{V}_k^-([0,T];\mathcal{R}_{d}(\mathcal{A}_N,\mathcal{O}_N))\times \mathbb{P}_k([0,T];\mathcal{P}_{d,N}),\label{triple-1'}\\
 			(\mathscr{X}_2^{k,M},\mathscr{Y}_2^{k,M},\varrho_{2}^{k,M})&\in \mathbb{V}_k^+([0,T];\mathcal{R}_{d}(\mathcal{A}_M,\mathcal{O}_M))\times \mathbb{V}_k^-([0,T];\mathcal{R}_{d}(\mathcal{A}_M,\mathcal{O}_M))\times \mathbb{P}_k([0,T];\mathcal{P}_{d,M})\label{triple-2'}
 		\end{align}
 	\end{subequations}  each satisfy the discrete Hamiltonian system \eqref{numerical-scheme-gen} with data \eqref{data-line-1''} and \eqref{data-line-2''}, respectively. Then, for all $n\in\{0,1,\cdots,M_k\}$ and all  $0<\epsilon<1$, there holds 
 		\begin{equation}\label{y-diff-bound-tn-gen}
 		\begin{split}
 			&\int_0^{t_n}\mathbb{E}[|\mathscr{Y}_1^{k,N}-\mathscr{Y}_2^{k,M}|^2]\mathrm{d}s
 			\lesssim \frac{1}{\epsilon}\mathfrak{E}_1
 			+\epsilon\mathfrak{E}_2+\int_0^{t_n}\mathbb{E}\left[|\mathscr{X}_1^{k,N}-\mathscr{X}_2^{k,M}|^2\right]\ds
 		\end{split}
 	\end{equation}
 	where the hidden constant in \eqref{y-diff-bound-tn-gen}  depends only on $c_0$ and the Lipschitz constants of $D_pH$ and $D_xH$, and where we define the quantities
 	\begin{equation}\label{defn-E_*}
 		\begin{split}
 			&\mathfrak{E}_1\coloneqq
 			\frac{1}{\lambda_1^2}\Bigl(\mathbb{E}\left[|(\mathscr{X}_1^{k,N}-\mathscr{X}_2^{k,M})(0)|^2+(\lambda_1-\lambda_2)^2\abs{D_xg(\mathscr{X}_2^{k,M}(T),\varrho_{2}^{k,M}(T))}^2\right]\\
 			&\qquad\qquad\qquad+\int_0^T\expectation{\abs{F_1-F_2}^2+\abs{G_1-G_2}^2}\ds
 			\\
 			& \quad\qquad+(\lambda_1-\lambda_2)^2\int_0^T\expectation{\abs{D_pH(\mathscr{X}_2^{k,M},\mathscr{Y}_2^{k,M},\varrho_{2}^{k,M})}^2+\abs{D_xH(\mathscr{X}_2^{k,M},\mathscr{Y}_2^{k,M},\varrho_{2}^{k,M})}^2}\ds\Bigr),
 		\end{split}
 	\end{equation}
 	\begin{multline}\label{defn-E_*2}
 		\mathfrak{E}_2\coloneqq \mathbb{E}\left[|(\mathscr{Y}_1^{k,N}-\mathscr{Y}_2^{k,M})(0)|^2+|(\mathscr{X}_1^{k,N}-\mathscr{X}_2^{k,M})(T)|^2\right]\\+\int_0^T\mathbb{E}\left[|\mathscr{X}_1^{k,N}-\mathscr{X}_2^{k,M}|^2+|\mathscr{Y}_1^{k,N}-\mathscr{Y}_2^{k,M}|^2\right]\ds.
 	\end{multline}
 \end{lemma}
 
 \begin{proof}
 	 	We prove the result in the case when the sum $\Lip(D_pH)+\Lip(D_xH)$ of the Lipschitz constants for $D_pH$ and $D_xH$ is positive. The proof in the the case when one or both of these constants are zero follows similarly by a more direct argument. 
 	 	
 	 	Let $k,N,M\in\mathbb{N}$ be given. We suppose that the triples $(\mathscr{X}_1^{k,N},\mathscr{Y}_1^{k,N},\varrho_1^{k,N})$   $(\mathscr{X}_2^{k,M},\mathscr{Y}_2^{k,M},\varrho_{2}^{k,M})$ with \eqref{triple-1'}, \eqref{triple-2'} each satisfy the discrete system \eqref{numerical-scheme-gen} with data \eqref{data-line-1'} and  \eqref{data-line-2'}, respectively. We then let $\Delta \mathscr{X}_{N,M}^{k}\coloneqq \mathscr{X}_1^{k,N}-\mathscr{X}_2^{k,M}\in \mathbb{V}_k^+([0,T];L^2(\Omega;\mathbb{R}^d))$ and $\Delta\mathscr{Y}_{N,M}^{k}\coloneqq \mathscr{Y}_1^{k,N}-\mathscr{Y}_2^{k,M}\in \mathbb{V}_k^-([0,T];L^2(\Omega;\mathbb{R}^d))$.  We begin the proof by  employing the strong convexity of $H$ w.r.t.\ $p$ to obtain an estimate for the generalized momentum difference $W_{N,M}^k=\mathscr{Y}_1^{k,N}-\mathscr{Y}_2^{k,M}$. This this end, we rearrange the bound \eqref{step-one-bound} from Lemma \ref{lemma-step-1-discrete-uniqueness} and introduce the term $ D_pH(\mathscr{X}_1^{k,N},\mathscr{Y}_2^{k,M},\varrho_1^{k,N})\cdot (\mathscr{Y}_1^{k,N}-\mathscr{Y}_2^{k,M})$ on both sides of the bound to get
 	\begin{equation}
 		\begin{split}
 			&\lambda_1 \int_0^{t_n} \mathbb{E}\left[\left(D_pH(\mathscr{X}_1^{k,N},\mathscr{Y}_1^{k,N},\varrho_1^{k,N}) - D_pH(\mathscr{X}_1^{k,N},\mathscr{Y}_2^{k,M},\varrho_1^{k,N})\right)\cdot (\mathscr{Y}_1^{k,N}-\mathscr{Y}_2^{k,M})\right]\mathrm{d}s
 			\\
 			&\leq (\lambda_2-\lambda_1)\expectation{D_xg(\mathscr{X}_2^{k,M}(T),\varrho_{2}^{k,M}(T))\cdot \Delta \mathscr{X}_{N,M}^{k}(T)} -\mathbb{E}\left[\Delta \mathscr{X}_{N,M}^{k}(0)\cdot \Delta\mathscr{Y}_{N,M}^{k}(0)\right]-\mathfrak{I}(0,T)
 			\\
 			&\quad\quad+\lambda_1\int_0^{t_n} \mathbb{E}\left[\left(D_pH(\mathscr{X}_2^{k,M},\mathscr{Y}_2^{k,M},\varrho_{2}^{k,M}) - D_pH(\mathscr{X}_1^{k,N},\mathscr{Y}_2^{k,M},\varrho_1^{k,N})\right)\cdot \Delta\mathscr{Y}_{N,M}^{k}\right]\mathrm{d}s
 			\\
 			&\quad\quad+\lambda_1\int_{0}^{t_n} \mathbb{E}\left[\Delta \mathscr{X}_{N,M}^{k}\cdot\left(D_xH(\mathscr{X}_1^{k,N},\mathscr{Y}_1^{k,N},\varrho_1^{k,N}) - D_xH(\mathscr{X}_2^{k,M},\mathscr{Y}_2^{k,M},\varrho_{2}^{k,M})\right)\right]\mathrm{d}s
 		\end{split}
 	\end{equation} for all $n\in\{0,1,\cdots,M_k\}$.  Using the strong convexity of $H$ w.r.t.\ $p$ via \eqref{ass-H:4}, we get
 	\begin{equation}\label{bound-1}
 		\begin{split}
 			&\lambda_1 c_0\int_0^{t_n}\mathbb{E}[|\Delta\mathscr{Y}_{N,M}^{k}|^2]\mathrm{d}s
 			\\
 			&\leq |\lambda_2-\lambda_1|\expectation{|D_xg(\mathscr{X}_2^{k,M}(T),\varrho_{2}^{k,M}(T))|| \Delta \mathscr{X}_{N,M}^{k}(T)|}+\mathbb{E}\left[|\Delta \mathscr{X}_{N,M}^{k}(0)| |\Delta\mathscr{Y}_{N,M}^{k}(0)|\right]
 			\\
 			&\quad+|\mathfrak{I}(0,T)|+\lambda_1\int_0^{t_n} \mathbb{E}\left[\left |D_pH(\mathscr{X}_2^{k,M},\mathscr{Y}_2^{k,M},\varrho_{2}^{k,M}) - D_pH(\mathscr{X}_1^{k,N},\mathscr{Y}_2^{k,M},\varrho_1^{k,N})\right || \Delta\mathscr{Y}_{N,M}^{k}|\right]\mathrm{d}s
 			\\
 			&\quad+\lambda_1\int_{0}^{t_n} \mathbb{E}\left[|\Delta \mathscr{X}_{N,M}^{k}|\left |D_xH(\mathscr{X}_1^{k,N},\mathscr{Y}_1^{k,N},\varrho_1^{k,N}) - D_xH(\mathscr{X}_2^{k,M},\mathscr{Y}_2^{k,M},\varrho_{2}^{k,M})\right |\right]\mathrm{d}s
 		\end{split}
 	\end{equation} for all $n\in\{0,1,\cdots,M_k\}$. Note that the uniform Lipschitz continuity of both $D_pH$, $D_xH$ with respect to the Euclidean norm in the vector arguments and $\mathcal{W}_1$-metric in the measure argument via \eqref{ass-H:2}, together with the Cauchy--Schwarz inequality applied to the integral in the definition of $\mathcal{W}_1$, implies the uniform Lipschitz continuity of both $D_pH$, $D_xH$ with respect to the Euclidean norm in the vector arguments and $\mathcal{W}_2$-metric in the measure argument. This leads to
 	\begin{multline}\label{bound-2}
 			\lambda_1\int_0^{t_n} \mathbb{E}\left[\left |D_pH(\mathscr{X}_2^{k,M},\mathscr{Y}_2^{k,M},\varrho_{2}^{k,M}) - D_pH(\mathscr{X}_1^{k,N},\mathscr{Y}_2^{k,M},\varrho_1^{k,N})\right || \Delta\mathscr{Y}_{N,M}^{k}|\right]\mathrm{d}s\\
 			\leq \lambda_1 \Lip(D_pH)\int_0^{t_n} \mathbb{E}\left[\left (|\Delta \mathscr{X}_{N,M}^{k}|+\mathcal{W}_2(\varrho_1^{k,N},\varrho_{2}^{k,M})\right)| \Delta\mathscr{Y}_{N,M}^{k}|\right]\mathrm{d}s
 	\end{multline}
 	\begin{multline}\label{bound-3}
 		\lambda_1\int_{0}^{t_n} \mathbb{E}\left[|\Delta \mathscr{X}_{N,M}^{k}|\left |D_xH(\mathscr{X}_1^{k,N},\mathscr{Y}_1^{k,N},\varrho_1^{k,N}) - D_xH(\mathscr{X}_2^{k,M},\mathscr{Y}_2^{k,M},\varrho_{2}^{k,M})\right |\right]\mathrm{d}s
 		\\
 		\leq \lambda_1\Lip(D_xH)\int_0^{t_n}\mathbb{E}\left[|\Delta \mathscr{X}_{N,M}^{k}|\left (|\Delta \mathscr{X}_{N,M}^{k}|+|\Delta\mathscr{Y}_{N,M}^{k}|+\mathcal{W}_2(\varrho_1^{k,N},\varrho_{2}^{k,M})\right)\right]\mathrm{d}s.
 	\end{multline}

 	With $n\in\{0,1,\cdots,M_k\}$ fixed, we now aim to extract the integral $\int_0^{t_n}\mathbb{E}[|\Delta\mathscr{Y}_{N,M}^{k}|^2]\mathrm{d}s$ from the upper bounds in \eqref{bound-2} and \eqref{bound-3} by applying Young's inequality repeatedly. 
 	Indeed, by introducing  a parameter $\gamma\in (0,1)$, we apply Young's inequality to the integrand in the upper bound of \eqref{bound-2} to get
 	\begin{multline}\label{bound-2-youngs}
 		\lambda_1\int_0^{t_n} \mathbb{E}\left[\left |D_pH(\mathscr{X}_2^{k,M},\mathscr{Y}_2^{k,M},\varrho_{2}^{k,M}) - D_pH(\mathscr{X}_1^{k,N},\mathscr{Y}_2^{k,M},\varrho_1^{k,N})\right || \Delta\mathscr{Y}_{N,M}^{k}|\right]\mathrm{d}s\\
 		\leq {\lambda_1 \Lip(D_pH)}\int_0^{t_n} \mathbb{E}\left[\left (\frac{1}{2\gamma}|\Delta \mathscr{X}_{N,M}^{k}|^2+\frac{1}{2\gamma}\mathcal{W}_2^2(\varrho_1^{k,N},\varrho_{2}^{k,M})\right)+\frac{\gamma}{2}| \Delta\mathscr{Y}_{N,M}^{k}|^2\right]\mathrm{d}s,
 	\end{multline}
 	 and likewise Young's inequality applied to the integrand in the upper bound of \eqref{bound-3} gives
 	 \begin{multline}\label{bound-3-youngs}
 	 	\lambda_1\int_{0}^{t_n} \mathbb{E}\left[|\Delta \mathscr{X}_{N,M}^{k}|\left |D_xH(\mathscr{X}_1^{k,N},\mathscr{Y}_1^{k,N},\varrho_1^{k,N}) - D_xH(\mathscr{X}_2^{k,M},\mathscr{Y}_2^{k,M},\varrho_{2}^{k,M})\right |\right]\mathrm{d}s
 	 	\\
 	 	\leq \lambda_1\Lip(D_xH)\int_0^{t_n}\mathbb{E}\left[\left(\frac{3}{2}+\frac{1}{2\gamma}\right)|\Delta \mathscr{X}_{N,M}^{k}|^2+\frac{\gamma}{2}|\Delta\mathscr{Y}_{N,M}^{k}|^2+\frac{1}{2}\mathcal{W}_2^2(\varrho_1^{k,N},\varrho_{2}^{k,M})\right]\mathrm{d}s
 	 \end{multline}
 	 Adding these bounds \eqref{bound-2-youngs}, \eqref{bound-3-youngs} together gives
 	 \begin{multline}\label{bound-2+3-youngs}
 	 	\lambda_1\int_0^{t_n} \mathbb{E}\left[\left |D_pH(\mathscr{X}_2^{k,M},\mathscr{Y}_2^{k,M},\varrho_{2}^{k,M}) - D_pH(\mathscr{X}_1^{k,N},\mathscr{Y}_2^{k,M},\varrho_1^{k,N})\right || \Delta\mathscr{Y}_{N,M}^{k}|\right]\mathrm{d}s\\+\lambda_1\int_{0}^{t_n} \mathbb{E}\left[|\Delta \mathscr{X}_{N,M}^{k}|\left |D_xH(\mathscr{X}_1^{k,N},\mathscr{Y}_1^{k,N},\varrho_1^{k,N}) - D_xH(\mathscr{X}_2^{k,M},\mathscr{Y}_2^{k,M},\varrho_{2}^{k,M})\right |\right]\mathrm{d}s\\
 	 	\leq {\lambda_1}\int_0^{t_n} \mathbb{E}\left[\left(\frac{3\Lip(D_xH)}{2}+\frac{\Lip(D_pH)+\Lip(D_xH)}{2\gamma}\right)|\Delta \mathscr{X}_{N,M}^{k}|^2\right]
 	 	\\
 	 	+\lambda_1\int_0^{t_n}\left[\left(\frac{\Lip(D_pH)}{2\gamma}+\frac{\Lip(D_xH)}{2}\right)\mathcal{W}_2^2(\varrho_1^{k,N},\varrho_{2}^{k,M})\right]\mathrm{d}s\\+\lambda_1\int_0^{t_n}\mathbb{E}\left[\left(\frac{\left(\Lip(D_pH)+\Lip(D_xH)\right)\gamma}{2}\right)| \Delta\mathscr{Y}_{N,M}^{k}|^2\right]\mathrm{d}s,
 	 \end{multline}
 	 To obtain coefficients in the upper bound that are more compact, let us assume further that $0<\gamma<1/3$, to have
 	 \begin{equation}
 	 	\frac{3\Lip(D_xH)}{2}+\frac{\Lip(D_pH)+\Lip(D_xH)}{2\gamma}\leq \frac{\Lip(D_pH)+2\Lip(D_xH)}{2\gamma}\leq \frac{\Lip(D_pH)+\Lip(D_xH)}{\gamma}
 	 \end{equation}
 	 \begin{equation}
 	 	\frac{\Lip(D_pH)}{2\gamma}+\frac{\Lip(D_xH)}{2}\leq \frac{\Lip(D_pH)+\Lip(D_xH)}{\gamma}
 	 \end{equation}
 	 and clearly $\left(\Lip(D_pH)+\Lip(D_xH)\right)\gamma/2\leq \left(\Lip(D_pH)+\Lip(D_xH)\right)\gamma$.
 	 We thus obtain
 	  \begin{multline}\label{bound-2+3-youngs-scaled}
 	 	\lambda_1\int_0^{t_n} \mathbb{E}\left[\left |D_pH(\mathscr{X}_2^{k,M},\mathscr{Y}_2^{k,M},\varrho_{2}^{k,M}) - D_pH(\mathscr{X}_1^{k,N},\mathscr{Y}_2^{k,M},\varrho_1^{k,N})\right || \Delta\mathscr{Y}_{N,M}^{k}|\right]\mathrm{d}s\\+\lambda_1\int_{0}^{t_n} \mathbb{E}\left[|\Delta \mathscr{X}_{N,M}^{k}|\left |D_xH(\mathscr{X}_1^{k,N},\mathscr{Y}_1^{k,N},\varrho_1^{k,N}) - D_xH(\mathscr{X}_2^{k,M},\mathscr{Y}_2^{k,M},\varrho_{2}^{k,M})\right |\right]\mathrm{d}s\\
 	 	\leq {\lambda_1}\frac{\Lip(D_pH)+\Lip(D_xH)}{\gamma}\int_0^{t_n} \left(\mathbb{E}\left[|\Delta \mathscr{X}_{N,M}^{k}|^2\right]+\mathcal{W}_2^2(\varrho_1^{k,N},\varrho_{2}^{k,M})\right)\mathrm{d}s\\+\lambda_1 \left(\Lip(D_pH)+\Lip(D_xH)\right)\gamma\int_0^{t_n}\mathbb{E}\left[|\Delta\mathscr{Y}_{N,M}^{k}|^2\right]\mathrm{d}s,
 	 \end{multline}
 	 With this simplified upper bound, we return to the bound \eqref{bound-1} and rearrange to get
 	 	\begin{equation}\label{bound-1-update}
 	 	\begin{split}
 	 		&\lambda_1\left( c_0- \left(\Lip(D_pH)+\Lip(D_xH)\right)\gamma\right)\int_0^{t_n}\mathbb{E}[|\Delta\mathscr{Y}_{N,M}^{k}|^2]\mathrm{d}s
 	 		\\
 	 		&\leq |\lambda_2-\lambda_1|\expectation{|D_xg(\mathscr{X}_2^{k,M}(T),\varrho_{2}^{k,M}(T))|| \Delta \mathscr{X}_{N,M}^{k}(T)|}+\mathbb{E}\left[|\Delta \mathscr{X}_{N,M}^{k}(0)| |\Delta\mathscr{Y}_{N,M}^{k}(0)|\right]
 	 		+|\mathfrak{I}(0,T)|
 	 		\\
 	 		&\quad+
 	 		 {\lambda_1}\frac{\Lip(D_pH)+\Lip(D_xH)}{\gamma}\int_0^{t_n} \left(\mathbb{E}\left[|\Delta \mathscr{X}_{N,M}^{k}|^2\right]+\mathcal{W}_2^2(\varrho_1^{k,N},\varrho_{2}^{k,M})\right)\mathrm{d}s
 	 	\end{split}
 	 \end{equation} for all $n\in\{0,1,\cdots,M_k\}$ and $0<\gamma<1/3$. By fixing $0<\gamma<\min\left\{\frac{1}{3}, \frac{c_0}{\Lip(D_pH)+\Lip(D_xH)}\right\}$ we obtain the bound
 	 \begin{multline}\label{final-young-bound-1}
 	 	\lambda_1\int_0^{t_n}\mathbb{E}[|\Delta\mathscr{Y}_{N,M}^{k}|^2]\mathrm{d}s
 	 	\\
 	 	\leq C_{\gamma}\left( |\lambda_2-\lambda_1|\expectation{|D_xg(\mathscr{X}_2^{k,M}(T),\varrho_{2}^{k,M}(T))|| \Delta \mathscr{X}_{N,M}^{k}(T)|}+\mathbb{E}\left[|\Delta \mathscr{X}_{N,M}^{k}(0)| |\Delta\mathscr{Y}_{N,M}^{k}(0)|\right]
 	 	\right.
 	 	\\\left.
 	 	+|\mathfrak{I}(0,T)|+
 	 	{\lambda_1}\int_0^{t_n} \left(\mathbb{E}\left[|\Delta \mathscr{X}_{N,M}^{k}|^2\right]+\mathcal{W}_2^2(\varrho_1^{k,N},\varrho_{2}^{k,M})\right)\mathrm{d}s\right)
 	 \end{multline}
 	 for all $n\in\{0,1,\cdots,M_k\}$, with the positive constant 
 	 \begin{equation}\label{Cgamma-constant}
 	 	C_{\gamma}\coloneqq \max\left\{\frac{1}{c_0-(\Lip(D_pH)+\Lip(D_xH))\gamma},\frac{\Lip(D_pH)+\Lip(D_xH)}{\left(c_0-(\Lip(D_pH)+\Lip(D_xH))\gamma\right)\gamma}\right\}
 	 \end{equation}
 	 with fixed constant $0<\gamma<\min\left\{\frac{1}{3}, \frac{c_0}{\Lip(D_pH)+\Lip(D_xH)}\right\}$. 
 	 
 	Next, we aim to bound the term $$ |\lambda_2-\lambda_1|\expectation{|D_xg(\mathscr{X}_2^{k,M}(T),\varrho_{2}^{k,M}(T))|| \Delta \mathscr{X}_{N,M}^{k}(T)|}+\mathbb{E}\left[|\Delta \mathscr{X}_{N,M}^{k}(0)| |\Delta\mathscr{Y}_{N,M}^{k}(0)|\right]
 	+|\mathfrak{I}(0,T)|.$$ By Young's inequality with parameter $0<\lambda_1\epsilon<1$ where $\epsilon \in (0,1)$, have 
 	\begin{multline}\label{bound-4}
 		|\mathfrak{I}(0,T)|\leq 
 		|\lambda_1-\lambda_2|\int_{0}^T  \mathbb{E}\left[  |D_pH(\mathscr{X}_2^{k,M},\mathscr{Y}_2^{k,M},\varrho_{2}^{k,M})||\Delta\mathscr{Y}_{N,M}^{k}|\right.\\
 		\left.\qquad\qquad\qquad+|\Delta \mathscr{X}_{N,M}^{k}|  |D_xH(\mathscr{X}_2^{k,M},\mathscr{Y}_2^{k,M},\varrho_{2}^{k,M})|\right]\mathrm{d}s\\
 		+\int_{0}^T\mathbb{E}\left[|F_1-F_2||\Delta\mathscr{Y}_{N,M}^{k}|+|\Delta \mathscr{X}_{N,M}^{k}| |G_1-G_2|\right]\mathrm{d}s
 		\\
 		\leq \frac{(\lambda_1-\lambda_2)^2}{2\lambda_1\epsilon}\int_0^T\mathbb{E}\left[ |D_pH(\mathscr{X}_2^{k,M},\mathscr{Y}_2^{k,M},\varrho_{2}^{k,M})|^2+ |D_xH(\mathscr{X}_2^{k,M},\mathscr{Y}_2^{k,M},\varrho_{2}^{k,M})|^2\right]\mathrm{d}s \\+
 		\frac{1}{2\lambda_1\epsilon}\int_0^T\mathbb{E}\left[|F_1-F_2|^2+|G_1-G_2|^2\right]\mathrm{d}s
 		+{\lambda_1\epsilon}\int_0^T\mathbb{E}\left[|\Delta\mathscr{Y}_{N,M}^{k}|^2+|\Delta \mathscr{X}_{N,M}^{k}|^2\right]\mathrm{d}s
 	\end{multline}and
 	\begin{multline}
 		 |\lambda_2-\lambda_1|\expectation{|D_xg(\mathscr{X}_2^{k,M}(T),\varrho_{2}^{k,M}(T))|| \Delta \mathscr{X}_{N,M}^{k}(T)|}+\mathbb{E}\left[|\Delta \mathscr{X}_{N,M}^{k}(0)| |\Delta\mathscr{Y}_{N,M}^{k}(0)|\right]\\
 		 \leq \frac{(\lambda_1-\lambda_2)^2}{2\lambda_1\epsilon}\mathbb{E}\left[|D_xg(\mathscr{X}_2^{k,M}(T),\varrho_{2}^{k,M}(T))|^2\right]\\ + \frac{\lambda_1\epsilon}{2}\mathbb{E}\left[| \Delta \mathscr{X}_{N,M}^{k}(T)|^2\right]
 		 +\frac{1}{2\lambda_1\epsilon}\mathbb{E}\left[|\Delta \mathscr{X}_{N,M}^{k}(0)|^2\right]+\frac{\lambda_1\epsilon}{2}\mathbb{E}\left[|\Delta\mathscr{Y}_{N,M}^{k}(0)|^2\right].
 	\end{multline}
 	We then obtain
 	\begin{multline}
 		 |\lambda_2-\lambda_1|\expectation{|D_xg(\mathscr{X}_2^{k,M}(T),\varrho_{2}^{k,M}(T))|| \Delta \mathscr{X}_{N,M}^{k}(T)|}+\mathbb{E}\left[|\Delta \mathscr{X}_{N,M}^{k}(0)| |\Delta\mathscr{Y}_{N,M}^{k}(0)|\right]
 		+|\mathfrak{I}(0,T)|
 		\\\leq \mathscr{A}_1+\mathscr{A}_2+\mathscr{A}_3
 	\end{multline}
 	where
 	\begin{multline}
 		\mathscr{A}_1\coloneqq 
 		\frac{(\lambda_1-\lambda_2)^2}{2\lambda_1\epsilon}\left(\int_0^T\mathbb{E}\left[ |D_pH(\mathscr{X}_2^{k,M},\mathscr{Y}_2^{k,M},\varrho_{2}^{k,M})|^2+ |D_xH(\mathscr{X}_2^{k,M},\mathscr{Y}_2^{k,M},\varrho_{2}^{k,M})|^2\right]\mathrm{d}s\right.\\\left.+\mathbb{E}\left[|D_xg(\mathscr{X}_2^{k,M}(T),\varrho_{2}^{k,M}(T))|^2\right] \right),
 	\end{multline}
 	\begin{equation}
 		\mathscr{A}_2\coloneqq 
 		\frac{1}{2\lambda_1\epsilon}\left(\int_0^T\mathbb{E}\left[|F_1-F_2|^2+|G_1-G_2|^2\right]\mathrm{d}s+\mathbb{E}\left[|\Delta \mathscr{X}_{N,M}^{k}(0)|^2\right]\right),
 	\end{equation}
 	and
 	\begin{equation}
 		\mathscr{A}_3\coloneqq
 		{\lambda_1\epsilon}\left(\int_0^T\mathbb{E}\left[|\Delta\mathscr{Y}_{N,M}^{k}|^2+|\Delta \mathscr{X}_{N,M}^{k}|^2\right]\mathrm{d}s+\mathbb{E}\left[|\Delta\mathscr{Y}_{N,M}^{k}(0)|^2\right]+\mathbb{E}\left[| \Delta \mathscr{X}_{N,M}^{k}(T)|^2\right]\right).
 	\end{equation}
 	In view of the definition of $\mathfrak{E}_1$ in \eqref{defn-E_*} and $\mathfrak{E}_2$ in \eqref{defn-E_*2}, we see that $\mathscr{A}_1+\mathscr{A}_2 = \frac{\lambda_1}{2\epsilon}\mathfrak{E}_1$ and $\mathscr{A}_3 = \lambda_1\epsilon\mathfrak{E}_2$. Therefore,
 	\begin{multline}\label{final-young-bound-2}
 		|\lambda_2-\lambda_1|\expectation{|D_xg(\mathscr{X}_2^{k,M}(T),\varrho_{2}^{k,M}(T))|| \Delta \mathscr{X}_{N,M}^{k}(T)|}+\mathbb{E}\left[|\Delta \mathscr{X}_{N,M}^{k}(0)| |\Delta\mathscr{Y}_{N,M}^{k}(0)|\right]
 		+|\mathfrak{I}(0,T)|
 		\\\leq  \frac{\lambda_1}{2\epsilon}\mathfrak{E}_1+ \lambda_1\epsilon\mathfrak{E}_2.
 	\end{multline}
 	
 	In all, combining \eqref{final-young-bound-1} and \eqref{final-young-bound-2},  we get
 		\begin{equation}
 		\begin{split}
 			&\lambda_1\int_0^{t_n}\mathbb{E}[|\Delta\mathscr{Y}_{N,M}^{k}|^2]\mathrm{d}s
 			\leq  \frac{C_{\gamma}}{2\epsilon}\lambda_1\mathfrak{E}_1+{C_{\gamma}}\epsilon\lambda_1\mathfrak{E}_2+ C_{\gamma}\lambda_1\int_0^{t_n}\left(\mathbb{E}\left[|\Delta \mathscr{X}_{N,M}^{k}|^2\right]+\mathcal{W}_2^2(\varrho_1^{k,N},\varrho_{2}^{k,M})\right)\ds 
 		\end{split}
 	\end{equation}
 	for all $n\in\{0,1,\cdots,M_k\}$ and all constant $0<\epsilon<1$, where the quantities $\mathfrak{E}_1$ and $\mathfrak{E}_2$ are defined as in \eqref{defn-E_*} and \eqref{defn-E_*2}, respectively, and $C_{\gamma}$ is the constant defined in \eqref{Cgamma-constant}  which depends only on the Lipschitz constants $\Lip(D_pH)$ and $\Lip(D_xH)$, and the strong convexity constant $c_0>0$. Since $\varrho_1^{k,N}(t)$ and $\varrho_{2}^{k,M}(t)$ are laws for $\mathscr{X}_1^{k,N}(t)$ and $\mathscr{X}_2^{k,M}(t)$ for all $t\in [0,T]$, respectively, from the definition of $\mathcal{W}_2$ we have the bound
 	\begin{equation}\label{dist-diff-time-err-pf-gen}
 		\mathcal{W}_2^2(\varrho_1^{k,N}(t),\varrho_{2}^{k,M}(t))\leq \expectation{\abs{\Delta \mathscr{X}_{N,M}^{k}(t)}^2}\quad\forall t\in [0,T].
 	\end{equation}
	It then follows that 
 	\begin{equation}
 		\begin{split}
 			&\lambda_1\int_0^{t_n}\mathbb{E}[|\Delta\mathscr{Y}_{N,M}^{k}|^2]\mathrm{d}s
 			\leq  \frac{C_{\gamma}}{2\epsilon}\lambda_1\mathfrak{E}_1+{C_{\gamma}}\epsilon\lambda_1\mathfrak{E}_2+ 2C_{\gamma}\lambda_1\int_0^{t_n}\mathbb{E}\left[|\Delta \mathscr{X}_{N,M}^{k}|^2\right]\ds 
 		\end{split}
 	\end{equation}
 	for all $n\in \{0,1,\cdots,M_k\}$ and all constant $0<\epsilon<1$. Divide this bound by $\lambda_1\in (0,1]$ to obtain \eqref{y-diff-bound-tn-gen}, as desired.
\end{proof}

If in addition one has that $D_xg$ is Lipschitz continuous, we show that the preceding lemma and a discrete Gr\"onwall's inequality allow us to deduce an $L^{\infty}$-in-time stability bound for the discrete Hamiltonian system \eqref{numerical-scheme-gen}.
\begin{lemma}[Uniform $L^{\infty}$-stability in time as $k\to \infty$]\label{X-Y-diff-Linf-bound}
		Let the Hamiltonian $H$ satisfy the regularity assumptions \eqref{ass-H:1}, \eqref{ass-H:2}, \eqref{ass-H:3}, \eqref{ass-H:4} and the displacement monotonicity condition \eqref{ass-H:disp-mono-cond-H}. Let the terminal cost $g$ satisfy the regularity assumptions \eqref{ass-g:1}, \eqref{ass-g:2} and the displacement monotonicity condition \eqref{ass-g:disp-mono-cond-g}. Let $N,M,k\in\mathbb{N}$ and 
	\begin{subequations}
		\begin{align}
			&\lambda_1\in (0,1],  \quad F_1,G_1\in \mathbb{V}_k(0,T;\mathcal{R}_{d}(\mathcal{A}_N,\mathcal{O}_N)), \quad X_{0,1}^N\in  \mathcal{R}_{d}(\mathcal{A}_N,\mathcal{O}_N),\label{data-line-1'''}
			\\
			&\lambda_2\in [0,1], \quad F_2,G_2\in \mathbb{V}_k(0,T;\mathcal{R}_{d}(\mathcal{A}_M,\mathcal{O}_M)),\quad X_{0,2}^M\in  \mathcal{R}_{d}(\mathcal{A}_M,\mathcal{O}_M)\label{data-line-2'''}
		\end{align}
	\end{subequations}
	be given, and suppose that the triples
	\begin{subequations}
		\begin{align}
			(\mathscr{X}_1^{k,N},\mathscr{Y}_1^{k,N},\varrho_1^{k,N})&\in \mathbb{V}_k^+([0,T];\mathcal{R}_{d}(\mathcal{A}_N,\mathcal{O}_N))\times \mathbb{V}_k^-([0,T];\mathcal{R}_{d}(\mathcal{A}_N,\mathcal{O}_N))\times \mathbb{P}_k([0,T];\mathcal{P}_{d,N}),\label{triple-1''}\\
			(\mathscr{X}_2^{k,M},\mathscr{Y}_2^{k,M},\varrho_{2}^{k,M})&\in \mathbb{V}_k^+([0,T];\mathcal{R}_{d}(\mathcal{A}_M,\mathcal{O}_M))\times \mathbb{V}_k^-([0,T];\mathcal{R}_{d}(\mathcal{A}_M,\mathcal{O}_M))\times \mathbb{P}_k([0,T];\mathcal{P}_{d,M})\label{triple-2''}
		\end{align}
	\end{subequations}  each satisfy the discrete Hamiltonian system \eqref{numerical-scheme-gen} with data \eqref{data-line-1'''} and \eqref{data-line-2'''}, respectively. There exists $k_{\dagger}\in\mathbb{N}$ such that
	\begin{equation}\label{pre-X-Y-Linf-bound}
		\sup_{0\leq t\leq T}\expectation{\abs{(\mathscr{X}_1^{k,N}-\mathscr{X}_2^{k,M})(t)}^2}+\sup_{0\leq t\leq T}\expectation{\abs{(\mathscr{Y}_1^{k,N}-\mathscr{Y}_2^{k,M})(t)}^2}\lesssim \frac{1}{\epsilon}\mathfrak{E}_1
		+\epsilon\mathfrak{E}_2
	\end{equation}
	for each $k\geq k_{\dagger}$ and all $0<\epsilon<1$, where $\mathfrak{E}_1$ and $\mathfrak{E}_2$ are defined as in \eqref{defn-E_*} and \eqref{defn-E_*2}, respectively. The hidden constant in \eqref{pre-X-Y-Linf-bound}  depends only on $k_{\dagger}$, $T$, $c_0$ and the Lipschitz constants of $D_pH$, $D_xH$, and $D_xg$.
\end{lemma}
The proof of this result will rely on the following discrete version of the Gr\"onwall Lemma, c.f.\ \cite[Lemma 2]{jones1964fundamental}, \cite[Proposition 4.1]{Emmrich2000DiscreteVO}. 
\begin{lemma}[Discrete Gr\"onwall Lemma]\label{lemma-discrete-gronwall}
	Let $T>0$ be constant, and let $\tau_k>0$, $M_k\in\mathbb{N}$ satisfy $T=\tau_kM_k$ and $M_{k+1}\geq M_k$, for $k\in\mathbb{N}$, and $M_k\to\infty$ as $k\to\infty$. Let constants $b_x,b_y,c_x,c_y\geq 0$ and $k_*\in\mathbb{N}$ be given. For each $k\in\mathbb{N}$, let $\{x_n\}_{n=1}^{M_k}$ be a sequence of non-negative numbers which satisfy
	\begin{equation}\label{gronwall-sum-cond-x}
			x_n\leq b_x +c_x\tau_k\sum_{m=1}^nx_m,
	\end{equation}
	for $n\in \{0,1,2,\cdots,M_k\}$. Then, for each $k\geq k_{\dagger}^x$,
	\begin{equation}\label{gronwall-uniform-x-bound}
			\max_{0\leq n \leq M_k} x_{n}\leq b_x \exp\left(\frac{c_xT}{1-c_x\tau_{k_{\dagger}^x}}\right),
	\end{equation}
	where $k_{\dagger}^x$ is the smallest positive integer such that $1-c_{x}\tau_{k_{\dagger}^x}>0$. 
	
	For each $k\geq k_*$,  let $\{y_n\}_{n=1}^{M_k}$ be a sequence of non-negative numbers which satisfy
	\begin{equation}\label{gronwall-sum-cond-y}
		y_n\leq b_y +c_y\tau_k\sum_{m=n+1}^{M_k}y_{m-1}
	\end{equation} 
	for $n\in \{0,1,2,\cdots,M_k\}$. Then, for each $k\geq k_{\dagger}^y$,
	\begin{equation}\label{gronwall-uniform-y-bound}
		\max_{0\leq n \leq M_k} y_{n}\leq b_y \exp\left(\frac{c_yT}{1-c_y\tau_{k_{\dagger}^y}}\right),
	\end{equation}
	 where $k_{\dagger}^y$ is the smallest positive integer such that $1-c_{x}\tau_{k_{\dagger}^y}>0$ and $k_{\dagger}^y\geq k_*$.
\end{lemma}
For completeness, the proof of this lemma is provided in Appendix \ref{app-aux-disc-temp-ana}.

\begin{proof}[Proof of Lemma \ref{X-Y-diff-Linf-bound}]
 	Let $\Delta \mathscr{X}_{N,M}^{k}\coloneqq \mathscr{X}_1^{k,N}-\mathscr{X}_2^{k,M}\in \mathbb{V}_k^+([0,T];\LLspace)$ and $\Delta\mathscr{Y}_{N,M}^{k}\coloneqq \mathscr{Y}_1^{k,N}-\mathscr{Y}_2^{k,M}\in \mathbb{V}_k^-([0,T];\LLspace)$.  The integration-by-parts formula \eqref{eq:discrete_ibp-expectation_pp} then implies that
 	\begin{multline}
 		\expectation{\abs{\Delta \mathscr{X}_{N,M}^{k}(t_n)}^2}\\= \expectation{\abs{\Delta \mathscr{X}_{N,M}^{k}(0)}^2}+2\int_0^{t_n}\expectation{\Delta \mathscr{X}_{N,M}^{k}\cdot\partial_t\Ip \Delta \mathscr{X}_{N,M}^{k}}\mathrm{d}s- \sum_{j=1}^n\mathbb{E}\left[\left|\jump{\Delta \mathscr{X}_{N,M}^{k}}_{j-1}\right|^2\right]
 		\\
 		\leq \expectation{\abs{\Delta \mathscr{X}_{N,M}^{k}(0)}^2}+2\int_0^{t_n}\expectation{\Delta \mathscr{X}_{N,M}^{k}\cdot\partial_t\Ip \Delta \mathscr{X}_{N,M}^{k}}\mathrm{d}s
 	\end{multline}
 	for all $n\in \{0,1,2,\cdots,M_k\}$. The systems satisfied by $(\mathscr{X}_1^{k,N},\mathscr{Y}_1^{k,N},\varrho_1^{k,N}),(\mathscr{X}_2^{k,M},\mathscr{Y}_2^{k,M},\varrho_{2}^{k,M})$ respectively then imply that
 	\begin{equation}
 		\begin{split}
 			&\expectation{\abs{\Delta \mathscr{X}_{N,M}^{k}(t_n)}^2}
 			\\
 			&\leq \expectation{\abs{\Delta \mathscr{X}_{N,M}^{k}(0)}^2}
 			\\
 			&\qquad+2\int_0^{t_n}\expectation{\left(\lambda_1 D_pH(\mathscr{X}_1^{k,N},\mathscr{Y}_1^{k,N},\varrho_1^{k,N}) - \lambda_2 D_pH(\mathscr{X}_2^{k,M},\mathscr{Y}_2^{k,M},\varrho_{2}^{k,M})\right)\cdot \Delta \mathscr{X}_{N,M}^{k}}\ds
 			\\
 			&\qquad\qquad\qquad\qquad+2\int_0^{t_n}\expectation{(F_1-F_2)\cdot \Delta \mathscr{X}_{N,M}^{k}}\mathrm{d}s
 			\\
 			&=\expectation{\abs{\Delta \mathscr{X}_{N,M}^{k}(0)}^2}\\
 			&\quad+2\lambda_1\int_0^{t_n}\expectation{\left( D_pH(\mathscr{X}_1^{k,N},\mathscr{Y}_1^{k,N},\varrho_1^{k,N}) -  D_pH(\mathscr{X}_2^{k,M},\mathscr{Y}_2^{k,M},\varrho_{2}^{k,M})\right)\cdot \Delta \mathscr{X}_{N,M}^{k}}\ds
 			\\
 			&\qquad +2(\lambda_1-\lambda_2)\int_0^{t_n}\expectation{\left( D_pH(\mathscr{X}_2^{k,M},\mathscr{Y}_2^{k,M},\varrho_{2}^{k,M})\right)\cdot \Delta \mathscr{X}_{N,M}^{k}}\ds\\
 			&\qquad\qquad+2\int_0^{t_n}\expectation{(F_1-F_2)\cdot \Delta \mathscr{X}_{N,M}^{k}}\mathrm{d}s
 		\end{split}
 	\end{equation}
 	for $n\in\{0,1,2,\cdots,M_k\}$.
 	Lipschitz continuity of $D_pH$ via \eqref{ass-H:2} implies that
 	\begin{equation}
 		\begin{split}
 			&\expectation{\abs{\Delta \mathscr{X}_{N,M}^{k}(t_n)}^2}
 			\\
 			&\leq \expectation{\abs{\Delta \mathscr{X}_{N,M}^{k}(0)}^2}\\
 			&\quad+2
 			\Lip(D_pH)\int_0^{t_n}\expectation{\abs{\Delta \mathscr{X}_{N,M}^{k}}^2+\abs{\Delta\mathscr{Y}_{N,M}^{k}}\abs{\Delta \mathscr{X}_{N,M}^{k}}+\mathcal{W}_2(\varrho_1^{k,N},\varrho_{2}^{k,M})\abs{\Delta \mathscr{X}_{N,M}^{k}}}\ds
 			\\
 			&\quad+ 2|\lambda_1-\lambda_2|\int_0^{t_n}\expectation{ |D_pH(\mathscr{X}_2^{k,M},\mathscr{Y}_2^{k,M},\varrho_{2}^{k,M})||\Delta \mathscr{X}_{N,M}^{k}|}\ds+2\int_0^{t_n}\expectation{|F_1-F_2||\Delta \mathscr{X}_{N,M}^{k}|}\mathrm{d}s
 		\end{split}
 	\end{equation}
 	where we used the fact that $0<\lambda_1\leq 1$. We apply Young's inequality, without parameter, to each of the latter three integral terms of the upper bound to get
 	\begin{multline}
 		2
 		\Lip(D_pH)\int_0^{t_n}\expectation{\abs{\Delta \mathscr{X}_{N,M}^{k}}^2+\abs{\Delta\mathscr{Y}_{N,M}^{k}}\abs{\Delta \mathscr{X}_{N,M}^{k}}+\mathcal{W}_2(\varrho_1^{k,N},\varrho_{2}^{k,M})\abs{\Delta \mathscr{X}_{N,M}^{k}}}\ds\\\leq 2\Lip(D_pH)\left(\int_0^{t_n}\mathbb{E}\left[2\abs{\Delta \mathscr{X}_{N,M}^{k}}^2+\frac{1}{2}\abs{\Delta \mathscr{Y}_{N,M}^{k}}^2+\frac{1}{2}\mathcal{W}_2^2(\varrho_1^{k,N},\varrho_{2}^{k,M})\right]\mathrm{d}s\right),
 	\end{multline}
 	\begin{multline}
 		2|\lambda_1-\lambda_2|\int_0^{t_n}\expectation{ |D_pH(\mathscr{X}_2^{k,M},\mathscr{Y}_2^{k,M},\varrho_{2}^{k,M})||\Delta \mathscr{X}_{N,M}^{k}|}\ds\\\leq (\lambda_1-\lambda_2)^2\int_0^{t_n}\mathbb{E}\left[ |D_pH(\mathscr{X}_2^{k,M},\mathscr{Y}_2^{k,M},\varrho_{2}^{k,M})|^2\right]\ds +\int_0^{t_n}\mathbb{E}\left[ |\Delta \mathscr{X}_{N,M}^{k}|^2\right]\ds, 
 	\end{multline}
 	\begin{equation}
 		2\int_0^{t_n}\expectation{|F_1-F_2||\Delta \mathscr{X}_{N,M}^{k}|}\mathrm{d}s\leq  \int_0^{t_n}\mathbb{E}\left[ |F_1-F_2|^2\right]\ds +\int_0^{t_n}\mathbb{E}\left[ |\Delta \mathscr{X}_{N,M}^{k}|^2\right]\ds
 	\end{equation}
 	for each $n\in\{0,1,2,\cdots, M_k\}$. Adding these bounds, we then obtain 
 	\begin{multline}
 		\expectation{\abs{\Delta \mathscr{X}_{N,M}^{k}(t_n)}^2}
 		\\
 		\leq (4\Lip(D_pH)+2)\int_0^{t_n}\mathbb{E}\left[ |\Delta \mathscr{X}_{N,M}^{k}|^2\right]\ds + \Lip(D_pH)\int_0^{t_n}\mathcal{W}_2^2(\varrho_1^{k,N},\varrho_{2}^{k,M})\ds \\+ \Lip(D_pH)\int_0^{t_n}\mathbb{E}\left[ |\Delta \mathscr{Y}_{N,M}^{k}|^2\right]\ds
 		+\expectation{\abs{\Delta \mathscr{X}_{N,M}^{k}(0)}^2}\\+ (\lambda_1-\lambda_2)^2\int_0^{t_n}\mathbb{E}\left[ |D_pH(\mathscr{X}_2^{k,M},\mathscr{Y}_2^{k,M},\varrho_{2}^{k,M})|^2\right]\ds+ \int_0^{t_n}\mathbb{E}\left[ |F_1-F_2|^2\right]\ds 
 	\end{multline}
 	for each $n\in\{0,1,2,\cdots, M_k\}$. Since $0<\lambda_1\leq 1$ and $t_n\in (0,T)$ for 
 	for each $n\in\{0,1,2,\cdots, M_k\}$, we have the bound
 	\begin{multline}
 		 \expectation{\abs{\Delta \mathscr{X}_{N,M}^{k}(0)}^2}+ (\lambda_1-\lambda_2)^2\int_0^{t_n}\mathbb{E}\left[ |D_pH(\mathscr{X}_2^{k,M},\mathscr{Y}_2^{k,M},\varrho_{2}^{k,M})|^2\right]\ds+ \int_0^{t_n}\mathbb{E}\left[ |F_1-F_2|^2\right]\ds \leq \mathfrak{E}_1
 	\end{multline}
 	after recalling the definition of $\mathfrak{E}_1$ in \eqref{defn-E_*}. We therefore deduce that 	
 	\begin{equation}
 		\begin{split}
 			&\expectation{\abs{\Delta \mathscr{X}_{N,M}^{k}(t_n)}^2}
 			\lesssim \mathfrak{E}_1+\int_0^{t_n}\expectation{\abs{\Delta \mathscr{X}_{N,M}^{k}}^2+\abs{\Delta\mathscr{Y}_{N,M}^{k}}^2+\mathcal{W}_2^2(\varrho_1^{k,N},\varrho_{2}^{k,M})}\ds,
 		\end{split}
 	\end{equation} for $n\in\{0,1,2,\cdots, M_k\}$, where the hidden constant depends only on $\Lip(D_pH)$ and where $\mathfrak{E}_1$ is defined in \eqref{defn-E_*}. Using \eqref{dist-diff-time-err-pf-gen}, together  with \eqref{y-diff-bound-tn-gen} from Lemma \ref{lemma-X-bound-Y-discr-bound}, gives for all $n\in\{0,1,2,\cdots,M_k\}$ and $0<\epsilon<1$
 	\begin{equation}\label{X-diff-pre-gronwall}
 		\begin{split}
 			&\expectation{\abs{\Delta \mathscr{X}_{N,M}^{k}(t_n)}^2}
 			\lesssim \int_0^{t_n}\expectation{\abs{\Delta \mathscr{X}_{N,M}^{k}}^2}\ds
 			+ \frac{1}{\epsilon}\mathfrak{E}_1
 			+\epsilon\mathfrak{E}_2
 		\end{split}
 	\end{equation}  where the hidden constant in \eqref{X-diff-pre-gronwall} depends on only $c_0$ and the Lipschitz constants of $D_pH,D_xH$. 
 	
 	Recall that $\Delta \mathscr{X}_{N,M}^{k}$ is piecewise constant in time with values in $\LLspace$ and is left-continuous. As such, the bound \eqref{X-diff-pre-gronwall} can be written equivalently as a discrete summation inequality of the form 
 	\begin{equation}
 		\expectation{\abs{\Delta \mathscr{X}_{N,M}^{k}(t_n)}^2}
 		\leq C_{X}\left(\tau_k\sum_{m=1}^{n}\expectation{\abs{\Delta \mathscr{X}_{N,M}^{k}(t_m)}^2}
 		+ \frac{1}{\epsilon}\mathfrak{E}_1
 		+\epsilon\mathfrak{E}_2\right)
 	\end{equation}
 	for all $n\in\{0,1,2,\cdots,M_k\}$ where $C_{X}$ is the hidden constant in \eqref{X-diff-pre-gronwall}. Take the smallest positive integer $k_{\dagger}^x\in\mathbb{N}$ such that $1-C_{X}\tau_{k_{\dagger}^x}>0$. Note then that such $k_{\dagger}^x$ depends only $c_0$ and the Lipschitz constants of $D_pH,D_xH$.  We can therefore apply the forward-in-time discrete Gronwall inequality \eqref{gronwall-uniform-x-bound} from Lemma \ref{lemma-discrete-gronwall} to deduce that
 	\begin{equation}
 		\expectation{\abs{\Delta \mathscr{X}_{N,M}^{k}(t_n)}^2}\leq C_X \exp\left(\frac{C_XT}{1-C_X\tau_{k_{\dagger}^x}}\right)\left(\frac{1}{\epsilon}\mathfrak{E}_1
 		+\epsilon\mathfrak{E}_2\right)
 	\end{equation} for all $n\in\{0,1,2,\cdots,M_k\}$, $0<\epsilon<1$ and $k\geq k_{\dagger}^x$, where the constant $C_X \exp\left(\frac{C_XT}{1-C_X\tau_{k_{\dagger}^x}}\right)$ above depends only on $k_{\dagger}^x$, $T$, $c_0$ and the Lipschitz constants of $D_pH,D_xH$.
 	Since $\Delta \mathscr{X}_{N,M}^{k}$ is left-continuous in $[0,T]$, this bound holds for all $t\in [0,T]$, i.e.\
 	\begin{equation}\label{sup-norm-diff-X-time-approx-gen}
 		\sup_{0\leq t\leq T}\expectation{\abs{\Delta \mathscr{X}_{N,M}^{k}(t)}^2}\lesssim \frac{1}{\epsilon}\mathfrak{E}_1
 		+\epsilon\mathfrak{E}_2
 	\end{equation} for all $0<\epsilon<1$ and all $k\geq k_{\dagger}^x$. The hidden constant in the above bound depends on only $k_{\dagger}^x$, $c_0$, $T$, and the Lipschitz constants of $D_pH,D_xH$.
 	
 	To prove the analogous bound for the $Y$-component of the stability bound we follow the same approach. Indeed, we infer from the integration by parts formula \eqref{eq:discrete_ibp-expectation_nn} and the systems satisfied by $$(\mathscr{X}_1^{k,N},\mathscr{Y}_1^{k,N},\varrho_1^{k,N}),(\mathscr{X}_2^{k,M},\mathscr{Y}_2^{k,M},\varrho_{2}^{k,M}),$$ respectively, that
 	\begin{equation}
 		\begin{split}
 			&\expectation{\abs{\Delta\mathscr{Y}_{N,M}^{k}(t_n)}^2}
 			\\
 			&= \expectation{\abs{\Delta\mathscr{Y}_{N,M}^{k}(T)}^2}-2\int_{t_n}^T\expectation{\Delta\mathscr{Y}_{N,M}^{k}\cdot\partial_t\In \Delta\mathscr{Y}_{N,M}^{k}}\mathrm{d}s- \sum_{m=n+1}^{M_k}\mathbb{E}\left[|\jump{\Delta\mathscr{Y}_{N,M}^{k}}_{m}|^2\right]
 			\\
 			&\leq \expectation{\abs{\Delta\mathscr{Y}_{N,M}^{k}(T)}^2}-2\int_{t_n}^T\expectation{\Delta\mathscr{Y}_{N,M}^{k}\cdot\partial_t\In \Delta\mathscr{Y}_{N,M}^{k}}\mathrm{d}s
 			\\
 			&=\expectation{\abs{\lambda_1 D_xg(\mathscr{X}_1^{k,N}(T),\varrho_1^{k,N}(T)) -\lambda_2 D_xg(\mathscr{X}_2^{k,M}(T),\varrho_{2}^{k,M}(T))}^2}
 			\\
 			&\qquad+2\int_{t_n}^T\expectation{\left(\lambda_1 D_xH(\mathscr{X}_1^{k,N},\mathscr{Y}_1^{k,N},\varrho_1^{k,N}) - \lambda_2 D_xH(\mathscr{X}_2^{k,M},\mathscr{Y}_2^{k,M},\varrho_{2}^{k,M})\right)\cdot \Delta\mathscr{Y}_{N,M}^{k}}\mathrm{d}s
 			\\
 			&\qquad\qquad-2\int_{t_n}^T\expectation{(G_1-G_2)\cdot \Delta\mathscr{Y}_{N,M}^{k}}\ds
 		\end{split}
 	\end{equation}
 	which implies
 	\begin{equation}
 		\begin{split}
 			&\expectation{\abs{\Delta\mathscr{Y}_{N,M}^{k}(t_n)}^2}
 			\\
 			&\leq 2\lambda_1^2 \expectation{\abs{ D_xg(\mathscr{X}_1^{k,N}(T),\varrho_1^{k,N}(T)) - D_xg(\mathscr{X}_2^{k,M}(T),\varrho_{2}^{k,M}(T))}^2}
 			\\
 			&\quad+2(\lambda_1-\lambda_2)^2\expectation{\abs{D_xg(\mathscr{X}_2^{k,M}(T),\varrho_{2}^{k,M}(T))}^2}
 			\\
 			&\quad+2\lambda_1\int_{t_n}^T\expectation{\left| D_xH(\mathscr{X}_1^{k,N},\mathscr{Y}_1^{k,N},\varrho_1^{k,N}) -  D_xH(\mathscr{X}_2^{k,M},\mathscr{Y}_2^{k,M},\varrho_{2}^{k,M})\right| | \Delta\mathscr{Y}_{N,M}^{k}|}\mathrm{d}s
 			\\
 			&\quad+2|\lambda_1-\lambda_2|\int_{t_n}^T\expectation{| D_xH(\mathscr{X}_2^{k,M},\mathscr{Y}_2^{k,M},\varrho_{2}^{k,M})| | \Delta\mathscr{Y}_{N,M}^{k}|}\mathrm{d}s+2\int_{t_n}^T\expectation{|G_1-G_2| | \Delta\mathscr{Y}_{N,M}^{k}|}\ds
 		\end{split}
 	\end{equation}
 	for $n\in\{0,1,2,\cdots,M_k\}$. 
 	
 	Lipschitz continuity of $D_xg$ and $D_xH$ via \eqref{ass-g:2},\eqref{ass-H:2} and the assumption $0<\lambda_1\leq 1$ imply that
 	\begin{multline}\label{Y-delta-bound}
 		\expectation{\abs{\Delta\mathscr{Y}_{N,M}^{k}(t_n)}^2}
 		\\
 		\leq 4\Lip(D_xg)^2\expectation{\abs{\Delta \mathscr{X}_{N,M}^k(T)}^2+\mathcal{W}_2^2(\varrho_1^{k,N}(T),\varrho_{2}^{k,M}(T))}
 		+2(\lambda_1-\lambda_2)^2\expectation{\abs{D_xg(\mathscr{X}_2^{k,M}(T),\varrho_{2}^{k,M}(T))}^2}
 		\\
 		\quad+2\Lip(D_xH)\int_{t_n}^T\expectation{\abs{\Delta \mathscr{Y}_{N,M}^{k}}^2+\abs{\Delta\mathscr{Y}_{N,M}^{k}}\abs{\Delta \mathscr{X}_{N,M}^{k}}+\mathcal{W}_2(\varrho_1^{k,N},\varrho_{2}^{k,M})\abs{\Delta \mathscr{Y}_{N,M}^{k}}}\mathrm{d}s
 		\\
 		\quad+2|\lambda_1-\lambda_2|\int_{t_n}^T\expectation{| D_xH(\mathscr{X}_2^{k,M},\mathscr{Y}_2^{k,M},\varrho_{2}^{k,M})| | \Delta\mathscr{Y}_{N,M}^{k}|}\mathrm{d}s+2\int_{t_n}^T\expectation{|G_1-G_2| | \Delta\mathscr{Y}_{N,M}^{k}|}\ds.
 	\end{multline}
 	Using Young's inequality, without parameter, we obtain that
 	\begin{multline}
 		2\Lip(D_xH)\int_{t_n}^T\expectation{\abs{\Delta \mathscr{Y}_{N,M}^{k}}^2+\abs{\Delta\mathscr{Y}_{N,M}^{k}}\abs{\Delta \mathscr{X}_{N,M}^{k}}+\mathcal{W}_2(\varrho_1^{k,N},\varrho_{2}^{k,M})\abs{\Delta \mathscr{Y}_{N,M}^{k}}}\mathrm{d}s\\
 		\leq  2\Lip(D_xH)\left(\int_{t_n}^T\mathbb{E}\left[2\abs{\Delta \mathscr{Y}_{N,M}^{k}}^2+\frac{1}{2}\abs{\Delta \mathscr{X}_{N,M}^{k}}^2+\frac{1}{2}\mathcal{W}_2^2(\varrho_1^{k,N},\varrho_{2}^{k,M})\right]\mathrm{d}s\right),
 	\end{multline}
 	\begin{multline}
 		2|\lambda_1-\lambda_2|\int_{t_n}^T\expectation{| D_xH(\mathscr{X}_2^{k,M},\mathscr{Y}_2^{k,M},\varrho_{2}^{k,M})| | \Delta\mathscr{Y}_{N,M}^{k}|}\mathrm{d}s\\\leq (\lambda_1-\lambda_2)^2\int_{t_n}^T\mathbb{E}\left[ |D_xH(\mathscr{X}_2^{k,M},\mathscr{Y}_2^{k,M},\varrho_{2}^{k,M})|^2\right]\ds +\int_{t_n}^T\mathbb{E}\left[ |\Delta \mathscr{Y}_{N,M}^{k}|^2\right]\ds, 
 	\end{multline}
 	\begin{equation}
 	2\int_{t_n}^T\expectation{|G_1-G_2| | \Delta\mathscr{Y}_{N,M}^{k}|}\ds\leq  \int_{t_n}^T\mathbb{E}\left[ |G_1-G_2|^2\right]\ds +\int_{t_n}^T\mathbb{E}\left[ |\Delta \mathscr{Y}_{N,M}^{k}|^2\right]\ds,
 	\end{equation}
 	for each $n\in\{0,1,2,\cdots, M_k\}$.  Adding these bounds we obtain from  \eqref{Y-delta-bound}
 	\begin{multline}
 		\expectation{\abs{\Delta\mathscr{Y}_{N,M}^{k}(t_n)}^2}
 		\\
 		\leq 4\Lip(D_xg)^2\expectation{\abs{\Delta \mathscr{X}_{N,M}^k(T)}^2+\mathcal{W}_2^2(\varrho_1^{k,N}(T),\varrho_{2}^{k,M}(T))} 		+
 		 (4\Lip(D_xH)+2)\int_{t_n}^T\mathbb{E}\left[ |\Delta \mathscr{Y}_{N,M}^{k}|^2\right]\ds  \\+\Lip(D_xH)\int_{t_n}^T\mathcal{W}_2^2(\varrho_1^{k,N},\varrho_{2}^{k,M})\ds 
 		+ \Lip(D_xH)\int_{t_n}^T\mathbb{E}\left[ |\Delta \mathscr{X}_{N,M}^{k}|^2\right]\ds
 		\\+2(\lambda_1-\lambda_2)^2\expectation{\abs{D_xg(\mathscr{X}_2^{k,M}(T),\varrho_{2}^{k,M}(T))}^2}+ (\lambda_1-\lambda_2)^2\int_{t_n}^T\mathbb{E}\left[ |D_xH(\mathscr{X}_2^{k,M},\mathscr{Y}_2^{k,M},\varrho_{2}^{k,M})|^2\right]\ds\\+\int_{t_n}^T\mathbb{E}\left[ |G_1-G_2|^2\right]\ds
 	\end{multline}
 	Since $0<\lambda_1\leq 1$ and $t_n\in (0,T)$ for each $n\in\{0,1,2,\cdots,M_k\}$, we have that 
 	\begin{multline}
 		2(\lambda_1-\lambda_2)^2\expectation{\abs{D_xg(\mathscr{X}_2^{k,M}(T),\varrho_{2}^{k,M}(T))}^2}+ (\lambda_1-\lambda_2)^2\int_{t_n}^T\mathbb{E}\left[ |D_xH(\mathscr{X}_2^{k,M},\mathscr{Y}_2^{k,M},\varrho_{2}^{k,M})|^2\right]\ds\\+\int_{t_n}^T\mathbb{E}\left[ |G_1-G_2|^2\right]\ds\leq 2\mathfrak{E}_1
 	\end{multline}
 	where we recall that $\mathfrak{E}_1$ is defined in \eqref{defn-E_*}. We have therefore shown that
 	\begin{multline}
 			\expectation{\abs{\Delta\mathscr{Y}_{N,M}^{k}(t_n)}^2}
 			\lesssim \expectation{\abs{\Delta \mathscr{X}_{N,M}^{k}(T)}^2+\mathcal{W}_2^2(\varrho_1^{k,N}(T),\rho_{2}(T))}\\\qquad\qquad\qquad\qquad\qquad\qquad+ \int_{t_n}^T\expectation{\abs{\Delta \mathscr{X}_{N,M}^{k}}^2+\abs{\Delta\mathscr{Y}_{N,M}^{k}}^2+\mathcal{W}_2^2(\varrho_1^{k,N},\varrho_{2}^{k,M})}\ds+\mathfrak{E}_1
 	\end{multline}
 	whence \eqref{dist-diff-time-err-pf-gen} implies 
 	\begin{equation}
 		\begin{split}
 			&\expectation{\abs{\Delta\mathscr{Y}_{N,M}^{k}(t_n)}^2}
 			\lesssim \expectation{\abs{\Delta \mathscr{X}_{N,M}^{k}(T)}^2}+\mathfrak{E}_1+ \int_{t_n}^T\expectation{\abs{\Delta \mathscr{X}_{N,M}^{k}}^2+\abs{\Delta\mathscr{Y}_{N,M}^{k}}^2}\ds
 		\end{split}
 	\end{equation} for all $n\in\{0,1,2,\cdots,M_k\}$, where the hidden constant depends only on the Lipschitz constants of $D_xH$ and $D_xg$. The uniform bound \eqref{sup-norm-diff-X-time-approx-gen} for the difference in the solutions to the discrete state equations then gives
 	\begin{equation}\label{Y-Delta-gronwall condition}
 		\begin{split}
 			&\expectation{\abs{\Delta\mathscr{Y}_{N,M}^{k}(t_n)}^2} \leq C_Y\int_{t_n}^T\expectation{\abs{\Delta\mathscr{Y}_{N,M}^{k}}^2}\ds+ \tilde{C}_X\left(\frac{1}{\epsilon}\mathfrak{E}_1
 			+\epsilon\mathfrak{E}_2\right)
 		\end{split}
 	\end{equation} for all $n\in\{0,1,2,\cdots,M_k\}$, $k\geq k_{\dagger}^x$, and all $0<\epsilon<1$, where the constant $C_Y$ depends only on the Lipschitz constants of $D_xH$ and $D_xg$, while the constant $\tilde{C}_X$ takes the form $\tilde{C}_X\coloneqq C_Y(1+T)C_X \exp\left(\frac{C_XT}{1-C_X\tau_{k_{\dagger}^x}}\right)$ and depends only on $k_{\dagger}^x, c_0,T$ and the Lipschitz constants of $D_pH$,$D_xH,D_xg$. 
 	Recalling that $\Delta \mathscr{Y}_{N,M}^k$ is piecewise constant in time with values in $\LLspace$ and is right-continuous, the above bound \eqref{Y-Delta-gronwall condition} is equivalent to
 	 \begin{equation}
 	 	\expectation{\abs{\Delta \mathscr{Y}_{N,M}^{k}(t_n)}^2}
 	 	\leq C_{Y}\tau_k\sum_{m=n+1}^{M_k}\expectation{\abs{\Delta \mathscr{Y}_{N,M}^{k}(t_{m-1})}^2}+
 	 	\tilde{C}_X\left(\frac{1}{\epsilon}\mathfrak{E}_1
 	 	+\epsilon\mathfrak{E}_2\right)
 	 \end{equation}
 	 for all $n\in\{0,1,2,\cdots,M_k\}$ and $k\geq k_{\dagger}^x$. Taking $k_{\dagger}^y$ to be the smallest positive integer such that $1-C_Y\tau_{k_{\dagger}^y}>0$ and $k_{\dagger}^y\geq k_{\dagger}^x$, we take $k_*\coloneqq k_{\dagger}^x$ and apply the backward-in-time discrete Gronwall inequality \eqref{gronwall-uniform-y-bound} from Lemma \ref{lemma-discrete-gronwall} to deduce that
 	\begin{equation}
 		\expectation{\abs{\Delta\mathscr{Y}_{N,M}^{k}(t_n)}^2}\leq \tilde{C}_X \exp\left(\frac{C_YT}{1-C_Y\tau_{k_{\dagger}^y}}\right) \left(\frac{1}{\epsilon}\mathfrak{E}_1
 		+\epsilon\mathfrak{E}_2\right)
 	\end{equation} for all $n\in\{0,1,2,\cdots,M_k\}$, $0<\epsilon<1$, and $k\geq k_{\dagger}^y$. Note that $\Delta\mathscr{Y}_{N,M}^{k}$ is right-continuous in $[0,T]$ since $\Delta\mathscr{Y}_{N,M}^{k}\in  \mathbb{V}_k^-([0,T];\LLspace)$. Consequently, this bound holds for all $t\in [0,T]$, thereby implying
 	\begin{equation}\label{sup-norm-diff-Y-time-approx-gen}
 		\sup_{0\leq t\leq T}\expectation{\abs{\Delta\mathscr{Y}_{N,M}^{k}(t)}^2}\lesssim \frac{1}{\epsilon}\mathfrak{E}_1
 		+\epsilon\mathfrak{E}_2
 	\end{equation} for all $0<\epsilon<1$ and  $k\geq k_{\dagger}\coloneqq k_{\dagger}^y$. The hidden constant in the above bound takes the form $C_Y(1+T)C_X \exp\left(\frac{C_XT}{1-C_X\tau_{k_{\dagger}^x}}\right) \exp\left(\frac{C_YT}{1-C_Y\tau_{k_{\dagger}^y}}\right)$ depends on only $k_{\dagger}^x,c_0,T$ and the Lipschitz constants of the functions $D_pH,D_xH,D_xg$. The proof is completed upon summing \eqref{sup-norm-diff-X-time-approx-gen} and \eqref{sup-norm-diff-Y-time-approx-gen} for each $k\geq k_{\dagger}$, noting then that the hidden constant in the resultant bound depends on only $c_0,T$ and the Lipschitz constants of $D_pH,D_xH,D_xg$ since $k_{\dagger}^x$ depends only on $c_0$ and the Lipschitz constants of $D_pH,D_xH$ and $D_xg$.
 \end{proof}
 
 We now complete the proof of the main result of this section, namely Theorem  \ref{theorem-cont-dependence-discrete-Hamiltonian-sys}.
 
\begin{proof}[Proof of Theorem  \ref{theorem-cont-dependence-discrete-Hamiltonian-sys}]
	
		Assume that the Hamiltonian $H$ satisfies the regularity assumptions \eqref{ass-H:1}, \eqref{ass-H:2}, \eqref{ass-H:3}, \eqref{ass-H:4} and the displacement monotonicity condition \eqref{ass-H:disp-mono-cond-H}. Further assume that the terminal cost $g$ satisfies the regularity assumptions \eqref{ass-g:1}, \eqref{ass-g:2} and the displacement monotonicity condition \eqref{ass-g:disp-mono-cond-g}.  Let $k,N,M\in\mathbb{N}$ be given. Suppose that the triples $(\mathscr{X}_1^{k,N},\mathscr{Y}_1^{k,N},\varrho_1^{k,N})$,  $(\mathscr{X}_2^{k,M},\mathscr{Y}_2^{k,M},\varrho_{2}^{k,M})$ in \eqref{triple-1}, \eqref{triple-2} each satisfy the discrete Hamiltonian system \eqref{numerical-scheme-gen} with data \eqref{data-line-1} and  \eqref{data-line-2}, respectively. 
		Let $\Delta \mathscr{X}_{N,M}^{k}\coloneqq \mathscr{X}_1^{k,N}-\mathscr{X}_2^{k,M}\in \mathbb{V}_k^+([0,T];\LLspace)$ and $\Delta\mathscr{Y}_{N,M}^{k}\coloneqq \mathscr{Y}_1^{k,N}-\mathscr{Y}_2^{k,M}\in \mathbb{V}_k^-([0,T];\LLspace)$. By  Lemma \ref{X-Y-diff-Linf-bound}, there exists $k_{\dagger}$ such that, for each $k\geq k_{\dagger}$, we have the bound \eqref{pre-X-Y-Linf-bound}. From this bound we immediately deduce that, for $k\geq k_{\dagger}$, 
		\begin{multline}
			\mathfrak{E}_2= \mathbb{E}\left[|(\mathscr{Y}_1^{k,N}-\mathscr{Y}_2^{k,M})(0)|^2+|(\mathscr{X}_1^{k,N}-\mathscr{X}_2^{k,M})(T)|^2\right]\\+\int_0^T\mathbb{E}\left[|\mathscr{X}_1^{k,N}-\mathscr{X}_2^{k,M}|^2+|\mathscr{Y}_1^{k,N}-\mathscr{Y}_2^{k,M}|^2\right]\ds
			\\
			\leq (2+T)\left(\sup_{0\leq t\leq T}\expectation{\abs{(\mathscr{X}_1^{k,N}-\mathscr{X}_2^{k,M})(t)}^2}+\sup_{0\leq t\leq T}\expectation{\abs{(\mathscr{Y}_1^{k,N}-\mathscr{Y}_2^{k,M})(t)}^2}\right)\\\lesssim (2+T)\left( \frac{1}{\epsilon}\mathfrak{E}_1
			+\epsilon\mathfrak{E}_2\right),
	\end{multline}
	since $\Delta \mathscr{X}_{N,M}^k,\Delta \mathscr{Y}_{N,M}^k\in L^{\infty}(0,T;L^2(\Omega;\mathbb{R}^d))$. We have hence shown that
	\begin{equation}
		\begin{split}
		&\mathfrak{E}_2
		\leq C_* \left(\frac{1}{\epsilon}\mathfrak{E}_1
		+\epsilon\mathfrak{E}_2\right)
		\end{split} 
	\end{equation} for all $0<\epsilon<1$ and $k\geq k_{\dagger}$, where the constant $C_*$ depends on only $c_0,T$ and the Lipschitz constants of $D_pH$, $D_xH$, and $D_xg$. In fact, from the proof of the previous Lemma \ref{X-Y-diff-Linf-bound}, this constant can be taken to be 
	\begin{equation}\label{C_*-defn}
		C_*=2(T+2)(T+1)\max\{C_X,C_Y\}\exp\left(\left(\frac{C_X}{1-C_X\tau_{k_{\dagger}^x}}+\frac{C_Y}{1-C_Y\tau_{k_{\dagger}^y}}\right)T\right)
	\end{equation} Therefore, taking $\epsilon=\frac{1}{2}C_*>0$, we conclude that
	\begin{equation}
		\mathbb{E}\left[|\Delta\mathscr{Y}_{N,M}^{k}(0)|^2+|\Delta \mathscr{X}_{N,M}^{k}(T)|^2\right]+\int_0^T\expectation{\abs{\Delta\mathscr{Y}_{N,M}^{k}}^2+\abs{\Delta \mathscr{X}_{N,M}^{k}}^2}\mathrm{d}s=\mathfrak{E}_2\leq 4C_*^2 \mathfrak{E}_1
	\end{equation} for all $k\geq k_{\dagger}$. It then follows from  \eqref{pre-X-Y-Linf-bound} and \eqref{dist-diff-time-err-pf-gen} that, for $k\geq k_{\dagger}$, there holds
	\begin{equation}\label{unif-Linf-X-Y-diff-bound}
		\begin{split}
			&\sup_{t\in [0,T]}\mathcal{W}_2(\varrho_1^{k,N}(t),\varrho_{2}^{k,M}(t))+\sup_{t\in [0,T]}\left(\mathbb{E}\left[|\Delta\mathscr{Y}_{N,M}^{k}(t)|^2\right]\right)^{\frac{1}{2}}+\sup_{t\in [0,T]}\left(\mathbb{E}\left[|\Delta \mathscr{X}_{N,M}^{k}(t)|^2\right]\right)^{\frac{1}{2}}\lesssim \mathfrak{E}_1^{\frac{1}{2}}
		\end{split}
	\end{equation}
	where $\mathfrak{E}_1$ is as defined in \eqref{defn-E_*}. The hidden constant here takes the form $3C_*$, where $C_*$ is defined in \eqref{C_*-defn}, which depends on only $c_0$, $T$, and the Lipschitz constants of $D_pH$, $D_xH,$ and $D_xg$. This completes the proof of the lemma.
\end{proof}

\section{Existence of numerical approximations}\label{sec-existence}
In this section we establish Theorem \ref{theorem-existence-num-scheme} on the existence of solutions to the numerical scheme \eqref{numerical-scheme}. The strategy we employ will be based on the Leray--Schauder--Schaefer fixed point theorem. For completeness, we formulate here the version of this theorem that is given in \cite[Theorem 11.3]{gilbarg2015elliptic}. 
\begin{theorem}[Leray--Schauder--Schaefer fixed point theorem]\label{lemma-leray--schauder--schaefer-fpt}
	Let $\mathcal{T}$ be a continuous, compact mapping of a Banach space $\mathfrak{B}$ into itself, and suppose there exists a constant $M$ such that $\|x\|_{\mathfrak{B}}< M$ for all $x\in \mathfrak{B}$ and $\beta\in [0,1]$ satisfying $x = \beta \mathcal{T}x$. Then $\mathcal{T}$ has a fixed point.
\end{theorem}
\begin{remark}
	In the analysis of systems of partial differential equations {associated to MFGs}, such as \eqref{mfg-pde-sys}, a typical approach to establishing the existence of solutions is an application of Schauder's fixed point theorem, in the case where the system has differentiable Hamiltonians, e.g.\ \cite{achdou2020mean}, or its set-valued generalization in Kakutani's fixed point theorem,  when the Hamiltonian of the system is nondifferentiable, see e.g.\ \cite{ducasse2020second,osborne2022analysis,osborne2024erratum}. In either of those settings, formulating a fixed-point solution map with bounded images is often made feasible under some sufficient conditions which ensure uniform boundedness of solutions to the KFP equation of the system. 
    
    However, in our setting where we study discrete-in-time Hamiltonian ODE systems under some general hypotheses for the model data, no such boundedness property is ensured for solutions to either of the ODE subproblems if one unknown of the system is fixed, unless we further strengthen assumptions on the Hamiltonian. We show in the sequel that it instead suffices to prove uniform boundedness of solutions to a parametrized family of fixed point problems by leveraging the displacement monotonicity of the model data to ultimately conclude existence of a solution via the Leray--Schauder--Schaefer fixed point theorem.
\end{remark}

\subsection{Definition of discrete solution map}
Fix $N
\in
\mathbb{N}$, $X_0^N\in\mathcal{R}_{d}(\mathcal{A}_N,\mathcal{O}_N)$ distributed according to $
\rho_0^N$, for some weights $\mathcal{A}_N$ and partition $\mathcal{O}_{N}$ of $\Omega$, and let $T>0$ be arbitrary. To begin the proof of the main existence result, we introduce a discrete solution map and prove that it has a fixed point, thereby establishing existence as desired. For each $k,N\in\mathbb{N}$, let the map $T^{k,N}$ be defined as follows: for each $Z=(Z_1,Z_2)\in  \mathbb{V}_k^+([0,T];\mathcal{R}_{d}(\mathcal{A}_N,\mathcal{O}_N))\times  \mathbb{V}_k(0,T;\mathcal{R}_{d}(\mathcal{A}_N,\mathcal{O}_N))$, let 
$$T^{k,N}[Z]\coloneqq (T_{1}^{k,N}[Z],T_2^{k,N}[Z])\in  \mathbb{V}_k^+([0,T];\mathcal{R}_{d}(\mathcal{A}_N,\mathcal{O}_N))\times \mathbb{V}_k^-([0,T];\mathcal{R}_{d}(\mathcal{A}_N,\mathcal{O}_N))$$ uniquely satisfy 
\begin{equation}\label{numerical-scheme-fp-map}
	\left\{\begin{aligned}
		\mathcal{I}_+^kT_{1}^{k,N}[Z](t) &=X_0^{N}+ \int_0^tD_pH\left(Z_1(s),Z_2(s),(Z_1(s))_{\#}\mathbb{P}\right)\mathrm{d}s\qquad\qquad\text{ in }[0,T],
		\\
		\mathcal{I}_-^kT_{2}^{k,N}[Z](t) &= -D_xg\left(Z_1(T),(Z_1(T))_{\#}\mathbb{P} \right)+\int_t^TD_xH\left(Z_1(s),Z_2(s),(Z_1(s))_{\#}\mathbb{P}\right)\mathrm{d}s\quad\text{ in }[0,T].
	\end{aligned}\right.
\end{equation}
Since $\mathbb{V}_k^-([0,T];\mathcal{R}_{d}(\mathcal{A}_N,\mathcal{O}_N))$ is a vector subspace of $\mathbb{V}_k(0,T;\mathcal{R}_{d}(\mathcal{A}_N,\mathcal{O}_N))$, it is clear that $T^{k,N}$ maps $\mathbb{V}_k^+([0,T];\mathcal{R}_{d}(\mathcal{A}_N,\mathcal{O}_N))\times  \mathbb{V}_k(0,T;\mathcal{R}_{d}(\mathcal{A}_N,\mathcal{O}_N))$ into itself, and that any fixed point of $T^{k,N}$ is a solution of the numerical scheme \eqref{numerical-scheme}, whence $\rho^{k,N}(t) \coloneqq (X^{k,N}(t))_{\#}\mathbb{P}$. The following lemma confirms that this solution map is indeed well-defined since functions in $ \mathbb{V}_k^+([0,T];\mathcal{R}_{d}(\mathcal{A}_N,\mathcal{O}_N))\times \mathbb{V}_k^-([0,T];\mathcal{R}_{d}(\mathcal{A}_N,\mathcal{O}_N))$ are uniquely determined by their values at the temporal node points $\{0,t_1,t_2,\cdots,t_{M_k-1},T\}$.
\begin{lemma}\label{lemma-fp-map-nodal-charac}
	For each $Z=(Z_1,Z_2)\in  \mathbb{V}_k^+([0,T];\mathcal{R}_{d}(\mathcal{A}_N,\mathcal{O}_N))\times  \mathbb{V}_k(0,T;\mathcal{R}_{d}(\mathcal{A}_N,\mathcal{O}_N))$, $$T^{k,N}[Z]=(T_{1}^{k,N}[Z],T_2^{k,N}[Z])\in  \mathbb{V}_k^+([0,T];\mathcal{R}_{d}(\mathcal{A}_N,\mathcal{O}_N))\times \mathbb{V}_k^-([0,T];\mathcal{R}_{d}(\mathcal{A}_N,\mathcal{O}_N))$$ satisfies \eqref{numerical-scheme-fp-map} if and only if
	\begin{equation}\label{temp-node-cond-1}
		T_{1}^{k,N}[Z](t_n) =X_0^{N}+ \int_0^{t_n}D_pH\left(Z_1(s),Z_2(s),(Z_1(s))_{\#}\mathbb{P}\right)\mathrm{d}s
	\end{equation}
	\begin{equation}\label{temp-node-cond-2}
		T_{2}^{k,N}[Z](t_n) = -D_xg\left(Z_1(T),(Z_1(T))_{\#}\mathbb{P} \right)+\int_{t_n}^TD_xH\left(Z_1(s),Z_2(s),(Z_1(s))_{\#}\mathbb{P}\right)\mathrm{d}s
	\end{equation}
	for $n=0,1,2,\cdots,M_k$.
\end{lemma}
\begin{proof}
	It is clear that if a function $W=(W_1,W_2)\in  \mathbb{V}_k^+([0,T];\mathcal{R}_{d}(\mathcal{A}_N,\mathcal{O}_N))\times \mathbb{V}_k^-([0,T];\mathcal{R}_{d}(\mathcal{A}_N,\mathcal{O}_N))$ satisfies the equations 
	\begin{subequations}\label{numerical-scheme-fp-map-necessary}
		\begin{align}
			\mathcal{I}_+^kW_1(t) &=X_0^{N}+ \int_0^tD_pH\left(Z_1(s),Z_2(s),(Z_1(s))_{\#}\mathbb{P}\right)\mathrm{d}s\text{ in }[0,T], \label{disc-HJ-eqn-fp-map-nec}
			\\
			\mathcal{I}_-^kW_2(t) &= -D_xg\left(Z_1(T),(Z_1(T))_{\#}\mathbb{P} \right)+\int_t^TD_xH\left(Z_1(s),Z_2(s),(Z_1(s))_{\#}\mathbb{P}\right)\mathrm{d}s\text{ in }[0,T], \label{disc-KFP-eqn-fp-map-nec}
		\end{align}
	\end{subequations}
	then we immediately have that 
	\begin{equation}\label{temp-cond-w-1}
		W_1(t_n) =X_0^{N}+ \int_0^{t_n}D_pH\left(Z_1(s),Z_2(s),(Z_1(s))_{\#}\mathbb{P}\right)\mathrm{d}s
	\end{equation}
	\begin{equation}\label{temp-cond-w-2}
		W_2(t_n) = -D_xg\left(Z_1(T),(Z_1(T))_{\#}\mathbb{P} \right)+\int_{t_n}^TD_xH\left(Z_1(s),Z_2(s),(Z_1(s))_{\#}\mathbb{P}\right)\mathrm{d}s
	\end{equation}
	for $n=0,1,2,\cdots,M_k$.  This shows the necessity of the images of $T^{k,N}$ satisfying the temporal node conditions \eqref{temp-node-cond-1}, \eqref{temp-node-cond-2}. We now show the sufficiency of these temporal node conditions.
	
	
	Suppose   $W=(W_1,W_2)\in  \mathbb{V}_k^+([0,T];\mathcal{R}_{d}(\mathcal{A}_N,\mathcal{O}_N))\times \mathbb{V}_k^-([0,T];\mathcal{R}_{d}(\mathcal{A}_N,\mathcal{O}_N))$ satisfies  \eqref{temp-cond-w-1} and \eqref{temp-cond-w-2} for $n=0,1,2,\cdots,M_k$. We clearly then have that $W$ satisfies the equations \eqref{numerical-scheme-fp-map} for $t=0,t_1,t_2,\cdots,t_{M_k-1},T$. We thus seek to show that these equations hold for all $t\in [0,T]\backslash\{0,t_1,t_2,\cdots,t_{M_k-1},T\}$. Given $t_*\in [0,T]\backslash\{0,t_1,t_2,\cdots,t_{M_k-1},T\}$ there exists $m\in \{1,2,\cdots,M_k\}$ such that $t_*\in (t_{m-1},t_m)=I_m$. We have by definition of $\Ip W_1|_{I_m}$ and the fact that $W_1$ is left-continuous 
	\begin{equation}
		\begin{split} 
			\Ip W_1(t_*) &= W_1(t_*) + \tau_k^{-1}\left(t_m-t_*\right)\left(W_1|_{I_{m-1}}-W_1|_{I_m}\right)
			\\
			&=W_1(t_m) +  \tau_k^{-1}\left(t_m-t_*\right)\left(W_1(t_{m-1})-W_1(t_m)\right)
			\\
			&=W_1(t_m) +  \tau_k^{-1}\left(t_m-t_*\right) \int_{t_m}^{t_{m-1}}D_pH\left(Z_1(s),Z_2(s),(Z_1(s))_{\#}\mathbb{P}\right)\mathrm{d}s
		\end{split} 
	\end{equation}
	where in the final line we used the temporal node conditions \eqref{temp-cond-w-1}, \eqref{temp-cond-w-2}. Since functions in the space $ \mathbb{V}_k^+([0,T];\mathcal{R}_{d}(\mathcal{A}_N,\mathcal{O}_N))\times  \mathbb{V}_k(0,T;\mathcal{R}_{d}(\mathcal{A}_N,\mathcal{O}_N))$ are piecewise constant in time with respect to the temporal partition $\mathcal{J}_k$, we get 
	\begin{equation}
		\begin{split}
			&\Ip W_1(t_*) =W_1(t_m) +  \tau_k^{-1}\left(t_m-t_*\right)(t_{m-1}-t_m) D_pH\left(Z_1(s),Z_2(s),(Z_1(s))_{\#}\mathbb{P}\right)|_{I_m}
			\\
			&=X_0^{N}+ \int_0^{t_m}D_pH\left(Z_1(s),Z_2(s),(Z_1(s))_{\#}\mathbb{P}\right)\mathrm{d}s - \left(t_m-t_*\right) D_pH\left(Z_1(s),Z_2(s),(Z_1(s))_{\#}\mathbb{P}\right)|_{I_m}
			\\
			&=X_0^{N}+ \int_0^{t_{m-1}}D_pH\left(Z_1(s),Z_2(s),(Z_1(s))_{\#}\mathbb{P}\right)\mathrm{d}s +\left(t_*-t_{m-1}\right) D_pH\left(Z_1(s),Z_2(s),(Z_1(s))_{\#}\mathbb{P}\right)|_{I_m} 
			\\
			&\qquad+ \int_{t_{m-1}}^{t_m}D_pH\left(Z_1(s),Z_2(s),(Z_1(s))_{\#}\mathbb{P}\right)\mathrm{d}s + \left(t_{m-1}-t_{m}\right) D_pH\left(Z_1(s),Z_2(s),(Z_1(s))_{\#}\mathbb{P}\right)|_{I_m} 
		\end{split}
	\end{equation}
	Since $$\int_{t_{m-1}}^{t_m}D_pH\left(Z_1(s),Z_2(s),(Z_1(s))_{\#}\mathbb{P}\right)\mathrm{d}s=(t_m-t_{m-1}) D_pH\left(Z_1(s),Z_2(s),(Z_1(s))_{\#}\mathbb{P}\right)|_{I_m}$$ and $$\left(t_*-t_{m-1}\right) D_pH\left(Z_1(s),Z_2(s),(Z_1(s))_{\#}\mathbb{P}\right)|_{I_m} = \int_{t_{m-1}}^{t_*}D_pH\left(Z_1(s),Z_2(s),(Z_1(s))_{\#}\mathbb{P}\right)\mathrm{d}s,$$ we conclude that 
	\begin{equation}
		\Ip W_1(t_*)=X_0^{N}+ \int_0^{t_{*}}D_pH\left(Z_1(s),Z_2(s),(Z_1(s))_{\#}\mathbb{P}\right)\mathrm{d}s
	\end{equation}
	As $t_*\in [0,T]\backslash\{0,t_1,t_2,\cdots,t_{M_k-1},T\}$ was arbitrary, this shows the sufficiency of the temporal node condition for $W_1$. By exacty the same argument, we analogously deduce the sufficiency of the temporal node condition for $W_2$. In all, we deduce that the temporal node conditions \eqref{temp-node-cond-1}, \eqref{temp-node-cond-2} are sufficient for constructing a function $W$ that satisfies \eqref{numerical-scheme-fp-map}, thereby completing the proof.
\end{proof}

\subsection{Continuity and compactness of {of the} discrete solution map}
We aim to apply the Leray--Schauder--Schaefer fixed point theorem to prove the existence of a fixed point for $T^{k,N}$. To this end, we show in the next result that $T^{k,N}$ is a continuous and compact operator. Equip the finite-dimensional vector space $ \mathbb{X}_k\coloneqq \mathbb{V}_k^+([0,T];\mathcal{R}_{d}(\mathcal{A}_N,\mathcal{O}_N))\times  \mathbb{V}_k(0,T;\mathcal{R}_{d}(\mathcal{A}_N,\mathcal{O}_N))$ with the norm
\begin{equation}\label{Xk-space-norm}
	\lVert (Z_1,Z_2)\rVert_{\mathbb{X}_k}^2\coloneqq\|Z_1\|_{ \mathbb{V}_k^+([0,T];\mathcal{R}_{d}(\mathcal{A}_N,\mathcal{O}_N))}^2+\|Z_2\|_{ \mathbb{V}_k(0,T;\mathcal{R}_{d}(\mathcal{A}_N,\mathcal{O}_N))}^2\quad\forall Z=(Z_1,Z_2)\in\mathbb{X}_k.
\end{equation}
where 
\begin{subequations}
	\begin{align}
		&\|Z_1\|_{ \mathbb{V}_k^+([0,T];\mathcal{R}_{d}(\mathcal{A}_N,\mathcal{O}_N))}^2\coloneqq  \int_0^T\mathbb{E}\left[|\partial_t\Ip Z_1|^2+|Z_1|^2\right]\mathrm{d}s+\expectation{|Z_1(0)|^2}\label{Vkplus-norm}
		\\
		&\|Z_2\|_{ \mathbb{V}_k(0,T;\mathcal{R}_{d}(\mathcal{A}_N,\mathcal{O}_N))}^2\coloneqq  \int_0^T\mathbb{E}\left[|Z_2|^2\right]\mathrm{d}s\label{Vk-norm}
	\end{align}
\end{subequations}
for $(Z_1,Z_2)\in\mathbb{X}_k$.
\begin{remark}
	The finite dimensional normed spaces $( \mathbb{V}_k^+([0,T];\mathcal{R}_{d}(\mathcal{A}_N,\mathcal{O}_N)),\|\cdot\|_{ \mathbb{V}_k^+([0,T];\mathcal{R}_{d}(\mathcal{A}_N,\mathcal{O}_N))})$, $( \mathbb{V}_k(0,T;\mathcal{R}_{d}(\mathcal{A}_N,\mathcal{O}_N)),\|\cdot\|_{ \mathbb{V}_k(0,T;\mathcal{R}_{d}(\mathcal{A}_N,\mathcal{O}_N))})$ are Banach, and hence $\mathbb{X}_k$, equipped with the norm \eqref{Xk-space-norm}, is also finite-dimensional and Banach.
\end{remark}
\begin{remark}
Let us motivate the choice of the above norm for $ \mathbb{V}_k^+([0,T];\mathcal{R}_{d}(\mathcal{A}_N,\mathcal{O}_N))$. We note the following $L^{\infty}$-in-time bound for maps in $\mathbb{V}_k^+([0,T];\mathcal{R}_{d}(\mathcal{A}_N,\mathcal{O}_N))$: 
\begin{multline}\label{L-inf-bound-Vk+}
	\sup_{0\leq t\leq T}\mathbb{E}\left[|Z(t)|^2\right]\leq  \int_0^T\mathbb{E}\left[|\partial_t\Ip Z|^2+|Z|^2\right]\mathrm{d}s+\expectation{|Z(0)|^2}=\|Z\|_{ \mathbb{V}_k^+([0,T];\mathcal{R}_{d}(\mathcal{A}_N,\mathcal{O}_N))}^2\\
	\forall Z\in \mathbb{V}_k^+([0,T];\mathcal{R}_{d}(\mathcal{A}_N,\mathcal{O}_N)).
\end{multline}
This is proved directly as follows. The integration-by-parts formula \eqref{eq:discrete_ibp-expectation_pp} implies that 
$$\mathbb{E}[|V(t_n)|^2]\leq \mathbb{E}[|V(0)|^2]+2\int_{0}^{t_n}\mathbb{E}\left[\partial_t\Ip V\cdot V\right]\ds,$$ for $n=0,1,\cdots,M_k$. Apply Young's inequality to deduce the desired bound \eqref{L-inf-bound-Vk+} for times $t$ at the temporal nodes. But the bound therefore holds also for all time since functions in $\mathbb{V}_k^+([0,T];\mathcal{R}_{d}(\mathcal{A}_N,\mathcal{O}_N))$ are left-continuous. Therefore, for each $Z\in  \mathbb{V}_k^+([0,T];\mathcal{R}_{d}(\mathcal{A}_N,\mathcal{O}_N))$ we are able to bound the quantity $\sup_{0\leq t\leq T}\mathbb{E}\left[|Z(t)|^2\right]$ by the squared norm \eqref{Vkplus-norm}. This fact will play a role in establishing the continuity of the solution map $T^{k,N}$.
\end{remark}

The continuity and compactness of $T^{k,N}$ is now proved in the following result. 
\begin{lemma}\label{lemma-cont-compactness-of-fp-map}
		Let the Hamiltonian $H$ satisfy the regularity assumptions \eqref{ass-H:1}, \eqref{ass-H:2}. Let the terminal cost $g$ satisfy the regularity assumptions \eqref{ass-g:1}, \eqref{ass-g:2}. There exists a constant $C\geq 0$, which depends only on $T$ and the Lipschitz constants of $D_pH$, $D_xH$, and $D_xg$, such that, for each $k,N\in\mathbb{N}$, 
	\begin{equation}\label{disc-cont-dependence-bound-fp-map}
		\lVert T^{k,N}[Z]-T^{k,N}[W]\rVert_{\mathbb{X}_k}\leq C\lVert Z-W\rVert_{\mathbb{X}_k}\quad \forall Z,W\in \mathbb{X}_k.
	\end{equation}
	Furthermore, $T^{k,N}$ is a compact operator.
\end{lemma}
\begin{proof}
	Notice that, for each $n=1,2,\cdots,M_k$, and $Z\in\mathbb{X}_k$ there hold
	\begin{equation}
		T_{1}^{k,N}[Z]|_{I_n} = T_{1}^{k,N}[Z](t_n),\qquad T_{2}^{k,N}[Z]|_{I_n} = T_{2}^{k,N}[Z](t_{n-1})
	\end{equation}
	since $T_{1}^{k,N}[Z]$ is left-continuous and $T_{2}^{k,N}[Z]$ is right-continuous. Consequently, by the triangle inequality and H\"older inequality in time we get 
	\begin{equation}
		\begin{split}
			&\mathbb{E}\left[\left|T_{1}^{k,N}[Z]|_{I_n}-T_{1}^{k,N}[W]|_{I_n}\right|^2\right]=\mathbb{E}\left[\left|T_{1}^{k,N}[Z](t_n)-T_{1}^{k,N}[W](t_n)\right|^2\right]
			\\
			&=\mathbb{E}\left[\left| \int_0^{t_n}D_pH\left(Z_1(s),Z_2(s),(Z_1(s))_{\#}\mathbb{P}\right)-D_pH\left(W_1(s),W_2(s),(W_1(s))_{\#}\mathbb{P}\right)\mathrm{d}s\right|^2\right]
			\\
			&\leq \mathbb{E}\left[\left( \int_0^{t_n}|D_pH\left(Z_1(s),Z_2(s),(Z_1(s))_{\#}\mathbb{P}\right)-D_pH\left(W_1(s),W_2(s),(W_1(s))_{\#}\mathbb{P}\right)|\mathrm{d}s\right)^2\right]
			\\
			&\leq  \mathbb{E}\left[ t_n\int_0^{t_n}|D_pH\left(Z_1(s),Z_2(s),(Z_1(s))_{\#}\mathbb{P}\right)-D_pH\left(W_1(s),W_2(s),(W_1(s))_{\#}\mathbb{P}\right)|^2\mathrm{d}s\right]
			\\
			&\leq T \int_0^{t_n}\mathbb{E}\left[|D_pH\left(Z_1(s),Z_2(s),(Z_1(s))_{\#}\mathbb{P}\right)-D_pH\left(W_1(s),W_2(s),(W_1(s))_{\#}\mathbb{P}\right)|^2\right]\ds
		\end{split}
	\end{equation}
	where in the final line we used Fubini's theorem and the fact that
	$t_n\in [0,T]$ for $n\in \{0,1,2,\cdots,M_k\}$. By using Lipschitz continuity of  $D_pH$ via \eqref{ass-H:2} and the elementary bound $(a+b+c)^2\leq 4(a^2+b^2+c^2)$ for $a,b,c\in\mathbb{R}$,  we get
	\begin{equation}
		\begin{split}
			&\mathbb{E}\left[\left|T_{1}^{k,N}[Z]|_{I_n}-T_{1}^{k,N}[W]|_{I_n}\right|^2\right]\\
			&\leq  4\Lip(D_pH)^2T\left[\int_0^{T}\left(\mathbb{E}\left[|Z_1-W_1|^2+|Z_2-W_2|^2\right]+\mathcal{W}_2^2((Z_1(s))_{\#}\mathbb{P},(W_1(s))_{\#}\mathbb{P})\right)\mathrm{d}s\right]
		\end{split}
	\end{equation}
	for $Z,W\in\mathbb{X}_k$ and $n\in\{1,2,\cdots,M_k\}$. Since $\mathcal{W}_2^2((Z_1(s))_{\#}\mathbb{P},(W_1(s))_{\#}\mathbb{P})\leq \mathbb{E}\left[|Z_1(s)-W_1(s)|^2\right]$ for all $s\in [0,T]$, we get 
	\begin{equation}
		\begin{split}
			&\mathbb{E}\left[\left|T_{1}^{k,N}[Z]|_{I_n}-T_{1}^{k,N}[W]|_{I_n}\right|^2\right]\leq  8\Lip(D_pH)^2T\left[\int_0^{T}\left(\mathbb{E}\left[|Z_1-W_1|^2+|Z_2-W_2|^2\right]\right)\mathrm{d}s\right]
		\end{split}
	\end{equation}
	for $n\in \{1,2,\cdots,M_k\}$. Hence, integrating in time the piecewise constant in time map $[0,T]\ni t\mapsto \mathbb{E}\left[\left|T_{1}^{k,N}[Z]-T_{1}^{k,N}[W]\right|^2\right]$ and the above bound give
	\begin{multline}\label{disc-cont-bound-1}
		\int_0^T\mathbb{E}\left[\left|T_{1}^{k,N}[Z]-T_{1}^{k,N}[W]\right|^2\right]\mathrm{d}s=\tau_k\sum_{k=1}^{M_k}\mathbb{E}\left[\left|T_{1}^{k,N}[Z]|_{I_n}-T_{1}^{k,N}[W]|_{I_n}\right|^2\right]\\
		\leq  \tau_k\sum_{n=1}^{M_k}8\Lip(D_pH)^2T
		\lVert Z-W\rVert_{\mathbb{X}_k}^2\\
		=8\Lip(D_pH)^2T^2
		\lVert Z-W\rVert_{\mathbb{X}_k}^2\quad\forall Z,W\in \mathbb{X}_k
	\end{multline}
	since $\tau_kM_k=T$ for all $k\in\mathbb{N}$. For the difference in the second component $T_{2}^{k,N}[Z]-T_{2}^{k,N}[W]$ we likewise deduce that, for $n=1,2,\cdots,M_k$,
	\begin{equation}
		\begin{split}
			&\mathbb{E}\left[\left|T_{2}^{k,N}[Z]|_{I_n}-T_{2}^{k,N}[W]|_{I_n}\right|^2\right]=\mathbb{E}\left[\left|T_{2}^{k,N}[Z](t_{n-1})-T_{2}^{k,N}[W](t_{n-1})\right|^2\right]
			\\
			&=\mathbb{E}\left[\left| -D_xg\left(Z_1(T),(Z_1(T))_{\#}\mathbb{P} \right)+ D_xg\left(W_1(T),(W_1(T))_{\#}\mathbb{P} \right)\right.\right.
			\\
			&\left.\left. \qquad+\int_{t_{n-1}}^TD_xH\left(Z_1(s),Z_2(s),(Z_1(s))_{\#}\mathbb{P}\right)-D_xH\left(W_1(s),W_2(s),(W_1(s))_{\#}\mathbb{P}\right)\mathrm{d}s\right|^2\right]
			\\
			&\leq 2\mathbb{E}\left[\left| -D_xg\left(Z_1(T),(Z_1(T))_{\#}\mathbb{P} \right)+ D_xg\left(W_1(T),(W_1(T))_{\#}\mathbb{P} \right)\right|^2\right]
			\\
			&\qquad +2\mathbb{E}\left[\left|\int_{t_{n-1}}^TD_xH\left(Z_1(s),Z_2(s),(Z_1(s))_{\#}\mathbb{P}\right)-D_xH\left(W_1(s),W_2(s),(W_1(s))_{\#}\mathbb{P}\right)\mathrm{d}s\right|^2\right]
			\\
			&\leq  4\Lip(D_xg)^2\mathbb{E}\left[|Z_1(T)-W_1(T)|^2\right]
			+ 16\Lip(D_xH)^2T\left[\int_0^{T}\left(\mathbb{E}\left[|Z_1-W_1|^2+|Z_2-W_2|^2\right]\right)\mathrm{d}s\right]
			\\
			&\leq \left( 4\Lip(D_xg)^2+ 16\Lip(D_xH)^2T\right)
			\lVert Z-W\rVert_{\mathbb{X}_k}^2
		\end{split}
	\end{equation}
	due to the Lipschitz continuity of $D_xg$, $D_xH$ via \eqref{ass-g:2}, \eqref{ass-H:2} and the $L^{\infty}$-bound \eqref{L-inf-bound-Vk+}. This implies that 
	\begin{equation}\label{disc-cont-bound-2}
		\int_0^T\mathbb{E}\left[\left|T_{2}^{k,N}[Z]-T_{2}^{k,N}[W]\right|^2\right]\mathrm{d}s\leq \left( 4\Lip(D_xg)^2+ 16\Lip(D_xH)^2T\right)T
		\lVert Z-W\rVert_{\mathbb{X}_k}^2
	\end{equation} for all $Z,W\in \mathbb{X}_k$.
	
	Given $Z,W\in \mathbb{X}_k$, it is clear from the definition of $T^{k,N}$ that
	\begin{equation}
		\partial_t\mathcal{I}_+^kT_{1}^{k,N}[Z](t) = D_pH\left(Z_1(t),Z_2(t),(Z_1(t))_{\#}\mathbb{P}\right)
	\end{equation}
	\begin{equation}
		\partial_t\mathcal{I}_+^kT_{1}^{k,N}[W](t) = D_pH\left(W_1(t),W_2(t),(W_1(t))_{\#}\mathbb{P}\right)
	\end{equation} after differentiating \eqref{numerical-scheme-fp-map} for all $t\in [0,T]\backslash \{0,t_1,t_2,\cdots,t_{M_k-1},T\}$. From this we find that
	\begin{equation}
		\begin{split}
			&\int_0^T\mathbb{E}\left[\left|\partial_t\mathcal{I}_+^kT_{1}^{k,N}[Z]-\partial_t\mathcal{I}_+^kT_{1}^{k,N}[W]\right|^2\right]\mathrm{d}s\\
			&\qquad=\int_0^T\mathbb{E}\left[\left| D_pH\left(Z_1(s),Z_2(s),(Z_1(s))_{\#}\mathbb{P}\right)- D_pH\left(W_1(s),W_2(s),(W_1(s))_{\#}\mathbb{P}\right)\right|^2\right]\mathrm{d}s
			\\
			&\qquad\leq 8\Lip(D_pH)^2\left[\int_0^{T}\left(\mathbb{E}\left[|Z_1-W_1|^2+|Z_2-W_2|^2\right]\right)\mathrm{d}s\right]
		\end{split}
	\end{equation}where the last line was deduced from previous arguments. Hence,
	\begin{equation}\label{disc-cont-bound-3}
		\int_0^T\mathbb{E}\left[\left|\partial_t\mathcal{I}_+^kT_{1}^{k,N}[Z]-\partial_t\mathcal{I}_+^kT_{1}^{k,N}[W]\right|^2\right]\mathrm{d}s\leq  8\Lip(D_pH)^2
		\lVert Z-W\rVert_{\mathbb{X}_k}^2.
	\end{equation}
	Noting that $T_{1}^{k,N}[Z](0)-T_{1}^{k,N}[W](0)=\xi^N-\xi^N=0$, the bounds \eqref{disc-cont-bound-1}, \eqref{disc-cont-bound-2} and \eqref{disc-cont-bound-3} added together imply the desired continuous dependence bound \eqref{disc-cont-dependence-bound-fp-map}.
	
	Given a bounded sequence $\{Z_j\}_{j\in\mathbb{N}}$ in $\mathbb{X}_k$, the continuous dependence bound \eqref{disc-cont-dependence-bound-fp-map} implies that
	\begin{equation}
			\sup_{j\in\mathbb{N}}\lVert T^{k,N}[Z_j]-T^{k,N}[0]\rVert_{\mathbb{X}_k}\leq C\sup_{j\in\mathbb{N}}\lVert Z_j\rVert_{\mathbb{X}_k}<\infty.
 	\end{equation}
 	The triangle inequality consequently implies that the sequence $\{T^{k,N}[Z_j]\}_{j\in\mathbb{N}}$ is uniformly bounded in $\mathbb{X}_k$. But since $\mathbb{X}_k$ is a finite dimensional Banach space with norm $\|\cdot\|_{\mathbb{X}_k}$, we deduce that  $\{T^{k,N}[Z_j]\}_{j\in\mathbb{N}}$ is sequentially precompact. Therefore, $T^{k,N}$ is indeed a compact operator.
\end{proof}

\subsection{Proof of existence of numerical approximations}
We now show that solutions to a particular parametrized class of problems is uniformly bounded in the $\mathbb{X}_k$-norm, which will allow us to conclude the existence of a fixed point by the Leray--Schauder--Schaefer fixed point theorem. The class of fixed-point problems consider are defined as follows: given $\beta\in [0,1]$, find $Z_{\beta}^{k,N}=(X_{\beta}^{k,N},Y_{\beta}^{k,N})\in \mathbb{X}_k$ such that $Z_{\beta}^{k,N}=\beta T^{k,N}Z_{\beta}^{k,N}$, i.e.\
\begin{equation}\label{numerical-scheme-fp-map-lambda}
	\begin{split}
		\mathcal{I}_+^kX_{\beta}^{k,N}(t) &=\beta X_0^{N}+ \beta\int_0^tD_pH\left(X_{\beta}^{k,N}(s),Y_{\beta}^{k,N}(s),(X_{\beta}^{k,N}(s))_{\#}\mathbb{P}\right)\mathrm{d}s\qquad\qquad \text{ in }[0,T],
		\\
		\mathcal{I}_-^kY_{\beta}^{k,N}(t) &= -\beta D_xg\left(X_{\beta}^{k,N}(T),(X_{\beta}^{k,N}(T))_{\#}\mathbb{P} \right)\\
		&\qquad\qquad\qquad+\beta\int_t^TD_xH\left(X_{\beta}^{k,N}(s),Y_{\beta}^{k,N}(s),(X_{\beta}^{k,N}(s))_{\#}\mathbb{P}\right)\mathrm{d}s\quad\text{ in }[0,T]
	\end{split}
\end{equation} due to linearity of the interpolation operators $\Ip,\In$. We note in particular that $$Y_{\beta}^{k,N}\in  \mathbb{V}_k^-([0,T];\mathcal{R}_{d}(\mathcal{A}_N,\mathcal{O}_N)),$$ 
for all $\beta\in [0,1]$.

\begin{lemma}\label{lemma-unif-bound-for-lam-fp-map}
	Assume the hypotheses of Theorem \ref{theorem-existence-num-scheme}. For each and $k,N\in\mathbb{N}$, and $\beta\in [0,1]$, let $(X_{\beta}^{k,N},Y_{\beta}^{k,N})$ satisfy the scaled Hamiltonian system \eqref{numerical-scheme-fp-map-lambda}. There exists $k_{\dagger}\in\mathbb{N}$ such that 
	\begin{multline}\label{schaefer-unif-bound}
		\sup_{k\geq k_{\dagger}}\sup_{\beta\in[0,1]}\left[\sup_{0\leq t\leq T}\left(\mathbb{E}\left[|X_{\beta}^{k,N}(t)|^2+|Y_{\beta}^{k,N}(t)|^2\right]\right)^{\frac{1}{2}} \right.\\
		\left.+\text{\emph{ess}}\sup_{0\leq t\leq T}\left(\mathbb{E}\left[|\partial_t\Ip X_{\beta}^{k,N}(t)|^2\right]+\mathbb{E}\left[|\partial_t\In Y_{\beta}^{k,N}(t)|^2\right]\right)^{\frac{1}{2}} \right]\\\leq  C\left[\left(\mathbb{E}\left[|X_0^N|^2\right]\right)^{\frac{1}{2}}+1\right]
	\end{multline}	where the constants $C$ and $k_{\dagger}$ depend only on $,c_0$, $D_pH$,$D_xH$ and $D_xg$, and the time horizon $T>0$.
\end{lemma}

\begin{proof}
	It is clear that the bound holds if $\beta=0$, so let us assume that $\beta\in (0,1]$. Since $(X_{\beta}^{k,N},Y_{\beta}^{k,N})\in  \mathbb{V}_k^+([0,T];\mathcal{R}_{d}(\mathcal{A}_N,\mathcal{O}_N))\times \mathbb{V}_k^-([0,T];\mathcal{R}_{d}(\mathcal{A}_N,\mathcal{O}_N))$ satisfies the system \eqref{numerical-scheme-fp-map-lambda} by hypothesis and $(0,0,\delta_0)$ uniqely satisfies the generalized system \eqref{numerical-scheme-gen} with $\lambda=0$, $F,G=0$ in $\mathbb{V}_k(0,T;\mathcal{R}_{d}(\mathcal{A}_N,\mathcal{O}_N))$ and any fixed $X_0^N\in\mathcal{R}_{d}(\mathcal{A}_N,\mathcal{O}_N)$ (see Remark \ref{remark-uniqueness-gen-ham-sys}),  in Theorem  \ref{theorem-cont-dependence-discrete-Hamiltonian-sys} we take $M=N$, $\mathscr{X}_1^{k,N}=X_{\beta}^{k,N}$,  $\mathscr{Y}_1^{k,N}=Y_{\beta}^{k,N}$ with $\varrho_1^{k,N} = (\mathscr{X}_1^{k,N})_{\#}\mathbb{P} =(X_{\beta}^{k,N})_{\#}\mathbb{P}$, and $\mathscr{X}_2^{k,M}=0$ and $\mathscr{Y}_2^{k,M}=0$ with $\varrho_2^{k,M}=\delta_0$ where $F_1=F_2=G_1=G_2=0$ in $\mathbb{V}_k(0,T;\mathcal{R}_{d}(\mathcal{A}_N,\mathcal{O}_N))$, and $\lambda_1=\beta$, $\lambda_2=0$. As such, Theorem  \ref{theorem-cont-dependence-discrete-Hamiltonian-sys} implies
	\begin{equation}\label{beta-Linf-bound-X-Y}
		\begin{split}
			&\sup_{0\leq t\leq T}\left(\expectation{\abs{X_{\beta}^{k,N}}^2+\abs{Y_{\beta}^{k,N}}^2}\right)^{\frac{1}{2}}
			\\
			&\lesssim \left(\expectation{\abs{X_0^N}^2}\right)^{\frac{1}{2}}
			+\left(\left(\expectation{\abs{D_xg(0,\delta_0)}^2}\right)^{\frac{1}{2}}+\left\|D_pH(0,0,\delta_0)\right\|_{L^2(0,T;\LLspace)}+\left\|D_xH(0,0,\delta_0)\right\|_{L^2(0,T;\LLspace)}\right)
		\end{split}
	\end{equation}
	for all $k\geq k_{\dagger}$, $\beta\in (0,1]$, and $N\in\mathbb{N}$. 
	
	 We now prove the uniform bound on the partial derivatives of the interpolants. For each $N,k\in\mathbb{N}$ with $k\geq k_{\dagger}$ and each $\beta\in (0,1]$, the pair $(X_{\beta}^{k,N}, Y_{\beta}^{k,N})$ satisfies 
	 \begin{subequations}\label{numerical-scheme-non-gen}
	 	\begin{align}
	 		\partial_t\mathcal{I}_+^kX_{\beta}^{k,N}(t) &=  \beta D_pH\left(X_{\beta}^{k,N}(t),Y_{\beta}^{k,N}(t),\rho_{\beta}^{k,N}\right)\quad\text{in }\mathcal{R}_{d}(\mathcal{A}_N,\mathcal{O}_N)\label{disc-HJ-eqn-non-gen}
	 		\\
	 		\partial_t\mathcal{I}_-^kY_{\beta}^{k,N}(t) &=- \beta D_xH\left(X_{\beta}^{k,N}(t),Y_{\beta}^{k,N}(t),\rho_{\beta}^{k,N}\right)\quad\text{in }\mathcal{R}_{d}(\mathcal{A}_N,\mathcal{O}_N)\label{disc-KFP-eqn-non-gen}
	 		\\
	 		\rho_{\beta}^{k,N}(t) &= (X_{\beta}^{k,N}(t))_{\#}\mathbb{P} \label{disc-KFP-eqn-law-non-gen}
	 		\\
	 		X_{\beta}^{k,N}(0)&= \beta X_0^N,\quad Y_{\beta}^{k,N}(T) = -  \beta D_xg(X_{\beta}^{k,N}(T),\rho_{\beta}^{k,N}(T))&&\text{  }\label{disc-equiv-init-term-condition-non-gen}
	 	\end{align}
	 \end{subequations}
	 where \eqref{disc-HJ-eqn-non-gen} and \eqref{disc-KFP-eqn-non-gen} hold for all $t\in (0,T)\backslash \iota_k$, and the distribution is given by \eqref{disc-KFP-eqn-law-non-gen} for all $t\in [0,T]$. From the linear growth properties \eqref{linear-growth-D_pH},\eqref{linear-growth-D_xH} for the partial derivatives of $H$ w.r.t.\ $p$ and $x$, respectively, for each $\beta\in (0,1]$ we find that 
	 \begin{equation}
	 	\expectation{| D_pH\left(X_{\beta}^{k,N}(t),Y_{\beta}^{k,N}(t),\rho_{\beta}^{k,N}\right)|^2}\lesssim \expectation{|X_{\beta}^{k,N}(t)|^2+|Y_{\beta}^{k,N}(t)|^2}+\mathcal{W}_2^2(\rho_{\beta}^{k,N}(t),\delta_0)+1
	 \end{equation}
	 \begin{equation}
	 	\expectation{| D_xH\left(X_{\beta}^{k,N}(t),Y_{\beta}^{k,N}(t),\rho_{\beta}^{k,N}\right)|^2}\lesssim \expectation{|X_{\beta}^{k,N}(t)|^2+|Y_{\beta}^{k,N}(t)|^2}+\mathcal{W}_2^2(\rho_{\beta}^{k,N}(t),\delta_0)+1
	 \end{equation}
	 for all $t\in [0,T]$, where the hidden constants depend only on $D_pH$ and $D_xH$, respectively. Since we clearly have that $\mathcal{W}_2^2(\rho_{\beta}^{k,N}(t),\delta_0)\leq \expectation{|X_{\beta}^{k,N}(t)|^2}$ for all $t\in [0,T]$, we find that
	 \begin{multline}
	 	\expectation{| D_pH\left(X_{\beta}^{k,N}(t),Y_{\beta}^{k,N}(t),\rho_{\beta}^{k,N}\right)|^2}+\expectation{| D_xH\left(X_{\beta}^{k,N}(t),Y_{\beta}^{k,N}(t),\rho_{\beta}^{k,N}\right)|^2}\\
	 	\lesssim \expectation{|X_{\beta}^{k,N}|^2+|Y_{\beta}^{k,N}|^2}+1
	 \end{multline}
	 and hence the system \eqref{disc-HJ-eqn-non-gen},\eqref{disc-KFP-eqn-non-gen},\eqref{disc-KFP-eqn-law-non-gen}, \eqref{disc-equiv-init-term-condition-non-gen} implies further that
	 \begin{equation}
	 	\expectation{|	\partial_t\mathcal{I}_+^kX_{\beta}^{k,N}(t)|^2}+\expectation{|	\partial_t\mathcal{I}_-^kY_{\beta}^{k,N}(t)|^2}
	 	\lesssim \expectation{|X_{\beta}^{k,N}|^2+|Y_{\beta}^{k,N}|^2}+1
	 \end{equation} for all $t\in [0,T]\backslash \iota_k$. The $L^{\infty}$-bound \eqref{beta-Linf-bound-X-Y} on $X_{\beta}^{k,N}$ and $Y_{\beta}^{k,N}$ then yields
	 \begin{equation}\label{Linf-unif-bound-time-interp-deriv}
	 	\sup_{N\in\mathbb{N}}\sup_{k\geq k_{\dagger}}\text{ess}\sup_{0< t< T}\left(\mathbb{E}\left[|\partial_t\Ip X_{\beta}^{k,N}(t)|^2+|\partial_t \In Y_{\beta}^{k,N}(t)|^2\right]\right)^{\frac{1}{2}} \leq C''\left(\sup_{N\in\mathbb{N}}\left(\mathbb{E}\left[|X_0^N|^2\right]\right)^{\frac{1}{2}}+1\right).
	 \end{equation} 
	 where $C''$ is independent of $k,N\in\mathbb{N}$ and $\beta\in (0,1]$. This completes the proof.
\end{proof}

\begin{proof}[Proof of Theorem \ref{theorem-existence-num-scheme}]
Fixing $N\in\mathbb{N}$ and $k\geq k_{\dagger}$ where $k_{\dagger}$ is given by Lemma \ref{lemma-unif-bound-for-lam-fp-map}, we immediately deduce from \eqref{schaefer-unif-bound} that there exists a constant $M>0$ such that the set $$\{Z_{\beta}^{k,N}\in\mathbb{X}_k: Z_{\beta}^{k,N}=\beta T^{k,N}Z_{\beta}^{k,N},\text{ }\beta\in [0,1]\}$$ is contained in the open ball in $\mathbb{X}_k$ centred at the origin and of radius $M$. 
In view of Lemma \ref{lemma-cont-compactness-of-fp-map} which further shows that $T^{k,N}:\mathbb{X}_k\to\mathbb{X}_k$ is a continuous compact operator, we deduce by the Leray--Schauder-Schaefer fixed-point Theorem \ref{lemma-leray--schauder--schaefer-fpt} that $T^{k,N}$ admits a fixed point, i.e. there exists a solution to the numerical scheme \eqref{numerical-scheme} for each $k\geq k_{\dagger}$ and $N\in\mathbb{N}$. This proves Theorem \ref{theorem-existence-num-scheme}.
\end{proof}

\section{Convergence of {the} numerical approximations}\label{sec-convergence}
We aim to show that, in the joint limit as $k\to\infty$ and $N\to\infty$, then $(X^{k,N},Y^{k,N},\rho^{k,N})$ converges in a suitable norm to the unique solution of \eqref{definition-soln-of-continuous-pb}. 

\subsection{Convergence analysis of scheme for fixed $N$ as $k\to\infty$}

The main result of this section asserts the convergence of the numerical scheme \eqref{numerical-scheme} to a continuous time Hamiltonian system as the time-step vanishes while the number of samples $N$ of the support of $\rho_0$ is fixed. In this case, the limit continuous time problem is as follows: \emph{find $(X^N,Y^N,\rho^N)\in H^1(0,T;\mathcal{R}_{d}(\mathcal{A}_N,\mathcal{O}_N))\times H^1(0,T;\mathcal{R}_{d}(\mathcal{A}_N,\mathcal{O}_N))\times C([0,T];\mathcal{P}_2(\mathbb{R}^d))$ which satisfies} 
\begin{equation}\label{numerical-scheme-N-only}
	\left\{\begin{aligned}
		X^{N}(t) &=X_0^N+ \int_0^tD_pH\left(X^{N}(s),Y^{N}(s),\rho^{N}(s)\right)\mathrm{d}s, \qquad\text{in }\mathcal{R}_{d}(\mathcal{A}_N,\mathcal{O}_N)
		\\
		Y^{N}(t) &= -D_xg\left(X^{N}(T),\rho^{N}(T)\right)+\int_t^TD_xH\left(X^{N}(s),Y^{N}(s),\rho^{N}(s)\right)\mathrm{d}s, \qquad\text{in }\mathcal{R}_{d}(\mathcal{A}_N,\mathcal{O}_N)\\
		\rho^{N}(t) &= (X^{N}(t))_{\#}\mathbb{P} \qquad\text{in }\mathcal{P}_2(\mathbb{R}^d)
	\end{aligned}\right.
\end{equation} \emph{for all $t\in [0,T]$.} 
Under the hypotheses on the model data, importantly the displacement monotonicity of $H$ and $g$, this problem has at most one solution, see \cite[Theorem 4.5]{meszaros2024mean}. The main convergence result that we now proceed to establish also serves as an alternative proof of existence by a compactness argument.
\begin{lemma}[Convergence for fixed $N$]\label{theorem-convegrence-fixed-N}
	Assume the hypotheses of Theorem \ref{theorem-existence-num-scheme} and let $N\in\mathbb{N}$ be given. Let $\{(X^{k,N},Y^{k,N})\}_{k\geq k_{\dagger}}$ denote the sequence generated by the discrete Hamiltonian system \eqref{numerical-scheme} where $k_{\dagger}\in\mathbb{N}$ is given by Theorem  \ref{theorem-cont-dependence-discrete-Hamiltonian-sys}. Then,  there exists $(X^N,Y^N,\rho^N)\in C^1([0,T];\mathcal{R}_{d}(\mathcal{A}_N,\mathcal{O}_N))\times C^1([0,T];\mathcal{R}_{d}(\mathcal{A}_N,\mathcal{O}_N))\times C([0,T];\mathcal{P}_2(\mathbb{R}^d))$ which uniquely satisfies \eqref{numerical-scheme-N-only}. Moreover, 
	\begin{subequations}\label{numerical-scheme-N-only-conv}
		\begin{align}
			&X^{k,N}\to X^N {\rm\ and\ } Y^{k,N}\to Y^N {\rm\ in\ }L^p(0,T;\mathcal{R}_{d}(\mathcal{A}_N,\mathcal{O}_N)),\label{XkYkN-convergence}\\
			&\partial_t\Ip X^{k,N} \to \partial_t X^N {\rm\ and\ } \partial_t\In Y^{k,N} \to \partial_t Y^N {\rm \ in\  }L^p(0,T;\mathcal{R}_{d}(\mathcal{A}_N,\mathcal{O}_N)),\\
			&X^{k,N}(T)\to X^N(T) {\rm\ in\  }\mathcal{R}_{d}(\mathcal{A}_N,\mathcal{O}_N),\quad Y^{k,N}(0)\to Y^N(0) {\rm\ in\ }\mathcal{R}_{d}(\mathcal{A}_N,\mathcal{O}_N), \\
			&\rho^{k,N} \to \rho^N=(X^{N})_{\#}\mathbb{P} \quad {\rm\ in\  }L^2(0,T;\mathcal{P}_2(\mathbb{R}^d)),
			\quad\rho^{k,N}(t) \to \rho^N(t) \quad {\rm\ in\  }\mathcal{P}_2(\mathbb{R}^d){\rm{\ for}\  }t\in \{0,T\},
		\end{align}
	\end{subequations} as $k\to\infty$ for each  $p\in [1,\infty)$.
\end{lemma}

The proof of this result is based on precompactness of solutions to the discrete-time system \eqref{numerical-scheme}, which we establish in the next subsection. 

\subsubsection{Stability and precompactness}

The first step to proving Theorem \ref{theorem-convegrence-fixed-N} concerns establishing uniform stability for the solution of the scheme \eqref{numerical-scheme} and their time interpolants and corresponding time derivatives. 
\begin{lemma}\label{lemma-unif-k-N-bound-X-Y}
	Assume the hypotheses of Theorem \ref{theorem-existence-num-scheme}. There exist constants $C>0$ and $k_{\dagger}\in\mathbb{N}$ such that, for each $N\in\mathbb{N}$,
	\begin{multline}\label{Linf-unif-bound-time-interp}
		\sup_{k\geq k_{\dagger}}\left[\sup_{0\leq t\leq T}\left(\mathbb{E}\left[|X^{k,N}(t)|^2+|Y^{k,N}(t)|^2+|\Ip X^{k,N}(t)|^2+|\In Y^{k,N}(t)|^2\right]\right)^{\frac{1}{2}}\right. \\
		\left.+\text{\emph{ess}}\sup_{0< t< T}\left(\mathbb{E}\left[|\partial_t\Ip X^{k,N}(t)|^2+|\partial_t \In Y^{k,N}(t)|^2\right]\right)^{\frac{1}{2}}\right]\\\leq  C\left(\left(\mathbb{E}\left[|X_0^N|^2\right]\right)^{\frac{1}{2}}+1\right).
	\end{multline}
\end{lemma}
\begin{proof}
	For fixed $N$, we apply the uniform $L^{\infty}({L}^2)$-bound given by Lemma \ref{lemma-unif-bound-for-lam-fp-map} for $\beta=1$, which then implies 
	\begin{multline}\label{Linf-unif-bound}
		\sup_{k\geq k_{\dagger}}\left[\sup_{0\leq t\leq T}\left(\mathbb{E}\left[|X^{k,N}(t)|^2+|Y^{k,N}(t)|^2\right]\right)^{\frac{1}{2}} \right.\\
		\left.+\text{{ess}}\sup_{0< t< T}\left(\mathbb{E}\left[|\partial_t\Ip X^{k,N}(t)|^2+|\partial_t \In Y^{k,N}(t)|^2\right]\right)^{\frac{1}{2}}\right]\\\leq C\left(\left(\mathbb{E}\left[|X_0^N|^2\right]\right)^{\frac{1}{2}}+1\right)
	\end{multline}
	for each $N\in\mathbb{N}$, where $C>0$ is independent of $k,N\in\mathbb{N}$. It follows from this and the definition of the time-interpolation operators $\Ipm^k$ that the claimed uniform bound holds: 
	\begin{equation}
		\sup_{k\geq k_{\dagger}}\sup_{0\leq t\leq T}\left(\mathbb{E}\left[|\Ip X^{k,N}(t)|^2+|\In Y^{k,N}(t)|^2\right]\right)^{\frac{1}{2}}\leq  C'\left(\left(\mathbb{E}\left[|X_0^N|^2\right]\right)^{\frac{1}{2}}+1\right)\quad\forall N\in\mathbb{N}.
	\end{equation}
\end{proof}

Fixing $N\in\mathbb{N}$, the above result allows us to deduce precompactness of $\{(X^{k,N},Y^{k,N})\}_{k\geq k_{\dagger}}$ in the space $L^{p}(0,T;\mathcal{R}_{d}(\mathcal{A}_N,\mathcal{O}_N))\times L^{p}(0,T;\mathcal{R}_{d}(\mathcal{A}_N,\mathcal{O}_N))$ for each $p\in [1,\infty)$. 
\begin{lemma}\label{lemma-precompactness}
	Assume the hypotheses of Theorem \ref{theorem-existence-num-scheme} and let $N\in\mathbb{N}$ be given. The sequence $\{(X^{k,N},Y^{k,N})\}_{k\geq k_{\dagger}}$ generated by the numerical scheme \eqref{numerical-scheme} is precompact in the strong topology of $L^{p}(0,T;\mathcal{R}_{d}(\mathcal{A}_N,\mathcal{O}_N))\times L^{p}(0,T;\mathcal{R}_{d}(\mathcal{A}_N,\mathcal{O}_N))$ for each $p\in [1,\infty)$. Moreover, $\{(\Ip X^{k,N},\In Y^{k,N})\}_{k\geq k_{\dagger}}$ is precompact in the strong topology of $C([0,T];\mathcal{R}_{d}(\mathcal{A}_N,\mathcal{O}_N))\times C([0,T];\mathcal{R}_{d}(\mathcal{A}_N,\mathcal{O}_N))$. 
\end{lemma}
We prove this result by adapting the proof of \cite[Theorem 6.2.8]{osborne2024thesis}.
\begin{proof}
	Let $N\in\mathbb{N}$ be fixed. Since the bound in Lemma \ref{lemma-unif-k-N-bound-X-Y} implies uniform boundedness of the sequence $\{(X^{k,N},Y^{k,N})\}_{k\geq k_{\dagger}}$ in $L^{\infty}(0,T; \mathcal{R}_{d}(\mathcal{A}_N,\mathcal{O}_N))$, we have the uniform bound 
	\begin{equation}
		\sup_{k\geq k_{\dagger}}\sup_{0<t_1<t_2<T}\left[\left\|\int_{t_1}^{t_2}X^{k,N}(s)\mathrm{d}s\right\|_{\LLspace}+\left\|\int_{t_1}^{t_2}Y^{k,N}(s)\mathrm{d}s\right\|_{\LLspace}\right]
        \leq C''\left(\left(\mathbb{E}\left[|X_0^N|^2\right]\right)^{\frac{1}{2}}+1\right)
	\end{equation}
	where $C''>0$ is a constant that is independent of $N\in\mathbb{N}$. We recall here that $\|\cdot\|_{L^2(\Omega;\mathbb{R}^d)}$ is the standard norm on $L^2(\Omega;\mathbb{R}^d)$ that is induced by the expectation on $(\Omega,\mathbb{F},\mathbb{P})$: $\|W\|_{L^{2}(\Omega;\mathbb{R}^d)}=\left(\expectation{|W|^2}\right)^{1/2}$ for $W\in L^2(\Omega;\mathbb{R}^d)$.
	Since $\mathcal{R}_{d}(\mathcal{A}_N,\mathcal{O}_N)$ is a finite dimensional subspace of $L^2(\Omega;\mathbb{R}^d)$, we conclude that the sequences
	\begin{equation}\label{seq-int-precompact}
		\left\{\int_{t_1}^{t_2}X^{k,N}(s)\mathrm{d}s\right\}_{k\geq k_{\dagger}},\quad \left\{\int_{t_1}^{t_2}Y^{k,N}(s)\mathrm{d}s\right\}_{k\geq k_{\dagger}} \text{are precompact in $\mathcal{R}_{d}(\mathcal{A}_N,\mathcal{O}_N)$ }
	\end{equation}
	for all $0<t_1<t_2<T$.  For the remainder of the proof, we will show that $\{X^{k,N}\}_{k\geq k_{\dagger}}$ is precompact in $L^{p}(0,T;\mathcal{R}_{d}(\mathcal{A}_N,\mathcal{O}_N))$ for each $p\in [1,\infty))$. We omit the proof of precompactness for the sequence $\{Y^{k,N}\}_{k\geq k_{\dagger}}$ since it follows by the same argument.
    
    There holds the equicontinuity bound:
	\begin{multline}\label{equi-bound-general}
		\sup_{s\in (0,T)}\int_0^{T-s}s^{-1}\|V(t+s)-V(t)\|_{\LLspace}\mathrm{d}t\leq \int_0^T\|\partial_t\Ipm^k V\|_{\LLspace}\mathrm{d}t \quad
        \forall V\in \mathbb{V}_k^{\pm}([0,T];\mathcal{R}_{d}(\mathcal{A}_N,\mathcal{O}_N)).
	\end{multline}
	This is shown by the direct argument given in \cite[Lemma 6.2.9]{osborne2024thesis}, which uses the fact that maps in $\mathbb{V}_k^{\pm}([0,T];\mathcal{R}_{d}(\mathcal{A}_N,\mathcal{O}_N))$ are piecewise constant-in-time and defined everywhere in $[0,T]$ through left or right continuity. 
	We note that the equicontinuity bound \eqref{equi-bound-general} and Lemma \ref{lemma-unif-k-N-bound-X-Y} allow us to apply the argument of \cite[Theorem 6.2.8]{osborne2024thesis} to then obtain that
	\begin{equation}\label{equi-cont-bound}
		\sup_{k\geq k_{\dagger}}\|X^{k,N}(\cdot+s)-X^{k,N}(\cdot)\|_{L^{p}(0,T-s,\mathcal{R}_{d}(\mathcal{A}_N,\mathcal{O}_N))}\leq (2^{1-\frac{1}{p}}T^{\frac{1}{p}}M_{*,X,N})s^{\frac{1}{p}}\quad\forall s\in (0,T),
	\end{equation} for each $p\in [1,\infty)$, where the quantity
	\begin{equation}
		M_{*,X,N}\coloneqq \sup_{k\geq k_{\dagger}}\text{ess sup}_{0\leq t\leq T}\left(\| X^{k,N}(t)\|_{\LLspace}+\|\partial_t\Ip X^{k,N}(t)\|_{\LLspace}\right)
	\end{equation} is finite by Lemma \ref{lemma-unif-k-N-bound-X-Y}. Since \eqref{seq-int-precompact} and \eqref{equi-cont-bound} hold, we apply \cite[Theorem 1]{simon1986compact} to conclude that the sequence $\{X^{k,N}\}_{k\geq k_{\dagger}}$ is precompact in $L^{p}(0,T;\mathcal{R}_{d}(\mathcal{A}_N,\mathcal{O}_N))$ for any $p\in [1,\infty)$.
	
	We finally show that $\{(\Ip X^{k,N},\In Y^{k,N})\}_{k\geq k_{\dagger}}$ is precompact in the strong topology of space $$C([0,T];\mathcal{R}_{d}(\mathcal{A}_N,\mathcal{O}_N))\times C([0,T];\mathcal{R}_{d}(\mathcal{A}_N,\mathcal{O}_N)).$$ Observe that Lemma \ref{lemma-unif-k-N-bound-X-Y}  and the fact that  $\mathcal{R}_{d}(\mathcal{A}_N,\mathcal{O}_N)$ is a finite dimensional subspace of $\LLspace$ together imply that 
	for each $t\in (0,T)$, the sequence $\{(\Ip X^{k,N}(t),\In Y^{k,N}(t))\}_{k\geq k_{\dagger}}$ is bounded and hence precompact in $\mathcal{R}_{d}(\mathcal{A}_N,\mathcal{O}_N)\times \mathcal{R}_{d}(\mathcal{A}_N,\mathcal{O}_N)$. Lemma \ref{lemma-unif-k-N-bound-X-Y} further implies that the sequence $\{(\Ip X^{k,N},\In Y^{k,N})\}_{k\geq k_{\dagger}}$ is uniformly equicontinuous in time with values in $\mathcal{R}_{d}(\mathcal{A}_N,\mathcal{O}_N)\times\mathcal{R}_{d}(\mathcal{A}_N,\mathcal{O}_N)$. Thus, precompactness follows from the Arzela-Ascoli theorem, see e.g.\ \cite[Lemma 1]{simon1986compact}.
\end{proof}

We then establish the following result on convergence along subsequences.
\begin{lemma}\label{lemma-X-YkN-subseq-convergence}
	Assume the hypotheses of Theorem \ref{theorem-existence-num-scheme}. There exists a constant $C>0$ such that, for each $N\in\mathbb{N}$, there exist $X^N,Y^N\in C([0,T];\mathcal{R}_{d}(\mathcal{A}_N,\mathcal{O}_N))$ such that  
	\begin{subequations}
		\begin{align}
			&X^{k,N}\to X^N{\rm\ and\ } Y^{k,N}\to Y^N {\rm\ in\ }L^{p}(0,T;\mathcal{R}_{d}(\mathcal{A}_N,\mathcal{O}_N)),\quad X^{k,N}(T)\to X^N(T){\rm\ in\  }\mathcal{R}_{d}(\mathcal{A}_N,\mathcal{O}_N)
			\\
			& \Ip X^{k,N}\to {X}^N{\rm\ and\  }\In Y^{k,N}\to Y^N{\rm\ in\  }C([0,T];\mathcal{R}_{d}(\mathcal{A}_N,\mathcal{O}_N))
		\end{align}
	\end{subequations}
	as $k\to\infty$, along a subsequence for all $p\in [1,\infty)$, and
	\begin{equation}\label{Linf-unif-bound-N-post-time-limit}
		\sup_{0\leq t\leq T}\left(\mathbb{E}\left[| X^{N}(t)|^2+| Y^{N}(t)|^2\right]\right)^{\frac{1}{2}} \leq   C\left(\left(\mathbb{E}\left[|X_0^N|^2\right]\right)^{\frac{1}{2}}+1\right).
	\end{equation}
\end{lemma}
\begin{proof}
	Let $N
	\in\mathbb{N}$ be given. By Lemma \ref{lemma-precompactness} there exist ${X}^N,{Y}^N\in C([0,T];\mathcal{R}_{d}(\mathcal{A}_N,\mathcal{O}_N))$ such that $\Ip X^{k,N}\to {X}^N$ and $\In Y^{k,N}\to Y^N$ in $C([0,T];\mathcal{R}_{d}(\mathcal{A}_N,\mathcal{O}_N))$ as $k\to\infty$ along a subsequence. This fact allows us to pass to the limit in the uniform bound \eqref{Linf-unif-bound-time-interp} of Lemma \ref{lemma-unif-k-N-bound-X-Y} to obtain
	\begin{equation}
		\sup_{0\leq t\leq T}\left(\mathbb{E}\left[| X^{N}(t)|^2+| Y^{N}(t)|^2\right]\right)^{\frac{1}{2}} \leq  C\left(\left(\mathbb{E}\left[|X_0^N|^2\right]\right)^{\frac{1}{2}}+1\right)
	\end{equation}
	for a constant $C>0$ that is independent of $N\in\mathbb{N}$, thereby proving \eqref{Linf-unif-bound-N-post-time-limit}. By applying Lemma \ref{lemma-precompactness} once more, for each $p\in [1,\infty)$ we pass to further subsequences (without change of notation) to find that there exist $\tilde{X}_p^N,\tilde{Y}_p^N\in L^{p}(0,T;\mathcal{R}_{d}(\mathcal{A}_N,\mathcal{O}_N))$ (which depend on $p$ in general) such that $X^{k,N}\to \tilde{X}_p^N$ and $Y^{k,N}\to \tilde{Y}_p^N$ in $L^{p}(0,T;\mathcal{R}_{d}(\mathcal{A}_N,\mathcal{O}_N))$ as $k\to\infty$. We claim that $\tilde{X}_p^N=X^N$ and $\tilde{Y}_p^N=Y^N$ in $L^p(0,T;\mathcal{R}_{d}(\mathcal{A}_N,\mathcal{O}_N))$ for each $p\in [1,\infty)$, and hence
	\begin{equation}\label{interp-conv}
		X^{k,N}\to X^N \text{ and } Y^{k,N}\to Y^{N}\text{ in }L^p(0,T;\LLspace)
	\end{equation}
	as $k\to\infty$ for each $p\in [1,\infty)$. As such, the remainder of the proof seeks to establish  $\tilde{X}_p^N=X^N$ and $\tilde{Y}_p^N=Y^N$ in $L^p(0,T;\mathcal{R}_{d}(\mathcal{A}_N,\mathcal{O}_N))$ for each $p\in [1,\infty)$.
	
	We prove the claim for the sequence $\{X^{k,N}\}_{k\geq k_{\dagger}}$ as the argument for the sequence $\{ Y^{k,N}\}_{k\geq k_{\dagger}}$ follows analogously. Indeed, the uniform bound \eqref{Linf-unif-bound-time-interp} implies that the sequence $\{\partial_t\Ip X^{k,N}\}_{k\geq k_{\dagger}}$ is bounded in $L^2(0,T;\LLspace)$ uniformly in $N\in\mathbb{N}$. We may then adapt the proof of \cite[Theorem 6.2.7]{osborne2024thesis} to deduce that the weak derivative $\partial_t\tilde{X}_p^N\in L^2(0,T;L^2(\Omega;\mathbb{R}^d))$ exists such that $\partial_t\Ip X^{k,N}$ converges weakly to $\partial_t \tilde{X}_p^N$ in $L^2(0,T;\LLspace)$ along a further subsequence as $k\to\infty$, together with the fact that $\xi^N=\tilde{X}_p^N(0)=X^N(0)$ and $\tilde{X}_p^N(T)=X^N(T)$ in $\mathcal{R}_{d}(\mathcal{A}_N,\mathcal{O}_N)$. We then deduce that $X^{k,N}(T)=\Ip X^{k,N}(T)\to X^N(T)$ in $\mathcal{R}_{d}(\mathcal{A}_N,\mathcal{O}_N)$ as $k\to\infty$, while noting that $X^{k,N}(0)=\Ip X^{k,N}(0)= X_0^N=X^{N}(0)$ for all $k\in\mathbb{N}$. This, coupled with the fact that $X^{k,N}\to \tilde{X}_p^N$ and $\Ip X^{k,N}\to {X}^N$ in $L^2(0,T;\LLspace)$ as $k\to\infty$, allow us to deduce from integration-by-parts that
\begin{equation}
	\begin{split} 
		\int_0^T\expectation{-{X}^{N}\cdot \partial_t\phi}\ds
		&=\lim_{k\to \infty}\int_0^T\expectation{-\Ip X^{k,N}\cdot \partial_t\phi}\ds
		=\lim_{k\to \infty}\int_0^T\expectation{\partial_t\Ip X^{k,N}\cdot \phi}\ds
		=  \int_0^T\expectation{\partial_t \tilde{X}_p^{N}\cdot \phi }\ds
	\end{split} 
\end{equation} for all test functions $\phi\in C_c^{\infty}(0,T;\LLspace)$. Hence, the weak derivative $\partial_t {X}^N$ exists with $\partial_t{X}^N=\partial_t \tilde{X}_p^{N}$ in $L^2(0,T;\LLspace)$, which implies that $\tilde{X}_p^N(t)=X^N(t)+z_p^N$ in $\LLspace$ for a.e.\ $t\in (0,T)$, for some function $z_p^N\in\LLspace$. Since clearly ${X}^N+z_p^N\in C([0,T];\LLspace)$, we have $\tilde{X}_p^{N}\in C([0,T];\LLspace)$ as well, and so $\tilde{X}_p^N(t)=X^N(t)+z_p^N$ in $\LLspace$ for all $t\in [0,T]$. But since $\tilde{X}_p^N(0)=X^N(0)$ in $L^2(\Omega;\mathbb{R}^d)$, we deduce that $z_p^N=0$ in $L^2(\Omega;\mathbb{R}^d)$, which shows that $\tilde{X}_p^N=X^N$ in $C([0,T];\mathcal{R}_{d}(\mathcal{A}_N,\mathcal{O}_N))$ as claimed.
\end{proof}

\subsubsection{Proof of Lemma \ref{theorem-convegrence-fixed-N}}
\begin{proof}
Let $N\in\mathbb{N}$ be given. Lemma \ref{lemma-X-YkN-subseq-convergence} implies that there exist $X^N,Y^N\in C([0,T];\mathcal{R}_{d}(\mathcal{A}_N,\mathcal{O}_N))$ such that $X^{k,N}\to X^N$ and $Y^{k,N}\to Y^N$ in $L^{p}(0,T;\mathcal{R}_{d}(\mathcal{A}_N,\mathcal{O}_N))$ and $X^{k,N}(T)\to X^N(T)$ in $\LLspace$ as $k\to\infty$ along a subsequence for all $p\in [1,\infty)$. With this, it is straight-forward to show that $
\rho^{k,N}=(X^{k,N}(t))_{\#}\mathbb{P}\to \rho^{N} = (X^{N}(t))_{\#}\mathbb{P}$ in $L^2(0,T;\mathcal{P}_2(\mathbb{R}^d))$ and $\rho^{k,N}(T)\to \rho^N(T)$ in $\mathcal{P}_2(\mathbb{R}^d)$ as $k\to\infty$ along a subsequence. Indeed,  since 
\begin{equation}
	\mathcal{W}_2(\rho^{k,N}(t),\rho^{N}(t))\leq \left(\mathbb{E}\left[|X^{k,N}(t)-X^{N}(t)|^2\right]\right)^{\frac{1}{2}}\quad \forall t\in [0,T],
\end{equation} we conclude from the convergences  $X^{k,N}\to X^N$ in $L^{p}(0,T;\mathcal{R}_{d}(\mathcal{A}_N,\mathcal{O}_N))$ with $p=2$ and $X^{k,N}(T)\to X^N(T)$ in $\LLspace$ as $k\to\infty$ that $\rho^{k,N}\to \rho^N$ in $L^2(0,T;\mathcal{P}_2(\mathbb{R}^d))$ and $\rho^{k,N}(T)\to \rho^N(T)$ in $\mathcal{P}_2(\mathbb{R}^d)$ as $k\to\infty$ along a subsequence. In all, we have shown with Lemma \ref{lemma-X-YkN-subseq-convergence} that 
\begin{subequations}\label{numerical-scheme-N-only-conv-pre}
\begin{align}
	&X^{k,N}\to X^N \text{ and } Y^{k,N}\to Y^N \text{ in }L^p(0,T;\mathcal{R}_{d}(\mathcal{A}_N,\mathcal{O}_N)),\quad X^{k,N}(T)\to X^N(T) \text{ in }\mathcal{R}_{d}(\mathcal{A}_N,\mathcal{O}_N), \\
	&\Ip X^{k,N}\to {X}^N\text{ and }\In Y^{k,N}\to Y^N\text{ in }C([0,T];\mathcal{R}_{d}(\mathcal{A}_N,\mathcal{O}_N))\\
	&\rho^{k,N} \to \rho^N=(X^{N})_{\#}\mathbb{P} \quad \text{in }L^2(0,T;\mathcal{P}_2(\mathbb{R}^d)),\quad \rho^{k,N}(T) \to \rho^N(T) \quad \text{in }\mathcal{P}_2(\mathbb{R}^d),
\end{align}
\end{subequations} as $k\to\infty$ along a subsequence for any $p\in [1,\infty)$.

Using the convergence \eqref{numerical-scheme-N-only-conv-pre} and the uniform Lipschitz continuity of $D_pH$, $D_xH$ and $D_xg$ via \eqref{ass-H:2},\eqref{ass-g:2}, we deduce  that 
\begin{subequations}\label{D_pH-D_xH-D_xg-comp-convergence}
	\begin{align}
		&D_pH\left(X^{k,N}(t),Y^{k,N}(t),\rho^{k,N}(t)\right)\to D_pH\left(X^{N}(t),Y^{N}(t),\rho^{N}(t)\right) \text{ in }  L^2(0,T;\mathcal{R}_{d}(\mathcal{A}_N,\mathcal{O}_N))\\
		&D_xH\left(X^{k,N}(t),Y^{k,N}(t),\rho^{k,N}(t)\right)\to D_xH\left(X^{N}(t),Y^{N}(t),\rho^{N}(t)\right) \text{ in }  L^2(0,T;\mathcal{R}_{d}(\mathcal{A}_N,\mathcal{O}_N))\\
		&D_xg\left(X^{k,N}(T),\rho^{k,N}(T)\right)\to D_xg\left(X^{N}(T),\rho^{N}(T)\right) \text{ in }  \mathcal{R}_{d}(\mathcal{A}_N,\mathcal{O}_N)
	\end{align}
\end{subequations} as $k\to\infty$, along some subsequence. It then follows from H\"older's inequality and Fubini's theorem that
\begin{equation}
	\begin{split}
		&\lim_{k\to\infty}\sup_{0\leq t\leq T}\mathbb{E}\left|\int_0^tD_pH\left(X^{k,N},Y^{k,N},\rho^{k,N}\right)- D_pH\left(X^{N},Y^{N},\rho^{N}\right)\mathrm{d}s\right|^2
		\\
		&\qquad\qquad\leq T\lim_{k\to\infty}\int_0^T\mathbb{E}\left[\left|D_pH\left(X^{k,N},Y^{k,N},\rho^{k,N}\right)- D_pH\left(X^{N},Y^{N},\rho^{N}\right)\right|^2\right]\mathrm{d}s=0
	\end{split}
\end{equation}
\begin{equation}
	\begin{split}
		&\lim_{k\to\infty}\sup_{0\leq t\leq T}\mathbb{E}\left|\int_t^TD_xH\left(X^{k,N},Y^{k,N},\rho^{k,N}\right)- D_xH\left(X^{N},Y^{N},\rho^{N}\right)\mathrm{d}s\right|^2
		\\
		&\qquad\qquad\leq T\lim_{k\to\infty}\int_0^T\mathbb{E}\left[\left|D_xH\left(X^{k,N},Y^{k,N},\rho^{k,N}\right)- D_xH\left(X^{N},Y^{N},\rho^{N}\right)\right|^2\right]\mathrm{d}s=0.
	\end{split}
\end{equation}
We then deduce from the definition of the numerical scheme \eqref{numerical-scheme} and the uniform convergence of the interpolant terms via \eqref{numerical-scheme-N-only-conv-pre} that 
\begin{equation}\label{conv-result-1-alt}
	X^N(t)=\lim_{k\to\infty}\mathcal{I}_+^kX^{k,N}(t)=X_0^N+ \int_0^tD_pH\left(X^{N}(s),Y^{N}(s),\rho^{N}(s)\right)\mathrm{d}s\quad\text{in }\mathcal{R}_{d}(\mathcal{A}_N,\mathcal{O}_N)
\end{equation}
\begin{multline}\label{conv-result-2-alt}
	Y^N(t)=\lim_{k\to\infty}\mathcal{I}_-^kY^{k,N}(t)= -D_xg\left(X^{N}(T),\rho^{N}(T)\right)+\int_t^TD_xH\left(X^{N}(s),Y^{N}(s),\rho^{N}(s)\right)\mathrm{d}s\\
    \text{in }\mathcal{R}_{d}(\mathcal{A}_N,\mathcal{O}_N)
\end{multline}
for all $t\in [0,T]$, where $\rho^N(t):=(X^{N}(t))_{\#}\mathbb{P}$ is in $ C([0,T];\mathcal{P}_2(\mathbb{R}^d))$ since $X^N\in C([0,T];\mathcal{R}_{d}(\mathcal{A}_N,\mathcal{O}_N))$. We also clearly see that 
\begin{subequations}
	\begin{align}
		&\partial_tX^N=D_pH\left(X^{N}(t),Y^{N}(t),\rho^{N}(t)\right)\text{ in }L^2(0,T;\mathcal{R}_{d}(\mathcal{A}_N,\mathcal{O}_N))
		\\
		&\partial_tY^N=-D_xH\left(X^{N}(t),Y^{N}(t),\rho^{N}(t)\right)\text{ in }L^2(0,T;\mathcal{R}_{d}(\mathcal{A}_N,\mathcal{O}_N))
	\end{align}
\end{subequations}
We note that these time derivatives are continuous in $[0,T]$ since $D_pH$ and $D_xH$ are continuous while $X^N,Y^N\in C([0,T];\mathcal{R}_{d}(\mathcal{A}_N,\mathcal{O}_N))$ and $\rho^N\in C([0,T];\mathcal{P}_2(\mathbb{R}^d))$. As such, we have thus shown for each $N\in\mathbb{N}$ the existence of a solution $(X^N,Y^N,\rho^N)\in C^1([0,T];\mathcal{R}_{d}(\mathcal{A}_N,\mathcal{O}_N))\times C^1([0,T];\mathcal{R}_{d}(\mathcal{A}_N,\mathcal{O}_N))\times C([0,T];\mathcal{P}_2(\mathbb{R}^d))$ to the continuous time Hamiltonian system \eqref{numerical-scheme-N-only}. 

In addition, we deduce from the numerical scheme \eqref{numerical-scheme} and the convergence \eqref{D_pH-D_xH-D_xg-comp-convergence} that
\begin{equation}
	\begin{split}
		&\int_0^T\mathbb{E}\left[|\partial_t\Ip X^{k,N} - \partial_t X^N|^2\right]\mathrm{d}s+\int_0^T\mathbb{E}\left[|\partial_t\In Y^{k,N} - \partial_t Y^N|^2\right]\mathrm{d}s
		\\
		&=\int_0^T\mathbb{E}\left[|D_pH\left(X^{k,N},Y^{k,N},\rho^{k,N}\right) - D_pH\left(X^{N},Y^{N},\rho^{N}\right)|^2\right]\mathrm{d}s\\
		&\qquad\qquad\qquad+\int_0^T\mathbb{E}\left[|D_xH\left(X^{k,N},Y^{k,N},\rho^{k,N}\right) - D_xH\left(X^{N},Y^{N},\rho^{N}\right)|^2\right]\mathrm{d}s
		\\
		&\qquad\qquad\qquad\qquad\qquad\qquad\qquad\qquad\qquad\qquad\qquad\to 0 \text{ as }k\to\infty.
	\end{split}
\end{equation}
Thus, we have strong convergence: $\partial_t\Ip X^{k,N} \to \partial_t X^N$ and $\partial_t\Ip Y^{k,N} \to \partial_t Y^N$ in $L^2(0,T;\mathcal{R}_{d}(\mathcal{A}_N,\mathcal{O}_N))$ as $k\to\infty$ along a subsequence. But, by Lemma \ref{lemma-convegrence-fixed-N-inf}, we have that the solution to \eqref{numerical-scheme-N-only} is unique. Hence this convergence and \eqref{numerical-scheme-N-only-conv-pre} hold for the entire sequence in the limit as $k\to\infty$. This completes the proof.
\end{proof}

\subsection{{Uniform in $N$} temporal error bound}
In this section we prove the following error bound for the convergence \eqref{XkYkN-convergence} as $k\to\infty$ for fixed $N\in\mathbb{N}$, which shows in particular that this convergence also holds for $p=\infty$ uniformly in time. 
\begin{lemma}[Error bound in time uniform in $N$]\label{theorem-time-err-bound-fixed-N} Assume the hypotheses of Theorem \ref{theorem-convergence-k-N-joint}. There holds
	\begin{multline}\label{part-conv-bound-conv-rate-time}
			\sup_{t\in [0,T]}\left(\|(Y^{N}-Y^{k,N})(t)\|_{\LLspace}+\|(X^{N}-X^{k,N})(t)\|_{\LLspace}\right)+\sup_{t\in [0,T]}\mathcal{W}_2(\rho^{N}(t),\rho^{k,N}(t))\lesssim \tau_k
	\end{multline}
	for all $k\in\mathbb{N}$ sufficiently large and all $N\in\mathbb{N}$, where the hidden constant in \eqref{part-conv-bound-conv-rate-time} is independent of $k$ and $N$.
\end{lemma}
The analysis will be based on the following projection operators. Let $\qPikp:C([0,T];\mathcal{R}_{d}(\mathcal{A}_N,\mathcal{O}_N))\to \mathbb{V}_k^{+}([0,T];\mathcal{R}_{d}(\mathcal{A}_N,\mathcal{O}_N))$, $\qPikn:C([0,T];\mathcal{R}_{d}(\mathcal{A}_N,\mathcal{O}_N))\to \mathbb{V}_k^{-}([0,T];\mathcal{R}_{d}(\mathcal{A}_N,\mathcal{O}_N))$ be the operators given by
\begin{equation}
	\qPikp \mathscr{X}(t)\coloneqq \mathscr{X}(t_n)\text{ in }\mathcal{R}_{d}(\mathcal{A}_N,\mathcal{O}_N),\quad t\in (t_{n-1},t_n],\quad n\in\{1,2,\cdots,M_k\} 
\end{equation} with $\qPikp \mathscr{X}(0)\coloneqq \mathscr{X}(0)$ in $\mathcal{R}_{d}(\mathcal{A}_N,\mathcal{O}_N)$ for all $\mathscr{X}\in C([0,T];\mathcal{R}_{d}(\mathcal{A}_N,\mathcal{O}_N))$, and 
\begin{equation}
	\qPikn \mathscr{Y}(t)\coloneqq \mathscr{Y}(t_{n-1})\text{ in }\mathcal{R}_{d}(\mathcal{A}_N,\mathcal{O}_N),\quad t\in [t_{n-1},t_n),\quad n\in\{1,2,\cdots,M_k\} 
\end{equation} with $\qPikn \mathscr{Y}(T)\coloneqq \mathscr{Y}(T)$ in $\mathcal{R}_{d}(\mathcal{A}_N,\mathcal{O}_N)$ for all $\mathscr{Y}\in C([0,T];\mathcal{R}_{d}(\mathcal{A}_N,\mathcal{O}_N))$. Given $\mathscr{X}\in C([0,T];\mathcal{R}_{d}(\mathcal{A}_N,\mathcal{O}_N))$, we also define  $\pi_{\mathscr{X}}^k{\rho}^N(t)\coloneqq (\qPikp \mathscr{X}(t))_{\#}\mathbb{P}$, which is the law of the projection $\qPikp \mathscr{X}(t)$, for each $t\in [0,T]$.

The following result provides a temporal error in the projections of the solution $$(X^N,Y^N,\rho^N)\in  C^1(0,T;\mathcal{R}_{d}(\mathcal{A}_N,\mathcal{O}_N))\times C^1(0,T;\mathcal{R}_{d}(\mathcal{A}_N,\mathcal{O}_N))\times C([0,T];\mathcal{P}_2(\mathbb{R}^d))$$ to the continuous time system \eqref{numerical-scheme-N-only} for each $N\in\mathbb{N}$.
\begin{lemma}\label{lemma-temporal-inter-error-bound}
	Assume the hypotheses of Theorem \ref{theorem-convergence-k-N-joint}. There holds
	\begin{multline}\label{time-projection-err-bound}
			\sup_{t\in [0,T]}\left(\|(Y^{N}-\qPikn Y^{N})(t)\|_{\LLspace}+\|(X^{N}-\qPikp X^{N})(t)\|_{\LLspace}\right)+\sup_{t\in [0,T]}\mathcal{W}_2(\rho^{N}(t),\pi_{X^N}^k{\rho}^N(t))\lesssim \tau_k
	\end{multline}
	for all $k,N\in\mathbb{N}$, where the hidden constant in \eqref{part-conv-bound-conv-rate-time} is independent of $k$ and $N$.
\end{lemma}

\begin{proof}
	Fix $k,N\in\mathbb{N}$. Given $t\in (0,T]$, there exists $n\in \{1,2,\cdots,M_k\}$ such that $t\in (t_{n-1},t_n]$ and hence the system \eqref{numerical-scheme-N-only} satisfied uniquely by the triple $(X^N,Y^N,\rho^N)$ gives
	\begin{equation}
		\begin{split}
			&X^N(t)-\qPikp X^N(t)\\
			&= X_0^N+ \int_0^tD_pH\left(X^{N}(s),Y^{N}(s),\rho^{N}(s)\right)\mathrm{d}s - X^N(t_n)
			\\
			&=X_0^N+ \int_0^tD_pH\left(X^{N}(s),Y^{N}(s),\rho^{N}(s)\right)\mathrm{d}s - X_0^{N}- \int_0^{t_n}D_pH\left(X^{N}(s),Y^{N}(s),\rho^{N}(s)\right)\mathrm{d}s
			\\
			&=- \int_t^{t_n}D_pH\left(X^{N}(s),Y^{N}(s),\rho^{N}(s)\right)\mathrm{d}s
		\end{split}
	\end{equation}
	Since $|t-t_n|\leq \tau_k$, we apply H\"older's inequality and Fubini's theorem to obtain
	\begin{multline}
		\mathbb{E}\left[|X^N(t)-\qPikp X^N(t)|^2\right]
		\leq \mathbb{E}\left[\left(\int_t^{t_n}|D_pH\left(X^{N}(s),Y^{N}(s),\rho^{N}(s)\right)|\mathrm{d}s\right)^2\right]
		\\
		\leq2\tau_k\int_t^{t_n}\mathbb{E}\left[|D_pH(X^N,Y^N,\rho^N)-D_pH(0,0,\delta_0)|^2+|D_pH(0,0,\delta_0)|^2\right]\mathrm{d}s
		\\
		\lesssim\tau_k \int_t^{t_n}\mathbb{E}\left[|X^N|^2+|Y^N|^2+\mathcal{W}_2^2(\rho^N,\delta_0)\right]\mathrm{d}s+\tau_k^2|D_pH(0,0,\delta_0)|^2
	\end{multline}
	where in the final inequality we used Lipschitz continuity of $D_pH$ via \eqref{ass-H:2}. 
	Since $\mathcal{W}_2^2(\rho^N,\delta_0)\leq \expectation{|X^N|^2}$ in $[0,T]$, we get
	\begin{equation}
		\mathbb{E}\left[|X^N(t)-\qPikp X^N(t)|^2\right]
		\lesssim\tau_k \int_t^{t_n}\mathbb{E}\left[2|X^N|^2+|Y^N|^2\right]\mathrm{d}s+\tau_k^2|D_pH(0,0,\delta_0)|^2
	\end{equation}
	Noting the assumption that $X_0^N
    \to X_0$ in $L^2(\Omega;\mathbb{R}^d)$ as $N
    \to
    \infty$, we deduce from Lemma \ref{lemma-convegrence-fixed-N-inf} the strong convergence of $X^N\to X$ and $Y^N\to Y$ in $C([0,T], L^2(\Omega;\mathbb{R}^d))$ as $N\to\infty$. This implies that 
	\begin{equation}\label{Linf-unif-bound-N-post-time-limit-real}
		\sup_{N\in\mathbb{N}}\sup_{0\leq t\leq T}\left(\mathbb{E}\left[| X^{N}(t)|^2+| Y^{N}(t)|^2\right]\right)^{\frac{1}{2}}<\infty,
	\end{equation}
	and therefore 
	\begin{equation}
		\begin{split}
			&\mathbb{E}\left[|X^N(t)-\qPikp X^N(t)|^2\right] \lesssim \tau_k(t_n-t)+\tau_k^2\lesssim\tau_k^2
		\end{split}
	\end{equation}
	where the hidden constant is independent of $N\in\mathbb{N}$. 
	Since $t\in (0,T]$ was arbitrary and $X^N(0) -\qPikp X^N(0)=X_0^N-X^N(0)=0$, we conclude that
	\begin{equation}
		\sup_{0\leq t\leq T}\left(\mathbb{E}\left[|X^N(t)-\qPikp X^N(t)|^2\right]\right)^{\frac{1}{2}}\lesssim\tau_k.
	\end{equation} 
	We find that for each $t\in [t_{n-1},t_n)$ with $n\in \{1,2,\cdots,M_k\}$
	\begin{equation}
		\begin{split}
			&Y^N(t)-\qPikn Y^N(t)
			\\
			&= -D_xg\left(X^{N}(T),\rho^{N}(T)\right)+ \int_t^TD_xH\left(X^{N}(s),Y^{N}(s),\rho^{N}(s)\right)\mathrm{d}s - Y^N(t_{{n-1}})
			\\
			&=-D_xg\left(X^{N}(T),\rho^{N}(T)\right)+ \int_t^TD_xH\left(X^{N}(s),Y^{N}(s),\rho^{N}(s)\right)\mathrm{d}s + D_xg\left(X^{N}(T),\rho^{N}(T)\right)
			\\
			&\qquad\qquad\qquad\qquad\qquad- \int_{t_{n-1}}^TD_xH\left(X^{N}(s),Y^{N}(s),\rho^{N}(s)\right)\mathrm{d}s
			\\
			&=- \int_{t_{n-1}}^{t}D_xH\left(X^{N}(s),Y^{N}(s),\rho^{N}(s)\right)\mathrm{d}s
		\end{split}
	\end{equation}
	We then likewise deduce from the Lipschitz continuity of $D_xH$ via \eqref{ass-H:2}, the fact that $|t-t_{n-1}|\leq \tau_k$ for each $t\in [t_{n-1},t_n)$ with $n\in \{1,2,\cdots,M_k\}$, and the uniform bound \eqref{Linf-unif-bound-N-post-time-limit-real} that 
	\begin{equation}
		\sup_{0\leq t\leq T}\left(\mathbb{E}\left[|Y^N(t)-\qPikn Y^N(t)|^2\right]\right)^{\frac{1}{2}}\lesssim\tau_k.
	\end{equation} Finally, since $\mathcal{W}_2(\rho^{N}(t),\pi_{X^N}^k{\rho}^N(t))\leq \left(\mathbb{E}\left[|X^N(t)-\qPikp X^N(t)|^2\right]\right)^{\frac{1}{2}}$ for all $t\in [0,T]$, we conclude the proof of the bound \eqref{time-projection-err-bound}.	
\end{proof}
\begin{remark}
	The projection error bound \eqref{time-projection-err-bound} indicates a rate of convergence of order one in the time-step $\tau_k$. This rate is optimal since the projections are piecewise constant in time and exactly approximate the zero-order terms $X^N(0)$, $Y^N(T)$ in the first-order temporal Taylor expansions of $X^N,Y^N$, respectively, as shown in the above calculations. 
\end{remark}

In addition, the projections $\qPikp X^N,\qPikn Y^N$ satisfy a discrete Hamiltonian system with zeroth-order perturbations which are first order in the time-step in the norm of $L^2(0,T;\LLspace)$.
\begin{lemma}\label{lemma-deriv-interp-time-rep-XY}
	Assume the hypotheses of Theorem \ref{theorem-convergence-k-N-joint}.  For each $k,N\in\mathbb{N}$, there hold
	\begin{subequations}\label{deriv-interp-time-rep-XY}
		\begin{align}
			&\partial_t\Ip \qPikp X^{N}= D_pH\left(\qPikp X^{N},\qPikn Y^{N},\pi_{X^N}^k{\rho}^N\right)+X_{pert,avg}^{k,N}\quad \text{ in }\mathcal{R}_{d}(\mathcal{A}_N,\mathcal{O}_N),\\
			&\partial_t\In \qPikn Y^{N}= -D_xH\left(\qPikp X^{N},\qPikn Y^{N},\pi_{X^N}^k{\rho}^N\right) +Y_{pert,avg}^{k,N}\quad \text{ in }\mathcal{R}_{d}(\mathcal{A}_N,\mathcal{O}_N),
		\end{align}
	\end{subequations}
	for a.e.\ $t\in (0,T)$ with $\qPikp X^{N}(0)=X_0^N$ and $\qPikn Y^{N}(T)=-D_xg\left(\qPikp X^{N}(T),\pi_{X^N}^k{\rho}^N(T)\right)$, where $X_{pert,avg}^{k,N},Y_{pert,avg}^{k,N}\in\mathbb{V}_k(0,T;\mathcal{R}_{d}(\mathcal{A}_N,\mathcal{O}_N))$ with
	\begin{subequations}
		\begin{align}
			&X_{pert,avg}^{k,N}|_{I_n}\coloneqq \tau_k^{-1}\int_{t_{n-1}}^{t_n}\left(D_pH(X^N,Y^N,\rho^N) -D_pH\left(\qPikp X^{N},\qPikn Y^{N},\pi_{X^N}^k{\rho}^N\right) \right)\mathrm{d}s,\label{X-pert-avg-defn}\\
			&Y_{pert,avg}^{k,N}|_{I_n}\coloneqq -\tau_k^{-1}\int_{t_{n-1}}^{t_n}\left(D_xH(X^N,Y^N,\rho^N) -D_xH\left(\qPikp X^{N},\qPikn Y^{N},\pi_{X^N}^k{\rho}^N\right) \right)\mathrm{d}s,\label{Y-pert-avg-defn}
		\end{align}
	\end{subequations}
	in $\mathcal{R}_{d}(\mathcal{A}_N,\mathcal{O}_N)$ for $n\in\{1,2,\cdots,M_k\}$ and 
	\begin{equation}\label{XYdiffavg-error}
		\left\|X_{pert,avg}^{k,N}\right\|_{L^2(0,T;\LLspace)}+\left\|Y_{pert,avg}^{k,N}\right\|_{L^2(0,T;\LLspace)}\lesssim\tau_k 
	\end{equation}
	where the hidden constant is independent of $k$ and $N$.
\end{lemma}
\begin{proof}
	Given $k,N\in\mathbb{N}$, the definition of $\qPikp ,\qPikn $ and the regularity of $X^N,Y^N$ given in Theorem \ref{theorem-convegrence-fixed-N} clearly imply that  $X_{pert,avg}^{k,N},Y_{pert,avg}^{k,N}$ are well-defined as maps in $\mathbb{V}_k(0,T;\mathcal{R}_{d}(\mathcal{A}_N,\mathcal{O}_N))$. To derive the discrete Hamiltonian system \eqref{deriv-interp-time-rep-XY}, fix $n\in\{1,2,\cdots,M_k\}$. From the definitions of the interpolation operators $\Ip,\In$ and projection operators $\qPikp ,\qPikn $ we deduce that
	\begin{equation}
		\begin{split}
			\partial_t\Ip \qPikp X^{N}|_{I_n}
			&=-\tau_k^{-1}\jump{\qPikp X^{N}}_{n-1}=-\tau_k^{-1}\left(\qPikp X^{N}(t_{n-1})-\qPikp X^{N}(t_n)\right)\\
			&=-\tau_k^{-1}\left(X^{N}(t_{n-1})-X^{N}(t_n)\right)
		\end{split}
	\end{equation}
	The system \eqref{numerical-scheme-N-only} satisfied uniquely by the triple $(X^N,Y^N,\rho^N)$ gives
	\begin{equation}
			\begin{split} 
			&\partial_t\Ip \qPikp X^{N}|_{I_n}=\tau_k^{-1}\int_{t_{n-1}}^{t_n}D_pH(X^N,Y^N,\rho^N)\mathrm{d}s
			\\
			&=\tau_k^{-1}\int_{t_{n-1}}^{t_n}\left(D_pH(X^N,Y^N,\rho^N)- D_pH\left(\qPikp X^{N},\qPikn Y^{N},\pi_{X^N}^k{\rho}^N\right)\right)\mathrm{d}s\\
			&\qquad\qquad\qquad\qquad\qquad\qquad\qquad\qquad+ D_pH\left(\qPikp X^{N},\qPikn Y^{N},\pi_{X^N}^k{\rho}^N\right)|_{I_n}
			\\&= D_pH\left(\qPikp X^{N},\qPikn Y^{N},\pi_{X^N}^k{\rho}^N\right)|_{I_n}+X_{pert,avg}^{k,N}|_{I_n}
			\end{split}
	\end{equation} where $X_{pert,avg}^{k,N}$ is defined in \eqref{X-pert-avg-defn}. We likewise obtain
	\begin{equation}
		\begin{split}
			\partial_t\In \qPikn Y^{N}|_{I_n}&=-\tau_k^{-1}\jump{\qPikn Y^{N}}_{n}=-\tau_k^{-1}\left(\qPikn Y^{N}(t_{n-1})-\qPikn Y^{N}(t_n)\right)
			=-\tau_k^{-1}\left(Y^{N}(t_{n-1})-Y^{N}(t_n)\right)
		\end{split}
	\end{equation} and so
	\begin{equation}
		\begin{split}
			&\partial_t\In \qPikn Y^{N}|_{I_n}=-\tau_k^{-1}\int_{t_{n-1}}^{t_n}D_xH(X^N,Y^N,\rho^N)\mathrm{d}s
			\\
			&=-\tau_k^{-1}\int_{t_{n-1}}^{t_n}\left(D_xH(X^N,Y^N,\rho^N)- D_xH\left(\qPikp X^{N},\qPikn Y^{N},\pi_{X^N}^k{\rho}^N\right)\right)\mathrm{d}s\\
			&\qquad\qquad\qquad\qquad\qquad\qquad\qquad\qquad- D_xH\left(\qPikp X^{N},\qPikn Y^{N},\pi_{X^N}^k{\rho}^N\right)|_{I_n}
			\\
			&= -D_xH\left(\qPikp X^{N},\qPikn Y^{N},\pi_{X^N}^k{\rho}^N\right)|_{I_n}+Y_{pert,avg}^{k,N}|_{I_n}
		\end{split}
	\end{equation} where $Y_{pert,avg}^{k,N}$ is defined in \eqref{Y-pert-avg-defn}. Since $n\in\{1,2,\cdots,M_k\}$ was arbitrary this establishes the discrete Hamiltonian system \eqref{deriv-interp-time-rep-XY}. The initial and terminal conditions hold as well since the initial and terminal conditions of the system satisfied by $(X^N,Y^N,\rho^N)$ can be rewritten in terms of the projection operators $\qPikn$, $\qPikp$ and $\pi_{X^N}^k\rho^N$. Indeed, $\qPikp X^N(0)=X^N(0)=X_0^N$ and $\qPikp X^N(T)=X^N(T)$, so
	\begin{multline}
		\qPikn Y^N(T)=Y^N(T)=-D_xg(X^N(T),(X^N(T))_{\#}\mathbb{P})\\
		=-D_xg(\qPikp X^N(T),(\qPikp X^N(T))_{\#}\mathbb{P})=-D_xg(\qPikp X^N(T),\pi_{X^N}^k\rho^N(T)),
	\end{multline}
	as claimed. 
	
	We now prove \eqref{XYdiffavg-error}. The fact that $t_n-t_{n-1}=\tau_k$ for $n\in\{1,2,\cdots,M_k\}$ implies 
	\begin{equation}
		\begin{split}
			&\int_0^T\mathbb{E}\left[\left|X_{pert,avg}^{k,N}\right|^2\right]\mathrm{d}s=\tau_k\sum_{n=1}^{M_k}\mathbb{E}\left[\left|X_{pert,avg}^{k,N}\right|^2\right]
			\\
			&=\tau_k\sum_{n=1}^{M_k}\mathbb{E}\left[\left| \tau_k^{-1}\int_{t_{n-1}}^{t_n}\left(D_pH(X^N,Y^N,\rho^N) -D_pH\left(\qPikp X^{N},\qPikn Y^{N},\pi_{X^N}^k{\rho}^N\right) \right)\mathrm{d}s\right|^2\right]
			\\
			&\leq \tau_k^{-1}\sum_{n=1}^{M_k}\mathbb{E}\left[\left( \int_{t_{n-1}}^{t_n}\left|D_pH(X^N,Y^N,\rho^N) -D_pH\left(\qPikp X^{N},\qPikn Y^{N},\pi_{X^N}^k{\rho}^N\right) \right|\mathrm{d}s\right)^2\right]
			\\
			&\leq\tau_k^{-1}\sum_{n=1}^{M_k}\mathbb{E}\left[ (t_n-t_{n-1})\int_{t_{n-1}}^{t_n}\left|D_pH(X^N,Y^N,\rho^N) -D_pH\left(\qPikp X^{N},\qPikn Y^{N},\pi_{X^N}^k{\rho}^N\right) \right|^2\mathrm{d}s\right]
		\end{split}
	\end{equation}
	where we used H\"older's inequality to obtain the final inequality.
	We then use Lipschitz continuity of $D_pH$ via \eqref{ass-H:2} and  $t_n-t_{n-1}=\tau_k$ for $n\in\{1,2,\cdots,M_k\}$ to get
	\begin{equation}
		\begin{split}
			&\int_0^T\mathbb{E}\left[\left|X_{pert,avg}^{k,N}\right|^2\right]\mathrm{d}s
			\leq \sum_{n=1}^{M_k}\int_{t_{n-1}}^{t_n}\mathbb{E}\left[ \left|D_pH(X^N,Y^N,\rho^N) -D_pH\left(\qPikp X^{N},\qPikn Y^{N},\pi_{X^N}^k{\rho}^N\right) \right|^2\right]\mathrm{d}s
			\\
			&\lesssim \sum_{n=1}^{M_k}\int_{t_{n-1}}^{t_n}\mathbb{E}\left[ |X^N-\qPikp X^{N}|^2+|Y^N-\qPikn Y^{N}|^2+\mathcal{W}_2^2(\rho^N,\pi_{X^N}^k{\rho}^N)\right]\mathrm{d}s
		\end{split}
	\end{equation}
	Then, the uniform $L^{\infty}({L}^2)$-bound given in Lemma \ref{lemma-temporal-inter-error-bound} implies that
	\begin{equation}
		\begin{split}
			\int_0^T\mathbb{E}\left[\left|X_{pert,avg}^{k,N}\right|^2\right]\mathrm{d}s\lesssim \sum_{n=1}^{M_k}\int_{t_{n-1}}^{t_n}\tau_k^2\lesssim\tau_k^2
		\end{split}
	\end{equation}
	which yields $\left\|X_{pert,avg}^{k,N}\right\|_{L^2(0,T;\LLspace)}\lesssim\tau_k$, where the hidden constant is independent of $k$ and $N$. An analogous argument also establishes that $\left\|Y_{pert,avg}^{k,N}\right\|_{L^2(0,T;\LLspace)}\lesssim\tau_k$ where the hidden constant is independent of $k$ and $N$, thereby concluding the proof of the lemma. 
\end{proof}

We are now able to prove Lemma \ref{theorem-time-err-bound-fixed-N}.
\begin{proof}[Proof of Lemma \ref{theorem-time-err-bound-fixed-N}]
	We establish the desired temporal error bound by leveraging the displacement monotonicity of $H$ and $g$ via the general continuity bound of the discrete Hamiltonian system as shown in Theorem  \ref{theorem-cont-dependence-discrete-Hamiltonian-sys}. To this end, let $N,k\in\mathbb{N}$ with $k\geq k_{\dagger}$. The triangle inequality gives
	\begin{multline}
			\sup_{t\in [0,T]}\mathcal{W}_2(\rho^{N}(t),\rho^{k,N}(t))+\sup_{t\in [0,T]}\left(\mathbb{E}\left[|(Y^{N}-Y^{k,N})(t)|^2\right]\right)^{\frac{1}{2}}+\sup_{t\in [0,T]}\left(\mathbb{E}\left[|(X^{N}-X^{k,N})(t)|^2\right]\right)^{\frac{1}{2}}
			\leq A+B
	\end{multline}
	where \begin{equation}
		A\coloneqq \sup_{t\in [0,T]}\mathcal{W}_2(\rho^{N}(t),\pi_{X^N}^k{\rho}^N(t))+\sup_{t\in [0,T]}\left(\mathbb{E}\left[|(Y^{N}-\qPikn Y^{N})(t)|^2\right]\right)^{\frac{1}{2}}
		+\sup_{t\in [0,T]}\left(\mathbb{E}\left[|(X^{N}-\qPikp X^{N})(t)|^2\right]\right)^{\frac{1}{2}},
	\end{equation}
	\begin{equation}
		B\coloneqq \sup_{t\in [0,T]}\mathcal{W}_2(\pi_{X^N}^k{\rho}^N(t),\rho^{k,N}(t))+\sup_{t\in [0,T]}\left(\mathbb{E}\left[|\Delta\mathscr{Y}_{N,N}^{k}(t)|^2\right]\right)^{\frac{1}{2}}+\sup_{t\in [0,T]}\left(\mathbb{E}\left[|\Delta \mathscr{X}_{N,N}^{k}(t)|^2\right]\right)^{\frac{1}{2}}
	\end{equation}
	where  $\Delta \mathscr{X}_{N,N}^{k}\coloneqq X^{k,N}-\qPikp X^N$ in $\mathbb{V}_k^+([0,T];\mathcal{R}_{d}(\mathcal{A}_N,\mathcal{O}_N))$ and let $\Delta\mathscr{Y}_{N,N}^{k}\coloneqq Y^{k,N}-\qPikn Y^N$ in $\mathbb{V}_k^-([0,T];\mathcal{R}_{d}(\mathcal{A}_N,\mathcal{O}_N))$. We know from Lemma \ref{lemma-temporal-inter-error-bound} that $A\lesssim\tau_k$. We therefore seek to prove that
	\begin{equation}\label{disc-approx-proj-exact-N-error}
		\sup_{t\in [0,T]}\mathcal{W}_2(\pi_{X^N}^k{\rho}^N(t),\rho^{k,N}(t))+\sup_{t\in [0,T]}\left(\mathbb{E}\left[|\Delta\mathscr{Y}_{N,N}^{k}(t)|^2\right]\right)^{\frac{1}{2}}+\sup_{t\in [0,T]}\left(\mathbb{E}\left[|\Delta \mathscr{X}_{N,N}^{k}(t)|^2\right]\right)^{\frac{1}{2}}\lesssim\tau_k
	\end{equation}
	for all $k\in\mathbb{N}$ sufficiently large, where the hidden constant is independent of $N\in\mathbb{N}$. Now, this is a direct application of the continuous dependence Theorem  \ref{theorem-cont-dependence-discrete-Hamiltonian-sys}. Indeed, let $M=N$, $\lambda_1=\lambda_2=1$, $X_{0,1}^N= X_{0,2}^N=X_0^N$ in $\mathcal{R}_{d}(\mathcal{A}_N,\mathcal{O}_N)$, and let $F_1,G_1=0$ and $F_2\coloneqq X_{pert,avg}^{k,N}$, $G_2\coloneqq Y_{pert,avg}^{k,N}$ in $\mathbb{V}_k(0,T;\mathcal{R}_{d}(\mathcal{A}_N,\mathcal{O}_N))$, where $X_{pert,avg}^{k,N},Y_{pert,avg}^{k,N}$ are defined as in Lemma \ref{lemma-deriv-interp-time-rep-XY}. Since $(X^{k,N},Y^{k,N},\rho^{k,N})$ and $(\qPikp X^{N},\qPikn Y^{N},\pi_{X^N}^k{\rho}^N)$ are triples that (uniquely) satisfy the discrete Hamiltonian systems \eqref{numerical-scheme} and \eqref{deriv-interp-time-rep-XY}, respectively, we apply Theorem  \ref{theorem-cont-dependence-discrete-Hamiltonian-sys} with 
	$$(\mathscr{X}_1^{k,N},\mathscr{Y}_1^{k,N},\varrho_1^{k,N})\coloneqq (X^{k,N},Y^{k,N},\rho^{k,N})$$ and 
	$$(\mathscr{X}_2^{k,M},\mathscr{Y}_2^{k,M},\varrho_{2}^{k,M})=(\qPikp X^{N},\qPikn Y^{N},\pi_{X^N}^k{\rho}^N),$$ to obtain
	\begin{equation}
		\begin{split}
			&\sup_{t\in [0,T]}\mathcal{W}_2(\pi_{X^N}^k{\rho}^N(t),\rho^{k,N}(t))+\sup_{t\in [0,T]}\left(\mathbb{E}\left[|\Delta\mathscr{Y}_{N,N}^{k}(t)|^2\right]\right)^{\frac{1}{2}}+\sup_{t\in [0,T]}\left(\mathbb{E}\left[|\Delta \mathscr{X}_{N,N}^{k}(t)|^2\right]\right)^{\frac{1}{2}}
			\\
			&\qquad\qquad\qquad\qquad\qquad\qquad\lesssim \|X_{pert,avg}^{k,N}\|_{L^2(0,T;\LLspace)}+\|Y_{pert,avg}^{k,N}\|_{L^2(0,T;\LLspace)}\lesssim \tau_k
		\end{split}
	\end{equation}
	where in the last line we invoked the bound \eqref{XYdiffavg-error} from Lemma \ref{lemma-deriv-interp-time-rep-XY}. This proves \eqref{disc-approx-proj-exact-N-error}, thereby establishing the lemma.
\end{proof}

\subsection{Proof of Theorem \ref{theorem-convergence-k-N-joint}}
Theorem \ref{theorem-convergence-k-N-joint} on strong convergence of the approximations is deduced from the triangle inquality and application of both Lemma \ref{theorem-time-err-bound-fixed-N} and Lemma \ref{lemma-convegrence-fixed-N-inf}. Indeed, we find that for each $N,k\in\mathbb{N}$ with $k\geq k_{\dagger}$, the triangle inequality gives
\begin{equation}
	\begin{split} 
		&\sup_{t\in [0,T]}\mathcal{W}_2(\rho(t),\rho^{k,N}(t))+\sup_{t\in [0,T]}\|(Y-Y^{k,N})(t)\|_{L^2(\Omega;\mathbb{R}^d)}+\sup_{t\in [0,T]}\|(X-X^{k,N})(t)\|_{L^2(\Omega;\mathbb{R}^d)}\leq \mathfrak{A}+\mathfrak{B}
	\end{split}
\end{equation}
where 
\begin{equation}
	\mathfrak{A}
	\coloneqq \sup_{t\in [0,T]}\mathcal{W}_2(\rho(t),\rho^{N}(t))+\sup_{t\in [0,T]}\|(Y-Y^{N})(t)\|_{L^2(\Omega;\mathbb{R}^d)}+\sup_{t\in [0,T]}\|(X-X^{N})(t)\|_{L^2(\Omega;\mathbb{R}^d)},
\end{equation}
\begin{equation}
	\mathfrak{B}
	\coloneqq\sup_{t\in [0,T]}\mathcal{W}_2(\rho^N(t),\rho^{k,N}(t))+\sup_{t\in [0,T]}\|(Y^N-Y^{k,N})(t)\|_{L^2(\Omega;\mathbb{R}^d)}+\sup_{t\in [0,T]}\|(X^N-X^{k,N})(t)\|_{L^2(\Omega;\mathbb{R}^d)},
\end{equation}
and $(X,Y,\rho)\in C^1(0,T;L^2(\Omega;\mathbb{R}^d))\times C^1(0,T;L^2(\Omega;\mathbb{R}^d))\times C([0,T];\mathcal{P}_2(\mathbb{R}^d))$ denotes the triple which uniquely satisfies the continuous Hamiltonian system \eqref{characteristics-continuous} in the sense of \eqref{definition-soln-of-continuous-pb}, as given by Lemma \ref{lemma-convegrence-fixed-N-inf}. Lemma \ref{lemma-convegrence-fixed-N-inf} also gives $\mathfrak{A}\lesssim \|X_0-X_0^N\|_{\LLspace}$, while Lemma \ref{theorem-time-err-bound-fixed-N} implies that $\mathfrak{B}\lesssim\tau_k$ since $k\geq k_{\dagger}$. Therefore, we get
	\begin{equation}
	\begin{split} 
		&\sup_{t\in [0,T]}\mathcal{W}_2(\rho(t),\rho^{k,N}(t))+\sup_{t\in [0,T]}\|(Y-Y^{k,N})(t)\|_{L^2(\Omega;\mathbb{R}^d)}+\sup_{t\in [0,T]}\|(X-X^{k,N})(t)\|_{L^2(\Omega;\mathbb{R}^d)}\\
		&\qquad\qquad\qquad\qquad\qquad\qquad\qquad\qquad\qquad\qquad\qquad\qquad\qquad\qquad\lesssim \|X_0-X_0^N\|_{\LLspace}+ \tau_k
	\end{split}
\end{equation}
for all $N\in\mathbb{N}$ and $k\geq k_{\dagger}$. Because $X_0^N\to X_0$ in $\LLspace$ as $N\to\infty$ and $\tau_k\to 0$ as $k\to\infty$, we conclude the claimed convergence \eqref{Y-rho-convergence} in the joint limit as $N,k\to\infty$. This completes the proof of the theorem.
\hfill\proofbox

\subsection{Proof of Corollary \ref{cor-val-func-contin-dep-bound}}

Using the representation formulae \eqref{u-opt-trajectory-rep} and \eqref{u-opt-trajectory-rep-N} we get for each $t\in [0,T]$
\begin{align*}
    &\|(u(\cdot,X) - u^{N}(\cdot,X^N))(t)\|_{L^1(\Omega)}\\
    &\leq \int_t^T\int_{\Omega}\left|L(X(s),{D_pH(X(s),Y(s),\rho(s))},\rho(s))-L(X^N(s),{D_pH(X^N(s),Y^N(s),\rho^N(s))},\rho^N(s))\right|\mathrm{d}\mathbb{P}(\omega)\mathrm{d}s\\
    &+\int_{\Omega}\left|g(X(T),\rho(T))-g(X^N(T),\rho^N(T))\right|\mathrm{d}\mathbb{P}(\omega),
\end{align*}
where we used Fubini's Theorem. Furthermore, for each $t\in [0,T]$, we have 
$$\|(u(\cdot,X) - u^{N}(\cdot,X^N))(t)\|_{L^1(\Omega)}\leq Q_1(t)+Q_2(t) + Q_3 + Q_4,$$ 
where 
\small
\begin{equation}
    \begin{split}
    &Q_1(t)\coloneqq \int_t^T\int_{\Omega}\left|L(X(s),{D_pH(X(s),Y(s),\rho(s))},\rho(s))-L(X(s),{D_pH(X(s),Y(s),\rho(s))},\rho^N(s))\right|\mathrm{d}\mathbb{P}(\omega)\mathrm{d}s,\\
    &Q_2(t)\coloneqq \int_t^T\int_{\Omega}\left|L(X(s),{D_pH(X(s),Y(s),\rho(s))},\rho^N(s))-L(X^N(s),{D_pH(X^N(s),Y^N(s),\rho^N(s))},\rho^N(s))\right|\mathrm{d}\mathbb{P}(\omega)\mathrm{d}s,\\
    &Q_3\coloneqq \int_{\Omega}\left|g(X(T),\rho(T))-g(X(T),\rho^N(T))\right|\mathrm{d}\mathbb{P}(\omega),\\
    & Q_4\coloneqq \int_{\Omega}\left|g(X(T),\rho^N(T))-g(X^N(T),\rho^N(T))\right|\mathrm{d}\mathbb{P}(\omega).
    \end{split}
\end{equation}\normalsize For the remainder of the proof we seek to show that 
\begin{equation}
    \sup_{t\in[0,T]}Q_i(t)\lesssim \|X_0-X_0^N\|_{\LLspace},\ \    i\in\{1,2\},\ \ {{\rm and}\ \ Q_i\lesssim \|X_0-X_0^N\|_{\LLspace},  \ \ i\in\{3,4\},}
\end{equation} for all sufficiently large $N\in\mathbb{N}$.

\underline{Bounding the terms $Q_1$, $Q_3$}: 
Bounding the terms $Q_1$,$Q_3$ relies on the assumption that $L$ and $g$ exhibit local quantitative continuity via \eqref{L-loc-meas-contin} and \eqref{g-loc-meas-contin}, respectively. Indeed, fix any $R>0$ (e.g., set $R=1$ for concreteness). We prove the estimate for $Q_1$ as the proof for $Q_3$ is analogous and simpler. Note that the error bound \eqref{N-convergence-error-bound} from Lemma \ref{lemma-convegrence-fixed-N-inf} implies $\rho^N\to\rho$ in $C([0,T];\mathcal{P}_2(\mathbb{R}^d))$ as $N\to\infty$. Therefore, picking $N_R\in\mathbb{N}$ such that 
$$\sup_{t\in[0,T]}\mathcal{W}_1(\rho(t),\rho^N(t))\leq \sup_{t\in[0,T]}\mathcal{W}_2(\rho(t),\rho^N(t))<R,$$ 
for all $N>N_R$, the local Lipschitz continuity of $L$ via \eqref{L-loc-meas-contin} implies
\begin{multline}
    \left|L(X(s),{D_pH(X(s),Y(s),\rho(s))},\rho(s))-L(X(s),{D_pH(X(s),Y(s),\rho(s))},\rho^N(s))\right|\\\leq C_R \left(1+|X(s)|+|D_pH(X(s),Y(s),\rho(s))|+\mathcal{W}_1(\rho(s),\delta_0)+\mathcal{W}_1(\rho^N(s),\delta_0)\right)\mathcal{W}_1(\rho(s),\rho^N(s))
\end{multline}
for all $s\in[0,T]$ and a.e.\ $\omega\in \Omega$, and hence
\begin{equation}
    \begin{split}
    &\int_t^T\int_{\Omega}\left|L(X(s),{D_pH(X(s),Y(s),\rho(s))},\rho(s))-L(X(s),{D_pH(X(s),Y(s),\rho(s))},\rho^N(s))\right|\mathrm{d}\mathbb{P}(\omega)\mathrm{d}s\\
    &\leq C_R \int_t^T\int_{\Omega}\left(1+|X(s)|+|D_pH(X(s),Y(s),\rho(s))|+\mathcal{W}_1(\rho(s),\delta_0)+\mathcal{W}_1(\rho^N(s),\delta_0)\right)\mathcal{W}_1(\rho(s),\rho^N(s))\mathrm{d}\mathbb{P}(\omega)\mathrm{d}s
    \\&\lesssim\int_0^T\int_{\Omega}\left(1+|X(s)|+|Y(s)|+\mathcal{W}_1(\rho(s),\delta_0)+\mathcal{W}_1(\rho^N(s),\delta_0)\right)\mathcal{W}_1(\rho(s),\rho^N(s))\mathrm{d}\mathbb{P}(\omega)\mathrm{d}s
    \end{split}
\end{equation}
where the hidden constant arises due to linear growth of $D_pH$ (via \eqref{linear-growth-D_pH}) and is independent of $t\in [0,T]$ but depends in particular on $R$.
We then obtain
\begin{equation}
    \begin{split}
    &\int_t^T\int_{\Omega}\left|L(X(s),{D_pH(X(s),Y(s),\rho(s))},\rho(s))-L(X(s),{D_pH(X(s),Y(s),\rho(s))},\rho^N(s))\right|\mathrm{d}\mathbb{P}(\omega)\mathrm{d}s
    \\
    &\lesssim \sup_{s\in[0,T]}\left(1+\|X(s)\|_{L^2(\Omega;\mathbb{R}^d)}+\|Y(s)\|_{L^2(\Omega;\mathbb{R}^d)}+\mathcal{W}_2(\rho(s),\delta_0)+\mathcal{W}_2(\rho^N(s),\delta_0)\right)\sup_{s\in[0,T]}\mathcal{W}_2(\rho(s),\rho^N(s))
    \end{split}
\end{equation} for all $N>N_R$. Since {$\sup_{s\in[0,T]}\left(\|X(s)\|_{L^2(\Omega;\mathbb{R}^d)}+\|Y(s)\|_{L^2(\Omega;\mathbb{R}^d)}\right)$ is finite} and $\rho^N\to\rho$ in $C([0,T];\mathcal{P}_2(\mathbb{R}^d))$ as $N\to\infty$ with estimate $\sup_{s\in[0,T]}\mathcal{W}_2(\rho(s),\rho^N(s))\lesssim \|X_0-X_0^N\|_{\LLspace}$, we deduce that the first supremum of the left above is uniformly bounded with respect to\ $N$, thus completing the proof that $\sup_{t\in [0,T]}Q_1(t)\lesssim \|X_0-X_0^N\|_{\LLspace}$ for all $N\in\mathbb{N}$ with $N>N_R$ sufficiently large. Likewise, it holds that $Q_3\lesssim \|X_0-X_0^N\|_{\LLspace}$ for all $N\in\mathbb{N}$ with $N>\tilde{N}_R\geq N_R$ sufficiently large (where we note that $\tilde{N}_R$ depends on $g$ and $N_R$). It now remains to bound the terms $Q_2$ and $Q_4$.

\underline{Bounding the terms $Q_2$, $Q_4$}: Due to the regularity property \eqref{property-L:1} for $L$, we use the chain rule and the triangle inequality to obtain 
\begin{multline}
    \int_t^T\int_{\Omega}\left|L(X(s),{D_pH(X(s),Y(s),\rho(s))},\rho^N(s))-L(X^N(s),{D_pH(X^N(s),Y^N(s),\rho^N(s))},\rho^N(s))\right|\mathrm{d}\mathbb{P}(\omega)\mathrm{d}s \\\leq \int_t^T\int_{\Omega}\int_0^1|D_xL(rX(s)+(1-r)X^N(s),rV_*(s)+(1-r)V_*^N(s),\rho^N(s))||X(s)-X^N(s)|\mathrm{d}r\mathrm{d}\mathbb{P}(\omega)\mathrm{d}s\\+\int_t^T\int_{\Omega}\int_0^1|D_vL(rX(s)+(1-r)X^N(s),rV_*(s)+(1-r)V_*^N(s),\rho^N(s))||V_*(s)-V_*^N(s)|\mathrm{d}r\mathrm{d}\mathbb{P}(\omega)\mathrm{d}s,
\end{multline}
where $V_*(s,\omega)\coloneqq D_pH(X(s),Y(s),\rho(s))$ and $V_*^N(s,\omega)\coloneqq D_pH(X^N(s),Y^N(s),\rho^N(s))$ for all $s\in [0,T]$ and a.e.\ $\omega\in \Omega$. We consequently obtain from the linear growth of $D_vL$ and $D_xL$ via \eqref{linear-growth-D_qL} and \eqref{linear-growth-D_xL}
\begin{multline}
    Q_2(t)\lesssim
    \int_t^T\int_{\Omega}\left(|X(s)-X^N(s)|^2+|V(s)-V_*^N(s)||X(s)-X^N(s)|\right)\mathrm{d}\mathbb{P}(\omega)\mathrm{d}s\\+\int_t^T\int_{\Omega}(|X^N(s)|+|V_*^N(s)|+{\mathcal{W}_1}(\rho^N(s),\delta_0)+1)|X(s)-X^N(s)|\mathrm{d}\mathbb{P}(\omega)\mathrm{d}s \\+
    \int_t^T\int_{\Omega}\left(|V(s)-V_*^N(s)|^2+|X(s)-X^N(s)||V(s)-V_*^N(s)|\right)\mathrm{d}\mathbb{P}(\omega)\mathrm{d}s\\
    +
    \int_t^T\int_{\Omega}(|X^N(s)|+|V_*^N(s)|+{\mathcal{W}_1}(\rho^N,\delta_0)+1)|V(s)-V_*^N(s)|\mathrm{d}\mathbb{P}(\omega)\mathrm{d}s.
\end{multline} Due to Lipschitz continuity of $D_pH$ via \eqref{ass-H:2} and the linear growth of $D_pH$ (see \eqref{linear-growth-D_pH}), we bound each of $|V-V_*^N|$ and $|V_*^N|$ in the above and collect product terms to eventually obtain
\small
\begin{multline}
    Q_2(t)\lesssim
    \int_t^T\int_{\Omega}\left(|X(s)-X^N(s)|^2+|Y(s)-Y^N(s)||X(s)-X^N(s)|+\mathcal{W}_2(\rho(s),\rho^N(s))|X(s)-X^N(s)|\right)\mathrm{d}\mathbb{P}(\omega)\mathrm{d}s\\+
    \int_t^T\int_{\Omega}\left(|Y(s)-Y^N(s)|^2+\mathcal{W}_2^2(\rho(s),\rho^N(s))\right)\mathrm{d}\mathbb{P}(\omega)\mathrm{d}s\\
    +
    \int_t^T\int_{\Omega}(|X^N(s)|+|Y^N(s)|+{\mathcal{W}_2}(\rho^N,\delta_0)+1)(|X(s)-X^N(s)|+|Y(s)-Y^N(s)|+\mathcal{W}_2(\rho(s),\rho^N(s)))\mathrm{d}\mathbb{P}(\omega)\mathrm{d}s.
\end{multline}\normalsize
By applying Young's inequality, followed by Cauchy--Schwarz inequality applied to each of the integrals over $\Omega$ above, we get the estimate
\small
\begin{multline}
    Q_2(t)\lesssim
    \int_0^T\left(\|X(s)-X^N(s)\|_{L^2(\Omega;\mathbb{R}^d)}^2+\|Y(s)-Y^N(s)\|_{L^2(\Omega;\mathbb{R}^d)}^2\right)\mathrm{d}s\\
    +\int_0^T(\|X^N(s)\|_{\LLspace}+\|Y^N(s)\|_{\LLspace}+{\mathcal{W}_2}(\rho^N(s),\delta_0)+1)(\|X(s)-X^N(s)\|_{\LLspace}+\|Y(s)-Y^N(s)\|_{\LLspace})\mathrm{d}s,
\end{multline}\normalsize 
where we used the estimate $\mathcal{W}_2(\rho(s),\rho^N(s))\leq \|X(s)-X^N(s)\|_{L^2(\Omega;\mathbb{R}^d)}$, for all $s\in [0,T]$, to obtain an upper bound without the term $\mathcal{W}_2(\rho,\rho^N)$. We note that the hidden constant in this bound does not depend on $t\in [0,T]$. By a similar argument that uses the Lipschitz continuity of $D_xg$ via \eqref{ass-g:2}, we also deduce that
\begin{equation}
    Q_4\lesssim \|X(T)-X^N(T)\|_{L^2(\Omega;\mathbb{R}^d)}^2+(\|X^N(T)\|_{\LLspace}+{\mathcal{W}_2}(\rho^N(T),\delta_0)+1)\|X(T)-X^N(T)\|_{\LLspace}.
\end{equation}
Since $X_0^N\to X_0$ in $\LLspace$, we apply Lemma \ref{lemma-convegrence-fixed-N-inf} to deduce the uniform boundedness of the sequences $\{X^N\}_{N}$, $\{Y^N\}_N$ in $C([0,T];\LLspace)$ and $\{\rho^N\}_N$ in $C([0,T];\mathcal{P}_2(\mathbb{R}^d))$ from the strong convergence estimate \eqref{N-convergence-error-bound}, which in turn implies that
\begin{equation}
    \sup_{t\in[0,T]}Q_2(t)+Q_4\lesssim \|X_0-X_0^N\|_{\LLspace},
\end{equation} where we note that the hidden constant in this inequality is independent of $t\in [0,T]$. This completes the proof.

\hfill\proofbox

\subsection{Proof of Theorem \ref{theorem-convergence-k-N-joint-u-val}}
Let $k_{\dagger}$ be as given by Theorem \ref{theorem-convergence-k-N-joint}. We note that Corollary \ref{cor-val-func-contin-dep-bound} implies that there exists $N_{\dagger}\in\mathbb{N}$ such that $\sup_{t\in [0,T]}\|(u(\cdot,X)-u^{N}(\cdot,X^N))(t)\|_{L^1(\Omega;\mathbb{R}^d)}\lesssim \|X_0-X_0^N\|_{\LLspace}$ for all $N\geq N_{\dagger}$. By the triangle inequality, we then have
\begin{multline}
    \sup_{t\in [0,T]}\|(u(\cdot,X)-u^{k,N})(t)\|_{L^1(\Omega;\mathbb{R}^d)}\\
    \leq \sup_{t\in [0,T]}\|(u(\cdot,X)-u^{N}(\cdot,X^N))(t)\|_{L^1(\Omega;\mathbb{R}^d)}+\sup_{t\in [0,T]}\|(u^N(\cdot,X^N)-u^{k,N})(t)\|_{L^1(\Omega;\mathbb{R}^d)}
    \\
    \lesssim \|X_0-X_0^N\|_{\LLspace}+\sup_{t\in [0,T]}\|(u^N(\cdot,X^N)-u^{k,N})(t)\|_{L^1(\Omega;\mathbb{R}^d)}
\end{multline}
for all $N\geq N_{\dagger}$ and $k\geq k_{\dagger}$. We are then left with bounding the latter term to complete the proof. Using the representation formula \eqref{u-opt-trajectory-rep-N} and the definition of $u^{k,N}$ via \eqref{u-opt-trajectory-rep-N-k}, we can follow the proof of Corollary \ref{cor-val-func-contin-dep-bound} to likewise get
$\|(u^N(\cdot,X^N)-u^{k,N})(t)\|_{L^1(\Omega;\mathbb{R}^d)}\leq \sum_{i=1}^2\tilde{Q}_i(t) {+ \sum_{i=3}^4\tilde{Q}_i,}$ for all $t\in [0,T]$, where 
\begin{equation}
    \begin{split}
    &\tilde{Q}_1(t)\coloneqq\int_t^T\int_{\Omega}\left|L(X^N(s),V_*^N,\rho^N(s))-L(X^{N}(s),V_*^N,\rho^{k,N}(s))\right|\mathrm{d}\mathbb{P}(\omega)\mathrm{d}s,\\
    &\tilde{Q}_2(t)\coloneqq \int_t^T\int_{\Omega}\left|L(X^{N}(s),V_*^N,\rho^{k,N}(s))-L(X^{k,N}(s),V_*^{k,N},\rho^{k,N}(s))\right|\mathrm{d}\mathbb{P}(\omega)\mathrm{d}s,\\
    &\tilde{Q}_3\coloneqq \int_{\Omega}\left|g(X^N(T),\rho^N(T))-g(X^N(T),\rho^{k,N}(T))\right|\mathrm{d}\mathbb{P}(\omega),\\
    &\tilde{Q}_4\coloneqq \int_{\Omega}\left|g(X^N(T),\rho^{k,N}(T))-g(X^{k,N}(T),\rho^{k,N}(T))\right|\mathrm{d}\mathbb{P}(\omega).
    \end{split}
\end{equation} where $V_*^N\coloneqq D_pH(X^N(s),Y^N(s),\rho^N(s))$ and $V_*^{k,N}\coloneqq D_pH(X^{k,N}(s),Y^{k,N}(s),\rho^{k,N}(s))$ for all $s\in [0,T]$ and a.e.\ $\omega\in\Omega$.
Now, Lemma \ref{theorem-time-err-bound-fixed-N} implies that $$\sup_{N\in\mathbb{N}}\left(\sup_{t\in [0,T]}\left(\|(Y^{N}-Y^{k,N})(t)\|_{\LLspace}+\|(X^{N}-X^{k,N})(t)\|_{\LLspace}\right)+\sup_{t\in [0,T]}\mathcal{W}_2(\rho^{N}(t),\rho^{k,N}(t))\right)\lesssim\tau_k\to 0$$ as $k\to\infty$, while Theorem \ref{theorem-convergence-k-N-joint} and Lemma \ref{lemma-convegrence-fixed-N-inf} together imply the boundedness of the sequences $\{X^{k,N}\}_{k\geq k_{\dagger},N}$, $\{Y^{k,N}\}_{k\geq k_{\dagger},N}$ in $L^{\infty}(0,T;\LLspace)$, boundedness of the sequences $\{X^N\}_N$, $\{Y^N\}_N$ in $C([0,T];\LLspace)$ and boundedness of the sequences $\{\rho^N\}_N$, $\{\rho^{k,N}\}_{k\geq k_{\dagger},N}$ in $L^{\infty}(0,T;\mathcal{P}_2(\mathbb{R}^d))$. By taking $k_{\dagger}^*\geq k_{\dagger}$ sufficiently large so that $\sup_{N\in\mathbb{N}}\sup_{t\in [0,T]}\mathcal{W}_1(\rho^{N}(t),\rho^{k,N}(t))\leq \sup_{N\in\mathbb{N}}\sup_{t\in [0,T]}\mathcal{W}_2(\rho^{N}(t),\rho^{k,N}(t))<1=:R$ for all $k\geq k_{\dagger}^*$, we follow the remaining steps of the proof of Corollary \ref{cor-val-func-contin-dep-bound} to deduce that
\small
\begin{multline}
    \sup_{t\in[0,T]}\tilde{Q}_1(t)+\tilde{Q}_3\\
    \lesssim \sup_{ s\in[0,T]}\left(1+\|X^N(s)\|_{L^2(\Omega;\mathbb{R}^d)}+\|Y^N(s)\|_{L^2(\Omega;\mathbb{R}^d)}+\mathcal{W}_2(\rho^N(s),\delta_0)+\mathcal{W}_2(\rho^{k,N}(s),\delta_0)\right)\sup_{s\in[0,T]}\mathcal{W}_2(\rho^N(s),\rho^{k,N}(s))
    \\
    \lesssim\tau_k
\end{multline}\normalsize
and that
\begin{multline}
    \sup_{t\in[0,T]}\tilde{Q}_2(t)+\tilde{Q}_4
    \lesssim 
    \sup_{s\in [0,T]}\left(\|X^N-X^{k,N}\|_{L^2(\Omega;\mathbb{R}^d)}^2+\|Y^N-Y^{k,N}\|_{L^2(\Omega;\mathbb{R}^d)}^2\right)\\
    +\sup_{s\in [0,T]}\left(\|X^{k,N}\|_{\LLspace}+\|Y^{k,N}\|_{\LLspace}+{\mathcal{W}_2}(\rho^{k,N},\delta_0)+1\right)\\\times\sup_{s\in[0,T]}\left(\|X^N-X^{k,N}\|_{\LLspace}+\|Y^N-Y^{k,N}\|_{\LLspace}\right)
    \lesssim\tau_k^2+\tau_k
\end{multline} for all $k\geq k_{\dagger}^*$ and all $N\in\mathbb{N}$. The desired error bound \eqref{val-func-error-bound} then follows for all $k\geq k_{\dagger}^*$ and $N\geq N_*$.
\hfill\proofbox

\section{Numerical Experiments}\label{sec-num-experi-algorithms}

\subsection{Algorithms for approximating the discrete Hamiltonian system}
In this section we propose two algorithms for approximating the general discrete Hamiltonian system \eqref{numerical-scheme} for each fixed $k, N\in\mathbb{N}$.  The first approach is a standard Picard iteration performed globally in time and which is formulated in Algorithm \ref{alg:standard_picard_iteration} for generic time intervals of the form $[a,b]$ and arbitrary initial conditions $X_{init}^{k,N}$ and terminal time couplings $\mathcal{T}: \mathcal{R}_d(\mathcal{A}_N,\mathcal{O}_N)\to \mathcal{R}_d(\mathcal{A}_N,\mathcal{O}_N)$. Note that for Algorithm \ref{alg:standard_picard_iteration} we allow the trivial case where $\mathcal{T}$ is a constant map. Given a momentum approximation $\mathcal{Y}\in \mathbb{V}_k^-([0,T];\mathcal{R}_{d}(\mathcal{A}_N,\mathcal{O}_{N}))$, the deterministic forward McKean--Vlasov equation in Algorithm \ref{alg:standard_picard_iteration} is approximated by Algorithm \ref{alg:MKV_iteration}.  When using this first algorithm to approximate the discrete Hamiltonian system \eqref{numerical-scheme} globally in time, we take $a=0$, $b=T$, and $X_{init}^{k,N}\equiv X_0^N$, $\mathcal{T}[\mathcal{X}]\coloneqq  -D_xg\left(\mathcal{X},{\mathcal{X}}_{\#}\mathbb{P}\right)$, $\mathcal{X}\in \mathcal{R}_d(\mathcal{A}_N,\mathcal{O}_N)$. 
\begin{algorithm}
    \renewcommand{\thealgorithm}{A}
	\caption{A standard Picard Iteration on $[a,b]$ for the discrete Hamiltonian System}\label{alg:standard_picard_iteration}
	\begin{algorithmic}
		\Require $k,N\in\mathbb{N}$; $r = 0$; $0<tol<1$; $a,b\in\mathbb{R}$ s.t.\ $0\leq a<b$; $X_{init}^{k,N}\in \mathbb{V}_k^+([a,b];\mathcal{R}_{d}(\mathcal{A}_N,\mathcal{O}_{N}))$; $Y_{init}^{k,N}\in \mathbb{V}_k^-([a,b];\mathcal{R}_{d}(\mathcal{A}_N,\mathcal{O}_{N}))$; $E_{1}\geq 1$; and a map $\mathcal{T}: \mathcal{R}_d(\mathcal{A}_N,\mathcal{O}_N)\to \mathcal{R}_d(\mathcal{A}_N,\mathcal{O}_N)$. 
		\State Set $Y_0^{k,N} \gets Y_{init}^{k,N}$
		\While{$E_1>tol$}
		\State \underline{Forward McKean--Vlasov equation:} Compute $X_r^{k,N}\in  \mathbb{V}_k^+([a,b];\mathcal{R}_{d}(\mathcal{A}_N,\mathcal{O}_{N}))$ such that $$\mathcal{I}_+^kX_r^{k,N}(t) =X_{init}^{k,N}(a)+ \int_a^tD_pH\left(X_r^{k,N}(s),Y_r^{k,N}(s),(X_r^{k,N}(s))_{\#}\mathbb{P}\right)\mathrm{d}s\quad\text{for }t\in [a,b].$$
		\State \underline{Backward momentum equation:} Compute  $Y_{r+1}^{k,N}\in  \mathbb{V}_k^-([a,b];\mathcal{R}_{d}(\mathcal{A}_N,\mathcal{O}_{N}))$ such that 
		\begin{equation*}
			\mathcal{I}_-^kY_{r+1}^{k,N}(t) =\mathcal{T}[X_r^{k,N}(b)]+\int_t^bD_xH\left(X_r^{k,N}(s),Y_{r+1}^{k,N}(s),(X_r^{k,N}(s))_{\#}\mathbb{P}\right)\mathrm{d}s\quad\text{for }t\in [a,b].
		\end{equation*} 
		\State \underline{One-step error update:} Compute the one-step error $E_1=\|Y_{r+1}^{k,N}-Y_r^{k,N}\|_{L^2(a,b;L^2(\Omega;\mathbb{R}^d))}$
		\State If $E_1>tol$, set $r+1\gets r$ and compute the next iteration. 
		\EndWhile
	\end{algorithmic}
\end{algorithm}

\begin{algorithm}[h!]
    \renewcommand{\thealgorithm}{MKV}
	\caption{Picard Iteration for discrete McKean--Vlasov Equation on $[a,b]$}\label{alg:MKV_iteration}
	\begin{algorithmic}
		\Require $k,N\in\mathbb{N}$; $m = 0$; $a,b\in\mathbb{R}$ s.t.\ $0\leq a<b$; $0<tol<1$; $\chi\in \mathcal{R}_d(\mathcal{A}_N,\mathcal{O}_N)$;  $X_{init}^{k,N}\in \mathbb{V}_k^+([a,b];\mathcal{R}_{d}(\mathcal{A}_N,\mathcal{O}_{N}))$ with $X_{init}^{k,N}(0)=\chi$ in $\mathcal{R}_{d}(\mathcal{A}_N,\mathcal{O}_{N})$;  $\mathcal{Y}\in \mathbb{V}_k^-([a,b];\mathcal{R}_{d}(\mathcal{A}_N,\mathcal{O}_{N}))$; $E_{1}\geq 1$. 
		\State Set $X_0^{k,N}\gets X_{init}^{k,N}$
		\State Set $\mu_0^{k,N}(t) \gets (X_0^{k,N}(t))_{\#}\mathbb{P}$ $\forall t\in [a,b]$
		\While{$E_1>tol$}
		\State \underline{McKean--Vlasov equation:} Compute $X_{m+1}^{k,N}\in  \mathbb{V}_k^+([a,b];\mathcal{R}_{d}(\mathcal{A}_N,\mathcal{O}_{N}))$ such that $$\mathcal{I}_+^kX_{m+1}^{k,N}(t) =\chi+ \int_a^tD_pH\left(X_{m+1}^{k,N}(s),\mathcal{Y}(s),\mu_{m}^{k,N}\right)\mathrm{d}s\text{ for }t\in [a,b].$$ 
		\State \underline{One-step error:} Compute $E_1=\|X_{m+1}^{k,N}-X_m^{k,N}\|_{L^2(a,b;L^2(\Omega;\mathbb{R}^d))}$ (Note: $E_1\geq \|\mathcal{W}_2(\mu_{m+1}^{k,N},\mu_{m}^{k,N})\|_{L^2(a,b)}$)
		\State \underline{Update player distribution} Set $\mu_{m+1}^{k,N}(t) \gets (X_{m+1}^{k,N}(t))_{\#}\mathbb{P}$ $\forall t\in [a,b]$
		\State If $E_1>tol$, set $m+1\gets m$ and compute the next iteration.
		\EndWhile
	\end{algorithmic}
\end{algorithm}

However, it is known that, in general, the Picard iteration in Algorithm \ref{alg:standard_picard_iteration} for approximating the Hamiltonian system \eqref{numerical-scheme} is not guaranteed to converge for arbitrarily large time horizon $T$. Therefore, we propose a second algorithm that is inspired by \cite{chassagneux2019numerical}, where  Chassagneux--Crisan--Delarue introduced a numerical method for the  approximation of a forward--backward system of stochastic differential equations based on patching together the results of local Picard iterations developed on time \emph{subintervals of $[0,T]$} with sufficiently small time-step. The algorithm we propose also performs local Picard iterations but within a global-in-time approximation of the deterministic nonlinear forward-backward discrete Hamiltonian system \eqref{numerical-scheme} for arbitrary time horizon $T$ and with displacement monotone data. 

To formulate the algorithm more compactly, recall the definition of the spaces $ \mathbb{V}_k^{\pm}([a,b];A)$ from Section \ref{sec-numerical-discretization}. We define related the spaces  $ \{\mathbb{V}_k^{\pm}([t_n,t_{n+1}];\mathcal{R}_{d}(\mathcal{A}_N,\mathcal{O}_{N}))\}_{n=0}^{M_k-1}$ where the set $\{t_0,t_1,\cdots, t_{M_k}\}\subset [a,b]$ denotes the temporal nodes of the time mesh $\mathcal{J}_k$. Furthermore, note that for any function $Q\in \mathbb{V}_k^{\pm}([a,b];\mathcal{R}_{d}(\mathcal{A}_N,\mathcal{O}_{N}))$, its restriction to a given temporal subinterval $[t_n,t_{n+1}]$ satisfies $$Q|_{[t_n,t_{n+1}]}\in \mathbb{V}_k^{\pm}([t_n,t_{n+1}];\mathcal{R}_{d}(\mathcal{A}_N,\mathcal{O}_{N}))$$ for each $n\in\{0,1,2,\cdots,M_k-1\}$. For each $k,N\in\mathbb{N}$, $0<tol<1$, $a,b\in\mathbb{R}$ s.t.\ $0\leq a<b$, $X_{init}^{k,N}\in V_k^{+}([a,b];\mathcal{R}_{d}(\mathcal{A}_N,\mathcal{O}_{N}))$, $Y_{init}^{k,N}\in \mathbb{V}_k^-([a,b];\mathcal{R}_{d}(\mathcal{A}_N,\mathcal{O}_{N}))$, $E_1>1$, and given map $\mathcal{T}: \mathcal{R}_d(\mathcal{A}_N,\mathcal{O}_N)\to \mathcal{R}_d(\mathcal{A}_N,\mathcal{O}_N)$, we let $$\Phi^{k,N}[a,b,tol,X_{init}^{k,N},Y_{init}^{k,N},E_1,\mathcal{T}]=(\Phi_X,\Phi_Y)\in  \mathbb{V}_k^+([a,b];\mathcal{R}_{d}(\mathcal{A}_N,\mathcal{O}_{N}))\times  \mathbb{V}_k^-([a,b];\mathcal{R}_{d}(\mathcal{A}_N,\mathcal{O}_{N}))$$ denote the output of Algorithm \ref{alg:standard_picard_iteration}, whenever it converges after some iteration count $r_*\in\mathbb{N}$ within the tolerance $tol$, where $\Phi_X\coloneqq X_{r_*}^{k,N}$ and $\Phi_Y\coloneqq Y_{r_*}^{k,N}$. We then define in Algorithm \ref{alg:local_picard_iteration} a method based on performing Picard iterations  on each of the subintervals $\{[t_n,t_{n+1}]\}_{n=0}^{M_k-1}$ along downward induction, c.f.\ \cite{chassagneux2019numerical}. In this algorithm we also apply Algorithm \ref{alg:MKV_iteration} to approximate the associated discrete McKean--Vlasov equations along the iterations. A pictorial depiction of Algorithm \ref{alg:local_picard_iteration} is given in Figure \ref{algo:local_picard_iteration-pic}. We note that the use of $L^2(L^2)$-norms for computing the one-step error quantities (namely $E_1,E_{local},E_{global}$) in each of the algorithms is a choice made merely for formulating the algorithms; for each algorithm one can fix any norm  to check the stated convergence criteria since all norms are equivalent in the finite dimensional spaces $\mathbb{V}_k^{\pm}$.

\begin{figure}[htbp]
\footnotesize
  \centering
\begin{tikzpicture}[
    >={Stealth[length=2.6mm]},
    vin/.style   = {->, thick, red},      
    vout/.style  = {->, thick, magenta},  
    horiz/.style = {->, thick, blue},     
    line join=round]
\def\s{1.6}\def\d{0.55}\def\al{0.95}\def\sp{6}
\pgfmathsetmacro{\cc}{(\s+\d)/2}   %

\drawcube{0*\sp}{$(X_j^{k,N}(0),Y_j^{k,N}(0))$}{$Y_{j+1}^{k,N}(t_1)$}{$Y_{j+1}^{k,N}(0)$}{$X_{j+1}^{k,N}(t_1)$}{\tiny \textbf{Picard iter.\ } \\ \tiny on $[0,t_1]$ }
\drawcube{1*\sp}{$(X_j^{k,N}(t_n),Y_j^{k,N}(t_n))$}{$Y_{j+1}^{k,N}(t_{n+1})$}{$Y_{j+1}^{k,N}(t_n)$}{$X_{j+1}^{k,N}(t_{n+1})$}{\tiny \textbf{Picard iter.\ } \\ \tiny on $[t_n,t_{n+1}]$}
\drawcube{2*\sp}{$(X_j^{k,N}(t_{M_k-1}),Y_j^{k,N}(t_{M_k-1}))$}{$\quad\mathcal{T}=-D_xg$}{$Y_{j+1}^{k,N}(t_{M_k-1})$}{$X_{j+1}^{k,N}(T)$}{\tiny \textbf{Picard iter.\ } \\ \tiny on $[t_{M_k-1},T]$}

\foreach \k in {0,1} {
  \node[font=\Large] at ({\k*\sp + \sp/2 + \cc}, \cc) {$\cdots$};
}
\end{tikzpicture}
\normalsize
\caption{ Depiction of $j$-th iteration of Algorithm \ref{alg:local_picard_iteration} -- The direction of the blue arrows indicate downward induction along the temporal subintervals $\{[t_n,t_{n+1}]\}_{n=0}^{M_k-1}$ starting from the right-most grey block. Each grey block corresponds to performing standard Picard iteration by Algorithm \ref{alg:standard_picard_iteration} to approximate the discrete Hamiltonian system posed on the $n$-th subinterval $[t_n,t_{n+1}]$ with initialization being the value pair $(X_j^{k,N}(t_n),Y_j^{k,N}(t_n))$ (red arrow) and output pair being the pair $(X_{j+1}^{k,N}(t_{n+1}),Y_{j+1}^{k,N}(t_n))$ (magenta and blue arrows respectively). Communication between subintervals is done by setting the terminal time value of $Y$ on the $n$-th interval as the $Y$-component of the output of the $(n+1)$-th interval, where for $n=M_k-1$ the terminal time value $Y_{j+1}^{k,N}(T)$ is implicitly determined by $X_{j+1}^{k,N}(T)$ via the terminal time coupling $\mathcal{T}=-D_xg
$. With $X_{j+1}^{k,N}(0)\coloneqq X_0^N$, the nodal output values of one downward sweep completely determine the function pairs $(X_{j+1}^{k,N},Y_{j+1}^{k,N})$ on $[0,T]$ used as input for the next $(j+1)$-th iteration of the algorithm.}
  \label{algo:local_picard_iteration-pic}
\end{figure}
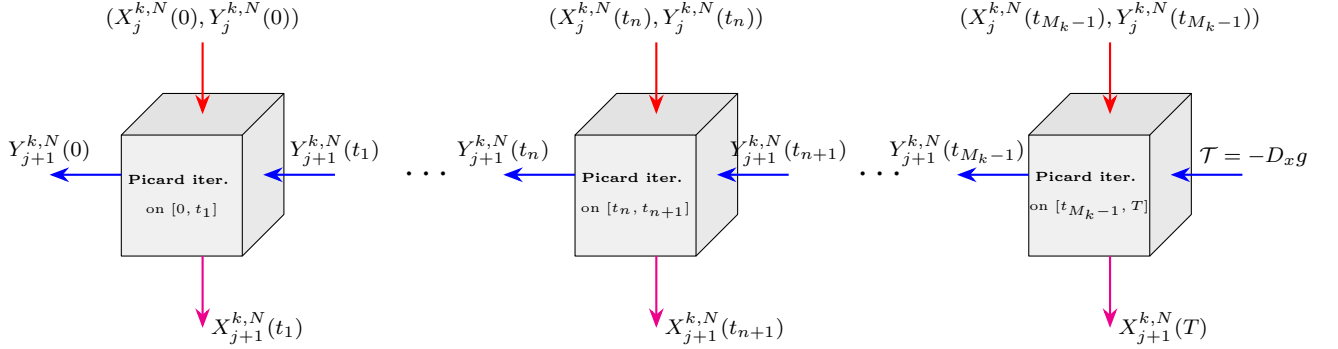

Proving the convergence of Algorithm \ref{alg:local_picard_iteration} is {outside the scope of the current paper and} will be the subject of future research. { Nonetheless, the following result sheds light on the relationship between Algorithm \ref{alg:local_picard_iteration} and the target numerical scheme \eqref{numerical-scheme} in an ideal scenario where all of the local Picard iterations of the algorithm are solved exactly, i.e.\ with zero one-step error.
\begin{proposition}[An a posteriori error bound for Algorithm \ref{alg:local_picard_iteration}]
	Assume the hypotheses of Theorem \ref{theorem-existence-num-scheme}. Let $k\geq k_{\dagger}$ be given, where $k_{\dagger}\in\mathbb{N}$ is as in Theorem \ref{theorem-existence-num-scheme} and let $N\in\mathbb{N}$ be given. Suppose that in Algorithm \ref{alg:local_picard_iteration} the computation of the map $\Phi^{k,N}$ is always exact, in the sense that the first one-step error made in employing Algorithm \ref{alg:standard_picard_iteration} to compute $\Phi^{k,N}$ is always equal to zero. Then, the sequence of approximations $\{(X_{j+1}^{k,N},Y_{j+1}^{k,N})\}_{j=0}^{\infty}$ generated by Algorithm \ref{alg:local_picard_iteration} satisfies the discrete Hamiltonian system
	 \begin{equation}\label{numerical-scheme-loc-algo}
	 	\left\{ \begin{aligned}
	 		\partial_t\mathcal{I}_+^k{X}_{j+1}^{k,N}(t) &=  D_pH\left({X}_{j+1}^{k,N}(t),{Y}_{j+1}^{k,N}(t),(X_{j+1}^{k,N}(t))_{\#}\mathbb{P}\right)+F_j(t)\quad{\rm in\ }\mathcal{R}_{d}(\mathcal{A}_N,\mathcal{O}_N)&&\forall t\in (0,T)\backslash \iota_k, 
	 		\\
	 		\partial_t\mathcal{I}_-^k{Y}_{j+1}^{k,N}(t) &=- D_xH\left({X}_{j+1}^{k,N}(t),{Y}_{j+1}^{k,N}(t),(X_{j+1}^{k,N}(t))_{\#}\mathbb{P}\right)\quad{\rm in\ }\mathcal{R}_{d}(\mathcal{A}_N,\mathcal{O}_N)&&\forall t\in (0,T)\backslash \iota_k, 
	 		\\
	 		{X}_{j+1}^{k,N}(0)&=X_0^N,\quad {Y}_{j+1}^{k,N}(T) = -  D_xg({X}_{j+1}^{k,N}(T),(X_{j+1}^{k,N}(T))_{\#}\mathbb{P})\quad{\rm in\ }\mathcal{R}_{d}(\mathcal{A}_N,\mathcal{O}_N)&&\text{  }
	 	\end{aligned}\right.
	 \end{equation}
	 where $F_j\in \mathbb{V}_k(0,T;\mathcal{R}_{d}(\mathcal{A}_N,\mathcal{O}_N))$ is a residual error term defined by $F_j|_{I_1}= 0$ in $\mathcal{R}_{d}(\mathcal{A}_N,\mathcal{O}_N)$ and
	 \begin{equation}\label{F_j-defn}
	 	F_j|_{I_n}\coloneqq \tau_k^{-1}\left(X_{j}^{k,N}(t_{n-1})-X_{j+1}^{k,N}(t_{n-1})\right)
	 \end{equation} in $\mathcal{R}_{d}(\mathcal{A}_N,\mathcal{O}_N)$ for $n\in\{2,3,\cdots, M_k\}$. Furthermore, there holds the error bound
	 \begin{equation}\label{it-approx-bound-a-posteriori}
	 	\sup_{0\leq t\leq T}\left(\norm{({X}^{k,N}-{X}_{j+1}^{k,N})(t)}_{L^2(\Omega;\mathbb{R}^d)}+\norm{({Y}^{k,N}-{Y}_{j+1}^{k,N})(t)}_{L^2(\Omega;\mathbb{R}^d)}\right)\lesssim \|F_j\|_{L^2(0,T;\LLspace)}
	 \end{equation}
	 where the hidden constant in \eqref{it-approx-bound-a-posteriori} depends only on the time horizon $T$, the uniform convexity constant $c_0>0$, and the Lipschitz constants of $D_pH$, $D_xH$ and $D_xg$. 
\end{proposition}

\begin{proof}
	We give a sketch of the proof as each step consists of merely direct calculations. Since we have assumed that $\Phi^{k,N}$ is always exactly evaluated,  we have from Algorithm \ref{alg:local_picard_iteration} that, for each $j\in\mathbb{N}$, 
	\begin{equation}\label{temp-node-cond-1-algo-3}
		X_{j+1}^{k,N}(t_n) =X_{j}^{k,N}(t_{n-1})+ \tau_kD_pH\left(X_{j+1}^{k,N}(t_n),Y_{j+1}^{k,N}(t_{n-1}),(X_{j+1}^{k,N}(t_n))_{\#}\mathbb{P}\right)
	\end{equation}
	\begin{equation}\label{temp-node-cond-2-algo-3}
		Y_{j+1}^{k,N}(t_{n-1}) = Y_{j+1}^{k,N}(t_n)+\tau_kD_xH\left(X_{j+1}^{k,N}(t_n),Y_{j+1}^{k,N}(t_{n-1}),(X_{j+1}^{k,N}(t_n))_{\#}\mathbb{P}\right)
	\end{equation}
	for all $n\in\{1,2,\cdots,M_k\}$ where $Y_{j+1}^{k,N}(T)=-  D_xg({X}_{j+1}^{k,N}(T),(X_{j+1}^{k,N}(T))_{\#}\mathbb{P})$ and $X_{j+1}^{k,N}(0)=X_0^N$. There are only two unknowns of the local forward-backward Hamiltonian system on $(t_{n-1},t_n)$, namely the values $X_{j+1}^{k,N}(t_n)$ and $Y_{j+1}^{k,N}(t_{n-1})$, whereas  $Y_{j+1}^{k,N}(t_{n})$ is known from the previous local forward-backward Hamiltonian system on the subinterval $(t_n,t_{n+1})$  and $X_{j}^{k,N}(t_{n-1})$ is known from the previous global iteration. The perturbed forward-backward Hamiltonian system \eqref{numerical-scheme-loc-algo} with residual term \eqref{F_j-defn} is then obtained by directly computing the partial derivatives $\partial_t\mathcal{I}_+^k{X}_{j+1}^{k,N}(t)$, $\partial_t\mathcal{I}_-^k{Y}_{j+1}^{k,N}(t)$  for $t\in (0,T)\backslash \iota_k$ by using the definition of the interpolation operators $\Ipm^k$ and the equations \eqref{temp-node-cond-1-algo-3}, \eqref{temp-node-cond-2-algo-3}, in similar fashion to the proof of Lemma \ref{lemma-fp-map-nodal-charac}. Because $k\geq k_{\dagger}$, we can apply our main continuous dependence bound from Theorem  \ref{theorem-cont-dependence-discrete-Hamiltonian-sys} to the numerical scheme \eqref{numerical-scheme} and the system \eqref{numerical-scheme-loc-algo} obtained from Algorithm \ref{alg:local_picard_iteration} to immediately get the claimed a posteriori error bound 
	\eqref{it-approx-bound-a-posteriori}. 
\end{proof}}

\begin{algorithm}[h!]
    \renewcommand{\thealgorithm}{B}
	\caption{Local Picard Iteration Method for the discrete Hamiltonian System \eqref{numerical-scheme}}\label{alg:local_picard_iteration}
	\begin{algorithmic}
		\Require $k,N\in\mathbb{N}$ fixed; $j = 0$; $0<tol_{global},tol_{local}<1$; $X_{init}^{k,N}\in \mathbb{V}_k^+([0,T];\mathcal{R}_{d}(\mathcal{A}_N,\mathcal{O}_{N}))$ with $X_{init}^{k,N}(0)=X_0^N$ in $L^2(\Omega;\mathbb{R}^d)$, $Y_{init}^{k,N}\in \mathbb{V}_k^-([0,T];\mathcal{R}_{d}(\mathcal{A}_N,\mathcal{O}_{N}))$; $E_{global}=1$ and $E_{local}=1$. \State Set $X_0^{k,N} \coloneqq X_{init}^{k,N}$ in $ \mathbb{V}_k^+([0,T];\mathcal{R}_{d}(\mathcal{A}_N,\mathcal{O}_{N}))$,  $Y_0^{k,N} \coloneqq Y_{init}^{k,N}$ in $ \mathbb{V}_k^-([0,T];\mathcal{R}_{d}(\mathcal{A}_N,\mathcal{O}_{N}))$.
		
		\While{$E_{global}>tol_{global}$}
		\State Use $X_j^{k,N}$ and $Y_j^{k,N}$ to compute $X_{j+1}^{k,N}\in \mathbb{V}_k^+([0,T];\mathcal{R}_{d}(\mathcal{A}_N,\mathcal{O}_{N}))$ with $X_{j+1}^{k,N}(0)\coloneqq X_0^N$ in $\mathcal{R}_d(\mathcal{A}_N,\mathcal{O}_N)$, $Y_{j+1}^{k,N}\in \mathbb{V}_k^-([0,T];\mathcal{R}_{d}(\mathcal{A}_N,\mathcal{O}_{N}))$ as follows:
		\State \underline{Computation on subinterval $[t_{M_k-1},T]$}
		\State Set $\mathcal{T}:\mathcal{R}_{d}(\mathcal{A}_N,\mathcal{O}_N)\to \mathcal{R}_{d}(\mathcal{A}_N,\mathcal{O}_N)$ via $\mathcal{T}[\mathcal{X}]\coloneqq  -D_xg\left(\mathcal{X},{\mathcal{X}}_{\#}\mathbb{P}\right)$, $\mathcal{X}\in \mathcal{R}_d(\mathcal{A}_N,\mathcal{O}_N)$.
		\State Compute $(\Phi_{X},\Phi_{Y})\coloneqq \Phi^{k,N}[t_{M_k-1},T,tol_{local},X_{j}^{k,N},Y_{j}^{k,N},E_{local},\mathcal{T}]$ by Algorithm \ref{alg:standard_picard_iteration}.
		\State Set $X_{j+1}^{k,N}(T)\coloneqq \Phi_{X}(T)$,  $Y_{j+1}^{k,N}(T)\coloneqq\Phi_{Y}(T)$, $Y_{j+1}^{k,N}(t_{M_k-1})\coloneqq \Phi_{Y}(t_{M_k-1})$ in $\mathcal{R}_{d}(\mathcal{A}_N,\mathcal{O}_{N})$.
		
		\For{$i = M_k-2,M_k-3,\cdots, 3, 2,1$}
		\State \underline{Computation on subinterval $[t_n,t_{n+1}]$}
		\State Set the constant map $\mathcal{T}:\mathcal{R}_{d}(\mathcal{A}_N,\mathcal{O}_N)\to \mathcal{R}_{d}(\mathcal{A}_N,\mathcal{O}_N)$ via $\mathcal{T}[\mathcal{X}]\coloneqq Y_{j+1}^{k,N}(t_{n+1})$, $\mathcal{X}\in \mathcal{R}_d(\mathcal{A}_N,\mathcal{O}_N)$.
		\State Compute $(\Phi_{X},\Phi_Y)\coloneqq \Phi^{k,N}[t_n,t_{n+1},tol_{local},X_{j}^{k,N},Y_{j}^{k,N},E_{local},\mathcal{T}]$ by Algorithm \ref{alg:standard_picard_iteration}.
		\State Set $X_{j+1}^{k,N}(t_{n+1})\coloneqq \Phi_{X}(t_{n+1})$,  $Y_{j+1}^{k,N}(t_n)\coloneqq \Phi_Y(t_n)$ in $\mathcal{R}_{d}(\mathcal{A}_N,\mathcal{O}_{N})$.
		\EndFor
		
		\State \underline{Computation on subinterval $[0,t_1]$}
		\State Set the constant map $\mathcal{T}:\mathcal{R}_{d}(\mathcal{A}_N,\mathcal{O}_N)\to \mathcal{R}_{d}(\mathcal{A}_N,\mathcal{O}_N)$ with $\mathcal{T}[\mathcal{X}]\coloneqq Y_{j+1}^{k,N}(t_1)$, $\mathcal{X}\in \mathcal{R}_d(\mathcal{A}_N,\mathcal{O}_N)$.
		\State Compute $(\Phi_{X},\Phi_{Y})\coloneqq \Phi^{k,N}[0,t_1,tol_{local},X_{j}^{k,N},Y_{j}^{k,N},E_{local},\mathcal{T}]$ by Algorithm \ref{alg:standard_picard_iteration}.
		\State Set $X_{j+1}^{k,N}(t_1)\coloneqq \Phi_{X}(t_1)$,  $Y_{j+1}^{k,N}(0)\coloneqq \Phi_{Y}(0)$ in $\mathcal{R}_{d}(\mathcal{A}_N,\mathcal{O}_{N})$.
		\State \underline{Global one-step error update:} Compute $$E_{global}=\left\|X_{j+1}^{k,N}-X_j^{k,N}\right\|_{L^2(0,T;L^2(\Omega;\mathbb{R}^d))}+\left\|Y_{j+1}^{k,N}-Y_j^{k,N}\right\|_{L^2(0,T;L^2(\Omega;\mathbb{R}^d))}.$$
		\State If $E_{global}>tol_{global}$, set $j+1\gets j$ and compute the next global iteration on $[0,T]$.
		\EndWhile
	\end{algorithmic}
\end{algorithm}

\subsection{The experiments}
In this section we present results of two numerical experiments that test the performance of the numerical scheme. To this end, we compare the performance of the standard Picard iteration from Algorithm \ref{alg:standard_picard_iteration} and the local Picard iteration scheme from Algortihm \ref{alg:local_picard_iteration}. Since there holds the estimate
\begin{equation}
	\sup_{t\in [0,T]}\mathcal{W}_2(\rho(t),\rho^{k,N}(t))\leq \sup_{t\in [0,T]}\|(X-X^{k,N})(t)\|_{L^2(\Omega;\mathbb{R}^d)},
\end{equation}
in the experiments we assess the approximations $(X_{app}^{k,N},Y_{app}^{k,N})$ of $(X^{k,N},Y^{k,N})$ generated by either algorithm by computing the following relative errors:
\begin{equation}
	E_X^{k}\coloneqq \frac{\sup_{t\in [0,T]}\|(X-X_{app}^{k,N})(t)\|_{L^2(\Omega;\mathbb{R}^d)}}{\sup_{t\in [0,T]}\|X(t)\|_{L^2(\Omega;\mathbb{R}^d)}}, \quad  E_Y^{k}\coloneqq \frac{\sup_{t\in [0,T]}\|Y-Y_{app}^{k,N}(t)\|_{L^2(\Omega;\mathbb{R}^d)}}{\sup_{t\in [0,T]}\|Y(t)\|_{L^2(\Omega;\mathbb{R}^d)}}.
\end{equation}

\subsubsection*{First experiment.}
In this test, we consider the approximation of the continuous MFG system with respect to a reference solution computed on a fine grid.
The model data that we consider for the continuous MFG system are as follows. Let the space dimension $d=2$ and let $T=1$. Let $\mathcal{S}\coloneqq \{ 2^{-3}(2m-1,2n-1)^\top\in\mathbb{R}^2:n,m\in \{1,2,3,4\}\}$, which we note is a set that consists of sixteen distinct points of the unit square $[0,1]^2$. Let $\{x_i\}_{i=1}^{16}$ denote a fixed denumeration of the points in $\mathcal{S}$. We assume that the initial measure $\rho_0\in\mathcal{P}_2(\mathbb{R}^2)$ is a deterministic empirical measure of the form $\rho_0\coloneqq (1/16)\sum_{i=1}^{16}\delta_{x_i}$. The Hamiltonian of the MFG system is taken to be 
\begin{multline}
	H(x,p,\mu)\coloneqq \frac{1}{4}(p_1^2+p_2^2-x_1^2-x_2^2)+\left(p_1+p_2\right)\int_{\mathbb{R}^2}\left(1-\sin(y_1+y_2)\right)\mathrm{d}\mu(y)
	\quad\forall (x,p,\mu)\in  \mathbb{R}^2\times\mathbb{R}^2\times\mathcal{P}_2(\mathbb{R}^2),
\end{multline}
with the terminal cost  set as
\begin{equation}
	g(x,\mu)\coloneqq  \frac{1}{2}(x_1^2+x_2^2)+\left(x_1+x_2\right)\int_{\mathbb{R}^2}\left(1-\cos(y_1+y_2)\right)\mathrm{d}\mu(y)
	\quad\forall (x,\mu)\in  \mathbb{R}^2\times\mathcal{P}_2(\mathbb{R}^2).
\end{equation}
From Example \ref{eg-1} we find that $H$ satisfies the regularity assumptions \eqref{ass-H:1}, \eqref{ass-H:2}, \eqref{ass-H:3}, \eqref{ass-H:4}, \eqref{ass-H:5} and the displacement monotonicity condition \eqref{ass-H:disp-mono-cond-H}, and that the $g$ above satisfies the regularity assumptions \eqref{ass-g:1}, \eqref{ass-g:2} and the displacement monotonicity condition \eqref{ass-g:disp-mono-cond-g}, and that the growth condition \eqref{ass-L-g} on $g$ and the Lagrangian $L$ of $H$ also holds. In addition, using \cite[Theorem 3.7, Theorem 4.5]{meszaros2024mean}, the MFG system \eqref{mfg-pde-sys} has a unique solution in the sense of \cite[Definition 3.2]{meszaros2024mean}.

For the associated continuous Hamiltonian system \eqref{characteristics-continuous}, we consider the reference probability space $(\Omega,\mathbb{F},\mathbb{P})$ where $\Omega=[0,1]^2$,  $\mathbb{F}$  is the $\sigma$-algebra of Lebesgue measurable subsets of $\Omega$, and $\mathbb{P}$ denotes the restriction of the two-dimensional Lebesgue measure to $\Omega$. Let  $\mathcal{O}\coloneqq \{\Omega_{i}\}_{i=1}^{16}\subset \Omega$ denote a uniform partition of $\Omega$ into (open) squares, each with centroid $w_i\in \Omega_i$ and area $\mathbb{P}(\overline{\Omega_{i}})=16^{-1}$, for $i\in\{1,2,\cdots, 16\}$. For this problem, we let the initial random variable $X_0$ take the form $X_0\coloneqq \sum_{i=1}^{16}{w_i}\chi_{\Omega_i}$ a.e.\ in $\Omega$. In this setting we do not have access to an explicit formula for the unique solution $(X,Y)$ of the continuous Hamiltonian system \eqref{characteristics-continuous}, so instead we consider the performance of the numerical scheme \eqref{numerical-scheme} with respect to the reference solution, which we illustrate in Figure \ref{fig:test-1-ref-soln} on a coarse temporal grid.

\begin{figure}[h!]
	\centering
	\includegraphics[width=0.7\textwidth]{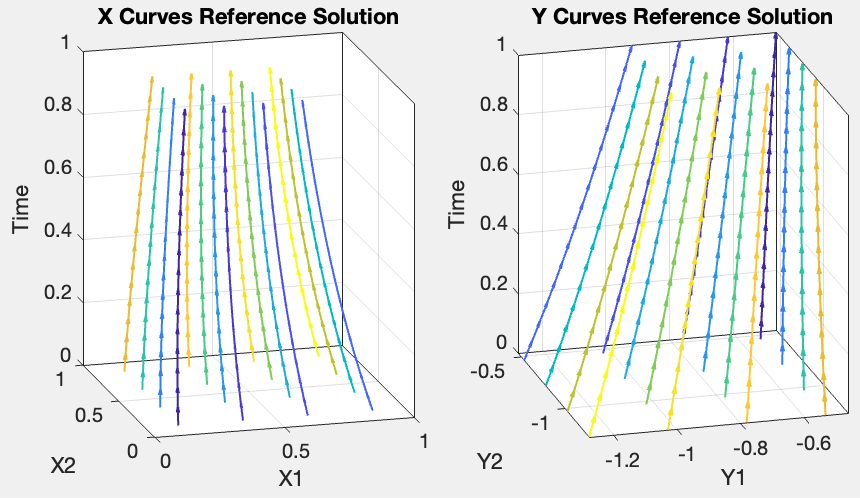}
	\caption{First experiment -- Illustration of reference solution $(X,Y)$ of continuous Hamiltonian system \eqref{characteristics-continuous} with $T=1$ on a coarse temporal grid.}
	\label{fig:test-1-ref-soln}
\end{figure}

In the discrete Hamiltonian system \eqref{numerical-scheme}, we simply set $\mathcal{A}_N\coloneqq \{\alpha_{i,N}\}_{i=1}^N$ with $\alpha_{i,N}\coloneqq 1/16$ for $i=1,2\cdots,N$, $\mathcal{O}_N\coloneq \mathcal{O}$, $\rho_0^N\coloneqq \rho_0$ in $\mathcal{P}_2(\mathbb{R}^2)$, and $X_0^N\coloneqq X_0$ in $\mathcal{R}_{2}(\mathcal{A}_N,\mathcal{O}_{N})$ for fixed $N=16$, since $\rho_0$ is already an empirical measure supported on a finite number of points. We consider a uniform partition of the time interval $[0,T]$ with time-step $\tau_k\coloneqq 2^{-k+1}$ with $k\in\mathbb{N}$. Theorem \ref{theorem-convergence-k-N-joint} then implies that the total error $E_{total}^{k}\coloneqq E_{X}^{k}+E_Y^{k}$ satisfies
\begin{equation}
	E_{total}^{k}\lesssim \|X_0-X_0^N\|_{\LLspace}+\tau_k\lesssim0+\tau_k=\tau_k
\end{equation}for sufficiently large $k\in\mathbb{N}$. To obtain a reference solution for $(X,Y)$ for computing the errors when $T=1$, we applied the standard Picard iteration scheme in Algorithm \ref{alg:standard_picard_iteration} with $tol_{global}=10^{-8}$, $k_{ref}=12$ (whence $\tau_{k_{ref}}=2^{-11}$), and we set $Y_{init}^{k_{ref},N}\coloneq 0$ in $\mathbb{V}_{k_{ref}}^-([0,T];\mathcal{R}_{2}(\mathcal{A},\mathcal{O}))$ and $X_{init}^{k_{ref},N}\in  \mathbb{V}_k^+([0,T];\mathcal{R}_{2}(\mathcal{A},\mathcal{O}))$ where $X_{init}^{k_{ref},N}|_{(0,T]}\coloneqq 0$ with $X_{init}^{k_{ref},N}(0)\coloneqq X_0$. This reference solution is depicted in Figure \ref{fig:test-1-ref-soln} on a coarser temporal partition. In the next sub-test where we let $T\in\{2,4,8,16,32\}$, for each $T$ we obtain a reference solution $(X,Y)$ for computing the errors by employing the local Picard iteration of Algorithm \ref{alg:local_picard_iteration} with the same initialisation $X_{init}^{k_{ref},N}$, $Y_{init}^{k_{ref},N}$, where $k_{ref}=12$ and global tolerance $tol_{global}=10^{-8}$, but we set the local tolerance to be $tol_{local}=10^{-12}$.

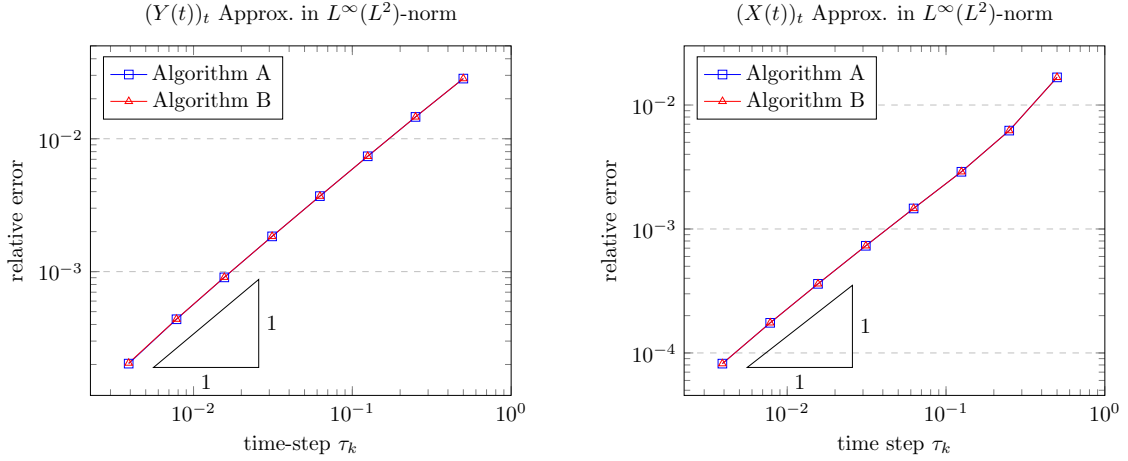
\begin{figure}[h!]
	\centering
	\begin{tabular}{c c} 
		\begin{subfigure}[b]{0.45\textwidth}
			\begin{adjustbox}{width=0.95\linewidth}
				\begin{tikzpicture}
					\begin{loglogaxis}[
						title={$(Y(t))_{t}$ Approx.\ in $L^{\infty}(L^2)$-norm},
						xlabel={time-step $\tau_k$},
						ylabel={relative error},
						xmax=1,
						ymax=0.05,
						legend pos=north west,
						ymajorgrids=true,
						grid style=dashed,
						]
						
						\addplot[
						color=blue,
						mark=square,]
						coordinates {
							(0.500000, 0.028365)
(0.250000, 0.014596)
(0.125000, 0.007381)
(0.062500, 0.003701)
(0.031250, 0.001842)
(0.015625, 0.000905)
(0.007812, 0.000438)
(0.003906, 0.000203)
						};
						
						\addplot[
						color=red,
						mark=triangle,]
						coordinates {
							(0.500000, 0.028367)
(0.250000, 0.014598)
(0.125000, 0.007382)
(0.062500, 0.003699)
(0.031250, 0.001841)
(0.015625, 0.000907)
(0.007812, 0.000439)
(0.003906, 0.000205)

						};
						
						\logLogSlopeTriangle{0.4}{0.25}{0.08}{1}{black};
						
						\legend{Algorithm A,Algorithm B}
					\end{loglogaxis}  
				\end{tikzpicture}
			\end{adjustbox}
		\end{subfigure}
		&
		\begin{subfigure}[b]{0.45\textwidth}
			\begin{adjustbox}{width=0.95\linewidth} 
				\begin{tikzpicture}
					\begin{loglogaxis}[
						title={$(X(t))_t$ Approx. in $L^{\infty}(L^2)$-norm},
						xlabel={time step $\tau_k$},
						ylabel={relative error},
						xmax=1,
						ymax= 0.03,
						legend pos=north west,
						ymajorgrids=true,
						grid style=dashed,
						]
						
						\addplot[
						color=blue,
						mark=square,]
						coordinates {
						(0.500000, 0.016713)
(0.250000, 0.006219)
(0.125000, 0.002894)
(0.062500, 0.001464)
(0.031250, 0.000732)
(0.015625, 0.000361)
(0.007812, 0.000175)
(0.003906, 0.000082)
						};
						
						\addplot[
						color=red,
						mark=triangle,]
						coordinates {
							(0.500000, 0.016713)
(0.250000, 0.006218)
(0.125000, 0.002895)
(0.062500, 0.001464)
(0.031250, 0.000732)
(0.015625, 0.000361)
(0.007812, 0.000175)
(0.003906, 0.000082)
						};

						\logLogSlopeTriangle{0.4}{0.25}{0.08}{1}{black};
						
						\legend{Algorithm A,Algorithm B}
					\end{loglogaxis}  
				\end{tikzpicture}
			\end{adjustbox}
		\end{subfigure}
	\end{tabular}
	\caption{First experiment: $L^{\infty}(L^2)$-Error plots for approximations of the reference  $(Y(t))_{t\in [0,T]}$ curve and the reference trajectories of the players  $(X(t))_{t\in [0,T]}$  for $T=1$ using Algorithm \ref{alg:standard_picard_iteration} and Algorithm \ref{alg:local_picard_iteration}.}
	\label{Fig-1}
\end{figure}
For $k\in \{2,3,\cdots, 9\}$, we apply both the standard Algorithm \ref{alg:standard_picard_iteration} and the local Algorithm \ref{alg:local_picard_iteration} in MATLAB to solve the discrete Hamiltonian system \eqref{numerical-scheme}, where for each algorithm we initialized by setting $Y_{init}^{k,N}\coloneqq 0$ in $ \mathbb{V}_k^-([0,T];\mathcal{R}_{2}(\mathcal{A},\mathcal{O}))$ and $X_{init}^{k,N}\in  \mathbb{V}_k^+([0,T];\mathcal{R}_{2}(\mathcal{A},\mathcal{O}))$ where $X_{init}^{k,N}|_{(0,T]}\coloneqq 0$ with $X_{init}^{k,N}(0)\coloneqq X_0$, and we set the global tolerance to be $tol_{global}=10^{-8}$. The local error tolerance for Algorithm \ref{alg:local_picard_iteration} was set to $tol_{local}=10^{-12}$. For solving the associated discrete McKean--Vlasov equations using Algorithm \ref{alg:MKV_iteration} in the respective implementations, throughout we initialized using $X_{init}^{k,N}\in  \mathbb{V}_k^-([0,T];\mathcal{R}_{2}(\mathcal{A},\mathcal{O}))$ where $X_{init}^{k,N}|_{(0,T]}\coloneqq 0$ with $X_{init}^{k,N}(0)\coloneqq \chi$ where $\chi$ denotes a generic initial condition that is determined by the given scheme (i.e. being fixed for Algorithm \ref{alg:standard_picard_iteration} but varying in Algorithm \ref{alg:local_picard_iteration} across subinterval computations). The algorithm for approximating the discrete McKean--Vlasov equation was run with error tolerance $tol=10^{-12}$ throughout both implementations of Algorithm \ref{alg:local_picard_iteration} and Algorithm \ref{alg:standard_picard_iteration} in this experiment.

The error plots given by the discrete system \eqref{numerical-scheme} for this experiment are shown in Figure \ref{Fig-1}. We observe that the error plots indicate an asymptotic rate of convergence of order one as the time-step vanishes, which is in agreement with our theoretical prediction. This observed rate of convergence is optimal since the approximation is based on piecewise constant-in-time approximations of $H^{1}$-regular time curves. We further see that the approximations generated by both Algorithms \ref{alg:standard_picard_iteration} and \ref{alg:local_picard_iteration} produce comparable discrete solutions since their respective error plots overlap well. 

We also assess the robustness of the numerical scheme \eqref{numerical-scheme} for longer time horizons $T\in\{2,4,8,16,32\}$ by using Algorithm \ref{alg:local_picard_iteration} based on local Picard iterations. We assess here robustness in the sense that the asymptotic rate of convergence of the method is independent of the time horizon $T$. The relative errors are plotted in Figure \ref{Fig-3-ref}. We observe in particular the decrease in the number of points in the curves from eight points to three points as $T$ increases.  We suspect that this is due to Algorithm \ref{alg:local_picard_iteration} requiring sufficiently large mesh parameter $k_{\dagger}$ to converge to the approximation $(X^{k,N_k},Y^{k,N_k})$ for each given $k\geq k_{\dagger}$ for the entire range of time horizons $T\in\{1,2,4,8,16,32\}$. For instance, in the experiments for $T=32$ we observed non-convergence of the algorithm for each $k<k_{\dagger}$ with $k_{\dagger}=6$. This size condition on $k$ for this experiment is perhaps not surprising, since there is a similar condition on $k$ in the statements of Theorem \ref{theorem-existence-num-scheme} and Theorem \ref{theorem-uniqueness-numerical-sheme} on well-posedness for the numerical scheme. We see that the asymptotic slope of each error plot is comparable to that of the base case $T=1$. Therefore, these results suggest that the numerical scheme \eqref{numerical-scheme} is robust as the time horizon increases.
\begin{figure}[h!]
	\centering
	\begin{tabular}{c c} 
		\begin{subfigure}[b]{0.45\textwidth}
			\begin{adjustbox}{width=0.95\linewidth}
				\begin{tikzpicture}
					\begin{loglogaxis}[
						title={$(Y(t))_{t}$ Approx.\ in $L^{\infty}(L^2)$-norm},
						xlabel={time-step $\tau_k$},
						ylabel={relative error},
						xmax=1,
						ymax=0.07,
						legend pos=north west,
						ymajorgrids=true,
						grid style=dashed,
						]
						
						\addplot[
						color=blue,
						mark=square,]
						coordinates {
                            (0.500000, 0.028367)
(0.250000, 0.014598)
(0.125000, 0.007382)
(0.062500, 0.003699)
(0.031250, 0.001841)
(0.015625, 0.000907)
(0.007812, 0.000439)
(0.003906, 0.000205)
						};
						
						\addplot[
						color=red,
						mark=triangle,]
						coordinates {
                            (0.500000, 0.039574)
(0.250000, 0.020530)
(0.125000, 0.010429)
(0.062500, 0.005226)
(0.031250, 0.002586)
(0.015625, 0.001255)
(0.007812, 0.000587)
						};
						
						\addplot[
						color=magenta,
						mark=star,]
						coordinates {
                            (0.500000, 0.039355)
(0.250000, 0.020288)
(0.125000, 0.010242)
(0.062500, 0.005085)
(0.031250, 0.002471)
(0.015625, 0.001156)
						};
						
							\addplot[
						color=brown,
						mark=diamond,]
						coordinates {
							(0.500000, 0.030338)
(0.250000, 0.015582)
(0.125000, 0.007809)
(0.062500, 0.003814)
(0.031250, 0.001789)
						};
						
						\addplot[
						color=teal,
						mark=oplus,]
						coordinates {
							(0.500000, 0.021744)
(0.250000, 0.011143)
(0.125000, 0.005511)
(0.062500, 0.002601)
						};

                        \addplot[
						color=black,
						mark=o,]
						coordinates {
							(0.500000, 0.016200)
(0.250000, 0.008505)
(0.125000, 0.004065)
						};
                        
						\logLogSlopeTriangle{0.9}{0.25}{0.5}{1}{black};
						
						\legend{$T=1$,$T=2$, $T=4$, $T=8$, $T=16$, $T=32$}
					\end{loglogaxis}  
				\end{tikzpicture}
			\end{adjustbox}
		\end{subfigure}
		&
		\begin{subfigure}[b]{0.45\textwidth}
			\begin{adjustbox}{width=0.95\linewidth} 
				\begin{tikzpicture}
					\begin{loglogaxis}[
						title={$(X(t))_t$ Approx. in $L^{\infty}(L^2)$-norm},
						xlabel={time step $\tau_k$},
						ylabel={relative error},
						xmax=1,
						ymax= 0.2,
						legend pos=north west,
						ymajorgrids=true,
						grid style=dashed,
						]
						
						\addplot[
						color=blue,
						mark=square,]
						coordinates {
                            (0.500000, 0.016713)
(0.250000, 0.006218)
(0.125000, 0.002895)
(0.062500, 0.001464)
(0.031250, 0.000732)
(0.015625, 0.000361)
(0.007812, 0.000175)
(0.003906, 0.000082)
						};
						
						\addplot[
						color=red,
						mark=triangle,]
						coordinates {
                            (0.500000, 0.025040)
(0.250000, 0.011501)
(0.125000, 0.005519)
(0.062500, 0.002661)
(0.031250, 0.001292)
(0.015625, 0.000621)
(0.007812, 0.000289)
						};
						
						\addplot[
						color=magenta,
						mark=star,]
						coordinates {
                            (0.500000, 0.045526)
(0.250000, 0.022677)
(0.125000, 0.011061)
(0.062500, 0.005418)
(0.031250, 0.002613)
(0.015625, 0.001217)

						};
						
						\addplot[
						color=brown,
						mark=diamond,]
						coordinates {
							(0.500000, 0.072242)
(0.250000, 0.036383)
(0.125000, 0.017820)
(0.062500, 0.008626)
(0.031250, 0.004024)
						};
						
							\addplot[
						color=teal,
						mark=oplus,]
						coordinates {
							(0.500000, 0.095711)
(0.250000, 0.048368)
(0.125000, 0.023323)
(0.062500, 0.010934)
						};

                        \addplot[
						color=black,
						mark=o,]
						coordinates {
							(0.500000, 0.104144)
(0.250000, 0.052484)
(0.125000, 0.024410)
						};

						\logLogSlopeTriangle{0.7}{0.25}{0.3}{1}{black};
						
						\legend{$T=1$,$T=2$, $T=4$, $T=8$, $T=16$, $T=32$}
					\end{loglogaxis}  
				\end{tikzpicture}
			\end{adjustbox}
		\end{subfigure}
	\end{tabular}
	\caption{First experiment: $L^{\infty}(L^2)$-Error plots for approximations of the reference flow  $(Y(t))_{t\in [0,T]}$ and the reference trajectories of the players  $(X(t))_{t\in [0,T]}$ for $T\in \{1,2,4,8,16,32\}$ using Algorithm \ref{alg:local_picard_iteration} for mesh counts $k\in \{1,2,3,4,5,6,7,8\}$}
	\label{Fig-3-ref}
\end{figure}
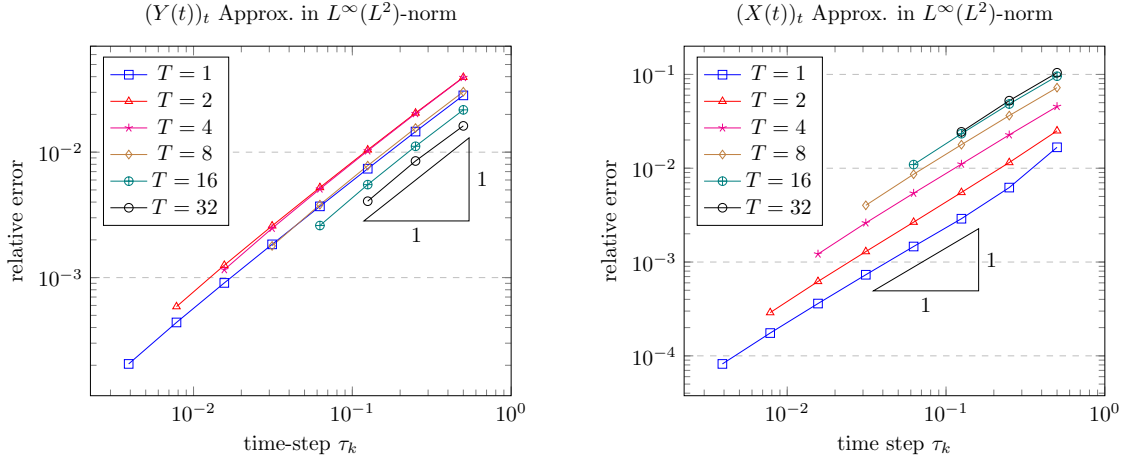

\subsubsection*{Second experiment.}
In this experiment we test the performance of the numerical scheme \eqref{numerical-scheme} for a MFG system with an exact solution.
We consider the continuous MFG system \eqref{mfg-pde-sys} with the following model data. We let the space dimension $d=2$, the time horizon $T\in (0,\infty)$ and we consider the initial measure $\rho_0 \in L^2(\mathbb{R}^2)$ with $\rho_0(x)\coloneqq \chi_{[0,1]^2}(x)$, $x\in\mathbb{R}^2$. Note that, the second moment of $\rho_0$ is $M_2(\rho_0)=\left(\int_{\mathbb{R}^2}|x|^2 {\rho}_0(x)dx\right)^{\frac{1}{2}}=\sqrt{\frac{2}{3}}$. We set the Hamiltonian  $H(x,p,\mu)\coloneqq \frac{1}{2}|p|^2$,  $(x,p,\mu)\in  \mathbb{R}^2\times\mathbb{R}^2\times\mathcal{P}_2(\mathbb{R}^2)$ and the terminal cost to be $g(x,\mu)\coloneqq \frac{1}{4}|x|^2M_2^2(\mu)$, $(x,\mu)\in \mathbb{R}^2\times\mathcal{P}_2(\mathbb{R}^2)$. 

It is clear that $H$ satisfies the regularity assumptions \eqref{ass-H:1}, \eqref{ass-H:2}, \eqref{ass-H:3}, \eqref{ass-H:4} and the displacement monotonicity condition \eqref{ass-H:disp-mono-cond-H}. Furthermore,  $g$ satisfies \eqref{ass-g:1} and the displacement monotonicity condition \eqref{ass-g:disp-mono-cond-g} {(to remain self-contained in our analysis -- as this is not a consequence of the results in Example \ref{eg-1} -- we supply the proof of this fact in Lemma \ref{lem:app-mon})}. However, the derivative $D_xg$ is only locally Lipschitz continuous { in the measure variable}, and so does not satisfy the uniform Lipschitz continuity assumption \eqref{ass-g:2}. Nonetheless, we now show by construction that the MFG system \eqref{mfg-pde-sys} has a solution $(u,\rho)$, whence the solution is unique by arguments in \cite[Corollary 4.3]{meszaros2024mean} due to displacement monotonicity of the data. Suppose we are given a distribution curve $\tilde{\rho}\in C([0,T];\mathcal{P}_2(\mathbb{R}^d))$. Then the unique viscosity solution $\tilde{u}$ of the Hamilton--Jacobi equation of the system \eqref{mfg-pde-sys} for given distribution $\tilde{\rho}$ takes the form
\begin{equation}
	\tilde{u}(t,x)=\inf_{y\in\mathbb{R}^2}\left(\frac{1}{4}|y|^2M_2^2(\tilde{\rho}(T))+\frac{1}{2(T-t)}|x-y|^2\right)\quad\forall (t,x)\in [0,T]\times \mathbb{R}^2
\end{equation}  
due to the Hopf--Lax formula.
The infimum here is uniquely attained at 
\begin{equation}
	y_*(t,x)=\frac{2}{2+(T-t)M_2^2(\tilde{\rho}(T))}x,\quad \forall (t,x)\in [0,T]\times \mathbb{R}^2,
\end{equation}
which yields the value function ansatz
\begin{equation}\label{val-func-ansatz}
	\tilde{u}(t,x)= \frac{|x|^2M_2^2(\tilde{\rho}(T))}{2(2+(T-t)M_2^2(\tilde{\rho}(T)))}, \quad  \forall (t,x)\in [0,T]\times \mathbb{R}^2.
\end{equation}
We then have that 
\begin{equation}
	D_x\tilde{u}(t,x)= \frac{M_2^2(\tilde{\rho}(T))}{2+(T-t)M_2^2(\tilde{\rho}(T))}x, \quad  \forall (t,x)\in [0,T]\times \mathbb{R}^2.
\end{equation}
This gradient results in the player distribution ansatz ${\rho}_*(t)\coloneqq (\mathcal{X}(t))_{\#}\rho_0$, $t\in [0,T]$, where the flow $(\mathcal{X}(t))_{t\in [0,T]}$ is uniquely given by the ODE
\begin{equation}
	\frac{d\mathcal{X}}{dt}(t,x) = - D_x\tilde{u}(t,\mathcal{X}(t)), \quad t\in [0,T], \quad \mathcal{X}(0)=x,\quad \forall x\in \text{spt}(\rho_0)=[0,1]^2.
\end{equation}
By a direct calculation we obtain
\begin{equation}
	\mathcal{X}(t,x)=\frac{2+(T-t)M_2^2(\tilde{\rho}(T))}{2+TM_2^2(\tilde{\rho}(T))}x, \quad  (t,x)\in [0,T]\times  [0,1]^2.
\end{equation}
Furthermore, using \cite[Lemma 3.6]{meszaros2024mean}, we have that $\rho_*\in C([0,T];\mathcal{P}_2(\mathbb{R}^2))$ is the unique distributional solution to the Cauchy problem $\partial_t\rho_* - D_x\cdot\left(\rho_* D_x\tilde{u}\right)=0 $ in $(0,T)\times \mathbb{R}^2$ with $\rho_*(0)=\rho_0$ in $\mathcal{P}_2(\mathbb{R}^2)$, where $\sup_{t\in [0,T]}M_2(\rho_*(t))<\infty$. Notice that the ansatz for both $\tilde{u}$ and $\rho_*$ depend on merely the value of $M_2(\tilde{\rho}(T))$. As such, we seek a fixed point relation between $M_2(\tilde{\rho}(T))$ and $M_2({\rho}_*(T))$ which fixes the value of the second-moment of $\rho_*$.

Using the definition of the push-forward, we find that
$M_2^2(\rho_*(T))=\int_{\mathbb{R}^2}|\mathcal{X}(T,x)|^2\mathrm{d}\rho_0(x)$ and hence
\begin{equation}
	M_2^2(\rho_*(T))=\int_{\mathbb{R}^2}|\mathcal{X}(T,x)|^2\mathrm{d}\rho_0(x)=\frac{4}{(2+TM_2^2(\tilde{\rho}(T)))^2}M_2^2(\rho_0)=\frac{8}{3(2+TM_2^2(\tilde{\rho}(T)))^2}.
\end{equation}
Therefore, we arrive at the fixed-point equation
\begin{equation}\label{second-moment-eqn}
	r(T)^2=\frac{8}{3(2+Tr(T)^2)^2}
\end{equation} for the second moment of the player distribution. One can show directly that, for each $T\in (0,\infty)$,  this equation is satisfied by a unique value of $r(T)\in (0,\sqrt{2/3})$. In addition, it can be shown from \eqref{second-moment-eqn} that 
\begin{equation}\label{second-moment-formula-eg}
	r^2(T)=\frac{2^{\frac{2}{3}}}{3T}\left(9T+3\sqrt{9T^2+4T}+2\right)^{\frac{1}{3}}+\frac{2^{\frac{4}{3}}}{T\left(9T+3\sqrt{9T^2+4T}+2\right)^{\frac{1}{3}}} - \frac{4}{3T}
\end{equation} for $T\in (0,\infty)$. In all, we have shown that the MFG has a solution $(u,\rho)$ where the player distribution is given by $\rho\coloneqq \rho_*$ with corresponding value function $u\coloneqq \tilde{u}$ given by \eqref{val-func-ansatz} with $M_2(\rho(T))=r(T)$ uniquely satisfying \eqref{second-moment-eqn}. The uniqueness of the solution is then a consequence of  \cite[Corollary 4.3]{meszaros2024mean}.

As before, we let $\Omega\coloneqq [0,1]^2$ and we consider the associated continuous Hamiltonian system 
\begin{subequations}\label{definition-soln-of-continuous-pb-eg}
	\begin{align}
		X(t) &=X_0+ \int_0^tY(s)\mathrm{d}s, \qquad&&\text{in }C([0,T];L^2(\Omega;\mathbb{R}^2))
		\\
		Y(t) &= - \frac{M_2^2(\rho(T))}{2}X(T), \qquad&&\text{in }C([0,T];L^2(\Omega;\mathbb{R}^2)) \\
		\rho(t) &= (X(t))_{\#}\mathbb{P} \qquad&&\text{in }C([0,T];\mathcal{P}_2(\mathbb{R}^2)).
	\end{align}
\end{subequations}
where we let $X_0\in L^2(\Omega;\mathbb{R}^2)$ be the identity map $X_0(\omega)\coloneqq \omega$, $\omega\in 
\Omega$, whose range is the support of $\rho_0$. Note that $\mathbb{P}$ is the restriction of $\rho_0$ to $[0,1]^2$.  It follows by a direct calculation that the unique solution of this Hamiltonian system is given by 
\begin{equation}
	X(t)=\left(\frac{2+(T-t)M_2^2(\rho(T))}{2+TM_2^2(\rho(T))}\right)X_0, \quad Y(t)=-\frac{M_2^2(\rho(T))}{2+TM_2^2(\rho(T))}X_0 \text{ in }C^{\infty}([0,T];H^1(\Omega;\mathbb{R}^d))
\end{equation}
where 	$M_2(\rho(T))\in (0,\sqrt{2/3})$ uniquely satisfies \eqref{second-moment-eqn} and is given by \eqref{second-moment-formula-eg}. Note that $X(t)=\mathcal{X}(t,X_0)$ and $Y(t)=-D_xu(t,X(t))$ in $C^{\infty}([0,T];H^1(\Omega;\mathbb{R}^d))$. 


Let us formulate the discrete Hamiltonian system that we aim to solve. We begin by setting the discretization parameters. For each $N\in\mathbb{N}$, we let  $\mathcal{O}_N\coloneqq \{\Omega_{i,N}\}_{i=1}^N\subset \Omega$ denote a uniform partition of $\Omega$ into (open) squares each of area $\mathbb{P}(\overline{\Omega_{i,N}})=N^{-1}$, and we let  $X_0^N\coloneqq \sum_{i=1}^Nx_{i,N}\chi_{\Omega_{i,N}}$ where $x_{i,N}$ denotes the centroid of $\Omega_{i,N}$. We set $N=N_k=2^{2(k-1)}$ for each $k\in\mathbb{N}$, and we consider a uniform partition of the time interval $[0,T]$ with time-step $\tau_k\coloneqq 2^{-k+1}T$ with $k\in\mathbb{N}$. In this case, we have that $\tau_k\eqsim N_k^{-\frac{1}{2}}$, and it can be shown from the definition of $X_0$ and $X_0^{N_k}$ that $\mathcal{W}_2(\rho_0,\rho_0^{N_k})\leq \|X_0-X_0^{N_k}\|_{L^2(\Omega;\mathbb{R}^2)}\lesssim N_k^{-\frac{1}{2}} $ for $k\in\mathbb{N}$.  
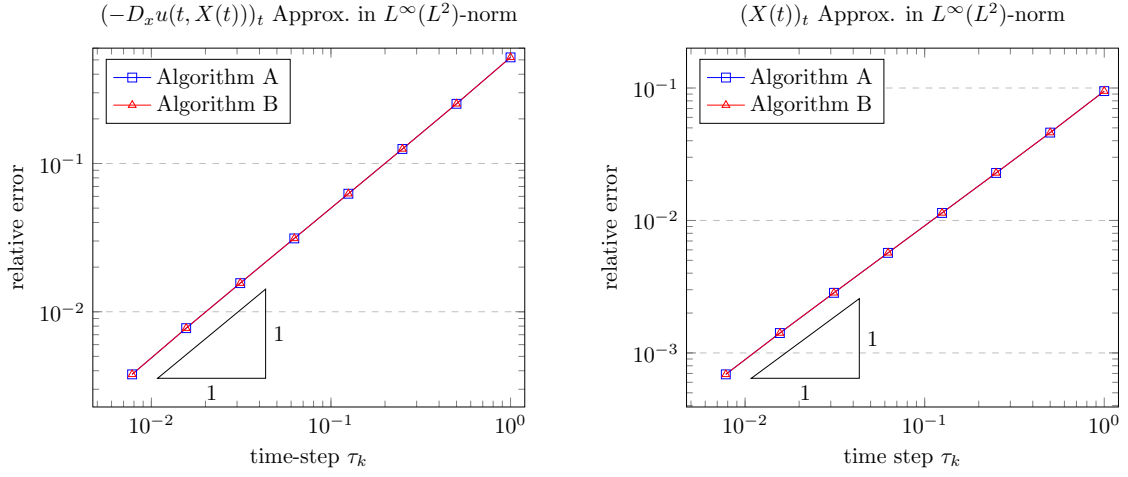
\begin{figure}[h!]
	\centering
	\begin{tabular}{c c} 
		\begin{subfigure}[b]{0.45\textwidth}
			\begin{adjustbox}{width=0.95\linewidth}
				\begin{tikzpicture}
					\begin{loglogaxis}[
						title={$(-D_xu(t,X(t)))_{t}$ Approx.\ in $L^{\infty}(L^2)$-norm},
						xlabel={time-step $\tau_k$},
						ylabel={relative error},
						xmax=1.2,
						ymax=0.6,
						legend pos=north west,
						ymajorgrids=true,
						grid style=dashed,
						]
						
						\addplot[
						color=blue,
						mark=square,]
						coordinates {
						(1.000000, 0.519889)
						(0.500000, 0.252732)
						(0.250000, 0.125345)
						(0.125000, 0.062536)
						(0.062500, 0.031240)
						(0.031250, 0.015595)
						(0.015625, 0.007751)
						(0.007812, 0.003782)
						};
						
							\addplot[
						color=red,
						mark=triangle,]
						coordinates {
							(1.000000, 0.519889)
							(0.500000, 0.252732)
							(0.250000, 0.125345)
							(0.125000, 0.062536)
							(0.062500, 0.031240)
							(0.031250, 0.015595)
							(0.015625, 0.007751)
							(0.007812, 0.003782)
						};
						
						\logLogSlopeTriangle{0.4}{0.25}{0.08}{1}{black};
						
						\legend{Algorithm A,Algorithm B}
					\end{loglogaxis}  
				\end{tikzpicture}
			\end{adjustbox}
		\end{subfigure}
		&
		\begin{subfigure}[b]{0.45\textwidth}
			\begin{adjustbox}{width=0.95\linewidth} 
				\begin{tikzpicture}
					\begin{loglogaxis}[
						title={$(X(t))_t$ Approx. in $L^{\infty}(L^2)$-norm},
						xlabel={time step $\tau_k$},
						ylabel={relative error},
						xmax=1.2,
						ymax= 0.2,
						legend pos=north west,
						ymajorgrids=true,
						grid style=dashed,
						]
						
						\addplot[
						color=blue,
						mark=square,]
						coordinates {
						(1.000000, 0.094759)
						(0.500000, 0.046065)
						(0.250000, 0.022846)
						(0.125000, 0.011398)
						(0.062500, 0.005694)
						(0.031250, 0.002843)
						(0.015625, 0.001413)
						(0.007812, 0.000689)
						};
						
						\addplot[
						color=red,
						mark=triangle,]
						coordinates {
						(1.000000, 0.094759)
						(0.500000, 0.046065)
						(0.250000, 0.022846)
						(0.125000, 0.011398)
						(0.062500, 0.005694)
						(0.031250, 0.002843)
						(0.015625, 0.001413)
						(0.007812, 0.000689)
						};

						\logLogSlopeTriangle{0.4}{0.25}{0.08}{1}{black};
						
						\legend{Algorithm A,Algorithm B}
					\end{loglogaxis}  
				\end{tikzpicture}
			\end{adjustbox}
		\end{subfigure}
	\end{tabular}
	\caption{Second experiment: $L^{\infty}(L^2)$-Error plots for approximations of the exact flow  $(-D_xu(t,X(t)))_{t\in [0,T]}$ and the exact trajectories of the players  $(X(t))_{t\in [0,T]}$ for $T=1$ using Algorithm \ref{alg:standard_picard_iteration} and Algorithm \ref{alg:local_picard_iteration}.}
	\label{Fig-2}
\end{figure}

In this experiment we first test the approximation of the player distribution $\rho$ and the value function gradient along the player trajectories, which is given by the flow $(-D_xu(t,X(t)))_{t\in [0,T]}$ for the case $T=1$ to compare the performance of Algorithm \ref{alg:standard_picard_iteration} and Algorithm \ref{alg:local_picard_iteration}. For $k\in \{1,2,\cdots, 8\}$, we apply both the standard Algorithm \ref{alg:standard_picard_iteration} and the local Algorithm \ref{alg:local_picard_iteration} in MATLAB to solve the discrete system \eqref{numerical-scheme}, where for each algorithm we initialised by setting $Y_{init}^{k,N_k}\coloneqq 0$ in $ \mathbb{V}_k^-([0,T];\mathcal{R}_{2}(\mathcal{A}_{N_k},\mathcal{O}_{N_k}))$ and $X_{init}^{k,N_k}\in  \mathbb{V}_k^-([0,T];\mathcal{R}_{2}(\mathcal{A}_{N_k},\mathcal{O}_{N_k}))$ where $X_{init}^{k,N_k}|_{(0,T]}\coloneqq 0$ with $X_{init}^{k,N_k}(0)\coloneqq X_0^{N_k}$. For both algorithms, we set the global tolerance to be $tol_{global}=10^{-8}$ with the local error tolerance for Algorithm \ref{alg:local_picard_iteration} set as $tol_{local}=10^{-12}$. For solving the associated discrete McKean--Vlasov equations using Algorithm \ref{alg:MKV_iteration} in the respective implementations, throughout we initialised using $X_{init}^{k,N_k}\in  \mathbb{V}_k^-([0,T];\mathcal{R}_{2}(\mathcal{A}_{N_k},\mathcal{O}_{N_k}))$ where $X_{init}^{k,N_k}|_{(0,T]}\coloneqq 0$ with $X_{init}^{k,N_k}(0)\coloneqq \chi$ where $\chi$ denotes a generic initial condition that is determined by the given scheme (i.e. being fixed for Algorithm \ref{alg:standard_picard_iteration} but varying in Algorithm \ref{alg:local_picard_iteration}), and we set the error tolerance to $tol=10^{-12}$. 

The error plots given by the numerical scheme \eqref{numerical-scheme} are shown in Figure \ref{Fig-2}. In this setting where the model data do not satisfy all the hypotheses of Theorem \ref{theorem-convergence-k-N-joint}, we observe an optimal asymptotic rate of convergence of order one as the time-step vanishes in the error plots. The optimality of this rate of convergence is likely due to the fact that the exact solution of the continuous Hamiltonian system are $H^1$-regular in $(0,T)\times \Omega$. Here too we see that the respective error plots given by Algorithm \ref{alg:standard_picard_iteration} and Algorithm \ref{alg:local_picard_iteration} overlap well with each other.

\begin{figure}[h!]
	\centering
	\begin{tabular}{c c} 
		\begin{subfigure}[b]{0.45\textwidth}
			\begin{adjustbox}{width=0.95\linewidth}
				\begin{tikzpicture}
					\begin{loglogaxis}[
						title={$(-D_xu(t,X(t)))_{t}$ Approx.\ in $L^{\infty}(L^2)$-norm},
						xlabel={time-step $\tau_k$},
						ylabel={relative error},
						xmax=4,
						ymax=0.6,
						legend pos=north west,
						ymajorgrids=true,
						grid style=dashed,
						]
						
						\addplot[
						color=blue,
						mark=square,]
						coordinates {
                            (1, 0.519889146605252)   (0.5, 0.252732338851463)
							(0.250000, 0.1253447863)
							(0.125000, 0.06253617577)
							(0.062500, 0.03124022714)
							(0.031250, 0.01559514628)
							(0.015625, 0.007751319244)
							(0.007812, 0.003782222816)
						};
						
						\addplot[
						color=red,
						mark=triangle,]
						coordinates {
                            (2, 0.513502210001673)   
                            (1, 0.251778683965082)
							(0.500000, 0.125221077)
							(0.250000, 0.0625205743)
							(0.125000, 0.03123827619)
							(0.062500, 0.01559489894)
							(0.031250, 0.007751285236)
							(0.015625, 0.003782246547)
						};
						
						\addplot[
						color=magenta,
						mark=star,]
						coordinates {
                            (2, 0.251099545116972)
							(1.000000, 0.1251342935)
							(0.500000, 0.06250966791)
							(0.250000, 0.03123691229)
							(0.125000, 0.01559472974)
							(0.062500, 0.007751262136)
							(0.031250, 0.003784015615)
						};
						
							\addplot[
						color=brown,
						mark=diamond,]
						coordinates {
							(2.000000, 0.1250790281)
							(1.000000, 0.06250274211)
							(0.500000, 0.03123604465)
							(0.250000, 0.01559462129)
							(0.125000, 0.007751248933)
							(0.062500, 0.003782214283)
						};
						
						\addplot[
						color=teal,
						mark=oplus,]
						coordinates {
							(2.000000, 0.06249855304)
(1.000000, 0.03123552157)
(0.500000, 0.01559455616)
(0.250000, 0.007751240111)
(0.125000, 0.003782213957)
						};

                        \addplot[
						color=black,
						mark=o,]
						coordinates {
							(2.000000, 0.0312352109)
(1.000000, 0.01559451731)
(0.500000, 0.007751235594)
(0.250000, 0.003782213069)
						};
                        
						\logLogSlopeTriangle{0.9}{0.25}{0.08}{1}{black};
						\legend{$T=1$,$T=2$, $T=4$, $T=8$, $T=16$, $T=32$}
					\end{loglogaxis}  
				\end{tikzpicture}
			\end{adjustbox}
		\end{subfigure}
		&
		\begin{subfigure}[b]{0.45\textwidth}
			\begin{adjustbox}{width=0.95\linewidth} 
				\begin{tikzpicture}
					\begin{loglogaxis}[
						title={$(X(t))_t$ Approx. in $L^{\infty}(L^2)$-norm},
						xlabel={time step $\tau_k$},
						ylabel={relative error},
						xmax=4,
						ymax= 0.15,
						legend pos=north west,
						ymajorgrids=true,
						grid style=dashed,
						]
						
						\addplot[
						color=blue,
						mark=square,]
						coordinates {
                            (1, 0.094759324516241)   (0.5, 0.046065101001758)
							(0.250000, 0.02284638353)
							(0.125000, 0.01139836442)
							(0.062500, 0.005694106496)
							(0.031250, 0.002842501199)
							(0.015625, 0.001412817886)
							(0.007812, 0.0006893795782)
						};
						
						\addplot[
						color=red,
						mark=triangle,]
						coordinates {
                            (2, 0.136006575820415)   
                            (1, 0.066686288551197)
							(0.500000, 0.03316614677)
							(0.250000, 0.01655924514)
							(0.125000, 0.008273793597)
							(0.062500, 0.004130477712)
							(0.031250, 0.002053010391)
							(0.015625, 0.001001759523)
						};
						
						\addplot[
						color=magenta,
						mark=star,]
						coordinates {
                            (2, 0.089402769083495)
							(1.000000, 0.04455345589)
							(0.500000, 0.02225626247)
							(0.250000, 0.01112175014)
							(0.125000, 0.005552427844)
							(0.062500, 0.00275979915)
							(0.031250, 0.001346936545)
						};
						
						\addplot[
						color=brown,
						mark=diamond,]
						coordinates {
							(2.000000, 0.05605448498)
							(1.000000, 0.02801076255)
							(0.500000, 0.01399851331)
							(0.250000, 0.006988769558)
							(0.125000, 0.003473741714)
							(0.062500, 0.001695009228)
						};
						
							\addplot[
						color=teal,
						mark=oplus,]
						coordinates {
							(2.000000, 0.03345408388)
(1.000000, 0.01671967911)
(0.500000, 0.00834741873)
(0.250000, 0.004149066389)
(0.125000, 0.002024534702)
						};

                        \addplot[
						color=black,
						mark=o,]
						coordinates {
							(2.000000, 0.01917552938)
(1.000000, 0.009573590745)
(0.500000, 0.004758541338)
(0.250000, 0.002321928828)
						};

						\logLogSlopeTriangle{0.9}{0.25}{0.3}{1}{black};
						
						\legend{$T=1$,$T=2$, $T=4$, $T=8$, $T=16$, $T=32$}
					\end{loglogaxis}  
				\end{tikzpicture}
			\end{adjustbox}
		\end{subfigure}
	\end{tabular}
	\caption{Second experiment: $L^{\infty}(L^2)$-Error plots for approximations of the exact flow  $(-D_xu(t,X(t)))_{t\in [0,T]}$ and the exact trajectories of the players  $(X(t))_{t\in [0,T]}$ for $T\in \{1,2,4,8,16,32\}$ using Algorithm \ref{alg:local_picard_iteration} for mesh counts $k\in \{1,2,3,4,5,6,7,8\}$}
	\label{Fig-3}
\end{figure}

In the second part of this experiment, we assess the robustness of the numerical scheme \eqref{numerical-scheme} for longer time horizons $T\in\{2,4,8,16,32\}$ by using Algorithm \ref{alg:local_picard_iteration} based on local Picard iterations. The relative errors are plotted in Figure \ref{Fig-3}. Like in the first experiment, we observe the decrease in the number of points in the curves as $T$ increases, namely from eight points to four points in this experiment.  In conducting this experiment, we observed convergence of Algorithm \ref{alg:local_picard_iteration} for each mesh level $k\geq k_{\dagger}$ where $k_{\dagger}=k_{\dagger}(T)$ increases from $1$ to $5$ as $T$ increases in the range $T\in\{1,2,4,8,16,32\}$, but otherwise we observed non-convergence of the algorithm for each mesh level $k<k_{\dagger}(T)$. We see that the slope of each error plot is comparable to that of the base case $T=1$, thereby showcasing the robustness of the numerical scheme \eqref{numerical-scheme} as the time horizon increases.

\subsubsection*{Third experiment.}
In this final test, we illustrate the space dimension dependence in the asymptotic rate of convergence implied by Theorem \ref{theorem-convergence-k-N-joint} for an example model problem. The model data that we consider for the continuous MFG system are as follows. We fix the time horizon to be $T=0.1$ and we assume that $d\in \{1,2,3,4,5,6\}$. We let the initial distribution $\rho_0$ be the regular probability measure given by the standard normal distribution on $\mathbb{R}^d$ with zero mean and identity covariance, {i.e. $\rho_0 = \frac{1}{(2\pi)^{d/2}}e^{-|x|^2/2}$}. Note that $\rho_0$ is thus supported on all of $\mathbb{R}^d$. We set the Hamiltonian of the MFG system to be 
\begin{equation}
	H(x,p,\mu)\coloneqq \frac{7}{6}(|p|^2-|x|^2)+\left(\phi(p)-\phi(x)\right)\int_{\mathbb{R}^d}\frac{|z|^2}{1+|z|^2}\mathrm{d}\mu(y)
	\quad\forall (x,p,\mu)\in  \mathbb{R}^d\times\mathbb{R}^d\times\mathcal{P}_2(\mathbb{R}^d),
\end{equation}
where the function $\phi:\mathbb{R}^d\to [0,\infty)$ is defined by $\phi(v)\coloneqq \frac{1}{2}|v|^2 - \frac{1}{12}|v|^4$ if $|v|\leq 1$ and $\phi(v)\coloneqq \frac{2}{3}|v|-\frac{1}{4}$ if $|v|\geq 1$.
The terminal cost of the system is set as
\begin{equation}
	g(x,\mu)\coloneqq  \frac{1}{2}|x|^2+ {(x_1+\ldots+x_d)}\int_{\mathbb{R}^d}\left(1-\sin({y_1+\ldots+y_d})\right)\mathrm{d}\mu(y)
	\quad\forall (x,\mu)\in  \mathbb{R}^d\times\mathcal{P}_2(\mathbb{R}^d).
\end{equation}

From Example \ref{eg-1} we find that $H$ satisfies the regularity assumptions \eqref{ass-H:1}, \eqref{ass-H:2}, \eqref{ass-H:3}, \eqref{ass-H:4}, \eqref{ass-H:5} and the displacement monotonicity condition \eqref{ass-H:disp-mono-cond-H}, that the $g$ above satisfies the regularity assumptions \eqref{ass-g:1}, \eqref{ass-g:2} and the displacement monotonicity condition \eqref{ass-g:disp-mono-cond-g}, and that the growth condition \eqref{ass-L-g} on $g$ and the Lagrangian $L$ of $H$ also holds. In addition, using \cite[Theorem 3.7, Theorem 4.5]{meszaros2024mean}, the MFG system \eqref{mfg-pde-sys} has a unique solution in the sense of \cite[Definition 3.2]{meszaros2024mean}. For the associated continuous Hamiltonian system \eqref{characteristics-continuous}, we consider the reference probability space $(\Omega,\mathbb{F},\mathbb{P})$ where $\Omega=[0,1]^d$,  $\mathbb{F}$  is the $\sigma$-algebra of Lebesgue measurable subsets of $\Omega$, and $\mathbb{P}$ denotes the $d$-dimensional Lebesgue measure {restricted to $\Omega$}. We let the initial random variable $X_0$ be distributed according to $\rho_0$ with $X_0(\omega)\coloneqq (\Phi^{-1}(\omega_1),\cdots, \Phi^{-1}(\omega_i),\cdots, \Phi^{-1}(\omega_d))$ for almost every $\omega=(\omega_i)_{i=1}^d\in \Omega$ where $\Phi:\mathbb{R}\to [0,1]$ denotes the cumulative distribution function of the one-dimensional standard normal distribution. 

For the numerical scheme, we let $N_k\coloneqq 2^{k}$ and set $\tau_k=T/N_k$ for $k=1,2,3,4,5,6$. Given $k\in \{1,2,3,4,5,6,7\}$, the approximate random variable $X_0^{N_k}$ is computed in MATLAB as an approximate $N_k$-point optimal quantization of $X_0$ using Lloyd's algorithm on a Monte Carlo sample of size 1,000,000. We estimate the corresponding quantization weights $\mathcal{A}_{N_k}$ via computing empirical Voronoi cell frequencies. We do not have access to an explicit formula for the unique solution $(X,Y)$ of the continuous Hamiltonian system \eqref{characteristics-continuous}, so we compute the approximation errors for the numerical scheme \eqref{numerical-scheme} with respect to the reference solution obtained at $N_{ref}=2048$ spatial samples with time-step $\tau_{ref}\coloneqq T/N_{ref}$. 

For this experiment, we employ the standard Picard iteration via Algorithm \ref{alg:standard_picard_iteration} to approximate the discrete problem, where we initialised the iteration by setting $Y_{init}^{k,N_k}\coloneqq 0$ in $ \mathbb{V}_k^-([0,T];\mathcal{R}_{2}(\mathcal{A}_{N_k},\mathcal{O}_{N_k}))$ and $X_{init}^{k,N_k}\in  \mathbb{V}_k^-([0,T];\mathcal{R}_{2}(\mathcal{A}_{N_k},\mathcal{O}_{N_k}))$ with $X_{init}^{k,N_k}|_{(0,T]}\coloneqq 0$ and $X_{init}^{k,N_k}(0)\coloneqq X_0^{N_k}$. The iteration was run with the global tolerance set to $tol_{global}=10^{-8}$, while the inner McKean--Vlasov solves were run via Algorithm \ref{alg:MKV_iteration} with tolerance $tol=10^{-12}$. For the chosen $N_k,\tau_k$ and the fact that $X_0\in L^{2+\delta}(\Omega;\mathbb{R}^d)$ for every finite $\delta>0$, we have that for each $d\in \{1,2,3,4,5,6\}$ Corollary \ref{cor-2-eg-rate-of-conv} gives an error estimate for the scheme of order $N_k^{-1/d}+\tau_k\eqsim N_k^{-1/d}$ for all $k\in\mathbb{N}$ sufficiently large, with rate of convergence of order $1/d$. In Figure \ref{Fig-4} we display error plots for $d\in\{1,2,3,4,5,6\}$, where we observe asymptotic rates of convergence of order $1/d$, in agreement with the conclusion of Corollary \ref{cor-2-eg-rate-of-conv}.

\begin{figure}[h!]
	\centering
	\begin{tabular}{c c} 
		\begin{subfigure}[b]{0.45\textwidth}
			\begin{adjustbox}{width=0.95\linewidth}
				\begin{tikzpicture}
					\begin{loglogaxis}[
						title={$(-D_xu(t,X(t)))_{t}$ Approx.\ in $L^{\infty}(L^2)$-norm},
						xlabel={time-step $\tau_k$},
						ylabel={relative error},
						xmax=0.1,
						ymax=0.9,
						legend pos=south east,
						ymajorgrids=true,
						grid style=dashed,
						]
						
						\addplot[
						color=blue,
						mark=square,]
						coordinates {
                            (0.050000, 0.444215)
(0.025000, 0.253073)
(0.012500, 0.137224)
(0.006250, 0.072154)
(0.003125, 0.037046)
(0.001563, 0.018797)
(0.000781, 0.009435)
						};
						
						\addplot[
						color=red,
						mark=triangle,]
						coordinates {
                            (0.050000, 0.608539)
(0.025000, 0.430135)
(0.012500, 0.320036)
(0.006250, 0.233820)
(0.003125, 0.170222)
(0.001563, 0.122671)
(0.000781, 0.087686)
						};
						
						\addplot[
						color=magenta,
						mark=star,]
						coordinates {
                            (0.050000, 0.674686)
(0.025000, 0.525727)
(0.012500, 0.424078)
(0.006250, 0.343183)
(0.003125, 0.277976)
(0.001563, 0.223678)
(0.000781, 0.179823)
						};
                        
						\logLogSlopeTriangle{0.35}{0.25}{0.08}{1}{black};
						\logLogSlopeTrianglei{0.34}{0.25}{0.52}{0.5}{black};
                        \logLogSlopeTriangleii{0.3}{0.2}{0.67}{0.333333333333333333333}{black};
                        
						\legend{$d=1$,$d=2$, $d=3$}
					\end{loglogaxis}  
				\end{tikzpicture}
			\end{adjustbox}
		\end{subfigure}
		&
		\begin{subfigure}[b]{0.45\textwidth}
			\begin{adjustbox}{width=0.95\linewidth} 
				\begin{tikzpicture}
					\begin{loglogaxis}[
						title={$(X(t))_t$ Approx. in $L^{\infty}(L^2)$-norm},
						xlabel={time step $\tau_k$},
						ylabel={relative error},
						xmax=0.1,
						ymax= 0.25,
						legend pos=south east,
						ymajorgrids=true,
						grid style=dashed,
						]
						
						\addplot[
						color=blue,
						mark=square,]
						coordinates {
                            (0.050000, 0.131229)
(0.025000, 0.074134)
(0.012500, 0.040082)
(0.006250, 0.021071)
(0.003125, 0.010815)
(0.001563, 0.005489)
(0.000781, 0.002756)
						};
						
						\addplot[
						color=red,
						mark=triangle,]
						coordinates {
                            (0.050000, 0.185431)
(0.025000, 0.132368)
(0.012500, 0.098185)
(0.006250, 0.071599)
(0.003125, 0.052066)
(0.001563, 0.037554)
(0.000781, 0.026835)
						};
						
						\addplot[
						color=magenta,
						mark=star,]
						coordinates {
                           (0.050000, 0.209600)
(0.025000, 0.163846)
(0.012500, 0.132119)
(0.006250, 0.106530)
(0.003125, 0.086235)
(0.001563, 0.069369)
(0.000781, 0.055709)
						};
						
						\logLogSlopeTriangle{0.35}{0.25}{0.08}{1}{black};
						\logLogSlopeTrianglei{0.34}{0.25}{0.53}{0.5}{black};
                        \logLogSlopeTriangleii{0.3}{0.2}{0.69}{0.333333333333333333333}{black};

						\legend{$d=1$,$d=2$, $d=3$}
					\end{loglogaxis}  
				\end{tikzpicture}
			\end{adjustbox}
		\end{subfigure}\\
        \begin{subfigure}[b]{0.45\textwidth}
			\begin{adjustbox}{width=0.95\linewidth}
				\begin{tikzpicture}
					\begin{loglogaxis}[
						title={$(-D_xu(t,X(t)))_{t}$ Approx.\ in $L^{\infty}(L^2)$-norm},
						xlabel={time-step $\tau_k$},
						ylabel={relative error},
						xmax=0.1,
						ymax=0.9,
						legend pos=south east,
						ymajorgrids=true,
						grid style=dashed,
						]

							\addplot[
						color=brown,
						mark=diamond,]
						coordinates {
							(0.050000, 0.761930)
(0.025000, 0.616840)
(0.012500, 0.486642)
(0.006250, 0.417782)
(0.003125, 0.354867)
(0.001563, 0.302037)
(0.000781, 0.256550)
						};
						
						\addplot[
						color=teal,
						mark=oplus,]
						coordinates {
							(0.050000, 0.760340)
(0.025000, 0.664528)
(0.012500, 0.577215)
(0.006250, 0.468331)
(0.003125, 0.410395)
(0.001563, 0.361442)
(0.000781, 0.316969)
						};

                        \addplot[
						color=black,
						mark=o,]
						coordinates {
							(0.050000, 0.774407)
(0.025000, 0.690177)
(0.012500, 0.592159)
(0.006250, 0.537847)
(0.003125, 0.457434)
(0.001563, 0.409330)
(0.000781, 0.365862)
						};
                        
                        \logLogSlopeTriangleiii{0.91}{0.125}{0.78}{0.25}{black};
                        \logLogSlopeTriangleiv{0.91}{0.1}{0.67}{0.2}{black};
                        \logLogSlopeTrianglev{0.91}{0.125}{0.56}{0.16666667}{black};
						\legend{$d=4$, $d=5$, $d=6$}
					\end{loglogaxis}  
				\end{tikzpicture}
			\end{adjustbox}
		\end{subfigure}
		&
		\begin{subfigure}[b]{0.45\textwidth}
			\begin{adjustbox}{width=0.95\linewidth} 
				\begin{tikzpicture}
					\begin{loglogaxis}[
						title={$(X(t))_t$ Approx. in $L^{\infty}(L^2)$-norm},
						xlabel={time step $\tau_k$},
						ylabel={relative error},
						xmax=0.1,
						ymax= 0.25,
						legend pos=south east,
						ymajorgrids=true,
						grid style=dashed,
						]

						\addplot[
						color=brown,
						mark=diamond,]
						coordinates {
							(0.050000, 0.238344)
(0.025000, 0.192560)
(0.012500, 0.151047)
(0.006250, 0.129682)
(0.003125, 0.109959)
(0.001563, 0.093557)
(0.000781, 0.079467)
						};
						
							\addplot[
						color=teal,
						mark=oplus,]
						coordinates {
							(0.050000, 0.235964)
(0.025000, 0.206852)
(0.012500, 0.179936)
(0.006250, 0.144467)
(0.003125, 0.126575)
(0.001563, 0.111470)
(0.000781, 0.097702)
						};

                        \addplot[
						color=black,
						mark=o,]
						coordinates {
							(0.050000, 0.239250)
(0.025000, 0.213810)
(0.012500, 0.182932)
(0.006250, 0.166861)
(0.003125, 0.140548)
(0.001563, 0.125569)
(0.000781, 0.112244)
						};
						
                        \logLogSlopeTriangleiii{0.91}{0.125}{0.765}{0.25}{black};
                        \logLogSlopeTriangleiv{0.91}{0.105}{0.65}{0.2}{black};
                        \logLogSlopeTrianglev{0.91}{0.125}{0.54}{0.16666667}{black};
						
						\legend{$d=4$, $d=5$, $d=6$}
					\end{loglogaxis}  
				\end{tikzpicture}
			\end{adjustbox}
		\end{subfigure}
	\end{tabular}
	\caption{Third experiment: $L^{\infty}(L^2)$-Error plots for approximations of the exact flow  $(-D_xu(t,X(t)))_{t\in [0,T]}$ and the exact trajectories of the players  $(X(t))_{t\in [0,T]}$ for $d\in \{1,2,3,4,5,6\}$ using Algorithm \ref{alg:standard_picard_iteration} for mesh counts $k\in \{1,2,3,4,5,6,7\}$}
	\label{Fig-4}
\end{figure}
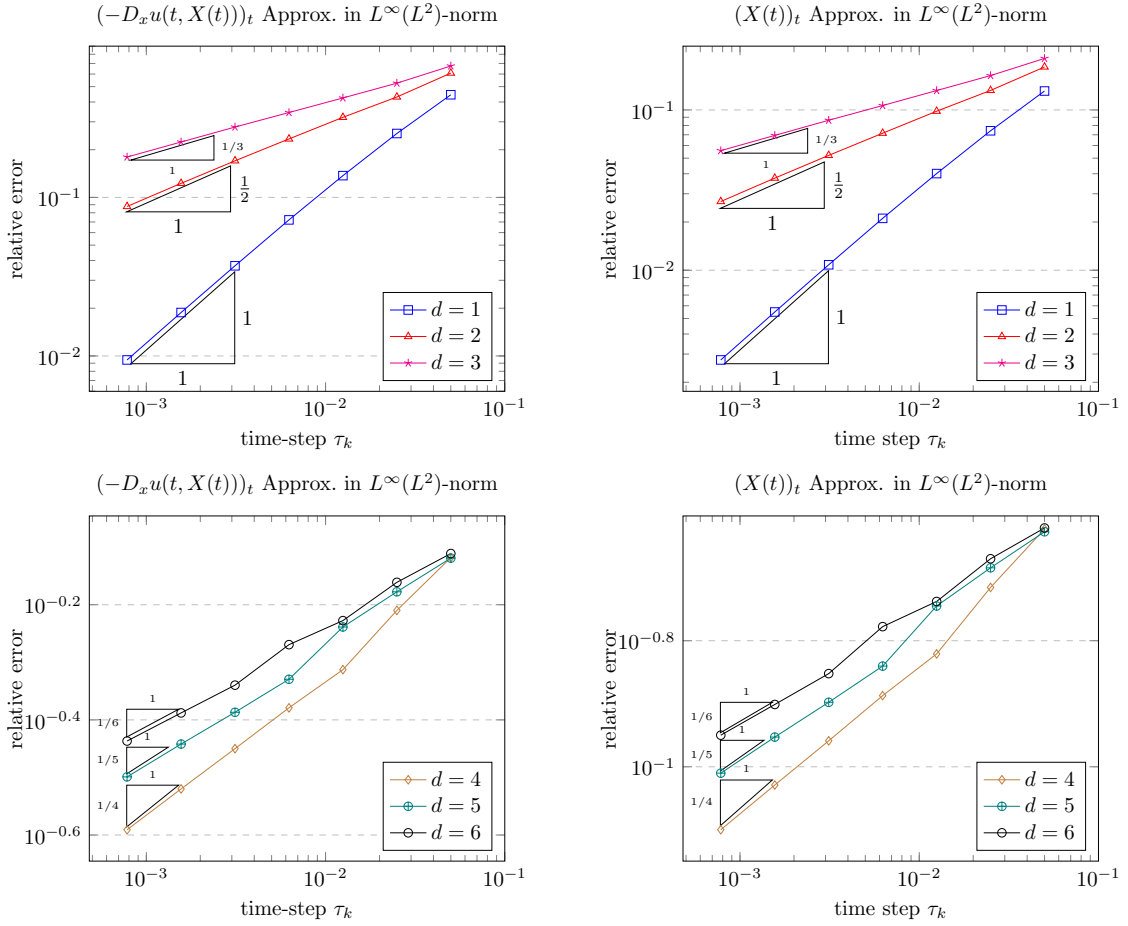

\section*{Acknowledgments.}
The work of the first author was supported by the Engineering and Physical Sciences Research Council New Investigator Award ``Mean Field Games and Master equations'' under award no. EP/X020320/1. The work of the second author was supported by The Royal Society Career Development Fellowship.

\appendix

\section{Proofs of auxiliary results}\label{app-aux-disc-temp-ana}

\subsection{Proof of Lemma \ref{lemma-ibp-time-formulae} on discrete integration-by-parts formulae}

	\begin{proof}
		We include only the proof of \eqref{eq:discrete_ibp-expectation_pn} as the identities \eqref{eq:discrete_ibp-expectation_pp} and \eqref{eq:discrete_ibp-expectation_nn} then follow analogously. Let $(V,W)\in \mathbb{V}_k^+([0,T];L^2(\Omega;\mathbb{R}^d))\times \mathbb{V}_k^-([0,T];L^2(\Omega;\mathbb{R}^d))$ be given. Without loss of generality assume that $n>m$ in $\{1,\cdots,M_k\}$ so that $t_n>t_m$. From definition of $\partial_s\Ip V$ and $\p_s \In W$ we have via their definitions \eqref{derivatives_Ip} that
		\begin{equation}
			\begin{split} 
				&\int_{t_{m}}^{t_{n}}\mathbb{E}\left[\partial_s\Ip V(s)\cdot W(s) + V(s) \cdot\p_s \In W(s)\right]\ds
				=\tau_k\sum_{i=m+1}^{n}\mathbb{E}\left[\partial_s\Ip V|_{I_i}\cdot W|_{I_i}+V|_{I_i} \cdot\p_s \In W|_{I_i}\right]\\
				&=\tau_k\sum_{i=m+1}^{n}\mathbb{E}\left[\left(-\frac{1}{\tau_k}\jump{V}_{i-1}\right)\cdot W|_{I_i}+V|_{I_i} \cdot\left(-\frac{1}{\tau_k} \jump{W}_i\right)\right]\\
                &=\sum_{i=m+1}^{n}\mathbb{E}\left[\left[V|_{I_i}-V|_{I_{i-1}}\right]\cdot W|_{I_i}+V|_{I_i} \cdot\left[W|_{I_{i+1}}-W|_{I_i}\right]\right]\\
				&=\sum_{i=m+1}^{n}\mathbb{E}\left[-V|_{I_{i-1}}\cdot W|_{I_i}+V|_{I_i} \cdot W|_{I_{i+1}}\right]
				=\mathbb{E}\left[V|_{I_{n}} \cdot W|_{I_{n+1}} - V|_{I_{m}}\cdot W|_{I_{m+1}}\right]
				\\&=\mathbb{E}\left[V(t_n)\cdot W(t_n) - V(t_m)\cdot W(t_m)\right]
			\end{split} 
		\end{equation}
		since $V$ is defined everywhere in time through left-continuity in time and $W$ is defined everywhere in time through right-continuity in time, by definition of $ \mathbb{V}_k^+([0,T];L^2(\Omega;\mathbb{R}^d))$ and $ \mathbb{V}_k^-([0,T];L^2(\Omega;\mathbb{R}^d))$, respectively. This proves 
		\begin{multline}\label{eq:discrete_ibp}
			\int_{t_m}^{t_n}\mathbb{E}\left[\partial_s\Ip V(s)\cdot W(s) + V(s) \cdot\p_s \In W(s)\right]\ds
			=\mathbb{E}\left[V(t_n)\cdot W(t_n)-V(t_m)\cdot W(t_m)\right]\quad\forall n,m\in \{0,1,\cdots, M_k\}
		\end{multline}
		for	all $(V,W)\in \mathbb{V}_k^+([0,T];L^2(\Omega;\mathbb{R}^d))\times \mathbb{V}_k^-([0,T];L^2(\Omega;\mathbb{R}^d))$. This completes the proof.
	\end{proof}

\subsection{Proof of Lemma \ref{lemma-discrete-gronwall} on discrete Gronwall inequalities}
\begin{proof}
	Let $k_{\dagger}^x$ be the smallest positive integer such that $1-c_{x}\tau_{k_{\dagger}^x}>0$ and fix $k\geq k_{\dagger}^x$. We then apply the discrete Gronwall inequality from \cite[Proposition 4.1]{Emmrich2000DiscreteVO} to obtain 
	\begin{equation}
		x_n\leq b_x(1-c_x\tau_k)^{-n}\quad\forall n\in\{0,1,2,\cdots, M_k\}.
	\end{equation}
	By taking $s= c_x\tau_k/(1-c_x\tau_k)$ in the standard estimate $\exp(s)\geq 1+s$, $s\geq 0$, we find that $(1-c_x\tau_k)^{-n}\leq \exp(c_xn\tau_k/(1-c_x\tau_k))\leq \exp(c_xT/(1-c_x\tau_{k_{\dagger}^x}))$ for all $n\in \{0,1,2,\cdots,M_k\}$ since $T=M_k\tau_k$ and $\tau_k\leq \tau_{k_{\dagger}^x}$ for all $k\geq k_{\dagger}^x$. We thus deduce \eqref{gronwall-uniform-x-bound}.
	
	The second bound is obtained via reversal of the index variable under the summation in the hypothesis \eqref{gronwall-sum-cond-y}. Indeed, fix $k\geq k_*$. We have from the hypothesis \eqref{gronwall-sum-cond-y} that
	\begin{equation}
		y_{M_k-n}\leq b_y +c_y\tau_k\sum_{m=M_k-n+1}^{M_k}y_{m-1}
		=b_y +c_y\tau_k\sum_{m=-n}^{-1}y_{M_k+m}=b_y +c_y\tau_k\sum_{m=1}^{n}y_{M_k-m}
	\end{equation} for all $n\in \{0,1,2,\cdots, M_k\}$. We let $z_n\coloneqq y_{M_k-n}$ for each $n\in \{0,1,2,\cdots, M_k\}$ to obtain that 
	\begin{equation}
		z_n\leq b_y +c_y\tau_k\sum_{m=1}^nz_m\quad\forall n\in \{0,1,2,\cdots, M_k\}.
	\end{equation} We then apply the discrete Gronwall inequality from \cite[Proposition 4.1]{Emmrich2000DiscreteVO} as in the previous case to find that, for each $k\geq k_{\dagger}^y$,
	\begin{equation}
		\max_{0\leq n \leq M_k} y_{n}=\max_{0\leq n \leq M_k} z_{n}\leq b_y \exp\left(\frac{c_yT}{1-c_y\tau_{k_{\dagger}^y}}\right)
	\end{equation} where $k_{\dagger}^y$ is the smallest positive integer such that $1-c_{x}\tau_{k_{\dagger}^y}>0$ and $k_{\dagger}^y\geq k_*$.
\end{proof}

\subsection{Proof of equivalence of numerical schemes}

\begin{proof}[Proof of Proposition \ref{Prop-equiv-scheme}]
	If we have a solution to \eqref{numerical-scheme} then the definition of $\mathcal{R}_{d}(\mathcal{A}_N,\mathcal{O}_N)$ implies that we can write the simple random variables as $ X^{k,N}(t) =\sum_{i=1}^{N}X^{k,N}(t)|_{\Omega_{i,N}}\chi_{\Omega_{i,N}}$ and $ Y^{k,N}(t) =\sum_{i=1}^{N}Y^{k,N}(t)|_{\Omega_{i,N}}\chi_{\Omega_{i,N}}$ for each $t\in [0,T]$. Therefore, for each $i=1,\cdots, N$, define $\mathcal{X}_{k,N}^i(t)\coloneqq X^{k,N}(t)|_{\Omega_{i,N}}$ and $\mathcal{Y}_{k,N}^i(t)\coloneqq Y^{k,N}(t)|_{\Omega_{i,N}}$ for all $t\in [0,T]$, and note that $\rho^{k,N}(t) = (X^{k,N}(t))_{\#}\mathbb{P}
			=\sum_{i=1}^N\alpha_{i,N}\delta_{\mathcal{X}_{k,N}^i(t)}$
	for all $t\in [0,T]$.	Now let $n\in \{1,\cdots M_k\}$ be given. For each $t\in I_n$ \eqref{disc-HJ-eqn} and the definition of $\Ip$ imply that
	\begin{equation}
		{X^{k,N}(t)+\frac{t_n-t}{\tau_k}\jump{X^{k,N}(t)}_{n-1}}=X_0^{N}+ \int_0^tD_pH\left(X^{k,N}(s),Y^{k,N}(s),\rho^{k,N}(s)\right)\mathrm{d}s
	\end{equation}
	Writing  $ X^{k,N}(t)|_{I_n} =\sum_{i=1}^{N}\left(X^{k,N}(t)|_{\Omega_{i,N}}\right)|_{I_n}\chi_{\Omega_{i,N}}=\sum_{i=1}^{N}\mathfrak{X}_{\Omega_{i,N}}(t)|_{I_n}\chi_{\Omega_{i,N}}$ where we recall that $\mathfrak{X}_{\Omega_{i,N}}(t)\coloneqq  X^{k,N}(t)|_{\Omega_{i,N}}$ is constant-in-time on $I_n$ for $n\in\{1,\cdots,M_k\}$, we obtain for all $t\in I_n$
	\begin{multline}
		\sum_{i=1}^{N}\mathfrak{X}_{\Omega_{i,N}}|_{I_n}\chi_{\Omega_{i,N}}+\frac{t_n-t}{\tau_k}\sum_{i=1}^{N}\left(\mathfrak{X}_{\Omega_{i,N}}|_{I_{n-1}}-\mathfrak{X}_{\Omega_{i,N}}|_{I_n}\right)\chi_{\Omega_{i,N}}=X_0^{N}+ \int_0^tD_pH\left(X^{k,N}(s),Y^{k,N}(s),\rho^{k,N}(s)\right)\mathrm{d}s
	\end{multline}
	where we have used the definition of the temporal jumps given in \eqref{def-time-jump}. We differentiate the above w.r.t.\ $t\in I_n$ to obtain that
	\begin{equation}
		-\tau_k^{-1}\sum_{i=1}^{N}\left(\mathfrak{X}_{\Omega_{i,N}}|_{I_{n-1}}-\mathfrak{X}_{\Omega_{i,N}}|_{I_n}\right)\chi_{\Omega_{i,N}} = D_pH\left(X^{k,N}(t),Y^{k,N}(t),\rho^{k,N}(t)\right)\quad\forall t\in I_n
	\end{equation}
	which is the same as
	\begin{multline}
		-\tau_k^{-1}\sum_{i=1}^{N}\left(\mathfrak{X}_{\Omega_{i,N}}|_{I_{n-1}}-\mathfrak{X}_{\Omega_{i,N}}|_{I_n}\right)\chi_{\Omega_{i,N}} = D_pH\left(\sum_{i=1}^{N}\mathfrak{X}_{\Omega_{i,N}}|_{I_n}\chi_{\Omega_{i,N}},\sum_{i=1}^{N}\mathfrak{Y}_{\Omega_{i,N}}|_{I_n}\chi_{\Omega_{i,N}},\rho^{k,N}(t)\right)\quad\forall t\in I_n
	\end{multline}
	where we similarly let $ Y^{k,N}(t)|_{I_n}=\sum_{i=1}^{N}\mathfrak{Y}_{\Omega_{i,N}}|_{I_n}\chi_{\Omega_{i,N}}$ where  $\mathfrak{Y}_{\Omega_{i,N}}(t)\coloneqq Y^{k,N}(t)|_{\Omega_{i,N}}$ is constant-in-time on $I_n$ for $n\in\{1,\cdots,M_k\}$. Evaluating the above equation at a sample point $\omega\in \Omega_{i,N}$ with $i=1,2,\cdots, N$, we then obtain
	\begin{equation}\label{x-discr-eqn}
		-\tau_k^{-1}\left(\mathfrak{X}_{\Omega_{i,N}}|_{I_{n-1}}-\mathfrak{X}_{\Omega_{i,N}}|_{I_n}\right) = D_pH\left(\mathfrak{X}_{\Omega_{i,N}}|_{I_n},\mathfrak{Y}_{\Omega_{i,N}}|_{I_n},\rho^{k,N}(t)\right)\quad\forall t\in I_n. 
	\end{equation}
	Since $\mathfrak{X}_{\Omega_{i,N}}(t)= X^{k,N}(t)|_{\Omega_{i,N}}=\mathcal{X}_{k,N}^i(t)|_{I_n}$ and $\mathfrak{Y}_{\Omega_{i,N}}(t)=Y^{k,N}(t)|_{\Omega_{i,N}}=\mathcal{Y}_{k,N}^i(t)|_{I_n}$, then \eqref{x-discr-eqn} is equivalent to 
	\begin{equation}
		\partial_t\mathcal{I}_+^k\mathcal{X}_{k,N}^i|_{I_n}=-\tau_k^{-1}\left(\mathcal{X}_{k,N}^i(t)|_{I_{n-1}}-\mathcal{X}_{k,N}^i(t)|_{I_n}\right) = D_pH\left(\mathcal{X}_{k,N}^i(t)|_{I_n},\mathcal{Y}_{k,N}^i(t)|_{I_n},\rho^{k,N}(t)\right)
	\end{equation}
	for all $t\in I_n$. Because this holds for all $t\in I_n$, $n=1,\cdots, M_k$, we conclude \eqref{disc-HJ-eqn-equiv}, as claimed. The derivation of \eqref{disc-KFP-eqn-equiv} follows by a similar direct calculation. It is direct to deduce the initial-terminal condition \eqref{disc-equiv-init-term-condition} from the scheme \eqref{numerical-scheme} by using the definition of $X_0^N$ and $\mathcal{R}_{d}(\mathcal{A}_N,\mathcal{O}_N)$. In all, any solution of the scheme \eqref{numerical-scheme} yields a solution to the alternative scheme \eqref{numerical-scheme-equiv}. 
	
	It is then clear from the above calculations that a solution to the first scheme \eqref{numerical-scheme} can be constructed using a solution to the alternative scheme \eqref{numerical-scheme-equiv}, by defining $ X^{k,N}(t) =\sum_{i=1}^{N}\mathcal{X}_{k,N}^i(t)\chi_{\Omega_{i,N}}$ and $ Y^{k,N}(t) =\sum_{i=1}^{N}\mathcal{Y}_{k,N}^i(t)\chi_{\Omega_{i,N}}$ for each $t\in [0,T]$ and noting the identities \eqref{nonlinearity-identity-1}, \eqref{nonlinearity-identity-2} and \eqref{nonlinearity-identity-3}.
\end{proof}

\subsection{Supplementary lemma for second numerical experiment}
{\begin{lemma}\label{lem:app-mon}
The function $g: \mathbb{R}^2\times\mathcal{P}_2(\mathbb{R}^2)\to \R$ defined as $g(x,\mu)\coloneqq \frac{1}{4}|x|^2M_2^2(\mu)$ satisfies the displacement monotonicity condition \eqref{ass-g:disp-mono-cond-g}.
\end{lemma}

\begin{proof}
We have that $D_xg(x,\mu) = \frac12 xM_2^2(\mu),$ from where for any two random variables $X_1,X_2 \in L^2(\Omega;\R^d)$ we have
\begin{align*}
2 \mathbb{E}&\left[\left(D_xg({X}_1,\mathcal{L}_{X_1})-D_xg({X}_2,\mathcal{L}_{X_2})\right)\cdot ({X}_1-{X}_2)\right] = \mathbb{E}\left[ \left( X_1 \mathbb{E}[|X_1|^2] - X_2 \mathbb{E}[|X_2|^2]\right)\cdot ({X}_1-{X}_2) \right]\\
& = \mathbb{E}[|X_1|^2] \mathbb{E}[|X_1 - X_2|^2] + \mathbb{E}[X_2 \cdot ({X}_1-{X}_2)] \mathbb{E}\left[|X_1|^2 - |X_2|^2\right]\\
& =\mathbb{E}[|X_1|^2] \left(\mathbb{E}[|X_1 - X_2|^2]  + \mathbb{E}[X_1\cdot X_2] -\mathbb{E}[|X_2|^2] \right) - \mathbb{E}[|X_2|^2]\left(\mathbb{E}[X_1\cdot X_2] - \mathbb{E}[|X_2|^2] \right)\\
& =\mathbb{E}[|X_1|^2] \left(\mathbb{E}[|X_1|^2 ]  - \mathbb{E}[X_1\cdot X_2]  \right) + \mathbb{E}[|X_2|^2]\left(-\mathbb{E}[X_1\cdot X_2] + \mathbb{E}[|X_2|^2] \right)\\
&\ge \frac{1}{2}\mathbb{E}[|X_1|^2] \left(\mathbb{E}[|X_1|^2 ] -  \mathbb{E}[|X_2|^2 \right) + \frac12\mathbb{E}[|X_2|^2]\left( \mathbb{E}[|X_2|^2] -\mathbb{E}[|X_1|^2] \right) = \frac12 \left(\mathbb{E}[|X_1|^2 ] -  \mathbb{E}[|X_2|^2\right)^2,
\end{align*}
where in the penultimate inequality we have used Young's inequality, i.e. 
$$- \mathbb{E}[X_1\cdot X_2] \ge - \frac12\mathbb{E}[|X_1|^2] - \frac12\mathbb{E}[|X_2|^2].$$
The result follows.
\end{proof}
}

\bibliographystyle{siamplain}
\bibliography{bib_file}

\begin{thebibliography}{10}
	
	\bibitem{achdou2013mean}
	{\sc Y.~Achdou, F.~Camilli, and I.~Capuzzo-Dolcetta}, {\em Mean field games: convergence of a finite difference method}, SIAM J.~Numer.~Anal., 51 (2013), pp.~2585--2612, \url{https://doi.org/10.1137/120882421}.
	
	\bibitem{achdou2010mean}
	{\sc Y.~Achdou and I.~Capuzzo-Dolcetta}, {\em Mean field games: numerical methods}, SIAM J.~Numer.~Anal., 48 (2010), pp.~1136--1162, \url{https://doi.org/10.1137/090758477}.
	
	\bibitem{achdou2020mean}
	{\sc Y.~Achdou, P.~Cardaliaguet, F.~Delarue, A.~Porretta, and F.~Santambrogio}, {\em Mean field games}, vol.~2281 of Lecture Notes in Mathematics, Springer, Cham; Centro Internazionale Matematico Estivo (C.I.M.E.), Florence, 2020, \url{https://doi.org/10.1007/978-3-030-59837-2}.
	
	\bibitem{achdou2016convergence}
	{\sc Y.~Achdou and A.~Porretta}, {\em Convergence of a finite difference scheme to weak solutions of the system of partial differential equations arising in mean field games}, SIAM J.~Numer.~Anal., 54 (2016), pp.~161--186, \url{https://doi.org/10.1137/15M1015455}.
	
	\bibitem{berry2025approximation}
	{\sc J.~Berry, O.~Ley, and F.~J. Silva}, {\em Approximation and perturbations of stable solutions to a stationary mean field game system}, J.~Math.~Pures Appl. (9), 194 (2025), pp.~Paper No. 103666, 28, \url{https://doi.org/10.1016/j.matpur.2025.103666}.

    \bibitem{boissard2014mean}
	{\sc E.~Boissard and T.~Le Gouic}, {\em On the mean speed of convergence of empirical and occupation measures in {W}asserstein distance}, Ann.~Inst.~Henri Poincar\'e{} Probab.~Stat., 50 (2014), pp.~539--563, \url{https://doi.org/10.1214/12-AIHP517}.
	
	\bibitem{bonnans2022error}
	{\sc J.~F. Bonnans, K.~Liu, and L.~Pfeiffer}, {\em Error estimates of a theta-scheme for second-order mean field games}, ESAIM Math.~Model.~Numer.~Anal., 57 (2023), pp.~2493--2528, \url{https://doi.org/10.1051/m2an/2023059}.

    \bibitem{bourne2018semi}
	{\sc D.~P. Bourne, B.~Schmitzer, and B.~Wirth}, {\em Semi-discrete unbalanced optimal transport and quantization}, arXiv preprint arXiv:1808.01962 (2018), \url{https://arxiv.org/abs/1808.01962}.

    \bibitem{brenier1991polar}
	{\sc Y.~Brenier}, {\em Polar factorization and monotone rearrangement of vector-valued functions}, Comm.~Pure Appl.~Math., 44 (1991), pp.~375--417, \url{https://doi.org/10.1002/cpa.3160440402}.

	\bibitem{CarliniSilva14}
	{\sc E.~Carlini and F.~J. Silva}, {\em A fully discrete semi-{L}agrangian scheme for a first order mean field game problem}, SIAM J.~Numer.~Anal., 52
	(2014), pp.~45--67, \url{https://doi.org/10.1137/120902987}.
	
	\bibitem{CarliniSilve15}
	{\sc E.~Carlini and F.~J. Silva}, {\em A semi-{L}agrangian scheme for a degenerate second order mean field game system}, Discrete Contin.~Dyn.~Syst., 35 (2015), pp.~4269--4292, \url{https://doi.org/10.3934/dcds.2015.35.4269}.

    \bibitem{chassagneux2019numerical}
	{\sc J.-F.~Chassagneux, D.~Crisan, and F.~Delarue}, {\em Numerical method for {FBSDE}s of {M}c{K}ean-{V}lasov type}, Ann. Appl. Probab., 29 (2019), pp.~1640--1684, \url{https://doi.org/10.1214/18-AAP1429}.
    
    \bibitem{chevallier2018uniform}
	{\sc J.~Chevallier}, {\em Uniform decomposition of probability measures: quantization, clustering and rate of convergence}, J.~Appl.~Probab., 55 (2018), pp.~1037--1045, \url{https://doi.org/10.1017/jpr.2018.69}.
    
	\bibitem{Ciarlet1978}
	{\sc P.~G. Ciarlet}, {\em The finite element method for elliptic problems}, vol.~40 of Classics in Applied Mathematics, Society for Industrial and Applied Mathematics (SIAM), Philadelphia, PA, 2002, \url{https://doi.org/10.1137/1.9780898719208}.

    \bibitem{du1999centroidal}
	{\sc Q.~Du, V.~Faber, and M.~Gunzburger}, {\em Centroidal {V}oronoi tessellations: applications and algorithms}, SIAM Rev., 41 (1999), pp.~637--676, \url{https://doi.org/10.1137/S0036144599352836}.

    \bibitem{ducasse2020second}
	{\sc R.~Ducasse, G.~Mazanti, and F.~Santambrogio}, {\em Second order local minimal-time {M}ean {F}ield {G}ames}, Nonlinear Differ. Equ. Appl., 29 (2022), \url{https://doi.org/10.1007/s00030-022-00767-2}.
    
    \bibitem{el2022new}
	{\sc R.~El Nmeir, H.~Luschgy, and G.~Pag\`es}, {\em New approach to greedy vector quantization}, Bernoulli, 28 (2022), pp.~424--452, \url{https://doi.org/10.3150/21-bej1350}.

    \bibitem{Emmrich2000DiscreteVO}
	{\sc E.~Emmrich}, {\em Discrete Versions of Gronwall's Lemma and Their Application to the Numerical Analysis of Parabolic Problems}, (2000), \url{https://api.semanticscholar.org/CorpusID:221213518}.

    \bibitem{fournier2015rate}
    {\sc N.~Fournier and A.~Guillin}, {\em On the rate of convergence in Wasserstein distance of the empirical measure}, Probab.~Theory Related Fields, 162 (2015), pp.~707--738, \url{https://link.springer.com/article/10.1007/s00440-014-0583-7}.

    \bibitem{fournier2023convergence}
    {\sc N.~Fournier}, {\em Convergence of the empirical measure in expected wasserstein distance: non-asymptotic explicit bounds in $\mathbb{R}^d$}, ESAIM Probab.~Stat., 27 (2023), pp.~749--775, \url{https://www.esaim-ps.org/articles/ps/abs/2023/01/ps220050/ps220050.html}.
    
	\bibitem{gilbarg2015elliptic}
	{\sc D.~Gilbarg and N.~S. Trudinger}, {\em Elliptic {P}artial {D}ifferential {E}quations of {S}econd {O}rder}, Springer, 2001, \url{https://doi.org/10.1007/978-3-642-61798-0}.
	
	\bibitem{GomesSaude2014}
	{\sc D.~A. Gomes and J.~a. Sa\'{u}de}, {\em Mean field games models---a brief survey}, Dyn.~Games~Appl., 4 (2014), pp.~110--154, \url{https://doi.org/10.1007/s13235-013-0099-2}.
	
	\bibitem{graf2000foundations}
	{\sc S.~Graf and H.~Luschgy}, {\em Foundations of quantization for probability distributions}, Springer-Verlag, 2000, \url{https://doi.org/10.1007/BFb0103945}.
	
	\bibitem{GueantLasryLions2003}
	{\sc O.~Gu\'{e}ant, J.-M. Lasry, and P.-L. Lions}, {\em Mean field games and applications}, in Paris-{P}rinceton {L}ectures on {M}athematical {F}inance
	2010, vol.~2003 of Lecture Notes in Math., Springer, Berlin, 2011, pp.~205--266, \url{https://doi.org/10.1007/978-3-642-14660-2\_3}.
	
	\bibitem{huang2006large}
	{\sc M.~Huang, R.~P. Malham\'e, and P.~E. Caines}, {\em Large population stochastic dynamic games: closed-loop {M}c{K}ean-{V}lasov systems and the {N}ash certainty equivalence principle}, Commun.~Inf.~Syst., 6 (2006), pp.~221--251, \url{https://doi.org/10.4310/cis.2006.v6.n3.a5}.

    \bibitem{jones1964fundamental}
	{\sc G.~S.~Jones}, {\em Fundamental Inequalities for Discrete and Discontinuous Functional Equations}, J. Soc. Indust. Appl. Math., 12 (1964), pp.~43--57, \url{http://www.jstor.org/stable/2946506}.

    \bibitem{kieffer1982exponential}
	{\sc J.~Kieffer}, {\em Exponential rate of convergence for {L}loyd's method. {I}}, IEEE Trans. Inform. Theory, 28 (1982), pp.~205--210, \url{https://doi.org/10.1109/TIT.1982.1056482}.
	
	\bibitem{lasry2006jeuxI}
	{\sc J.-M. Lasry and P.-L. Lions}, {\em Jeux \`a champ moyen. {I}. {L}e cas
		stationnaire}, C.~R.~Math.~Acad.~Sci.~Paris, 343 (2006), pp.~619--625,
	\url{https://doi.org/10.1016/j.crma.2006.09.019}.
	
	\bibitem{lasry2006jeuxII}
	{\sc J.-M. Lasry and P.-L. Lions}, {\em Jeux \`a champ moyen. {II}. {H}orizon fini et contr\^{o}le optimal}, C.~R.~Math.~Acad.~Sci.~Paris, 343 (2006),
	pp.~679--684, \url{https://doi.org/10.1016/j.crma.2006.09.018}.
	
	\bibitem{lasry2007mean}
	{\sc J.-M. Lasry and P.-L. Lions}, {\em Mean field games}, Jpn.~J.~Math., 2 (2007), pp.~229--260, \url{https://doi.org/10.1007/s11537-007-0657-8}.

    \bibitem{lloyd1982least}
	{\sc S.~P. Lloyd}, {\em Least squares quantization in {PCM}}, IEEE Trans. Inform. Theory, 28 (1982), pp.~129--137, \url{https://doi.org/10.1109/TIT.1982.1056489}.

    \bibitem{luschgy2015greedy}
	{\sc H.~Luschgy and G.~Pag\`es}, {\em Greedy vector quantization}, J. Approx. Theory, 198 (2015), pp.~111--131, \url{https://doi.org/10.1016/j.jat.2015.05.005}.
    
    \bibitem{max1960quantizing}
	{\sc J.~Max}, {\em Quantizing for minimum distortion}, Trans. IRE, IT-6 (1960), pp.~7--12,
	\url{https://doi.org/10.1109/tit.1960.1057548}.
    
	\bibitem{meszaros2024mean}
	{\sc A.~R. M\'esz\'aros and C.~Mou}, {\em Mean field games systems under displacement monotonicity}, SIAM J.~Math.~Anal., 56 (2024), pp.~529--553,
	\url{https://doi.org/10.1137/22M1534353}.
	
	\bibitem{osborne2024thesis}
	{\sc Y.~A.~P. Osborne}, {\em Analysis and numerical approximation of mean field game partial differential inclusions}, PhD thesis, University College London, UK, 2024. Available online at \url{https://discovery.ucl.ac.uk/id/eprint/10194241/}.
	
	\bibitem{osborne2022analysis}
	{\sc Y.~A.~P. Osborne and I.~Smears}, {\em Analysis and numerical approximation
		of stationary second-order mean field game partial differential inclusions}, SIAM J.~Numer.~Anal., 62 (2024), pp.~138--166,
	\url{https://doi.org/10.1137/22M1519274}.
	
	\bibitem{osborne2024erratum}
	{\sc Y.~A.~P. Osborne and I.~Smears}, {\em Erratum: Analysis and numerical approximation of stationary second-order mean field game partial differential inclusions}, SIAM J.~Numer.~Anal., 62 (2024), pp.~2415--2417, \url{https://doi.org/10.1137/24M165123X}.
	
	\bibitem{osborne2024regularization}
	{\sc Y.~A.~P. Osborne and I.~Smears}, {\em Regularization of stationary second-order mean field game partial differential inclusions}, SIAM J.~Math.~Anal., 57 (2025), pp.~5189--5215, \url{https://doi.org/10.1137/24M1686401}.
	
	\bibitem{osborne2023finite}
	{\sc Y.~A.~P. Osborne and I.~Smears}, {\em Finite element approximation of time-dependent mean field games with nondifferentiable {H}amiltonians}, Numer.~Math., 157 (2025), pp.~165--211, \url{https://doi.org/10.1007/s00211-024-01447-2}.
	
	\bibitem{osborne2024near}
	{\sc Y.~A.~P. Osborne and I.~Smears}, {\em Near and full quasi-optimality of finite element approximations of stationary second-order mean field games}, Math.~Comp.,  (2025), \url{https://doi.org/10.1090/mcom/4080}.
	
	\bibitem{osborne2025posteriori}
	{\sc Y.~A.~P. Osborne, I.~Smears, and H.~Wells}, {\em A posteriori error bounds for finite element approximations of steady-state mean field games}, IMA J.~Numer.~Anal., (2025), \url{https://doi.org/https://doi.org/10.1093/imanum/draf107}.

    \bibitem{pages2003optimal}
	{\sc G.~Pag\`es and J.~Printems}, {\em Optimal quadratic quantization for numerics: the {G}aussian case}, Monte Carlo Methods Appl., (2003), \url{https://doi.org/10.1163/156939603322663321}.
    
    \bibitem{quattrocchi2024asymptotics}
	{\sc F.~Quattrocchi}, {\em Asymptotics for Optimal Empirical Quantization of Measures}, arXiv preprint arXiv:2408.12924, (2024), \url{https://arxiv.org/abs/2408.12924}.

    \bibitem{sadhu2024stability}
	{\sc R.~Sadhu, Z.~Goldfeld, and K.~Kato}, {\em Stability and statistical inference for semidiscrete optimal transport maps}, Ann.~Appl.~Probab., (2025), \url{https://doi.org/10.1214/24-aap2104}.

    \bibitem{scott1990finite}
	{\sc L.~R. Scott and S.~Zhang}, {\em Finite element interpolation of nonsmooth functions satisfying boundary conditions}, Math. Comp., 54 (1990), pp.~483--493, \url{https://doi.org/10.2307/2008497}.
    
	\bibitem{seeger2025error}
	{\sc B.~Seeger}, {\em Error estimates for deterministic empirical approximations of probability measures}, arXiv preprint arXiv:2510.03451,  (2025), \url{https://arxiv.org/abs/2510.03451v1}.
	
	\bibitem{simon1986compact}
	{\sc J.~Simon}, {\em Compact sets in the space {$L^p(0,T;B)$}}, Ann. Mat. Pura Appl. (4), 146 (1987), pp.~65--96, \url{https://doi.org/10.1007/BF01762360}.

    \bibitem{tacskesen2023semi}
	{\sc B.~Ta\c skesen, S.~Shafieezadeh-Abadeh, and D.~Kuhn}, {\em Semi-discrete optimal transport: hardness, regularization and numerical solution}, Math. Program., 199 (2023), pp.~1033--1106, \url{https://doi.org/10.1007/s10107-022-01856-x}.
	
	\bibitem{zador1964development}
	{\sc P.~L.~Zador}, {\em D{evelopment} {and} {evaluation} {of} {procedures} {for} {quantizing} {multivariate} {distributions}}, PhD thesis, Stanford University, US, 1964. Available online at \url{http://gateway.proquest.com/openurl?url_ver=Z39.88-2004&rft_val_fmt=info:ofi/fmt:kev:mtx:dissertation&res_dat=xri:pqdiss&rft_dat=xri:pqdiss:6409855}.
	
\end{thebibliography}

\end{document}